\documentclass[10pt, reqno]{article}
\usepackage{amsthm}   
\usepackage{amsfonts, amsmath, amscd}
\usepackage{amssymb, latexsym}
\usepackage[all,cmtip]{xy}     
\usepackage{enumerate} 
\usepackage{color}
\usepackage[usenames,dvipsnames]{xcolor}
\usepackage{verbatim}
\usepackage{xcolor}
 
\usepackage{hyperref}
 \hypersetup{colorlinks=false}
 \usepackage{cite}

 \usepackage[T1]{fontenc}

 \usepackage{graphicx, psfrag}
 \usepackage{pstool}

    \usepackage{xspace}   
    \usepackage{tikz-cd}
\usepackage{tikz}
\usetikzlibrary{arrows}

\usepackage[titletoc]{appendix}
\usepackage[normalem]{ulem}     
\usepackage{amsthm}
\usepackage{amsmath}
\usepackage{amscd}
\usepackage[latin2]{inputenc}
\usepackage{t1enc}
\usepackage[mathscr]{eucal}
\usepackage{indentfirst}
\usepackage{graphicx}
\usepackage{graphics}
\usepackage{pict2e}
\usepackage{epic}
\newtheorem{theorem}{Theorem}

\usepackage{amssymb}
\usepackage{amsmath}
\usepackage{amsfonts}
\usepackage{amscd}
\usepackage{amsthm}
\usepackage{graphicx,psfrag}
\usepackage{wrapfig}
\usepackage{subeqnarray}
\usepackage{amssymb}
\usepackage{yfonts}
\usepackage[makeroom]{cancel}
\usepackage{trimclip}
\newif\iflclip
\newif\ifbclip
\newif\ifrclip
\newif\iftclip
\def\CLIP{\dimexpr\fboxrule+.2pt\relax}
\def\nulclip{0pt}
\newcommand\partbox[2]{%
  \lclipfalse\bclipfalse\rclipfalse\tclipfalse%
  \let\lkern\relax\let\rkern\relax%
  \let\lclip\nulclip\let\bclip\nulclip\let\rclip\nulclip\let\tclip\nulclip%
  \parseclip#1\relax\relax%
  \iflclip\def\lkern{\kern\CLIP}\def\lclip{\CLIP}\fi
  \ifbclip\def\bclip{\CLIP}\fi
  \ifrclip\def\rkern{\kern\CLIP}\def\rclip{\CLIP}\fi
  \iftclip\def\tclip{\CLIP}\fi
  \lkern\clipbox{\lclip{} \bclip{} \rclip{} \tclip}{\fbox{#2}}\rkern%
}
\def\parseclip#1#2\relax{%
  \ifx l#1\lcliptrue\else
  \ifx b#1\bcliptrue\else
  \ifx r#1\rcliptrue\else
  \ifx t#1\tcliptrue\else
  \fi\fi\fi\fi
  \ifx\relax#2\relax\else\parseclip#2\relax\fi
}
\usepackage{slashed} 

\theoremstyle{definition}

\theoremstyle{remark}
\newtheorem{remark}[theorem]{Remark}

\numberwithin{equation}{section}

\usepackage{amsmath}
\usepackage{amssymb}
\usepackage{amsfonts}
\usepackage{amsthm}
\usepackage{mathtools}
\usepackage{graphicx}
\usepackage{colonequals}
\usepackage{booktabs}
\usepackage{cancel}
\usepackage[numbers]{natbib}
\usepackage{hyperref}
\usepackage{fancyhdr}
\usepackage{multirow}
\usepackage{multirow}
\usepackage{braket}
\usepackage{tikz}
\usetikzlibrary{arrows,shapes,decorations.pathmorphing}
\usepackage{tikz-cd}
\usepackage[margin=2.9cm]{geometry}
\usepackage{mathabx}
\usepackage{ esint }
\usepackage{slashed}
\usepackage{mathrsfs}
\usepackage{amssymb}
\usepackage{blkarray}
\usepackage{xcolor}
\usepackage{pinlabel}   
\usepackage{scalerel,stackengine}
\stackMath
\newcommand\reallywidehat[1]{%
\savestack{\tmpbox}{\stretchto{%
  \scaleto{%
    \scalerel*[\widthof{\ensuremath{#1}}]{\kern-.6pt\bigwedge\kern-.6pt}%
    {\rule[-\textheight/2]{1ex}{\textheight}}
  }{\textheight}%
}{0.5ex}}%
\stackon[1pt]{#1}{\tmpbox}%
}
\newcommand{\R}{\mathbb{R}}

\newcommand{\Z}{\mathbb{Z}}

\newcommand{\N}{\mathbb{N}}

\newcommand{\C}{\mathbb{C}}

\newcommand{\br}{\langle}
\newcommand{\kt}{\rangle}

\usepackage{graphicx}

\theoremstyle{definition}
\newtheorem{thm}{Theorem}[section]
\newtheorem{prop}[thm]{Proposition}
\newtheorem{lm}[thm]{Lemma}
\newtheorem{defn}[thm]{Definition}
\newtheorem{rem}[thm]{Remark}
\newtheorem{q}[thm]{Question}
\newtheorem{eg}[thm]{Example}
\newtheorem{claim}{Claim}[thm]
\newtheorem{cor}[thm]{Corollary}

\newtheorem{ansatz}[thm]{Ansatz}

\newtheorem{innercustomthm}{Assumption}
\newenvironment{customthm}[1]
  {\renewcommand\theinnercustomthm{#1}\innercustomthm}
  {\endinnercustomthm}

\newtheoremstyle{exercise}{}{}{\itshape}{}{\bfseries}{:}{.5em}{\thmname{#1} \thmnumber{#2}\thmnote{(#3)}}
\theoremstyle{exercise}

\usepackage{amsfonts}
\DeclareFontFamily{U}{wncy}{}
\DeclareFontShape{U}{wncy}{m}{n}{<->wncyr10}{}
\DeclareSymbolFont{mcy}{U}{wncy}{m}{n}
\DeclareMathSymbol{\sha}{\mathord}{mcy}{"58}

\newcommand{\f}[2]{\frac{#1}{#2}}
\renewcommand{\d}[2]
{\frac{d#1}{d#2}}

\newcommand{\p}[2]
{\frac{\partial#1}{\partial#2}}

\newcommand{\nc}{\newcommand}
\nc {\Prod}{(\underset{I}\ph\, r_I^{n_I})}
\nc {\wlw}{\wedge\ldots\wedge}
\nc {\olo}{\otimes\ldots\otimes}
\nc{\bea}{\begin{eqnarray*}}
\nc{\eea}{\end{eqnarray*}}
\nc {\e}{\varepsilon}
\nc{\de}{\delta}
\nc{\A}{\hat a}
\nc{\Ad}{\hat a^\dag}
\nc{\grad}{\nabla}
\nc{\nlim}{\underset{n\to\infty}\lim}
\nc{\ilim}{\underset{i\to\infty}\lim}
\nc{\isum}{\sum\limits_{i=1}^N}
\nc{\xx}{\underset{x\to x_0}\lim}
\nc{\Bnrm}{\Big |\Big|}
\nc{\sym}{\text{Sym}}
\nc{\inte}{\overset{\circ}}
\nc{\del}{\partial}
\nc{\plp}{+\ldots+}
\nc{\lre}{\longrightarrow}
\nc{\be}{\begin{equation}}
\nc{\ee}{\end{equation}}
\nc{\CP}{\mathbb{CP}}
\nc{\End}{\text{End}}
\nc{\mf}{\mathfrak}
\nc{\Hom}{\text{Hom}}
\nc{\spec}{\text{Spec}}
\nc{\sub}{\subseteq}
\nc{\weakto}{\rightharpoonup}
\nc{\ph}{\varphi}
\nc{\leqc}{\lesssim}
\nc{\delbar}{\overline \del}
\nc{\Tr}{\text{Tr}}
\nc{\Lhe}{\mathcal L_{(\Phi^{h_\e}, A^{h_\e})}}

\newcommand{\Yminus}{Y\mathord{-}}

\title{}
\author{Gregory J. Parker\\  }
\date{}

\usetikzlibrary{positioning,arrows,patterns}
\usetikzlibrary{decorations.markings}
\usetikzlibrary{calc}
\usetikzlibrary{arrows}

\tikzset{
  photon/.style={decorate, decoration={snake}, draw=black},
  gluon/.style={decorate, decoration={snake}, draw=black},
  scalar/.style={dotted, draw=black, postaction={decorate},decoration={markings,mark=at position .55 with {\arrow[ultra thick]{>}}}},
  ghost/.style={dotted, draw=black, postaction={decorate},decoration={markings,mark=at position .55 with {\arrow[ultra thick]{stealth}}}},
  fermion/.style={thick, draw=black, postaction={decorate},decoration={markings,mark=at position .55 with {\arrow[ultra thick]{stealth}}}},
  vertex/.style={draw,shape=circle,fill=black,minimum size=3pt,inner sep=0pt},
}

\numberwithin{equation}{section}
\begin{document}
\title{Concentrating Local Solutions of the Two-Spinor Seiberg-Witten Equations on 3-Manifolds}

\maketitle
\begin{abstract} 

Given a compact 3-manifold $Y$ and a $\Z_2$-harmonic spinor $(\mathcal Z_0, A_0,\Phi_0)$ with singular set $\mathcal Z_0$, this article constructs a family of local solutions to the two-spinor Seiberg-Witten equations parameterized by $\e \in (0,\e_0)$ on tubular neighborhoods of $\mathcal Z_0$. These solutions concentrate in the sense that the $L^2$-norm of the curvature near $\mathcal Z_0$ diverges as $\e\to 0$, and after renormalization they converge locally to the original $\Z_2$-harmonic spinor. In a sequel to this article, these model solutions are used in a gluing construction showing that any $\Z_2$-harmonic spinor satisfying some mild assumptions arises as the limit of a family of two-spinor Seiberg-Witten solutions on $Y$.        

\end{abstract}

\setcounter{tocdepth}{2}
\tableofcontents

\maketitle

\section{Introduction}
\label{section1}

Equations of generalized Seiberg-Witten type are conjectured to have deep connections to the geometry and topology of manifolds. Examples include the Vafa-Witten equations \cite{VWOriginalPaper,TT1,TT2} and the Kapustin-Witten equations \cite{KWOriginal, WittenKhovanovGaugeTheory,WittenFivebranesKnots}, which are predicted to connect 3 and 4-dimensional topology to other areas. Another example is the ADHM Seiberg-Witten equations, which are expected to play a key role in Donaldson-Segal's program to construct invariants of manifolds with special holonomy in dimensions 6,7, and 8  \cite{DonaldsonSegal, DWAssociatives, HaydysG2SW}. In contrast to the standard Seiberg-Witten (SW) equations, the moduli spaces of solutions to generalized Seiberg-Witten equations may not be compact, and the lack of well-understood compactifications is one of the main barriers in the study of these equations. 

Pioneering work of Taubes \cite{Taubes3dSL2C, Taubes4dSL2C, TaubesVW, TaubesKWNahmPole, TaubesU1SW, TaubesZeroLoci}, Haydys-Walpuski \cite{HWCompactness}, and Walpuski-Zhang \cite{WZCompactness} has shown that sequences of solutions to generalized SW equations on a manifold $Y$ can diverge, but after renormalization must converge to {\it $\Z_2$-harmonic spinors} or more general types of {\it Fueter sections}---solutions of a different, in general non-linear PDE on $Y\mathrm{-}\mathcal Z$ where $\mathcal Z$ is a codimension 2 singular set. On a compact Riemannian 3-manifold $(Y,g_0)$, the first of these is defined for our purposes as follows. Given an embedded submanifold $\mathcal Z\subset Y$ of dimension 1 and a spinor bundle $S_0\to Y$, fix a real line bundle $\ell \to Y-\mathcal Z$ equipped with its unique flat connection $A_0$ with holonomy in $\Z_2$. A {\bf $\Z_2$-harmonic spinor} consists of a triple $(\mathcal Z, A_0,\Phi)$ where $\Phi \in \Gamma(S_0\otimes_\R \ell)$ is a spinor satisfying \be \slashed D_{A_0}\Phi=0  \ \ \text{ on } \ \ \Yminus\mathcal Z_0 \hspace{1.5cm} \text{and} \hspace{1.5cm} \int_{Y-\mathcal Z} |\nabla_{A_0}\Phi|^2 < \infty,\label{Z2harmonicpreliminary}\ee
where $\slashed D_{A_0}$ denotes the Dirac operator twisted by the connection $A_0$ on $\ell$. Said more simply, a $\Z_2$-harmonic spinor is a harmonic spinor on the open manifold $Y\mathrm{-}\mathcal Z$ in a spin structure which does not necessarily extend over $\mathcal Z$ and whose covariant derivative is $L^2$. 

The above convergence results suggest that the moduli spaces of solutions to a specific generalized Seiberg-Witten equation should admit natural compactifications obtained by including $\Z_2$-harmonic spinors or, more generally, the type of Fueter section arising for that equation as boundary strata. Constructing these compactifications requires addressing the converse to the convergence question: 

\begin{q}{\it Which $\Z_2$-harmonic spinors or Fueter sections arise as the limit of a sequence of Seiberg-Witten solutions}?
\label{question1}
\end{q} 
\noindent In terms of PDE, this question is a gluing problem, requiring the patching together of solutions of the two different equations. When the singular set $\mathcal Z=\emptyset$ is empty, this gluing problem was solved in a general setting by Doan-Walpuski \cite{DWDeformations}, though in this case the ``gluing'' is straightforward because no model solutions are required. Indeed, they instead refer to this case as a ``deformation'' problem. In the situation that $\mathcal Z\neq \emptyset $, which is a stable condition under perturbations \cite{DWExistence,RyosukeThesis,PartII}, the gluing problem requires model solutions near the singular set.  
 
 The present work provides two crucial steps towards solving the gluing problem in the case of the two-spinor Seiberg-Witten equations on a closed 3-manifold. In this case, the $\Z_2$-harmonic spinors that arise are as defined in (\refeq{Z2harmonicpreliminary}). Given a $\Z_2$-harmonic spinor $(\mathcal Z_0, A_0,\Phi_0)$ satisfying some mild assumptions, the first step accomplished in this article is to construct a 1-parameter family of model solutions to the two-spinor Seiberg-Witten equations in a neighborhood of the singular set $\mathcal Z_0$ which converge locally to $(A_0,\Phi_0)$. The second step is to analyze the linearized equations at this family of model solutions whose limiting linearization at $(A_0,\Phi_0)$ is a degenerate elliptic operator whose symbol vanishes along $\mathcal Z_0$.    
In a sequel to this article \cite{PartIII}, these model solutions are used in a gluing construction which gives affirmative answer to Question \ref{question1} in this setting.   
 \subsection{Main Results}

To state the main results, let us first describe the set-up briefly. Additional details are given in Section \ref{section2}. Let $(Y,g_0)$ denote a closed, oriented Riemannian 3-manifold, and let $S\to Y$ be the spinor bundle associated to a $\text{Spin}^c$ structure. Furthermore, let $E\to Y$ be a rank 2 complex vector bundle with structure group $SU(2)$ endowed with a fixed connection $B_0$. The {\bf two-spinor Seiberg-Witten equations} are the following system of equations for a pair $(\Psi,A)\in \Gamma(S\otimes_\C E) \times \mathcal A_{U(1)}$ of an $E$-valued spinor and a $U(1)$-connection lifted from $\text{det}(S)$: 
   
\begin{eqnarray}
\slashed D_{A}\Psi&=&0 \label{SW1intro}\\ 
\star F_A +\tfrac{1}{2}\mu(\Psi, \Psi)& =& 0
\label{SW2intro}
\end{eqnarray}
\smallskip

\noindent where $\slashed D_{A}$ is the Dirac operator on $S\otimes E$ twisted by $A$ and the fixed connection $B_0$ on $E$, $F_A$ is the curvature of $A$, and $\mu$ a point-wise quadratic map. The equations are invariant under $U(1)$-gauge transformations.

As mentioned above, there may be sequences $(\Psi_i, A_i)$ of solutions to (\refeq{SW1intro})-(\refeq{SW2intro}) that have no convergent sub-sequences modulo gauge. For such sequences, a straightforward argument using the Weitzenb\"ock formula \cite[Prop 1.16]{WZCompactness} shows that the $L^2$ norm $\|\Psi_i\|_{L^2}\to \infty$ must diverge. To highlight the role of the $L^2$ norm for these sequences, one can renormalize the spinor by setting 
    
    $$ \e:= \frac{1}{\|\Psi\|_{L^2}} \hspace{3cm} \Phi:= \e \Psi $$ 

\noindent and instead consider the equations (\refeq{SW1intro})-(\refeq{SW2intro}) for the pair $(\tfrac{\Phi}{\e}, A)$. The results of Haydys-Walpuski in \cite{HWCompactness} show that if the sequence $(\Psi_i, A_i)$ has no subsequences for which $\|\Psi_i\|$ remains bounded, then the renormalized sequence $(\Phi_i, A_i)$ converges subsequentially to a $\Z_2$-harmonic spinor as $\e\to 0$ modulo gauge transformations on the complement of the singular set $\mathcal Z$ (see Theorem \ref{compactness} in Section \ref{section2.1} for a precise statement). At present, it is not known that the singular set of a $\Z_2$-harmonic spinor arising in this way  necessarily has more regularity than being a closed, rectifiable subset of Hausdorff codimension 2. We do not attempt to address these regularity issues here, and consider only that case that $\mathcal Z$ is a smooth embedded submanifold.

Reversing the convergence statement to address the gluing question, let $(\mathcal Z_0, A_0, \Phi_0)$ be a $\Z_2$-harmonic spinor on $(Y,g_0)$ with respect to a perturbation induced by $B_0$. We assume that it satisfies the following.

\begin{defn} A $\Z_2$-harmonic spinor $(\mathcal Z_0, A_0, \Phi_0)$ is said to be {\bf regular} if it obeys the following three assumptions. 
\begin{customthm}{1}\label{assumption1}(Smoothness) the singular set $\mathcal Z_0\subseteq Y$ is a smooth, embedded link. \label{assumption1} 
\end{customthm}

\begin{customthm}{2}\label{assumption2} (Non-degeneracy)
the spinor $\Phi_0$ has non-vanishing leading-order, i.e. there is a constant $c_1$ such that $$|\Phi_0| \geq c_1\text{dist}(-,\mathcal Z_0)^{1/2}.$$  \label{assumption2} 
\end{customthm}
\begin{customthm}{3}\label{assumption3} (Isolated)
$\Phi_0$ is the unique $\Z_2$-harmonic spinor with respect to $(\mathcal Z_0, A_0,g_0,B_0)$ up to scaling and sign.   \label{assumption3} 
\end{customthm}

\end{defn}

We remark that analysis of the local polyhomogeneous expansion of $\Phi_0$ near $\mathcal Z_0$ (cf. Definition \ref{polyhomogeneous} and \cite{MazzeoEdgeOperators}) shows that non-degeneracy (Assumption \ref{assumption2}) implies that $A_0$ has holonomy $-1$ around the meridian of each component of $\mathcal Z_0$. By the Riemann-Hilbert correspondence, the set of flat connections with holonomy in $\Z_2$ is in bijection with the flat real line bundles $\ell \to \Yminus\mathcal Z_0$, thus with  $\text{{C}ech}$ cocycles in $H^1(\Yminus\mathcal Z_0; \Z_2)$. From this perspective, non-degeneracy implies that the corresponding line bundle $\ell$ restricts to the M\"obius bundle on every small disk transverse to $\mathcal Z_0$, and in particular does not extend to $Y$.

\noindent

\bigskip 

The main result is the following construction of model solutions: 

\begin{thm} \label{main}  Given a regular $\Z_2$-harmonic spinor $(\mathcal Z_0, A_0,\Phi_0)$ as in (\refeq{Z2harmonicpreliminary}) and an orientation on $\mathcal Z_0$, there exists a $\text{Spin}^c$ structure $S$ on $Y$ such that the following hold.

\begin{enumerate}

\item[{ (i)}]  $S$ extends $S_0\otimes_\R \ell$ to $Y$ in the sense that $S |_{\Yminus\mathcal Z_0}\simeq  S_0\otimes_\R \ell$, and $S$ is determined by a complex line bundle $\mathcal L\to Y$ such that $$ S=S_0\otimes_\C \mathcal L \hspace{1.5cm} \text{ and } \hspace{1.5cm} c_1(\det S)= -\text{PD}[\mathcal Z_0].$$

Additionally, $\Phi_0$ is naturally a section of a rank 4 subbundle of $(S\otimes_\C E)|_{\Yminus\mathcal Z_0}$.

\item[ (ii)] For the $\text{Spin}^c$ structure $S$, there is an $\e_0>0$ such that for $\e\in (0, \e_0)$ there exist model solutions $(\Phi_\e, A_\e)$ on the tubular neighborhood $N_{\lambda}(\mathcal Z_0)$ of radius $\lambda=\tfrac{1}{2}\e^{1/2}$ satisfying the two-spinor Seiberg-Witten equations 
\begin{eqnarray}
\slashed D_{A_\e}\Phi_\e &=& 0 \\
\star F_{A_\e} + \tfrac{1}{2}\tfrac{\mu(\Phi_\e,\Phi_\e)}{\e^2} &=&0
\end{eqnarray} 

\noindent with respect to the restrictions of $(g_0, B_0)$ to $N_\lambda(\mathcal Z_0)$. 

\item[{(iii)}] $\tfrac{\Phi_\e}{\e}$ extends via a cut-off function to a smooth section of $S\otimes_\C E$ on $Y$ that is equal to $\tfrac{\Phi_0}{\e}$ away from $N_{\lambda/2}(\mathcal Z_0)$ and has $L^2$ norm $\tfrac{1}{\e}+ O(\e^{-1/4})$ on $Y$.

\end{enumerate}

\label{main}
\end{thm}

\bigskip 
\smallskip 

The second main result shows that these model solutions approach $(\Phi_0, A_0)$ as $\e\to 0$. This follows from applying the main results of \cite{ConcentratingDirac} to the model solutions constructed in Theorem \ref{main}.

\begin{cor}
\label{mainb}

The model solutions $(\Phi_\e, A_\e)$ converge to $(\Phi_0, A_0)$ in the following sense. Fix a family of compact subsets  $K_\e \Subset  N_{\lambda}(\mathcal Z_0)-\mathcal Z_0$ such that $\text{dist}(K_\e, \mathcal Z_0) \geq c_1 \e^{2/3}$ for a positive constant $c_1$. Then there are constants $C,c$ independent of $\e$ such that the un-renormalized difference \be (\ph_\e, a_\e)=\left(\frac{\Phi_\e}{\e}, A_\e\right)-\left(\frac{\Phi_0}{\e}, A_0\right)\label{unrenormalizeddiff}\ee obeys the following properties.

\begin{enumerate}
\item[{ (i)}] There is a half-dimensional subbundle $S^\text{Im}\subseteq S\otimes_\C E$ such that the components of the spinor in $S^\text{Im}$ and the connection decay to $(\tfrac{\Phi_0}{\e},A_0)$ exponentially on $K_\e$.  That is,
$$\|(\ph^\text{Im}_\e, a_\e)\|_{C^0(K_\e)} \ \leq \  \frac{C}{|\text{dist}(K_\e, \mathcal Z)|^{3/2}\e}\ \text{Exp}\left(-\frac{c}{\e}\text{dist}(K_\e,\mathcal Z)^{3/2}\right).$$

\item[{ (ii)}] The remaining spinor components decay to $(\tfrac{\Phi_0}{\e}, A_0)$ like  $\text{dist}(-,\mathcal Z_0)^{-\nu}$ for any $0<\nu<\frac{1}{4}$; i.e. there is an $\e$-independent constant $\gamma'<<1$ such that 

$$\|\text{dist}(-,\mathcal Z_0)^\nu\ph_\e\|_{L^{1,2}(K_\e)} \ \leq \ C_\nu\, \e^{1/12-\gamma'}.$$ 
\end{enumerate}
\end{cor}

\noindent Note that $\text{dist}(-,\mathcal Z_0)$ is a function on $Y$, whereas $\text{dist}(K_\e,\mathcal Z)= \inf_{y\in K_\e} \text{dist}(y,\mathcal Z_0)$ is a constant depending on $\e$. In particular, if $\text{dist}(K_\e,\mathcal Z_0)\geq \e^{2/3-\gamma_1}$ for some $\gamma_1>0$, then $(\ph^\text{Im}_\e,a_\e)$ decays to 0 faster than any polynomial on $K_\e$ as $\e\to 0$. 

\bigskip 

The third and final main result is about the linearization of the Seiberg-Witten equations at the model solutions. Since in the eventual gluing construction these model solutions are pasted onto the manifold using a cut-off function to form global approximate solutions on $Y$, the statement of the orem is given for these. Let $\chi(r)$ denote a cut-off function supported on $N_{\lambda}(\mathcal Z_0)$ equal to 1 for radii $r\leq \lambda/2$. The {\bf Approximate solutions} are defined as 

 \be \left(\frac{\Phi^\text{\text{App}}_\e}{\e}, A^\text{\text{App}}_\e\right):= \left(\frac{\Phi_0}{\e}, A_0\right) + \chi (\ph_\e, a_\e)\label{approx},\ee

 \noindent where $(\ph_\e, a_\e)$ is the un-renormalized difference defined in  \refeq{unrenormalizeddiff}. 
 
 Let $\mathcal L_\e$ denote the extended, gauge-fixed linearized Seiberg-Witten equations at these approximate solutions, which are defined precisely in Section \ref{section2}. This linear equation is viewed as a first-order boundary value problem on a tubular neighborhood $N_{2\lambda}(\mathcal Z_0)$ with $\lambda$ as in Theorem \ref{main}(ii) by introducing a Hilbert space $H$ and a projection \be \Pi^{\mathcal L}: L^{1,2}(N_{2\lambda}(\mathcal Z_0))\to H \label{bdorthogonalityconstraints}\ee
    
\noindent so that $\ker(\Pi^\mathcal L)$ is the subspace of sections satisfying certain boundary and orthogonality conditions. The precise definitions of $H$ and  $\Pi^\mathcal L$ are given in Section \ref{section7}. The statement also references certain weighted norms $\| -\|_{H^1_{\e,\nu}}$ and $\|-\|_{L^2_{\e, \nu}}$ defined in Section \ref{section5}; these are equivalent (not uniformly) to the standard norms on $L^{1,2}(N_{2\lambda}(\mathcal Z_0))$ and $L^{2}(N_{2\lambda}(\mathcal Z_0))$ respectively.

 \begin{thm}
\label{mainc}
Subject to the boundary and orthogonality conditions defined by \refeq{bdorthogonalityconstraints}, the extended gauge-fixed linearization of the two-spinor Seiberg-Witten equations at the approximate solutions \refeq{approx}  

\be \mathcal L_\e: \ker(\Pi^\mathcal L)\subseteq L^{1,2}(N_{2\lambda}(\mathcal Z_0)) \lre L^2(N_{2\lambda}(\mathcal Z_0))\ee

\noindent is Fredholm of Index 0. Additionally, there is an $\e_0>0$ such that for $\e<\e_0$,  $\mathcal L_\e$ is invertible, and there are positive constants $C,\gamma^\text{in}\!<\!<1$ independent of $\e$ such that  the bound

\smallskip
\be \|(\ph,a)\|_{H^1_{\e,\nu}} \leq  \frac{C_\nu}{\e^{1/12+\gamma^\text{in}}} \  \|\mathcal L_{\e}(\ph,a)\|_{L^2_{\e,\nu}}\label{Linvertibleboundmainc}\ee
\smallskip
\noindent holds.
\end{thm}

The model solutions constructed in part (ii) of Theorem \ref{main} will sometimes be referred to as ``fiducial solutions''. This terminology is taken from \cite{MWWW} where fiducial solutions of a similar nature were found for Hitchin's equations (see Section \ref{Section1.2} and Section \ref{section4} for further discussion). The model solutions of Theorem \ref{main} actually solve the {\it extended} Seiberg-Witten equations as defined in Section 2.4. That is, they include an auxiliary 0-form component $a_0$.  We make several remarks contextualizing the main results above. 

\begin{rem} The proof of Theorem \ref{main} relies only on the local form of $(\mathcal Z_0, A_0, \Phi_0)$, thus holds equally well in the case that $(Y,g_0)$ are non-compact provided that $\mathcal Z_0$ is compact and $\Phi_0$, and $\nabla \Phi_0 \in L^2_{loc}(Y)$. The compactness results for convergence to $\Z_2$-harmonic spinors \cite{HWCompactness, WZCompactness} assume that $Y$ is compact. 
\end{rem}

\begin{rem}It is conjectured that Assumption \ref{assumption1} holds generically within the set of $(g_0, B_0)$ that admit $\Z_2$-harmonic spinors. The genericity of the embedding condition is the subject of ongoing work by other authors \cite{MazzeoHaydysTakahashi}. This and other questions on the regularity of the singular set $\mathcal Z_0$ involve significant detours into geometric measure theory (see \cite{TaubesZeroLoci, ZhangRectifiability, Haydysblowup}) and are beyond the scope of the present article. This assumption could readily be weakened (e.g. $\mathcal Z_0$ is an $L^{k,2}$-embedding for  $k=3$), but the required analysis would distract the main goals of the present article.    
\end{rem}

\begin{rem}  The non-degeneracy Assumption \ref{assumption2} is expected to hold generically following similar results for $\Z_2$-harmonic 1-forms (\cite{SiqiZ3}). The results of \cite{PartII} show that non-degeneracy stable under perturbations of the metric and background connection $B_0$. It is likely that the non-degeneracy assumption can be weakened to construct model solutions for $\Z_2$-harmonic spinors with high-order vanishing along $\mathcal Z_0$. The results of \cite{PartII, PartIII}, however, do not extend to this case without significant alteration. Indeed, the leading order term $b(t,\theta)$ such that $\Phi_0=b(t,\theta)r^{1/2} + O(r^{3/2})$ where $(r,t,\theta)$ are geodesic normal coordinates on $N\mathcal Z_0$ (see Section \ref{section2.5.2}) appears as the symbol of an elliptic operator used in the gluing construction. The gluing construction relies heavily on this ellipticity. 

The isolated Asumption \ref{assumption3} is likewise expected to hold generically. The results of \cite{PartII} again show this property is stable under perturbations of the metric and background connection. 
\end{rem}

\begin{rem}  There is some freedom in the choice of $\lambda=\tfrac12 \e^{1/2}$ for the size of the neighborhood in Theorem \ref{main}(ii). It will be shown in Section \ref{section6} that dilation by $\e^{2/3}$ produces an invariant scale for the fiducial solutions. For the gluing construction of \cite{PartIII} to succeed, the model solutions must be constructed on a tubular neighborhood of radius large than the invariant scale and such that the decay of Corollary \ref{mainb} can be invoked. This requires $O(\e^{2/3+\delta})< \lambda < O(1)$ for some $\delta>0$. With sufficient alternation of the definition of $\Pi^\mathcal L$ of the boundary conditions, model solutions could be constructed on neighborhoods with other choices of $\lambda$. If $\lambda=\e^\beta$, the coefficient in the bound in Theorem \ref{mainc} is of size  $O(\e^{-(2-3\beta)/6-\gamma})$ for $\gamma<<1$. 

There are explicit counterexamples to a uniform bound in Theorem \ref{mainc}, which show this estimate is sharp without significant alterations of the function spaces used. 
\end{rem}

\subsection{Motivation for Approach}\label{Section1.2}

This section briefly motivates and summarizes the approach taken to the proofs of Theorem \ref{main} and Theorem \ref{mainc}.  

\medskip

\noindent { \it 1.2.1. Degenerating Linearizations}\label{sectiondegenerating}
\smallskip

The gluing problem for $\Z_2$-harmonic spinors does not fit into the standard framework used in many other gluing problems. These differences are due to the existence of the singular set $\mathcal Z_0$, near which the equations degenerate and standard elliptic theory breaks down. Indeed, under the assumption that $\mathcal Z_0=\emptyset$, Doan-Walpuski  \cite{DWDeformations} solved the gluing problem in great generality using standard elliptic theory. Unfortunately, none of their approach extends to the case that $\mathcal Z\neq \emptyset$.

To be more precise, the standard elliptic theory breaks down in the following way. The linearized Seiberg-Witten equations

  \be \mathcal L_{(\Phi_0, A_0)}: L^{1,2}(\Yminus\mathcal Z_0) \lre L^2(\Yminus\mathcal Z_0) \label{deglinearization}\ee 
  
  \vspace{.25cm} 
 \noindent at a $\Z_2$-harmonic spinor with $\mathcal Z_0\neq \emptyset$ are a degenerate elliptic system whose symbol vanishes along $\mathcal Z_0$. Operators with this type of degeneracy are known as {\bf elliptic ``edge'' operators}, and are well-studied in microlocal analysis \cite{MazzeoEdgeOperators, MazzeoEdgeOperatorsII,Melrosebcalculus,grieser2001basics}. For the edge operator $\mathcal L_{(\Phi_0, A_0)}$, there is no natural choice of function spaces on which it is  Fredholm;  in particular, (\refeq{deglinearization}) has an infinite-dimensional cokernel. For any family of model solutions $(\Phi_\e, A_\e)$, the resulting family of linearized equations 
 
 \be \mathcal L_{(\Phi_\e, A_\e)}\overset{\e\to 0} \lre \mathcal L_{(\Phi_0, A_0)}\label{deglinearizations2}\ee
 
   \vspace{.25cm} 
\noindent  is converging to this limiting operator with infinite-dimensional cokernel (in no precise sense, as the function spaces change in the limit). As a result,  one {\it cannot} expect the linearizations to be uniformly invertible in any reasonable sense. 
 
 The consequences of this are two-fold. In the present article, this manifests in the difficulty of proving Theorem \ref{mainc}, where the subspace limiting to the infinite-dimensional cokernel ruins any naive approach. The proof unavoidably requires delicate analysis of the degenerating family (\refeq{deglinearizations2}), which is carried out in Sections \ref{section6}-\ref{section7}. The second consequence is for the eventual gluing: even with Theorem \ref{mainc} in hand, the gluing problem still appears at first to have an infinite-dimensional obstruction coming from the cokernel of \ref{deglinearization}. This is addressed in \cite{PartIII} by considering deformations of the singular set $\mathcal Z_0$, which requires the study of the infinite-dimensional family of operators $\mathcal L_{\e}(\mathcal Z)$ parameterized by nearby singular sets $\mathcal Z$, though no more is said about this issue here.   \bigskip 
  
\noindent {\it 1.2.2. Relation to Limiting Configurations}
\smallskip

The gluing problem for $\Z_2$-harmonic spinors is effectively a generalization of the gluing problem that arises at the boundary of the moduli space of solutions to Hitchin's equations on a Riemann surface $\Sigma$, and this observation guides parts of our approach.

 The boundary objects in the Hitchin moduli space, known as {\it limiting configurations}, are singular Higgs fields whose  singular set $\mathcal Z_\Sigma \subset \Sigma$   is a finite collection of points.  Given a limiting configuration $\Phi_0$ and a singular point $z\in\mathcal Z_\Sigma$,  one makes   the ansatz that there are local model solutions which differ from $\Phi_0$ by a complex-gauge transformation $h_\e(r)$ which depends only on the distance $r$ from $z$. That is, are locally  of the form
\be \Phi_\e= e^{h_\e(r)}\cdot \Phi_0.\label{ansatzintro}\ee

\noindent This leads to an $\e$-parameterized family of ODEs for $h_\e(r)$ that can be solved to yield model solutions. These are then spliced onto $\Sigma$ and corrected to true solutions using methods that exploit the 
holomorphic structure of Hitchin's equations to circumvent the problem of the degenerating linearization (see \cite{MWWW,FredricksonSLnC} for details).

The relation of this case with the gluing problem for $\Z_2$-harmonic spinors is, essentially, that it is a dimensional reduction. More precisely, the gluing problem in the case of Hitchin's equations is the dimensional reduction of the gluing problem for the closely related (though more difficult) Kapustin-Witten equations. For a 3-manifold $Y=S^1\times \Sigma$, the limiting configurations at the boundary of the Hitchin moduli space on $\Sigma$ can be lifted to $\Z_2$-harmonic 1-forms (which are $\Z_2$-harmonic spinors for the Dirac-type operator $(d+d^\star)$) that are invariant in the $S^1$ direction. The singular set is lifted to $\mathcal Z= S^1 \times \mathcal Z_\Sigma$. Up to some minor differences between the equations, the three-dimensional gluing problem for the two-spinor Seiberg-Witten equations can be viewed as a generalizing of the construction for Hitchin's equation to the non-$S^1$-invariant case. 

Unfortunately, for the case of a ${\Bbb Z}_2$-harmonic spinor on a general  3-manifold $Y$,  the lack of a holomorphic structure on  $Y$ means virtually none of the techniques used for Hitchin's equations are applicable. First, there is no analogue of the holomorphic structure which can be exploited to circumvent the problem of the degenerating family of linearized operators, and this problem must be confronted. Even disregarding this issue with the linearization, there are several critical issues in extending the 2-dimensional approach to find 3-dimensional model solutions. For one, the holomorphic structure on $\Sigma$ allows one to choose local coordinates putting $\Phi_0$ in a standard form, thus the 2-dimensional model solution is unique up to coordinate change. In contrast, in three-dimensions the local form of $\Phi_0$ lies in an infinite-dimensional space of possibilities. Secondly, one cannot make an effective simplifying ansatz akin to \refeq{ansatzintro}. In addition to having to upgrade the ODE for $h_\e(r)$ to a PDE for $h_\e(r,t)$ depending also on $t$ the tangential coordinate $\mathcal Z$, the lack of a holomorphic structure, means the number of equations also increases. In combination, these features mean there is no analogue of the ansatz \refeq{ansatzintro} that will lead to a system of PDEs near $\mathcal Z$ that is meaningfully simpler than the full Seiberg-Witten equations. 

Despite these differences, our approach still relies heavily on a very close analogue of the two-dimensional model ODE solutions, as we now explain. 

\bigskip 

\noindent {\it 1.2.3. Our Approach}
\smallskip 

Given the above, one must abandon the hope of finding explicit model solutions and instead turn to abstract methods. One reliable abstract method is the Implicit Function Theorem (IFT), and in fact, as explained momentarily, any other method would be redundant. Our use of the IFT here relies on the following observation: although the solution to the local PDE near the singular set cannot be found explicitly, its leading order term must be given by the $t$-parameterized family of 2-dimensional model solutions on the normal planes. These are not solutions, and in fact the error from being a solution does not approach 0 in $L^2$ as $\e\to 0$. Yet, surprisingly, it is sufficiently small that with the correctly weighted function spaces the IFT can correct these to true model solutions. Of course, applying the IFT requires analyzing the linearization at these, which has the same shortcoming as described above: this family of linearizations degenerates to an operator which is not Fredholm. 

Thus our approach produces model solutions in two steps. The first is to first introduce a $t$-parameterized family of 2-dimensional model solutions which smoothes the $\Z_2$-harmonic spinor to a ``de-singularized'' pair $(\Phi^{h_\e}, A^{h_\e})$. After this, an application of the IFT corrects them to the desired 3-dimensional model solutions, proving Theorem \ref{main}. With these model solutions in hand, one forms approximate solutions on the closed manifold by introducing a cut-off function, and the global gluing argument proceeds from there. Schematically, the steps of the gluing are

\begin{center}
\tikzset{node distance=3.5cm, auto}
\begin{tikzpicture}[decoration=snake]
\node(A){$(\Phi_0,A_0)$};
\node(C)[right of=A]{$(\Phi^{h_\e}, A^{h_\e})$};
\node(D)[right of=C]{$(\Phi_\e^{\text{mod}}, A_\e^{\text{mod}})$};
\node(E)[right of=D]{$(\Phi_\e^{\text{app}}, A_\e^{\text{app}})$};
\node(F)[right of=E]{$(\Phi_\e, A_\e).$};
\draw[->, decorate] (A) to node {$\text{de-sing.}$} (C);
\draw[->, decorate] (C) to node {$\text{correct}$} (D);
\draw[->, decorate] (D) to node {$\text{cutoff}$} (E);
\draw[->, decorate] (E) to node {$\begin{matrix} \text{gluing} \\ \text{iteration}\end{matrix}$} (F);
\end{tikzpicture}

\end{center}
\noindent where the first two steps are accomplished in the present work, and the last two relegated to the sequel \cite{PartIII} as explained in the introduction. 

A key advantage of this approach is that it proves Theorems  \ref{main} and \ref{mainc} simultaneously.  The proofs both rely on the study of the degenerating family of linearizations at the de-singularized pair $(\Phi^{h_\e}, A^{h_\e})$. This study extends across Sections~\ref{section4}--\ref{section7} and culminates in Theorem~\ref{invertibleL}, which describes  the invertibility of this family of operators. Theorems \ref{main} and \ref{mainc} then follow immediately from Theorem~\ref{invertibleL}.   While other approaches to Theorem~1.2 (such as finding more explicit local solutions) might be possible, any such approach would be redundant, since Theorem~\ref{invertibleL} is needed anyway to establish Theorem~\ref{mainc}.

\subsection{Outline}

Section 2 introduces background material and provides an overlay of technical statements on what was said in the introduction. Section 2.1 gives the precise statement of the convergence theorem of Haydys-Walpuski. Section 2.2 give relevant linear algebra constructions, and Section 2.3 gives a more precise definition of $\Z_2$-harmonic spinors which arises from a version of the Haydys correspondence. Section 2.4 states the Weitzenb\"ock formula for the linearized equations, which is used later.  

\medskip

Section 3 covers some basic properties of the singular Dirac operator \refeq{Z2harmonicpreliminary}. Section 3.1 covers its semi-Fredholm properties, and Section 3.2 establishes local forms for $\Z_2$-harmonic spinors which are of key importance. Relying on these local forms, Section 3.3 is devoted to the topological question of how to reconstruct the $\text{Spin}^c$ structure in  part (i) of Theorem \ref{main}.  
 
 \medskip 
 

\medskip 

Section 4 constructs the de-singularized configurations and estimates their failure to be true solutions. Section 4.1 reviews the dimensionally reduced problem, which is essentially identical to the corresponding problem for Hitchin's equations found in \cite{MWWW}. Section 4.2 extends these to the parameterized ODE case, and Section 4.3 contains the error calculation. 

\medskip 

Section 5 begins the analysis of the linearized equations at the de-singularized configurations. Section 5.1 defines the relevant function spaces. Section 5.2 defines a model operator given by the situation where metric near $\mathcal Z_0$ is Euclidean. In Section 5.3 it is shown that the $\e$-parameterized family of model operators on the planes normal to $\mathcal Z$ are all re-scaling of a single $\e$-invariant operator $\mathcal N$ at the invariant scale $O(\e^{2/3})$. 

\medskip 

Section 6 begins the bulk of the technical analysis by studying the scale-invariant normal operator $\mathcal N$. This operator can be understood via complex geometry, and viewing it via this lens makes certain properties manifest. In this section, it is found that $\mathcal N$ naturally has a two (real) dimensional kernel that cannot be perturbed away, despite the fact that the Seiberg-Witten equations on a compact 3-manifold are index 0. This kernel is the first manifestation of the infinite-dimensional cokernel that arises as $\e\to 0$. Section 6.1 and Section 6.2 provide background and review the relevant standard Fredholm theory. Sections 6.3-6.5 study the normal operator $\mathcal N$ in its holomorphic guise, and Section 6.6 provides details on the aforementioned two-dimensional kernel. 

\medskip 

Section 7 generalizes the results of the previous section to the 3-dimensional case. This follows essentially from integration by parts and the observation that all the tangential derivatives along $\mathcal Z$ are comparatively mild. In this section, the kernel of the normal operator which is isomorphic to $\C$ is upgraded to a high-dimensional subspace which approaches $L^2(\mathcal Z;\C)$ as $\e\to 0$ to become the infinite-dimensional cokernel in the limit.  To make the integration by parts work involves setting up the quite intricate collection of boundary and projection conditions $\Pi^\mathcal L$ (cf. Theorem \ref{mainc}) for the linearized operator, which accounts for this section's length despite the simplicity of the underlying idea. Section 7.1 reviews some standard results about APS boundary conditions for Dirac operators. Section 7.2 discusses the high-dimensional subspace approaching the limiting cokernel, and Section 7.3 sets up the boundary conditions accounting for this. Sections 7.4 and 7.5 then carry out the integration by parts argument, which by that point becomes rather involved. Section 7.6 generalizes to the case of an arbitrary metric near $\mathcal Z$.    

\medskip 

Section 8 concludes the proofs of Theorems \ref{main} and \ref{mainc}, which after the analysis of the linearization in Sections 5-7 are essentially immediate. The Appendices cover some calculations that would disrupt the flow of the rest of the article.

\section*{Acknowledgements}

This article constitutes a portion of the author's Ph.D. thesis. The author is grateful to his advisors Clifford Taubes and Tomasz Mrowka for their insights and suggestions. The author would also like to thank Aleksander Doan, Andriy Haydys, Rafe Mazzeo, Rohil Prasad, and Thomas Walpuski for helpful discussions. The author is supported by a National Science Foundation Graduate Research Fellowship and by National Science Foundation Grant No. 2105512.    
\section{$\Z_2$-Harmonic Spinors and Compactness}

\label{section2}

Let $(Y, g_0)$ be a closed, oriented Riemannian 3-manifold. Choose a $\text{Spin}^c$ structure $\frak s$ on $Y$, and let $S\to Y$ be the associated spinor bundle. We denote Clifford multiplication by $\gamma: \Omega^1(Y)\to \text{End}(S)$. Since every 3-manifold is spin, we may alternatively specify a $\text{Spin}^c$ structure by choosing a spin structure $\frak s_0$ with spinor bundle $S_0$ and taking $S=S_0\otimes_\C \mathcal L$ where $\mathcal L$ is a complex line bundle. In this second description, the $\text{Spin}^c$ structure obtained depends on the choice of $\frak s_0$.   

The two-spinor Seiberg-Witten equations are an extension of the standard Seiberg-Witten equations (\cite{MorganSW, KM}) that instead consider spinors valued in two (possibly twisted) copies of $S$. Let $E\to Y$ denote an $SU(2)$-bundle equipped with a fixed smooth background connection $B_0$. Define $$S_E:= S\otimes_\C E,$$

\noindent and denote by $\br \ph, \psi\kt$ the real inner-product on sections of $S_E$ arising from the Hermitian inner-products on $S$ and $E$. Pairs $(\Psi, A) \in \Gamma(S_E)\otimes \mathcal A_{\mathcal L}$ consisting of a spinor in $S_E$ and a $U(1)$ connection on $\mathcal L$ are called {\bf configurations}. 

\begin{defn} The {\bf Two-Spinor Seiberg-Witten Equations} for configurations $(\Psi, A)$ are 

\begin{eqnarray}
\slashed D_{A}\Psi&=&0 \label{SW1}\\ 
\star F_A +\tfrac{1}{2}\mu(\Psi, \Psi)& =& 0
\label{SW2}
\end{eqnarray}

\noindent where $\slashed D_{A}$ is the Dirac operator on $S_E$ formed using the Spin connection on $S_0$, the background connection $B_0$ on $E$, and the connection $A$ on $\mathcal L$, and $\tfrac{1}{2}\mu$ is a pointwise quadratic {\bf moment map}. These equations are invariant under the action of the {\bf gauge group} $\mathcal G= \text{Maps}(Y; U(1))$.   
\end{defn}

\bigskip 

The moment map $\tfrac{1}{2}\mu: S_E\to \Omega^1(i\R)$ is given in a local orthonormal coframe $e^j$ by $$\frac{1}{2}\mu(\Psi, \Psi)=\sum_{j=1}^3\frac{i}{2} \br \gamma(ie^j)\Psi, \Psi\kt e^j.$$ 

\noindent In a local trivialization $E|_U\simeq \C^2\times U$, we may write $\Psi=(\Psi_1, \Psi_2)$ as a pair of spinors in $S$ in which case $\mu(\Psi,\Psi)= \mu_\circ(\Psi_1, \Psi_1) + \mu_\circ(\Psi_2, \Psi_2)$, where $\mu_\circ$ is the moment map in the standard Seiberg-Witten equations. 

\medskip

  \subsection{Compactness Theorem}
\label{section2.1} 

It is a well-known fact  that the moduli space of solutions to the standard Seiberg-Witten equations modulo the action of the gauge group $\mathcal G$ is compact (\cite{MorganSW}, Chapter 5). The proof of this relies on the pointwise equality   

\be \br\gamma( \mu(\Psi, \Psi)) \Psi, \Psi\kt= \tfrac{1}{4}|\Psi|^4,\label{mubound}\ee

\noindent which via the Weitzenb\"ock formula for $\slashed D_{A}$ leads to an {\it a priori} bound $$\|\Psi\|^2_{L^2}\leq \int_{Y} |s| \ dV$$ for the spinor component of solutions, where $s$ is the scalar curvature of $g$. Starting with this, the proof of compactness is a standard application of elliptic theory (\cite{MorganSW}, Sections 5.2-5.3). 

For the case of the two-spinor Seiberg-Witten equations (\refeq{SW1})-(\refeq{SW2}), there are non-zero spinors for which $\mu(\Psi,\Psi)=0$, thus no bound akin to (\refeq{mubound}) can hold. The consequence is that for the two-spinor Seiberg-Witten equations
$$\text{\it There may be sequences of solutions $(\Psi_i, A_i)$ such that } \|\Psi_i\|_{L^2}\to \infty.$$

\noindent Note that this $L^2$-norm is a gauge-invariant quantity. To understand the behavior of such sequences of solutions, one considers renormalizing by dividing by the $L^2$-norm. Equivalently, we ``blow-up'' the space of configurations by adding the sphere are infinity in $L^2(Y; S_E)$. 

Thus consider re-normalized spinors to replace configurations $(\Psi, A)$ with {\bf blown-up configurations} $(\Phi, A, \e)$ by setting $$ \Phi= \e \Psi \hspace{1cm}\text{ where }\hspace{1cm} \e = \frac{1}{\|\Psi\|_{L^2}}. $$

\begin{defn}
The {\bf blown-up Seiberg-Witten Equations} for a blown-up configuration $(\Phi, A,\e)\in \Gamma(S_E)\times \mathcal A(\mathcal L) \times [0, \infty)$ are 
\begin{eqnarray}
\slashed D_{A}\Phi&=&0 \label{SWBlownup1}\\ 
\star \e^2F_A +\tfrac{1}{2} \mu(\Phi, \Phi)& =& 0 \\
\label{SWBlownup2}
\|\Phi\|_{L^2}&=&1. 
\label{SWBlownup3}
\end{eqnarray}
As before, these equations are invariant under the action $\mathcal G=\text{Maps}(Y;U(1))$. Solutions with $\e\neq 0$ are solutions of the original equations (\refeq{SW1})-(\refeq{SW2}) where $\|\Psi\|_{L^2}=\tfrac{1}{\e}$.  
\label{blownupdef}
\end{defn}

\bigskip 
The upcoming theorem, due to Haydys-Walpuski \cite{HWCompactness} and building on the work of Taubes in \cite{Taubes3dSL2C}, describes the limiting behavior of sequences of solutions for which the $L^2$-norm diverges. Additional regularity results were proved by Taubes \cite{TaubesZeroLoci}, and Zhang \cite{ZhangRectifiability};  a more general approach to the original result was later given by Walpuski-Zhang in \cite{WZCompactness}. The precise statement of the orem is rather intricate, and it merits preliminary explanation.

One would naively expect that a sequence of solutions $(\Phi_i, A_i, \e_i)$ with $\e_i\to 0$ would converge to a solution of  (\refeq{SWBlownup1})-(\refeq{SWBlownup3}) with $\e=0$, i.e. a pair $(\Phi_0, A_0)$ solving 

\medskip 
 \be \slashed D_{A_0}\Phi_0=0     \ \ \ \  \ \text{ such that } \ \ \ \  \  \Phi_0 \in \mu^{-1}(0) \ \ \ \ \  \text{ and } \ \ \ \  \|\Phi_0\|_{L^2}=1.\label{e=0equation}\ee

\medskip 
\noindent A version of this statement is true, but there are several caveats. 

The first caveat arises from the fact that $\mu^{-1}(0)$ is not fiberwise a manifold; instead, it is singular at  the point $0\in \mu^{-1}(0)$; it is therefore unclear what it means for $\Phi_0$ to solve the equation $\slashed D_{A_0}\Phi_0=0$ at the singular locus $|\Phi|^{-1}(0)$. The second difficulty is describing the limiting process for the connection, since it no longer appears in the limiting $\e=0$ equations  (\refeq{e=0equation}). It turns out that the connection converges to a well-defined limit away from a second singular locus around which the energy density $|F_A|^2$ concentrates and becomes unbounded. The coupling of the equations dictates, however, that this concentration may only occur where the spinor hits the singular point of $\mu^{-1}(0)$ and {\it these two singular loci therefore coincide}. Consequently, the statement of the  convergence theorem makes reference to a singular set $\mathcal Z_0$ which plays the dual role of 
\vspace{.2cm}
\begin{enumerate}
\item The set of $y\in Y$ for which the limiting spinor  hits the singularity, i.e. has $\Phi_0(y)=0 \in \mu^{-1}(0)$.
\item The set of $y\in Y$ away from which $|F_{A_i}|^2$ remains bounded.   
\end{enumerate}

\vspace{.2cm}

The following theorem makes these ideas precise. The statement given here combines the result of Haydys-Walpuski, Taubes, and Zhang referenced above.

\begin{thm}{\bf (Haydys-Walpuski \cite{HWCompactness}, Taubes \cite{TaubesZeroLoci}, Zhang \cite{ZhangRectifiability})} \label{compactness} Let $(\Phi_i, A_i, \e_i) \in \Gamma(S_E)\times \mathcal A(\mathcal L)\times (0, \infty)$ denote a sequence of solutions to the blown-up Seiberg-Witten equations $$\slashed D_{A_i}\Phi_i=0 \hspace{1cm} \star \e_i^2 F_{A_i}+ \tfrac{1}{2}\mu(\Phi_i,\Phi_i)=0 \hspace{1cm} \|\Phi_i\|_{L^2}=1 $$ with respect to a sequence of converging metrics  $g_i\to g_0$ on $Y$ and connections $B_i\to B_0$ on $E$. Then, either 

\begin{enumerate}
\item[(i)] If $\limsup \e_i>0$, then $(\Phi_i, A_i, \e_i)$ converges subsequentially modulo gauge to a solution with $\e>0$. 
\smallskip 

\noindent \hspace{-1.1cm} OR

\smallskip

\item[(ii)] If $\limsup \e_i=0$, there exists a triple $(\mathcal Z_0,\Phi_0, A_0)$ where  
\begin{itemize}
\item $\mathcal Z_0 \subseteq Y$ is a closed rectifiable subset of Haudorff codimension at least 2. 
\item $\Phi_0$ is a spinor on $\Yminus\mathcal Z_0$ such that $|\Phi_0|$ extends as a continuous function to $Y$ with $\mathcal Z_0=|\Phi_0|^{-1}(0)$. 
\item $A_0$ is a flat connection on $\mathcal L|_{\Yminus\mathcal Z_0}$ with holonomy in $\Z_2$, 
\end{itemize}
\noindent such that $(\Phi_0, A_0)$ satisfies the $\e=0$ version of the blown-up Seiberg-Witten equations \refeq{e=0equation} on $\Yminus\mathcal Z_0$ with respect to the metric $g_0$ and the connection $B_0$ on $E$. Furthermore, there is an $\alpha>0$ such that and after passing to a subsequence and up to gauge transformations defined on $\Yminus\mathcal Z_0$, 
 
 \be
 \Phi_i \overset{L^{2,2}_{loc}}\lre \Phi_0 \hspace{1cm} A_i \overset{L^{1,2}_{loc}}\lre A_0 \hspace{1cm} |\Phi_i| \overset{C^{0,\alpha}}\lre |\Phi_0|  \label{convergencetolim}
 \ee 
 where local convergence means on compact subsets of $\Yminus\mathcal Z_0$. \end{enumerate}

\qed

\end{thm}

As we will see in the next two subsections, the data of case (ii) is equivalent to that of a $\Z_2$-harmonic spinor. The main result of \cite{ConcentratingDirac} shows that the convergence (\refeq{convergencetolim}) in this Theorem is $C^\infty_{loc}$ on $\Yminus\mathcal Z_0$. 

\begin{rem}
Although $\Phi_0$ is a section of a bundle of dimension $>2$, solutions of the equations (\refeq{e=0equation}) are topologically constrained and do not behave generically. The stability of a singular set $\mathcal Z_0$ of Hausdorff codimension 2 follows from the main results of \cite{DWExistence, PartII, RyosukeThesis}.
\end{rem}

\vspace{5mm}

\subsection{The Hyperk\"ahler Quotient}
\label{section2.2} 

This section explicitly identifies $\mu^{-1}(0)$ in fibers of $S_E$ and gives important linear algebra constructions (see also \cite{DWExistence}, Appendix A).  

Consider the vector space $$V=\C^2\otimes_\C \mathbb H$$ 

\noindent equipped with its real inner product as a model for the fibers of $S_E$. It carries a pointwise action of $U(1)$ via the first factor, and a Clifford multiplication $\gamma: \Lambda^1(\R^3) \to \End(V)$ given by  $$\gamma(dt)= \begin{pmatrix} i & 0 \\ 0 & -i \end{pmatrix} \otimes Id\hspace{1cm}\gamma(dx)= \begin{pmatrix} 0 & -1 \\ 1 & 0 \end{pmatrix} \otimes Id\hspace{1cm}\gamma(dy)= \begin{pmatrix} 0 & i \\ i & 0  \end{pmatrix}\otimes Id. $$

\noindent where $\R^3$ is given coordinates $(t,x,y)$. A pointwise spinor $\Phi\in V$ may be written in the form \be\Phi= \begin{pmatrix} \alpha_1 \\ \beta_1 \end{pmatrix}\otimes 1 \ + \ \begin{pmatrix} \alpha_2 \\ \beta_2 \end{pmatrix}\otimes j.\label{Phiform}\ee

\noindent In this form the pointwise moment map is given by 

\begin{eqnarray}
\frac{1}{2}\mu(\Phi,\Phi) & = & \frac{i}{2}(|\beta_1|^2 + |\beta_2|^2 - |\alpha_1|^2 -|\alpha_2|^2) \ dt    \ \  \label{mut}\\
& +&   \frac{i}{2}\text{Re}( -\overline \alpha_1 \beta_1 - \overline \alpha_2 \beta_2) \ dx\label{mux}\\ 
& +&   \frac{i}{2}\text{Im}( -\overline \alpha_1 \beta_1 - \overline \alpha_2 \beta_2) \ \label{muy}dy.   
\end{eqnarray}

\noindent Notice that the sign convention here differs from many authors since we have written the Seiberg-Witten equations as $\star F_A + \mu=0$ rather than $\star F_A = \mu$. It is easy to check the under the identification $(\R^3)^* = \text{Im}\mathbb H$ given by $dt \mapsto I, dx\mapsto J, dy\mapsto K$, the map $\tfrac{1}{2}\mu$ is indeed the hyperk\"ahler moment map associated to the $U(1)$ action, justifying the name. 
\noindent 

We can identify $V\simeq \text{End}(\C^2; \C^2)$ so that \refeq{Phiform} is written as the matrix \be\Phi_0= \begin{pmatrix} \alpha_1 & \alpha_2  \\ \beta_1 & \beta_2 \end{pmatrix}.\label{matrixPhi1}\ee

\begin{lm} Under the above isomorphism, $$\mu^{-1}(0) \ \simeq \ \text{Cone}(U(2)).$$

\noindent In particular, it is a smooth 5-dimensional manifold away from $0\in V$. 

\end{lm}

\begin{proof}
In terms of the matrix (\refeq{matrixPhi1}), the second and third moment map equations  (\refeq{mux})-(\refeq{muy}) show that the columns are orthogonal in the Hermitian metric, and the first equation (\refeq{mut}) requires that the rows have the same norm. Thus the matrix is a possibly 0 multiple of a unitary matrix.  
\end{proof}

Next, we establish the form of the hyperk\"ahler quotient orbifold $\mu^{-1}(0)/U(1)$. To do this, we construct slices for the $U(1)$ action. We will show, in fact, that there is a global slice for the action up to a stabilizer of $\Z_2=\{\pm 1\}$. To begin, each factor of $V=\C^2 \otimes_\C \mathbb H$ carries a complex anti-linear involution, denoted by $J: \C^2 \to \C^2$ and $j:\mathbb H\to \mathbb H$ respectively, such that $J^2 = j^2=-1$. Explicitly, these are 

\be
J\begin{pmatrix}\alpha \\ \beta \end{pmatrix}: = \begin{pmatrix}-\overline \beta \\ \overline \alpha \end{pmatrix}\hspace{3cm} j(q):= qj.
\label{Jdef}
\ee

\noindent Together these give rise to a real structure $\tau:V\to V$ satisfying $\tau^2=Id$ given by $$\tau:= J\otimes j.$$
We denote by $V^\text{Re}$ and $V^\text{Im}$ the $+1$ and $-1$ eigenspaces of $\tau$ respectively. So that 
\begin{eqnarray}
V^\text{Re}&=& \{ \psi\otimes 1 + J\psi \otimes j \ | \ \psi \in \C^2 \}\label{SRe} \\ 
V^\text{Im}&=& \{\psi\otimes 1 - J\psi \otimes j \ | \ \psi \in \C^2 \}.
\label{SIm}
\end{eqnarray}

\begin{lm}\label{slicelemma}
The subspace $V^\text{Re}$ provides a global slice for the $U(1)$ action up to a $\Z_2$-stabilizer. That is, $$V^\text{Re}\subset \mu^{-1}(0)$$
and each $U(1)$-orbit intersects $V^\text{Re}$ in two point which differ by multiplication by $-1$. Consequently, the hyperk\"ahler quotient is given by $$\mu^{-1}(0)/U(1) \simeq \mathbb H/\Z_2.$$
\end{lm}

\begin{proof}
A pointwise spinor $\Phi \in V^\text{Re}$ has the form \be \Phi= \begin{pmatrix}\alpha_1 \\ \beta_1 \end{pmatrix}\otimes 1 +  \begin{pmatrix}-\overline \beta_1 \\ \overline \alpha_1  \end{pmatrix}\otimes j \label{Sre}\ee

\noindent which automatically satisfies (\refeq{mut})-(\refeq{muy}). Moreover, if $e^{i\phi}\Phi\in V^\text{Re}$ is another element in the same $U(1)$ orbit in $V^\text{Re}$ it must also be of the form (\refeq{Sre}). Since $$e^{i\phi}\Phi= e^{i\phi}\begin{pmatrix}\alpha_1 \\ \beta_1 \end{pmatrix}\otimes 1 +e^{i\phi}  \begin{pmatrix}-\overline \beta_1 \\ \overline \alpha_1  \end{pmatrix}\otimes j =\begin{pmatrix}e^{i\phi}\alpha_1 \\ e^{i\phi}\beta_1 \end{pmatrix}\otimes 1 +e^{2i\phi}  \begin{pmatrix}-\overline {e^{i\phi}\beta_1} \\ \overline {e^{i\phi}\alpha_1}  \end{pmatrix}\otimes j$$
\noindent this implies $e^{2i\phi}=1$ hence $e^{i\ph}=\pm 1$.  
\end{proof}
\noindent 

\medskip

To complete this section we note one more simple lemma, which is central to many arguments in this paper.  Let \be V=V^\text{Re}\oplus V^\text{Im}\label{Vsplitting}\ee

\noindent denote the decomposition from (\refeq{SRe})-(\refeq{SIm}). The linearization of $\mu$ at $\Phi$ is given by its polarization, which we denote $\mu(-,\Phi)$ . We may extend this to a map $\mu(-,\Phi): V\to (\Lambda^0\oplus \Lambda^1)(i\R)$ by redefining \be \mu(\Psi,\Phi) \Rightarrow (-i\br i\Psi  ,  \Phi\kt \ , \ \mu(\Psi, \Phi))\label{extendedmu}\ee

\noindent where the previous definition now constitutes the 1-form component. Notice polarizing cancels the factor of $\tfrac{1}{2}$. Similarly, we extend Clifford multiplication to $(\Lambda^0\oplus \Lambda^1)(\R^3)$ by scalar multiplication in the first factor.

\begin{lm}
\label{splittinglem}
The following statements hold: 

\begin{enumerate}
\item[(A)]  Clifford multiplication by real and purely imaginary forms $$\gamma: \Lambda^0(\R)\oplus \Lambda^1(\R) \to \text{End}(V) \hspace{2cm}\gamma: \Lambda(i\R)\oplus \Lambda^1(i\R) \to \text{End}(V)$$
respectively preserve and reverse the splitting (\refeq{Vsplitting}).

\item[(B)] If $\Phi \in V^\text{Re}$ is non-zero, then $$V^\text{Re}=\{\gamma(b)\Phi \ | \ b \in \Lambda^0(\R)\oplus \Lambda^1(\R)\} \hspace{1cm} \text{and}\hspace{1cm} V^\text{Im}=\{\gamma(a)\Phi \ | \ a \in \Lambda^0(i\R)\oplus \Lambda^1(i\R)\}$$

\item[(C)] If $\Phi \in V^\text{Re}$ then $$\ker(\mu(-,\Phi))= V^\text{Im}$$
\noindent and the reverse for $\Phi\in V^\text{Im}$.  
\end{enumerate}
\end{lm}

\begin{proof} For (A), simply note that Clifford multiplication by real forms commutes with $J$, hence with $\tau$ and preserves the splitting. For purely-imaginary forms, it anti-commutes by the anti-linearity of $J$. This implies (B) since $\gamma(e^j)\Phi$ is orthogonal to $\gamma(e^k)\Phi$ in the real inner product for $j\neq k$ ranging over $j,k=0,1,2,3$. (C) follows in turn from (B) since $\mu(-,\Phi)$ and $\gamma(-)\Phi$ are adjoints.    
\end{proof}

\medskip 

\subsection{The Haydys Gauge}
\label{section2.3}

In the $\e=0$ limit of
the blown-up Seiberg-Witten equations (\refeq{SWBlownup1})-(\refeq{SWBlownup2}) the variables are $\Phi_0$ and $A_0$, but $A_0$ no longer satisfies an elliptic equation. There is a special choice of gauge, however, that effectively eliminates $A_0$ as a variable. This allows the limiting $\e=0$ equation to be reinterpreted as an equation for only $\Phi_0$ that is elliptic on $\Yminus\mathcal Z_0$ (though not uniformly so). This gauge is a key part of the Haydys correspondence (\cite{HaydysCorrespondence,DWExistence}) in the case that $\mathcal Z_0=\emptyset$.

Since the limiting connection $A_0$ in Theorem \ref{compactness} is flat with holonomy in $\Z_2$, it follows that \be (\mathcal L|_{\Yminus\mathcal Z_0})^2 \simeq \underline \C\label{L^2triv}\ee is trivial. Indeed, it carries a flat connection whose holonomy is trivial, this being the one induced by $A_0$. It follows also that $\mathcal L|_{\Yminus\mathcal Z_0}$ admits a reduction of structure group to a real line bundle. More precisely, 

\begin{lm} Suppose $\mathcal L|_{\Yminus\mathcal Z_0}$ admits a flat connection $A_0$ with holonomy in $\Z_2$. Then 

\begin{enumerate}
\item[(A)] The first chern class $c_1(\mathcal L|_{\Yminus\mathcal Z_0})$ is 2-torsion. 
\item[(B)] There exists a real line bundle $\ell\to \Yminus\mathcal Z_0$ such that $$\mathcal L|_{\Yminus\mathcal Z_0}\simeq \ell \otimes_\R \underline \C.$$
\item[(C)]The set of gauge equivalence classes of connections $A_0$ on $\mathcal L|_{\Yminus\mathcal Z_0}$ is a torsor on the kernel of the integral Bockstein homomorphism $$\beta: H^1(\Yminus\mathcal Z_0; \Z_2)\to H^2(\Yminus\mathcal Z_0; \Z).$$
\end{enumerate}
\label{reallinebundles}
\end{lm}

\begin{proof}
The short exact sequence $\Z\overset{\times 2}\to \Z\to \Z_2$ induces the long exact sequence 
$$\ldots \lre H^1(\Yminus\mathcal Z_0;\Z)\lre H^1(\Yminus\mathcal Z_0;\Z_2)\overset{\beta}\lre H^2(\Yminus\mathcal Z_0;\Z)\overset{\times 2}\lre H^2(\Yminus\mathcal Z_0;\Z)\lre \ldots$$
and \refeq{L^2triv} shows that $c_1(\mathcal L|_{\Yminus\mathcal Z_0})$ is in the kernel of $\times 2$, which is (A). Exactness implies $c_1(\mathcal L|_{\Yminus\mathcal Z_0})$ the image of a class in $H^1(\Yminus\mathcal Z_0;\Z_2)$. For (B)-(C), note that flat  connections with holonomy in $\Z_2$ up to gauge (on $\Yminus\mathcal Z_0$) are in one-to-one correspondence with $\Z_2$-valued representations in $Hom(\pi_1(\Yminus\mathcal Z_0); \Z_2)\simeq H^1(\Yminus\mathcal Z_0;\Z_2)$, thus with real line bundles $\ell$ via $w_1(\ell)=\text{hol}_{A_0}$ under this isomorphism.  
The complex line bundle whose first chern class is $c_1(\mathcal L|_{\Yminus\mathcal Z_0})=\beta(w_1(\ell))$ is simply $\mathcal L|_{\Yminus\mathcal Z_0} =\ell \otimes \C$ (this is the $\check{\text{C}}\text{ech}$ description of $\beta$). (B) therefore holds by exactness, and the set of flat $A_0$ with $\Z_2$-holonomy on a given isomorphism class of complex line bundle is in one-to-one correspondence with the fiber of $\beta$ over its first chern class, which gives (C).

\end{proof}

Given the above lemma, we may fix an isomorphism \be\sigma: \mathcal L|_{\Yminus\mathcal Z_0}\ \simeq \  \ell \otimes_\R \underline\C,\label{sigmadef}\ee

\noindent where $\ell$ is the real line bundle specified by the holonomy representation of $A_0$. Such a choice is only determined up to gauge transformations on $\Yminus\mathcal Z_0$, as composing with a gauge transformation (now thought of as acting on the $\underline \C$ factor) gives another such choice. A choice of $\sigma$ gives a reduction of structure group of the  spinor bundle $S=S_0\otimes_\C(\ell \otimes_\R \C)$ from $\text{Spin}^c$ to $SU(2)$. The auxiliary bundle $E$ has structure group $SU(2)$ by definition, hence there are global versions of the maps \be J: S|_{\Yminus\mathcal Z_0}\to S|_{\Yminus\mathcal Z_0} \hspace{2cm} j: E\to E\label{Jdefglobal}\ee
from the previous subsection (which depend on the choice of $\sigma$). They are given in local trivializations by the same expressions \refeq{Jdef} which commute with the action $SU(2)$. It follows that there is a global splitting \be S_E|_{\Yminus\mathcal Z_0}= S^\text{Re}\oplus S^\text{Im}\label{ReImSplitting}\ee 

\noindent which is  determined by $\sigma$. Thus this splitting is specified up to gauge transformations on $\Yminus\mathcal Z_0$, and the gauge equivalence classes are given by the gauge equivalence classes of $A_0$ as in Lemma \ref{reallinebundles}. 

The following lemma gives the special gauge choice advertised at the beginning of the section. 
\begin{lm}\label{Haydysgauge} Suppose that $(\mathcal Z_0, A_0, \Phi_0)$ is a triple as in Theorem \ref{compactness}. There exists a choice of a gauge $u\in \mathcal G|_{\Yminus\mathcal Z_0}$ such that the following equivalent conditions hold: 

\begin{enumerate}
\item[{\bf (1)}] After replacing the isomorphism $\sigma$ by $u\circ \sigma$, $$\sigma: \mathcal L|_{\Yminus\mathcal Z_0} \ \simeq  \ \ell \otimes_\R \underline \C$$ sends \be A_0 \ \ \mapsto \nabla^\text{flat}\otimes 1 + 1\otimes \text{d}\label{A_0def}\ee
where $\nabla^\text{flat}$ is the unique flat connection with holonomy in $\Z_2$ on $\ell$.  
\item[{\bf (2)}] In the splitting \refeq{ReImSplitting} determined by the new $\sigma$, one has $$\Phi_0 \in \Gamma(S^\text{Re}).$$
\end{enumerate}

\noindent This choice of gauge is referred to as the {\bf Haydys gauge}. It is unique up to the action of $\Z_2\subseteq \mathcal G|_{\Yminus\mathcal Z_0}$. 
\end{lm}

\begin{proof}
We construct $u$ so that {\bf (2)} holds and show this implies {\bf (1)}. Let $\{U_\alpha\}$ be a finite open cover of $\Yminus\mathcal Z_0$ obtained by restricting a finite open cover on the compact $Y$ to the subspace topology. We may assume that on each $U_\alpha$ there is a trivialization $$g_\alpha \times f_\alpha: (S_{0}\otimes_\C  \mathcal L \otimes_\C E)|_{U_\alpha}\simeq U_\alpha \times (\C^2 \otimes_\C \C \otimes_\C \mathbb H) $$

\noindent with transition functions \bea g_{\alpha\beta}: U_\alpha\cap U_\beta & \to& SU(2)\times \{1\}\times  SU(2)) \\ f_{\alpha \beta}& \to & \{1\}\times U(1)\times \{1\}\eea

\noindent for $S_0\otimes E$ and $\mathcal L$ respectively. We may assume that $g_{\alpha\beta}$ are induced by trivializations that extend to $Y$. 

Let $\nabla$ denote the connection induced by the spin connection and $B$ and the product connection $\text{d}$ on the middle factor. Thus on $U_\alpha$ each we may write $$\nabla_{A_0}= \nabla + \gamma(i a_\alpha) $$
for a connection form $a_\alpha\in \Omega^1(U_\alpha; \R)$. Locally, in each trivialization there is a real structure given by $$\tau_\alpha:= J \otimes 1 \otimes j$$

\noindent where $J,j$ are as in \refeq{Jdef}. $\tau_\alpha$ do not {\it a priori} give rise to a global real structure $\tau$ as the transition functions $g_{\alpha\beta}f_{\alpha\beta}$ do not respect the $SU(2)$ structure.

By Lemma \refeq{slicelemma}, we may choose local gauge transformations $u_\alpha$ on each $U_\alpha$ unique up to a $\Z_2$ factor such that \be u_\alpha\Phi_0|_{U_\alpha}\in \text{Re}(\C^2\otimes_\C \C\otimes_\C  \mathbb H)\label{inslice}\ee
since $\Phi_0\in \mu^{-1}{(0)}$ In the new trivializations $u_{\alpha} f_\alpha: \mathcal L|_{\Yminus\mathcal Z_0}\to U_\alpha\times \C $, the transition functions $$f_{\alpha\beta}'= u_\alpha^{-1} u_\beta f_{\alpha\beta}$$ preserve the condition (\refeq{inslice}) thus by Lemma  \refeq{slicelemma}, we must have that $f_{\alpha\beta}'\in \{\pm 1\}.$ Let $\ell$ be the real line bundle determined by $f_{\alpha\beta}': U_\alpha \cap U_{\beta}\to \Z_2$. Since $f_{\alpha\beta}$ and $f_{\alpha\beta}'$ differ by a $\check{\text{C}}\text{ech}$ coboundary, we obtain an isomorphism $\sigma':  \mathcal L|_{\Yminus\mathcal Z_0}\simeq \ell\otimes_\R \underline{\C}$, and the gauge transformations $u_\alpha$ patch to form a global gauge transformation $\sigma' \circ u = \sigma$. This yields {\bf (2)}. Moreover, in this gauge $J,j$ are respected by the transition functions, hence $\tau_\alpha= \tau|_{U_\alpha}$ agrees with the global structure defined by (\refeq{Jdefglobal}) using the trivialization $\sigma'$.

Now we show that in this trivialization $\sigma'$, item {\bf (1)} holds. In the local trivialization on each $U_\alpha$ write $\nabla$ to be connection formed from $\nabla^\text{Spin}$ on $S_0$ and $B_0$ on $E$. We claim that in this trivialization, $u_\alpha A_0= \nabla$, i.e. $u_{\alpha}\cdot  (\text{d}+ia_\alpha)=\text{d}$ is the product connection on the $\C$ factor. To see this, write the Dirac equation $$\slashed D_{A_0}\Phi_0=\left(\sum_{j=1}^3\gamma(e^j) \nabla_j + \gamma(iu_\alpha\cdot a_\alpha)\right)  u_\alpha\Phi_0=0$$

\noindent and $\nabla, \gamma$ preserve $\text{Re}(\C^2\otimes_\C \C\otimes_\C \mathbb H)$, while $\gamma(ia_\alpha)$ exchanges it with $\text{Im}(\C^2\otimes_\C \mathbb H)\otimes \C$ by part (A of Lemma \refeq{splittinglem} and the fact that $\nabla$ is an $SU(2)\times SU(2)$ connection hence respects $J,j$. It follows that   

\bea
\sum_{j=1}^3\gamma(e^j) \nabla_j   (u_{\alpha}\Phi_0) &\in& \text{Re}(\C^2\otimes_\C \mathbb H)\otimes \C\\
\gamma(iu_\alpha \cdot a_\alpha)  u_{\alpha}\Phi_0 &\in& \text{Im}(\C^2\otimes_\C \mathbb H)\otimes \C\\
\eea
must individually vanish, implying $a_\alpha=0$ since $\Phi_0(y)\neq 0$ for $y\in \Yminus\mathcal Z_0$. Thus in the trivializations $f_{\alpha}'$, $A_0$ is the product connection, so globally it patches to the connection $\nabla^\text{flat}\otimes 1 + 1\otimes \text{d}$ on $\ell\otimes_\R \underline \C$ in the trivialization $\sigma'$. This shows {\bf (1)}. In fact, since there is always a unique gauge transformation up to constants so that {\bf (1)} holds, the two statements are equivalent up to constant gauge transformations.



\end{proof}

From now on, we fix the association  (\refeq{sigmadef}) to be one of the two determined by the Haydys gauge defined by the previous lemma. This choice subsequently fixes the splitting (\refeq{ReImSplitting}).

\begin{lm} The splitting $S_E=S^\text{Re} \oplus S^\text{Im}$ determined by the Haydys gauge satisfies the following. 

\begin{enumerate}
\item[(A)] The conclusions of Lemma \ref{splittinglem} hold globally. 
\item[(B)] The splitting is parallel with respect to $\nabla_{A_0}$. In particular, the Dirac operator splits as $$\slashed D_{A_0}^\text{Re}: \Gamma(S^\text{Re})\to \Gamma(S^\text{Re}) \hspace{3cm}\slashed D_{A_0}^\text{Im}: \Gamma(S^\text{Im})\to \Gamma(S^\text{Im})$$
\end{enumerate}
\end{lm}
\begin{proof} (A) is immediate from the pointwise version Lemma \ref{splittinglem}. For (B), note that in the Haydys gauge of Lemma \ref{Haydysgauge}, the connection formed from $A_0$ and the spin connection respects the $SU(2)$ structure, hence commutes with $J$. The connection $B_0$ on $E$ is an $SU(2)$ connection hence automatically commutes with $j$, thus $\nabla_{A_0}$ commutes with $\tau=J\otimes j$.  
\end{proof}

\bigskip 

With the above preparation, we may give a more precise definition of $\Z_2$-harmonic spinors which refines Definition \ref{Z2harmonicpreliminary} in the introduction.  The upcoming Proposition \ref{asymptoticexpansion} in Section 3 implies that the continuous extension in item (iii) is equivalent to the integrability condition in \ref{Z2harmonicpreliminary}. 

\begin{defn}\label{Z2harmonicdef} Let $\mathcal Z_0\subseteq Y$ be a smooth, embedded link. Fix a real line bundle $\ell\to \Yminus\mathcal Z_0$, and set $\mathcal L_0:= \ell \otimes_\R \underline\C$.  Denote by $A_0$ the connection (\refeq{A_0def}) formed from $\nabla^\text{flat}$ on $\ell$ and the product connection as in (\refeq{A_0def}).  An (unoriented) {\bf $\Z_2$-harmonic Spinor} is a triple $(\mathcal Z_0, A_0, \Phi_0)$ where $\Phi_0 \in \Gamma(S^\text{Re})$ satisfies

\begin{enumerate}
\item[(i)] $\|\Phi_0\|_{L^2}=1$ 
\item[(ii)]$\slashed D_{A_0}^\text{Re}\Phi_0=0 \hspace{1cm} \text{on }\hspace{.3cm} \Yminus\mathcal Z_0.$
\item[(iii)] $|\Phi_0|$ extends continuously to $Y$ with $\mathcal Z_0= |\Phi_0|^{-1}(0)$. 
\end{enumerate} 
\smallskip 
Such triples are considered up to the action of $\Z_2=\{\pm 1\}$. When $\mathcal Z_0$ is equipped with an orientation, the parenthetical descriptor is removed. 
\end{defn}

Notice that, although we have reached it in a circuitous way, this definition makes no reference to the Seiberg-Witten equations. The bundle $S=S_0\otimes_\C \mathcal L_0\simeq S_0\otimes_\R \ell$ is simply the spinor bundle associated to another spin structure on $Y-\mathcal Z$, which need not extend over $\mathcal Z_0$. Conversely, given a spinor bundle $S_1\to \Yminus\mathcal Z_0$ we form $S^\text{Re}\subset S_1\otimes_\C E$ as before. In fact,  it is straightforward to show that since an $SU(2)$ bundle $E$ on a 3-manifold is necessarily trivial, that $S^\text{Re}\simeq S_1$ and the only effect of introducing $E$ is a perturbation to the spin Dirac operator of $S_1$ arising from $B_0$. 

Another key point is that Definition \ref{Z2harmonicdef} makes no reference to a complex line bundle $\mathcal L$ such that $\mathcal L|_{\Yminus\mathcal Z_0}=\mathcal L_0$: the information about the isomorphism class of $\mathcal L$ is lost in the limit $\e\to 0$. There are many choices of extensions $\mathcal L\to Y$ whose restriction to $\Yminus\mathcal Z_0$ is isomorphic to $\mathcal L_0$, and before beginning any analysis of the gluing question one must first answer the topological question of {\it which} $\text{Spin}^c$ structure the $\Z_2$-harmonic spinor should be glued into. This is addressed in Section \ref{section3new} in the setting where Assumptions \ref{assumption1}-\ref{assumption3} hold.

\begin{rem} Definition \ref{Z2harmonicdef} makes sense if $\mathcal Z_0$ is simply a closed, rectifiable subset of Hausdorff codimension 2. The extension of the definition of an unoriented $\Z_2$-harmonic spinor is trivial; the oriented case requires some geometric measure theory arguments (see \cite{Haydysblowup}). Note that other authors generally do not assume that the definition includes an orientation of $\mathcal Z_0$. The results of \cite{Haydysblowup}, however, show that when a $\Z_2$-harmonic spinor arises as a limit of solutions to the Seiberg-Witten equations it carries a preferred orientation. 

\end{rem}

\subsection{The Weitzenb\"ock Formula}
\label{section2.4}

This section derives the Weitzenb\"ock formula for the gauge-fixed Seiberg-Witten equations with two spinors linearized at a possibly singular configuration. This formula is the two-spinor version of the one appearing in \cite{TaubesWeinstein} Equation 5.21.

In dimension 3, it is standard (\cite{KM}) to supplement the equations (\refeq{SW1})-(\refeq{SW2}) by an auxiliary 0-form field $a_0\in \Omega^0(i\R)$; this extends them to an elliptic system modulo gauge. Extend Clifford multiplication to a map $\gamma: (\Omega^0\oplus \Omega^1)\to \text{End}(S)$ denoted by the same letter. The {\bf extended (Two-Spinor) Seiberg-Witten Equations} for a configuration $(\Psi, A, a_0)\in \Gamma(S_E)\times \mathcal A(\mathcal L)\times \Omega^0(Y;i\R)$ are  
\begin{eqnarray}
\slashed D_{A}\Psi + \gamma(a_0)\Psi&=&0\label{SWextended1} \\ 
\star F_A - da_0 +\tfrac{1}{2}\mu(\Psi, \Psi)& =& 0.  
\label{SWextended2}
\end{eqnarray}  

\noindent This system is again invariant under the action of the gauge group $\mathcal G$ (which acts trivially on $a_0$). For irreducible configurations ($\Psi$ not identically 0), integration by parts shows that $a_0=0$, thus irreducible solutions of (\refeq{SWextended1})-(\refeq{SWextended2}) are the same as irreducible solutions of the original equations (\refeq{SW1})-(\refeq{SW2}). For the purposes of the eventual gluing result, it suffices to only consider irreducible solutions. The {\bf extended blown-up Seiberg-Witten equations} are defined analogously with the addition of the auxiliary $0$-form $a_0$ and the term $\e^2 da_0$ in the second equation.

Let $(\tfrac{\Phi}{\e}, A)\in C^\infty(Y; S_E)\times (\mathcal A(\mathcal L)\otimes \Omega^0(i\R))$ denote a smooth configuration with $\|\Phi\|_{L^2}=1$. Here, we have condensed the notation by replacing $A$ with $A+a_0$.   Differentiating a 1-parameter family of nearby configurations $(\frac{\Phi}{\e},A)+ s(\ph, a)$ shows that the linearization of the equations at $(\tfrac{\Phi}{\e},A)$ acting on the variation $(\ph,a)$ is given by 

$$\d{}{s}\Big |_{s=0} SW(\tfrac{\Phi}{\e}+ s\ph, A+ sA)=\begin{pmatrix} \slashed D_A\ph + \gamma(a)\tfrac{\Phi}{\e} \vspace{.15cm}\\ \tfrac{\mu(\ph, \Phi)}{\e}+( \star d -d)a\end{pmatrix}.$$

\noindent Supplementing the pair $(\Phi,A)$ with an auxiliary $0$-form $a_0\in \Omega^0(i\R)$ and the gauge-fixing condition 
$$-d^\star a -{ i\text{Re}\br i \ph,\tfrac{\Phi}{\e} \kt}=0$$
extends the linearization to the elliptic system 

$$\mathcal L_{(\Phi, A,\e)}\begin{pmatrix} \ph \\ a \end{pmatrix}= \begin{pmatrix}\slashed D_A & \gamma(\_)\tfrac{\Phi}{\e} \vspace{.15cm}\\ \tfrac{\mu(\_,\Phi)}{\e} & \bold d \end{pmatrix} \begin{pmatrix} \ph   \vspace{.15cm} \\ a \end{pmatrix} \hspace{2cm}\text{where}\hspace{1cm}\bold da= \begin{pmatrix} 0 & -d^\star \\ -d & \star d \end{pmatrix}\begin{pmatrix}a_0 \\ a_1\end{pmatrix}.$$

\noindent and $a=(a_0, a_1)\in (\Omega^0\oplus \Omega^1)(i\R)$. The moment map is extended as in (\refeq{extendedmu}) to $\mu=(-i \br i\ph, \Phi\kt, \mu^1)$ with $\mu^1$ being the 1-form components. The reader is cautioned that because of the singular nature of the connection $A_0$, the linearization at a $\Z_2$-harmonic spinor $(\Phi_0, A_0)$ is not a bounded operator on $L^2(\Yminus\mathcal Z_0)$.  

The Weitzenb\"ock formula is given below. Notice that in this formula something rather miraculous has occurred: {\it a priori} one would expect the Weitzenb\"ock formula to contain first order terms in $(\ph, a)$. The fact that these terms cancel is a special property of the Seiberg-Witten equations. 

\begin{prop}{\bf (Weitzenb\"ock Formula)} \label{Weitzenb\"ock} Let $\mathcal Z_0\subseteq Y$ be a closed subset and $(\Phi_0, A_0)$ a configuration smooth on $\Yminus\mathcal Z_0$. Then on $\Yminus\mathcal Z_0$ the operator $\mathcal L_{({\Phi_0}, A_0,\e)}$ satisfies 
\bigskip 

$$ \mathcal L^\star \mathcal L_{(\Phi_0, A_0,\e)} (\ph, a)= \begin{pmatrix} \slashed D_{A_0}\slashed D_{A_0} \ph  \vspace{.1cm}\\   \bold d\bold d a \end{pmatrix}+\frac{1}{\e^2}\begin{pmatrix} \gamma(\mu(\ph, \Phi_0))\Phi_0 \vspace{.1cm} \\  \mu(\gamma(a)\Phi_0, \Phi_0)\end{pmatrix}+\frac{1}{\e}\mathfrak B (\ph,a)$$
 
 \bigskip
 
\noindent where  the latter is the off-diagonal zeroth-order term 

$$\mathfrak B(\ph , a )=\begin{pmatrix} 0 & \gamma(\epsilon \_)\slashed D_{A_0}\Phi_0 - 2\_ \cdot \nabla_{A_0} \Phi_0 \\  \epsilon \mu(\_ ,\slashed D_{A_0}\Phi_0) + 2i \br i \_ , \nabla_{A_0} \Phi_0 \kt  & 0 \end{pmatrix} \begin{pmatrix}\ph \\ a \end{pmatrix}. $$  Here, $a\cdot \nabla\Phi_0$ denotes the contraction of 1-form indices, $2i \br i\ph, \nabla \Phi_0\kt$ is the contraction of spinor components (yielding a 1-form), and $\epsilon$ acts by $(-1)^k$ on $k$-forms. 

\end{prop}

\begin{proof} The operator $\mathcal L$ is formally self-adjoint. Expanding the expression and abbreviating $\slashed D= \slashed D_{A_0}$ and $\nabla=\nabla_{A_0}$,  
\bea
 \mathcal L^\star \mathcal L_{(\Phi_0, A_0,\e)}(\ph, a)&=&  \begin{pmatrix} \slashed D & \mathcal \gamma(\_)\tfrac{\Phi_0}{\e} \\ \tfrac{\mu(\_, \Phi_0)}{\e} & \bold d\end{pmatrix}\begin{pmatrix} \slashed D & \mathcal \gamma(\_)\tfrac{\Phi_0}{\e} \\ \tfrac{\mu(\_,\Phi_0)}{\e} & \bold d\end{pmatrix}\begin{pmatrix}\ph \\ a \end{pmatrix} \\
 &=&  \begin{pmatrix}\slashed D\slashed D \ph + \tfrac{\gamma(\mu(\ph, \Phi_0))\Phi_0}{\e^2} \vspace{.1cm}  \\ \bold d \bold d  +\tfrac{ \mu((\gamma(a)\Phi_0), \Phi_0)}{\e^2} \end{pmatrix} + \frac{1}{\e}\begin{pmatrix}    \slashed D( \gamma(a){\Phi_0}) + \gamma(\bold d a){\Phi_0}\vspace{.1cm} \\ \bold d \mu(\ph, \Phi_0)+ {\mu(\slashed D\ph, \Phi_0)}  \end{pmatrix}
\eea

\noindent Next, we use the following identities, which are proved in \cite{ConcentratingDirac}: 
\begin{eqnarray}
\slashed D(\gamma(a)\Phi_0) & =& - \gamma(\bold d a)\Phi_0 - \gamma((-1)^{\deg}a)\slashed D\Phi_0 - 2a \cdot \nabla\Phi_0\label{appendix1identities1}\\
\bold d\mu(\ph ,\Phi_0)&=& -\mu((\slashed D\ph), \Phi_0)  +(-1)^{\deg} \mu(\ph,  (\slashed D\Phi_0)) + 2i \br i \ph  \ , \ \nabla \Phi_0 \kt.
\label{appendix1identities}
\end{eqnarray}
Substituting these yields the formula. 
\end{proof}
In the case that $\mathcal Z_0$ is empty, i.e. if the configuration $(\Phi, A)$ is smooth, integration by parts and the above yields immediately yields the 
the following $L^2$ version of the Weitzenb\"ock formula will be used in later sections. This formula does not apply to the linearization at the singular configuration $(\Phi_0, A_0)$. 
\begin{cor} \label{L2Weitzenb\"ock} If $(\Phi, A)$ is a smooth configuration, then $\mathcal L_{({\Phi_0}, A_0,\e)}$ is formally self-adjoint and 
\bea \|\mathcal L_{(\Phi,A,\e)}(\ph, a)\|^2_{L^2(Y)} =   \|\slashed D_{A}\ph\|^2_{L^2(Y)} + \|\bold d a\|^2_{L^2(Y)} + \frac{1}{\e^2} \|\gamma(a)\Phi\|^2_{L^2(Y)}+\frac{1}{\e^2} \|\mu(\ph,\Phi)\|^2_{L^2(Y)} \\ + \frac{1}{\e} \br (\ph ,a ) \ , \ \mathfrak B (\ph, a)\kt \eea 
where the inner product is in $L^2$ and $\mathfrak B$ is as above. An equivalent statement holds in the case that $(Y,\del Y)$ is a manifold with boundary, up to the addition of a boundary term.  
\end{cor}

\begin{proof}
 The formal self-adjointness follows from the fact that  $\mu,\gamma$ are fiberwise adjoints. This also implies 
  $$ \br b , \mu(\gamma(b)\Phi, \Phi) ) \kt= \br \gamma(b)\Phi, \gamma(b)\Phi\kt= |b|^2 |\Phi_0|^2$$
for all $b$ and 
$$ \br \ph  , \gamma(\mu(\ph, \Phi)\Phi ) \kt= \br \mu(\ph,\Phi),\mu(\ph,\Phi)\kt= |\mu(\ph,\Phi)|^2$$
for all $\Phi$. The expression for the $\tfrac{1}{\e^2}$ terms follows.
\end{proof}

\begin{rem}
Notice that this cancellation of the first-order terms does {\it not} hold for the $\e$-version of the blown-up Seiberg equations unless $\e=1$. The unequal renormalization of the spinor and connection components in the blown-up equations disrupts the cancellation. It is for this reason that we prefer to work with the un-renormalized equation wherever possible. 
\end{rem}


\section{The Singular Dirac Operator}

\label{section3new}
Locally near a component of $\mathcal Z_0$, the Dirac operator $\slashed D_{A_0}$ takes the form 

\be \slashed D_{A_0}= \slashed D \  + \  O\left(\frac{1}{r}\right) \label{singularDirac} \ee

\noindent where $\slashed D$ is the Dirac operator on $S_E$ formed using a smooth background connection that extends over $\mathcal Z_0$, and $r$ denotes the distance to $\mathcal Z_0$. In particular, the zeroth order term is unbounded on $L^2$. An equivalent viewpoint is to consider $r\slashed D_{A_0}= r\slashed D+ O(1)$ in which case the zeroth order term is bounded, but the symbol degenerates along $\mathcal Z_0$. Elliptic operators of this type are known as {\bf elliptic edge operators}, and have been studied extensively in microlocal analysis dating back to the 1980s. Authoritative sources on similar operators include \cite{MazzeoEdgeOperators, MazzeoEdgeOperatorsII,Melrosebcalculus,grieser2001basics} and the references therein.

 This section gives some necessary results about the singular Dirac operator. The results are stated here without proof, and the reader is referred to \cite{MazzeoHaydysTakahashi,MazzeoEdgeOperators, MazzeoEdgeOperatorsII} for proofs, as well as \cite{PartII} which provides more detailed analysis in this particular case.  The final subsection uses these to address the topological problem of reconstructing the $\text{Spin}^c$ structure explained at the end of Section \ref{section2}. It is instructive in understanding the operator (\refeq{singularDirac}) to first consider the following example. 
 \begin{eg} \label{euclideanexample}
 Consider $Y=S^1\times D^2$ with coordinates $(t,x,y)$ and take $\mathcal Z= S^1\times \{0\}$. Consider the trivial bundle $\underline \C^2 \to Y$ of rank 2, and let $\ell \to Y-\mathcal Z$ be the real line bundle that restricts to the mobius bundle on $\{t\}\times \R^2$ equipped with its unique flat connection $A_0$. The spinor bundle $\underline \C^2\otimes_\R \ell$ is globally trivial (its determinant is trivial and $H^2(Y;\Z)$ has no 2-torsion here). In fact, 
 
 \bea
  \underline \C^2 &\to& \underline \C^2 \otimes_\R \ell \\
 \psi &\mapsto& e^{i\theta/2}\psi
 \eea  
 provides an explicit trivialization. Indeed, $e^{i\theta/2}$ provides a nowhere-vanishing section of each factor $\underline \C \otimes \ell$ with the proper monodromy condition. In this trivialization, we may write $$\nabla_{A_0}= \text{d}+ \frac{i}{2} d\theta= \text{d}+ \frac{1}{4}\left(\frac{dz}{z}-\frac{d\overline z}{\overline z}\right) \hspace{2.5cm} \slashed D_{A_0}= \begin{pmatrix}i\del_t & 2\del \\ 2 \delbar & -i\del_t \end{pmatrix}+\frac{1}{4}\gamma\left(\frac{dz}{z}-\frac{d\overline z}{\overline z}\right)$$  
 where $(r,\theta)$ are polar coordinates on the $\R^2$ factor and $z=x+iy$ and $z=x-iy$ complex ones. After decomposing a spinor in Fourier series $$\psi= e^{ik\theta}e^{i\ell t} \begin{pmatrix}\psi^+_{k\ell}(r) \\ \psi^-_{k\ell}(r)\end{pmatrix}$$
 the Dirac operator becomes a decoupled family of ODEs which may be solved using Bessel functions (see \cite{MazzeoHaydysTakahashi,RyosukeThesis}). One finds that \be\psi_\ell^\text{Euc}= e^{-i\theta/2} \sqrt{|\ell|} e^{i\ell t} e^{-|\ell| r} \begin{pmatrix} \tfrac{1}{\sqrt{z}} \\ \tfrac{\text{sgn}(\ell)}{\sqrt{\overline z}}\end{pmatrix}\label{euclideancokernel}\ee
 are an $L^2$-orthonormalized set of solutions parameterized by $\ell \in \Z^{\neq 0}$.  
 \end{eg}
 
 \medskip 
 
 This example is the local model of the infinite-dimensional cokernel alluded to in the introduction (section \ref{sectiondegenerating}). Notice that this phenomenon is not an artifact of the non-compactness of $Y$ as $r\to\infty$~\ --- \ these solutions concentrate exponentially near $\mathcal Z$ for large $|\ell|$.  
 
 The family of solutions (\refeq{euclideancokernel}) display the two following key properties: 
 
 \begin{enumerate}
 \item[(1)] $\psi_\ell^\text{Euc} \in L^2$ but $\nabla_{A_0} \psi_\ell^\text{Euc} \notin L^2$, thus these are {\it not} $\Z_2$-harmonic spinors as defined in (\refeq{Z2harmonicpreliminary}). 
 \item[(2)] $\psi_\ell^\text{Euc}$ do not extend smoothly across $r=0$; instead they have asymptotic expansions  with half-integer powers of $r$. \label{listitem2}
 \end{enumerate}
 In this setting there are no $\Z_2$-harmonic spinors because all solutions whose derivative is $L^2$ along $\mathcal Z$ are not $L^2$ as $r\to \infty$. These are the key properties which generalize to the case of a general 3-manifold. 
\subsection{(Semi)-Fredholm Theory}
\label{section2.5.1}

Returning to the setting of a general closed 3-manifold $(Y,g_0)$, let $r: Y \to \R^{\geq 0}$ denote a weight function equal to $\text{dist}(-, \mathcal Z_0)$ in a neighborhood of $\mathcal Z_0$ and bounded away from it. Consider the weighted Sobolev spaces defined by: 

\bea  rH^1_e(\Yminus\mathcal Z_0 ; S^\text{Re})&:=& \left\{\ph  \ \ \Big | \ \ \int_{Y\setminus\mathcal Z_0} |\nabla_{A_0} \ph|^2 + \frac{|\ph|^2}{r^2} \ dV   <\infty\right\}\\ 
L^{2}(\Yminus\mathcal Z_0 ; S^\text{Re})&:=& \left\{\psi \ \ \Big | \ \ \int_{Y\setminus\mathcal Z_0}|\psi|^2 \ dV  <\infty\right\}. \eea

\noindent Here, the subscript $e$ stands for ``edge''.  

The next proposition follows from the general theory of \cite{MazzeoEdgeOperators}. It is also proved using elementary methods in Section 2 of \cite{PartII}.

\begin{prop}
The operator $$\slashed D_{A_0}: rH^1_e(\Yminus\mathcal Z_0 ; S^\text{Re}) \lre L^2(\Yminus\mathcal Z_0 ; S^\text{Re}).$$  is (left) Semi-Fredholm, i.e. 
\begin{itemize}\item $\ker(\slashed D_{A_0})$ is finite-dimensional, and
\item  $\text{Range}(\slashed D_{A_0})$ is closed. \end{itemize}

\label{semifredholm}\qed 
\end{prop}

\noindent Notice that elements $\Phi\in \ker(\slashed D_{A_0})$ satisfy the integrability condition of Definition \ref{Z2harmonicpreliminary}, thus this space constitutes the $\Z_2$-harmonic spinors (integrability $|\nabla_{A_0}\Phi|^2$ implies the weighted $L^2$-term is finite as well which gives the reverse inclusion). Assumption \ref{assumption3} imposes the requirement that this finite-dimensional space is 1-dimensional and spanned by $\Phi_0$. 

The next proposition is not explicitly needed, but it is at the heart of the analysis for the question of gluing $\Z_2$-harmonic spinors (recall \ref{sectiondegenerating}). It states that the cokernel of the singular Dirac operator $\slashed D_{A_0}$ is a small perturbation of the case of Example \ref{euclideanexample}; in particular, it is infinite-dimensional and concentrates strongly near $\mathcal Z_0$. This proposition is proved in Section 4 of \cite{PartII}. 

\begin{prop} \label{infcokernel}There is a bounded linear isomorphism 

$$ \bigoplus_{\pi_0(\mathcal Z_0)} L^2(S^1;\C) \ \oplus \ \ker(\slashed D_{A_0}|_{rH^1_e}) \ \overset{\simeq} \lre \ \text{Coker}(\slashed D_{A_0}) $$
where the direct sum is over components of $\mathcal Z_0$. It is given by the inclusion on $\ker(\slashed D_{A_0}|_{rH^1_e}) $, and on the summand corresponding to a component $\mathcal Z_j$, by the linear extension of  

$$ e^{i \ell t} \ \mapsto \ \psi_{j,\ell}^\text{Re}  \ + \ \xi_{j,\ell}$$

\noindent where  

\begin{itemize} \item $\psi_{j,\ell}^\text{Re}$ is given in a local trivialization of $S_E\simeq \underline \C^2 \otimes_\C{\underline{ \mathbb H}}$ extending across $\mathcal Z_j$ by  $$ \psi_{j,\ell}^\text{Re}= \pi^\text{Re} \left(\chi\psi^\text{Euc}_\ell\otimes 1 \right)$$
\noindent where $\psi_\ell^\text{Euc}$ is given by Equation (\refeq{euclideancokernel}), $\pi^\text{Re}$ is the projection to $S^\text{Re}$ and $\chi$ is a cut-off function supported on a neighborhood of $\mathcal Z_j$. 
\item $\xi_\ell$ is a perturbation satisfying $\|\xi_{j,\ell}\|_{L^2}\leq \tfrac{C}{|\ell|}$. 
\end{itemize} 
\qed
\end{prop}

\subsection{Local Forms}
\label{section2.5.2}

Because of the effective degeneracy of the symbol, standard elliptic regularity fails for operators of the form ( \refeq{singularDirac}). The proper replacement of elliptic regularity is, as suggested by the form of (\refeq{euclideancokernel}), the existence of asymptotic expansions near $\mathcal Z_0$ that generalize Taylor expansions by allowing non-integral powers. Before writing these local expansions, let us choose local coordinates and express the Dirac operator using these coordinates and an appropriate trivialization. 

 We endow a tubular neighborhood diffeomorphic to a solid torus $N_{r_0}(\mathcal Z_j)$ of a component $\mathcal Z_j$ of $\mathcal Z_0$ with local coordinates as follows. Let $\gamma: S^1\to \mathcal Z_0$ denote an arclength parameterization of the component $\mathcal Z_j$ whose length is denoted $|\mathcal Z_j|$. When $\mathcal Z_0$ is oriented, it is assumed that $\gamma$ is chosen respecting its orientation. Next, choose a global orthonormal frame $\{n_1, n_2\}$ of the pullback $\gamma^*N\mathcal Z_0$ of the normal bundle to $\mathcal Z_0$. We require that $\{\dot \gamma, n_1, n_2\}$ is an oriented frame with respect to the orientation on $Y$.

\begin{defn}
\label{geodesicnormal}
A system of {\bf geodesic normal coordinates} for $r_0 < r_{\text{inj}}$ where $r_{\text{inj}}$ is the injectivity radius of $Y$ is the diffeomorphism $S^1\times D_{r_0}\simeq N_{r_0}(\mathcal Z_j)$ for a chosen component of $\mathcal Z_0$ defined by $$(t,x,y)\mapsto \text{Exp}_{\gamma(t)}(xn_1 + yn_2).$$

\noindent Here $t$ is the coordinate on the $S^1$ factor, which has radius normalized so that $t\in [0,|\mathcal Z_j|)$. In these coordinates the Riemannian metric $g$ can be written $$g=dt^2 + dx^2 + dy^2 +\  [2x\frak m_x(t)  +2y\frak m_y(t) ] dt^2\  + \ [ \mu(t) y ]  dt dx \  +\  [-\mu(t)x ] dt dy \ + \ O(r^2)$$

\noindent where $\mu(t), \frak m_x(t), \frak m_y(t)$ are defined by  
\bea
\mu(t)&=& \br \nabla_{\dot \gamma} n_x, n_y\kt= -\br \nabla_{\dot \gamma} n_y, n_x\kt\\
\frak m_\alpha(t) &=& \br \nabla_{\dot \gamma}\dot \gamma , n_\alpha \kt
\eea
for $\alpha=x,y$. Given such a coordinate system, $(t,r,\theta)$ are used to denote the corresponding cylindrical coordinates, and $(t,z,\overline z)$ the complex ones on the $D_{r_0}$ factor. 
\end{defn}

We also have the following trivialization. Recall that as in Definition \ref{Z2harmonicdef}, we denote $\mathcal L_0=\ell\otimes_\R \underline \C$.

\begin{lm}\label{localtrivialization} For each component $\mathcal Z_j$ of $\mathcal Z_0$, there exists a local trivialization $$\sigma_j: (S_0 \otimes \mathcal L_0 \otimes E)|_{N(\mathcal Z_j)-\mathcal Z_j} \simeq (N_{r_0}(\mathcal Z_j) -\mathcal Z_j ) \ \times \ (\C^2 \otimes_\C \C) \otimes_\C \mathbb H$$
in which 

\begin{itemize}
\item The connection $A_0$ on the middle factor is given by $$A_0 \ = \ \text{d} + \frac{i}{2}d\theta +\epsilon_j \tfrac{i}{2}dt$$ 
\noindent where $\epsilon_j=0$ or $1$ depending on $\ell$ and $S_0$. 

\item The Dirac operator may be written 
\be \slashed D_{A_0}= \slashed D + \frac{i}{2}\gamma(d\theta) + \frak d_1 + \frak d_0\label{Diraclocal}\ee 
\noindent where $\slashed D$ is the standard Euclidean Dirac operator, and $\frak d_1,\frak d_0$ are respectively a first order and zeroth order term satisfying 

$$|\frak d_1 \psi| \leq r |\nabla \psi| \hspace{3cm} |\frak d_2 \psi| \leq C |\psi|.$$

\item The anti-linear involution $J$ defined in \refeq{Jdefglobal} is given by $ e^{-i\epsilon_j t/2}e^{-i\theta} J_0 $ where $J_0$ is given by the expression \refeq{Jdef}. Consequently, a spinor $\Phi\in \Gamma(S^\text{Re})$ takes the form \be \Phi_0 = \begin{pmatrix} \alpha \\ \beta \end{pmatrix}\otimes 1 + e^{-i\theta} e^{-i\epsilon_j t}  \begin{pmatrix}-\overline \beta \\ \overline \alpha \end{pmatrix}\otimes j.\label{realformlocal}\ee

\end{itemize}
\end{lm}

\begin{proof}
First, we trivialize the middle factor $ \ell \otimes_\R\underline \C$. Fix a vector $s_0 \in (\ell \otimes_\R\underline \C )_{(0,r_0/2,0)}$ in the fiber above the point $(0, r_0/2,0) \in~N(\mathcal Z_j)$. Parallel transport using $A_0$ in the $+\theta$ and $+t$ directions defines a section $s$ with monodromy $-1$ around the meridian of $\mathcal Z_j$ and monodromy $\epsilon_j=0$ or $1$ around the longitude. The latter is determined by the line bundle $\ell$ and by the choice of spin structure $S_0$. The section $e^{i\theta/2} e^{i\epsilon_jt /2}s$ therefore defines a global nowhere-vanishing section of $\underline \C\otimes \ell$. Since $s$ is parallel by construction, in the trivialization 
\begin{eqnarray}
\underline \C & \mapsto & \underline \C\otimes_\R \ell\\
f  & \mapsto & e^{i\theta/2} e^{i\epsilon_jt /2}s f\label{trivialization3}
\end{eqnarray}     
\noindent the connection becomes $${A_0}=\text{d} + \frac{i}{2}d\theta + \epsilon_j\tfrac{i}{2}dt.$$
\noindent The first bullet point follows. 

The second bullet point, we extend the above trivialization to $\sigma_j$ by choosing local trivializations of $S_0$ and $E$ that extend across $\mathcal Z_j$. We may additionally specify that in the trivialization of $S_0$, the two factors of $\underline \C$ are the $\pm i$ eigenspaces of $\gamma(e^1)$ where $\{e^1, e^2, e^3\}$ is the orthonormal co-frame on $N(\mathcal Z_j)$ extending $dt,dx,dy$. In this trivialization, the connection on $S_0 \otimes \mathcal L_0 \otimes E$ is given by 

$$\nabla_{A_0}=\text{d} + \frac{i}{2}d\theta + \text{b}_0 + \Gamma+ \epsilon_j\tfrac{i}{2}dt$$

\noindent where $\text{b}_0$ and $\Gamma$ are zeroth order terms arising from the connection $B_0$ and the Christoffel symbols of $g_0$ respectively. The second bullet point is then immediate from the first, where $\frak d_1$ arises from the $O(r)$ failure of $dt, dx, dy$ to be an orthonormal frame for $g_0$, and $\frak d_1$ arises from a combination of $\text{b}_0, \Gamma$ and $\epsilon_j$.

For the third bullet point, recall that $J$ is given by the local expression \refeq{Jdef} in local trivialization respecting the $SU(2)$ structure on $\underline \C^2 \otimes_\C (\ell \otimes_\R \underline \C)$ (which $\sigma_j$ is not).  In a system of local trivializations on a a contractible open sets $U_\alpha\subset N(\mathcal Z_j)-\mathcal Z_j$ respecting the $SU(2)$ structure (i.e. one in which the transition functions on $\ell\otimes_\R \underline \C$ are simply), $A_0$ is given in each of these trivializations simply by $\text{d}$. These differ from the trivialization \refeq{trivialization3} by transition functions $e^{i\theta/2}e^{i\epsilon_j t/2}$, hence in the trivialization $\sigma_j$ constructed using  \refeq{trivialization3}, $J$ is given by $$J=(e^{-i\theta/2} e^{-i\epsilon_jt /2})  J_0(e^{i\theta/2} e^{i\epsilon_jt /2}) $$
and the third bullet point follows from the complex anti-linearity of $J$. 
\end{proof}

\medskip 

Using these local coordinates and local trivialization, the Dirac operator \refeq{Diraclocal} takes the local form \refeq{singularDirac} and the general regularity theory of \cite{MazzeoEdgeOperators} applies to give local asymptotic expansions. We consider the following type of asymptotic expansion.

\begin{defn} A spinor $\psi\in L^2(\Yminus\mathcal Z_0; S_E)$ is said to admit a {\bf Polyhomogenous Expansion} with index set $\Z+\tfrac{1}{2}$ if $$\psi\sim \sum_{ k\geq 0  }\sum_{m+n=k} r^{1/2} \begin{pmatrix} c_{k,m,n}(t)  z^m \overline {z}^n  \ \ \ \ \  \\ d_{k,m,n}(t)e^{-i\theta} z^m \overline {z}^n\end{pmatrix} $$

\noindent where $c_{k,m,n}(t), d_{k,m,n}(t)\in C^\infty(S^1;\mathbb H)$, and where $\sim$ denotes convergence in the following sense: for every $N\in \N$, the partial sums $$\psi_N=   \sum_{ k\leq N  }\sum_{m+n=k} r^{1/2}\begin{pmatrix} c_{k,m,n}(t)  z^m \overline {z}^n  \ \ \ \ \  \\ d_{k,m,n}(t)e^{-i\theta} z^m \overline {z}^n\end{pmatrix}$$ satisfies the pointwise bounds 

\begin{eqnarray} |\psi-\psi_N| &\leq& C_{N} r^{N+1}\label{polyhom1} \\ |\nabla_t^{\alpha}\nabla^{\beta}_{A_0}(\psi-\psi_N) |&\leq& C_{N,\alpha,\beta} r^{N+1-|\beta|}\label{polyhom2}\end{eqnarray}

\noindent for constants $C_{N,\alpha,\beta}$ determined by the background data and choice of local coordinates and trivialization. Here, $\beta$ is a multi-index of derivatives in the normal directions.

 \label{polyhomogeneous}

\end{defn}

The appropriate version of elliptic regularity for $\Z_2$-harmonic spinors, which follows from \cite{MazzeoEdgeOperators}, is the following. 

\begin{prop} Suppose that $\Phi_0\in rH^1_e(\Yminus\mathcal Z_0; S^\text{Re})$ is a $\Z_2$-harmonic spinor. Then $\Phi_0$ admits a polyhomogenous expansion. Thus in the trivialization of Lemma \refeq{localtrivialization}, $\Phi_0$ has a local expression 

\begin{eqnarray} \Phi_0&\sim &\begin{pmatrix} c(t) r^{1/2}  \ \ \ \ \  \\ d(t)r^{1/2}e^{-i\theta} \end{pmatrix}\otimes 1+\begin{pmatrix} -\overline d(t) r^{1/2}  \ \ \ \ \  \\ \overline c(t)r^{1/2}e^{-i\theta} \end{pmatrix} \otimes j \label{polyhomogeneous1} \\ & &  +   \sum_{ \underset{m+n=k}{k\geq 1}  }r^{1/2}\begin{pmatrix} c_{m,n}(t)  z^m \overline {z}^n  \ \ \ \ \  \\ d_{m,n}(t)e^{-i\theta} z^m \overline {z}^n\end{pmatrix}  \otimes 1  +r^{1/2}\begin{pmatrix}-\overline d_{m,n}(t)  z^m \overline {z}^n  \ \ \ \ \  \\ \overline c_{m,n}(t)e^{-i\theta} z^m \overline {z}^n\end{pmatrix} \otimes j  \label{polyhomogeneous2}\end{eqnarray}   
 
\noindent where $c(t), d(t),c_{k,m,n}(t), d_{k,m,n}(t) \in C^\infty(S^1;\C)$. The terms on the second line of the expression will often be abbreviated by adding $O(r^{3/2})$ to the leading order term, with the understanding that this notation refers to a collection of bounds as in (\refeq{polyhom1})-(\ref{polyhom2}). 
\label{asymptoticexpansion}
\end{prop}

\begin{proof}
The existence of such an expansion is a consequence of the regularity theory in \cite{MazzeoEdgeOperators, MazzeoHaydysTakahashi}. The relation between the two components follows from the form of $S^\text{Re}$ Lemma \refeq{localtrivialization}. 
\end{proof}

\begin{rem}
The smooth functions $c(t), d(t), c_{m,n}(t), d_{m,n}(t)$ depend on the choice of frame $\{n_x, n_y\}$ made to define local coordinates in Definition \ref{geodesicnormal}. More invariantly, these are sections of $(N\mathcal Z_0)^{\otimes(1/2+m-n)}$ where $m=n=0$ for $c(t),d(t)$. 
\end{rem}
\subsection{Reconstructing $Spin^c$ Structures}
\label{section2.5.3}

In this subsection, the above local forms are used to show that the $\text{Spin}^c$ structure arising from the line bundle $\mathcal L_0\to \Yminus\mathcal Z_0$ can be extended to one on $Y$ satisfying the conclusion of Theorem \ref{main}. This section therefore finishes the topological portion of Theorems \ref{main}-\ref{mainc}.

Since $\Phi_0$ behaves like $r^{1/2}$ along $\mathcal Z_0$, it does not extend to the closed manifold as a smooth section. A version of $\Phi_0$ squared does, however. For any section $\Phi$, let $\det \Phi \in \Gamma(\Yminus\mathcal Z_0; \mathcal L_0^2)$ be the section defined as follows. 
The identification used in (\refeq{matrixPhi1}) induces a bundle isomorphism $S_E\simeq \text{Hom}(E^*, S)$ so that $\Phi$ can be regarded as a bundle map $E^*\to S$; $\det(\Phi)$ is then a section of $\det(E^*)^{-1}\otimes \det(S)=\mathcal L_0^{2}$. As in (\refeq{matrixPhi1}), there are local trivialization in which $\det(\Phi)$ is given by the determinant of the matrix  

 \be\Phi_0 = \begin{pmatrix} \alpha_1 & \alpha_2  \\ \beta_1 & \beta_2 \end{pmatrix}\label{matrixPhi}.\ee

\begin{lm} \label{detminant}The following inequality holds pointwise for any $\Phi$: 

$$\tfrac{1}{4}|\Phi|^4 \leq |\mu(\Phi,\Phi)|^2 + |\det\Phi|^2.$$

\noindent In particular, $\det(\Phi_0)$ vanishes nowhere on $\Yminus\mathcal Z_0$. 

\end{lm}

\begin{proof}
Suppose that $\Phi$ has the above form, and write  $\alpha=(\alpha_1, \alpha_2)$ and $\beta=(\beta_1, \beta_2)$. Using the expressions for the moment map (\refeq{mut}-\refeq{muy}) denote the $dt$ component by $\mu_\R$ and the $dx+idy$ component by $\mu_\C$. Using the Hermitian inner product on $\mu_\C$ and $\det(\Phi)$ shows: 

\bea
|\mu_\R(\Phi,\Phi)|^2&=& (|\alpha_1|^2 +|\alpha_2|^2 - |\beta_1|^2 +|\beta_2|^2)^2 \\
 &=& |\alpha|^4 +|\beta|^4 - 2|\alpha|^2 |\beta|^2 \\
 |\mu_\C(\Phi,\Phi)|^2 &=&  |\overline \alpha_1 \beta_1 +\overline \alpha_2\beta_2|^2\\
 &=& |\alpha_1|^2 |\beta_1|^2 +|\alpha_2|^2 |\beta_2|^2 + 2 \text{Re}(\overline \alpha_1 \beta_1 \alpha_2 \overline \beta_2)\\
 |\det\Phi|^2 &=&|\alpha_1 \beta_2 - \alpha_2 \beta_1|^2 \\
&=&  |\alpha_1|^2 |\beta_2|^2 + |\alpha_2|^2|\beta_1|^2 - 2 \text{Re}(\alpha_1 \beta_2 \overline \alpha_2 \overline \beta_1)
\eea

\noindent and the two real parts are negatives after conjugating. Adding these yields $$|\Phi|^4= |\alpha|^4 + |\beta|^4 + 2|\alpha|^2 |\beta|^2 \leq |\mu_\R(\Phi,\Phi)|^2 + 4 |\alpha|^2|\beta|^2 \leq  |\mu_\R(\Phi,\Phi)|^2 + 4|\mu_\C(\Phi,\Phi)|^2 + 4|\det(\Phi)|^2.$$ 

The second statement follows directly from this inequality applied to $\Phi_0$; one has $\mu(\Phi_0, \Phi_0)=0$ by definition of a $\Z_2$-harmonic spinor, hence $\det(\Phi_0)$ vanishing on $\Yminus\mathcal Z_0$ would imply that $\Phi_0$ vanished there as well, which is forbidden by Assumption \ref{assumption2}. 
\end{proof}

Using $\det(\Phi_0)$, we can now construct the $\text{Spin}^c$ structure used in the gluing. 

\begin{prop}\label{spinc}There exists a unique $\text{Spin}^c$ structure with spinor bundle $S=S_0\otimes_\C \mathcal L$ where $\mathcal L \to Y$ is a complex line bundle such that
\begin{enumerate}
\item[(i)] $\mathcal L|_{\Yminus\mathcal Z_0}\simeq (\ell \otimes_\R \underline \C)$.
\item[(ii)] The trivialization of Lemma \ref{localtrivialization} extend to local trivialization of $\mathcal L$ where the conclusion of Lemma \ref{localtrivialization} continue to hold. 
\item[(iii)] The determinant $\det(S)=\mathcal L^2$ satisfies \be c_1(\mathcal L^2)= - \text{PD}[\mathcal Z_0].\label{c1constraint}\ee
\end{enumerate} 
\end{prop} 

\begin{proof} Define $\mathcal L$ by extending the trivializations of Lemma \ref{localtrivialization} on $N_{r_0}(\mathcal Z_j)$ across $\mathcal Z_j$ for every component of $\mathcal Z_0$. Thus items (i) and (ii) hold automatically. Note, however, that while the expressions for the connection form $id\theta$ and the involution $J=e^{-i\theta}J_0$ they are not smooth across the origin. 

Thus it suffices to show (iii). The local trivializations for $\mathcal L$ as defined induce local trivializations $\mathcal L^2|_{N(\mathcal Z_j)}\simeq N(\mathcal Z_j)\times \C$ of $\mathcal L^2$. In these, the local expression (\refeq{polyhomogeneous1})-(\refeq{polyhomogeneous2}) show that locally $$\det(\Phi_0)=(|c(t)|^2 + |d(t)|^2) \overline z  \ + \ \overline zf(z,\overline z).$$
Assumption \ref{assumption2} implies that $|c(t)|^2 + |d(t)|^2>0$, thus using Lemma \ref{detminant} it follows that $\det(\Phi_0)$ is a smooth section of $\mathcal L^2$ vanishing transversely along $|\det \Phi_0|^{-1}(0)= \mathcal Z_0$. Since $\overline z$ is orientation reversing, and the coordinates for the local expressions respect the orientation of $\mathcal Z_0$ by construction, (iii) follows.  

\end{proof}

\begin{rem}Notice that the orientation convention here is opposite that in \cite{Haydysblowup}. As observed in \cite{Haydysblowup}, the condition (\refeq{c1constraint}) places topological restrictions on the homology class represented by $\mathcal Z_0$. We emphasize that these restrictions follow from the existence of a $\Z_2$-harmonic spinor and do not require an additional assumption.  
\end{rem}

\subsection*{Restatement of Assumptions} \label{restatementsection}
The above proposition completes any global topological statements required for the proofs of Theorems \ref{main}-\ref{mainc}. The remainder of the article is analytic in nature and works exclusively in the local coordinates and trivializations constructed using Lemma \ref{localtrivialization} and Proposition \ref{spinc}. To summarize briefly, the starting point of the local analysis is the following local expressions for the pair $(\Phi_0, A_0)$: 

\begin{eqnarray}
 \Phi_0&=& \begin{pmatrix} c(t) r^{1/2}  \ \ \ \ \  \\ d(t)r^{1/2}e^{-i\theta} \end{pmatrix}\otimes 1+\begin{pmatrix} -\overline d(t) r^{1/2}  \ \ \ \ \  \\ \ \overline c(t)r^{1/2}e^{-i\theta} \end{pmatrix} \otimes j+O(r^{3/2}) \label{finallocalforms1})\\
 A_0 &=& \frac{i}{2}d\theta\label{finallocalforms2}
\end{eqnarray}

\noindent and Assumption \ref{assumption2} in its pragmatic form is the statement that the quantity $K(t)=\tfrac{2}{3}(|c(t)|^2+ |d(t)|^2)$ is nowhere vanishing. Additionally, in a slight abuse of notation we have switched from letting $A_0$ denote the connection to letting it denote the connection form in the local trivialization. Additionally, for notational simplicity we will assume that $\mathcal Z_0$ consists of a single component and that $\epsilon_j=0$ and $B_0$ is the product connection in the trivialization for this component. It is a trivial matter at the end of Section \ref{section8} to eliminate these restrictions.

\section{De-Singularized Configurations}
\label{section4}

In this section we begin main portion of the analysis required for the proof of Theorem \ref{main} and Theorem \ref{mainc}. As explained in the introduction, the construction requires several steps, the first of which is the ``de-singularization'' step 

\begin{center}
\tikzset{node distance=3.6cm, auto}
\begin{tikzpicture}[decoration=snake]
\node(A){$(\Phi_0,A_0)$};
\node(C)[right of=A]{$(\Phi^{h_\e}, A^{h_\e})$};
\draw[->, decorate] (A) to node {$\text{de-sing.}$} (C);
\end{tikzpicture}

\end{center}  of the limiting $\Z_2$-harmonic spinor to an $\e$-parameterized family of nearby smooth configuration, which we undertake in the current section.

The de-singularized configurations are a  family of solutions to an ODE parameterized by $(t,\e)\in S^1\times [0,\e_0)$. More specifically, they are the $S^1$-parameterized family of the two-dimensional fiducial solutions on planes normal to $\mathcal Z_0$ for the leading order term of $\Phi_0$. These two-dimensional fiducial solutions are exact solutions in the Euclidean metric, and first appeared in the context of Hitchin's Equations in \cite{MWWW,FredricksonSLnC}, though their existence may have been known to physicists before that. They are obtained from the limiting $\Z_2$-harmonic spinor by applying a singular complex gauge transformation which solves a degenerate second order ODE. Section 4.1 gives the construction of these two-dimensional fiducial solutions, and follows \cite{MWWW} quite closely. Section 4.2 departs from the approach of \cite{MWWW} and from the holomorphic setting to introduce the parameterized version which yields the de-singularized configurations. Section 4.3 calculates the size of the error term by which these fail to be true solutions. 

As explained at the end of the previous section, we now fix local coordinates on $N_{r_0}(\mathcal Z_0)$ and a trivialization of $S_E|_{N_{r_0}(\mathcal Z_0)}$ in which the local expressions are given by (\refeq{finallocalforms1}-(\refeq{finallocalforms2}). 

\subsection{Dimensional Reduction}
\label{section4.1}
This subsection constructs fiducial solutions on the complex plane. Let $(\C,g_0)$ denote the complex plane equipped with the flat Euclidean metric. The spinor bundle $S_{g_0}\simeq \C \times \C^2$ is identified with the trivial $\C^2$-bundle, and $E \simeq \C \times  {\mathbb H}$ with the trivial quaternionic line bundle. Assume in this case that $B_0$ is the product connection.  We may write a configuration $(\Phi, A)\in \Gamma(S_E)\times \mathcal A_{U(1)}$ as 
\bea
\Phi&=& \begin{pmatrix}\alpha \\ \beta \end{pmatrix}=  \begin{pmatrix}\alpha_1 \\ \beta_1 \end{pmatrix}\otimes 1+  \begin{pmatrix}\alpha_2 \\ \beta_2 \end{pmatrix}\otimes j \\ 
A &=& \tfrac{1}{2}\left(a dz - \overline a d\overline z  \right)
\eea 
where $\alpha,\beta$ are $\mathbb H$-valued functions, and $\alpha_i, \beta_i, a$ are complex-valued functions, and where we associate a connection form $A$ with the connection $d+A$ in the given trivialization. To convert from the complex coordinates to the real ones, we use the isomorphism 
\bea
\Omega^1(i\R)&\overset{\simeq}\lre& \Omega^{1,0}(\C) \\
i(a_x dx + a_y dy)& \mapsto & (a_y + i a_x)dz
\eea
on $1$-forms. Under this association (see also Section 6.1 and Section 3 of \cite{DoanThesisCh4}), the dimensionally-reduced blown-up Seiberg-Witten equations become 

\begin{eqnarray}
\begin{pmatrix} 0 & -2\del_A \\ 2\delbar_A & 0 \end{pmatrix}\begin{pmatrix}\alpha \\ \beta \end{pmatrix}&=& 0 \label{1stSWdr}\\
\mu_\C(\Phi)&=&0\label{2ndSWdr} \\
 F_A + \frac{\mu_\R(\Phi)}{\e^2}&=& 0\label{3rdSWdr}
\end{eqnarray}

\noindent where $(\mu_\R, \mu_\C)=\star_3 \mu$ under the isomorphism $\R\omega \oplus \Omega^{0,1}\simeq i T^* (S^1\times \C)$ so that, explicitly, 

\bea
\mu_\C(\Phi)&=&\left(- \frac{1}{2} \sum_{j=1,2} \overline \alpha_j \beta_j\right) d\overline z\\
\mu_\R(\Phi)&=&\left(- \frac{1}{2} \sum_{j=1,2} |\alpha_i|^2-  |\beta_i|^2\right) i dx\wedge dy. 
\eea

\noindent  Note also that we do not assume that $\|\Phi\|_{L^2}=1$ (in fact we won't even assume it is in $L^2$). As we are looking for local solutions which will later be transferred to the closed manifold $Y$ where the normalization is global, this is irrelevant for our immediate purposes. 

The configuration space of pairs $(\Phi, A) \in \Gamma(C; S_E \oplus \Omega^{1}({i\R}))$ carries an action of the complex gauge group $$\mathcal G^\C:= \{ e^h \ | \ h: \C \to \C\}$$  
by 
$$
e^h\cdot (\alpha,\beta,A)\mapsto (e^h \alpha, e^{-\overline h}\beta, A+ \del \overline h -\delbar h).
$$
\noindent The first two Seiberg-Witten equations (\refeq{1stSWdr} - \refeq{2ndSWdr}) are invariant under the action of $\mathcal G^\C$, while the third (\refeq{3rdSWdr}) is invariant only under the action of the real gauge group ($h\in i\R$).

Consider a $\Z_2$-harmonic spinor which is equal to the leading order term in its asymptotic expansion, so that
\bea
\Phi_0&=&\begin{pmatrix} cr^{1/2}  \ \  \ \  \ \\ d r^{1/2}e^{-i\theta}\end{pmatrix}\otimes 1+ \begin{pmatrix} -\overline dr^{1/2}  \ \  \ \ \  \ \\ \overline cr^{1/2}e^{-i\theta}\end{pmatrix}\otimes j  \\
A_0 &=& \frac{1}{4}\left(\frac{dz}{z}-\frac{d\overline z}{\overline z}\right)
\eea  
where, $|c|^2 + |d|^2>0$. Later we will take $c=c(t)$ and $d=d(t)$ for each fixed $t\in S^1$. We look for an $\e$-parameterized family of solutions satisfying the following ansatz: 

\begin{ansatz}
Assume that there is an $\e$-parameterized family of solutions is complex gauge equivalent to $(\Phi_0, A_0)$ via a complex gauge transformation $h_\e$ for every $\e$. Assume additionally that $h_\e=h_\e(r)$ is rotationally invariant, so that $$(\Phi^{h_\e}, A^{h_\e}):= e^{h_\e(r)}\cdot (\Phi_0, A_0).$$

\noindent We do not assume that $h_\e(r)$ is smooth and allow the possibility that $h_\e(r)\to \infty$ as $r\to 0$. \label{ansatz1}
\end{ansatz}

\begin{prop}
Let $(\Phi^{h_\e}, A^{h_\e})$ be configurations satisfying Ansatz \ref{ansatz1} above. Then $(\Phi^{h_\e}, A^{h_\e})$ satisfy the dimensionally reduced blown-up Seiberg-Witten equations (\refeq{1stSWdr}-\refeq{3rdSWdr}) on $\C$  if and only if the gauge transformations $h_\e(r)$ satisfy the $\e$-parameterized family of ODEs 

\be r^2 \Delta_r h_\e(r)= \frac{9}{4}\frac{K^2 r^3}{2\e^2}\sinh(2h_\e(r)) \label{ODEfamily}\ee 
where $\Delta_r$ is the radial part of the Laplacian $\Delta_r=\tfrac{1}{r}\del_r(r\del_r)$ and $K^2=\tfrac{2}{3}(|c|^2+ |d|^2)$.
\end{prop}

\begin{proof}  To begin, recall the polar coordinate expressions 

\begin{equation}
\delbar_z = \f{1}{2}e^{i\theta}(\del_r+\f{i}{r}\del_\theta)\hspace{1cm}
\del_z = \f{1}{2}e^{-i\theta}(\del_r-\f{i}{r}\del_\theta).
\label{delbarpolar}
\end{equation}

\noindent The gauge-transformed configurations may be written as follows, where the second expression is the definition of the function $f_\e(r)$.   
\bea
\Phi^{h_\e}&=& \begin{pmatrix} e^{h_\e(r)} c r^{1/2} \ \ \ \  \\ e^{-h_\e(r)}dr^{1/2}e^{-i\theta}\end{pmatrix} \otimes 1+ \begin{pmatrix} -e^{h_\e(r)} \overline d r^{1/2} \ \  \ \ \  \  \\ e^{-h_\e(r)}\overline cr^{1/2}e^{-i\theta}\end{pmatrix} \otimes j \\ A^{h_\e}&=&A_0+ \del \overline h_\e(r) -\delbar h_\e(r):=  f_\e(r) \left(\frac{dz}{z}-\frac{d\overline z}{\overline z}\right).
\eea

We may now substitute these expressions into the Seiberg-Witten equations (\refeq{3rdSWdr}). By complex gauge invariance, $\Phi^{h_\e}$ automatically satisfies the $\mu_\C=0$ equation. For the Dirac equation,  the $\otimes 1$ and $\otimes j$ components behave identically so it suffices to calculate the first. This equation becomes

\bea
\slashed D_{A^{h_\e}}\Phi^{h_\e}&=& \begin{pmatrix}0 & -2(\del + (A^{h_\e})_z)\\ 2(\delbar+ (A^{h_\e})_{\overline z})& 0 \end{pmatrix} \begin{pmatrix}ce^{h_\e(r)}r^{1/2}  \ \ \  \\ de^{-h_\e(r)}r^{1/2}e^{-i\theta} \end{pmatrix}\otimes 1 + \ldots \otimes j. 
\eea

\noindent Focusing on the first component, this becomes

\bea
-\left[e^{-i\theta}\left(\del_r-\tfrac{i}{r}\del_\theta \right)+\frac{2f_\e(r)e^{-i\theta}}{r}\right]de^{-h_\e(r)}r^{1/2} e^{-i\theta} &=& \frac{e^{i\theta}e^{h_\e(r)}}{2r^{1/2}}\Big( -2{f_\e(r)}+(r\del_r)h_\e(r)+\tfrac{1}{2}\Big) 
\eea

\noindent and the first factor is non-zero, hence $f_\e(r)$ must satisfy $$f_\e(r)=\tfrac{1}{4} +\tfrac{1}{2}r\del_r h_\e(r).$$
and the second component and the $\otimes j$ term give the same equation.

 For the third equation (\refeq{3rdSWdr}), we compute 
\bea
\mu_\R&=&\left(-\frac{i}{2}\sum_j |\alpha_j|^2-|\beta_j|^2\right) dx \wedge dy \\
&=& -\frac{i}{2}(r|c|^2 e^{2h_\e(r)}-r|d|^2 e^{-2h_\e(r)} +r |d|^2 e^{2h_\e(r)}-r|c|^2e^{-2h_\e(r)})dx \wedge dy \\
&=& -\frac{i9}{4}\frac{K^2 r}{ 2}  \sinh(2h_\e(r)) dx \wedge dy\\
F_{A^{h_\e}}&=& \left((\delbar f) \tfrac{d\overline z\wedge dz}{z}- (\del f) \tfrac{d z\wedge d\overline z}{\overline z} \right) \\
&=& \left(- \tfrac{1}{2}e^{i\theta}\tfrac{1}{z}\del_r f  -\tfrac{1}{2}e^{-i\theta}\tfrac{1}{\overline z}\del_r f \right) dz\wedge d\overline z \\ 
&=& \left(-\tfrac{1}{r}\del_r f\right) dz\wedge d\overline z=  \left(\tfrac{1}{r}\del_r f\right) 2i dx\wedge dy.
\eea
Combining these we obtain the system of ODEs 
\begin{eqnarray}
f_\e(r)&=&\frac{1}{4} +\frac{1}{2}r\del_r h_\e(r)\\
\tfrac{2}{r}\del_r f_\e(r) &=&\frac{9}{4}\frac{K^2 r }{2\e^2}\sinh(2h_\e(r))
\end{eqnarray}
and substitution the first into the second then multiplying by $r^2$ yields the proposition. 
\end{proof}

Up to a constant factor, Equation (\refeq{ODEfamily}) is the same equation obtained for the corresponding situation using Hitchin's equations. It is of Painlev\'e type and it is solved in \cite{MWWW}, Section 3 via the following substitution. It turns out that all the $\e$-parameterized family of solutions are all re-scalings of a single invariant solution. 

Let $\tau= \frac{K}{\e} r^{3/2}$ and $h_\e(r)=G(\tau)$ so that $$r\del_r = \frac{3}{2} \tau \del_\tau$$ \noindent and Equation (\refeq{ODEfamily}) becomes

\be
(\tau \del_\tau)^2 G= \frac{1}{2}\tau^2 \sinh (2G). 
\label{scaledODE}
\ee

This equation admits a distinguished solution which is defined by the two conditions that $h$ has an asymptote at $\tau=0$ and decays to $0$ as $\tau\to \infty$. The next below proposition collects the essential properties of this solution, which we do not prove and instead refer the reader to \cite{MWWW} (see Equation (25) and the accompanying discussion).

First, we change variables once more. It turns out that it is quite confusing to use the variable $\tau$ which depends non-linearly on $r$. We will instead opt for a linear scaling by replacing $\tau$ by $\rho=\tau^{2/3}$.  
 
 \begin{defn}
 Define the {\bf $\e$-invariant length} by $$\rho:= \left(\frac{K}{\e}\right)^{2/3}r.$$
 
 \noindent Then define $H(\rho):=G(\rho^{3/2})$ so that the $\e$-parameterized family of solutions are given by $$h_\e(r)= H(\rho).$$
 \end{defn}

The essential properties of the solution are now expressed in terms of $H(\rho)$. For the proof of these properties, see Lemma 3.3 of \cite{MWWW} (performing the above substitution for $\rho$ into their results).

\begin{prop} \label{Hrhoproperties}
There exists a unique $\e$-parameterized family of solutions $h_\e(r)$ to (\refeq{ODEfamily}) such that $h_\e(r)$ has an asymptote at $r=0$ and decays to $0$ as $r\to \infty$. This family $h_\e(r)= H(\rho)$ are all dilations of a single $\e$-independent function such that $H(\rho)=G(\tau)$ solves  (\refeq{scaledODE}). It satisfies the following properties

\begin{enumerate}
\item$H(\rho)$ is strictly positive and monotonically decreasing. 
\item $H(\rho)$ decays faster than exponentially as $\rho\to \infty$. More specifically, there are constant $C,c, \rho_0$ such that for $\rho \geq \rho_0$ $$H(\rho)\leq C \text{Exp} (-c\rho^{3/2})$$
and similarly for the derivatives of $H(\rho)$. 
\item At $\rho=0$ there is an asymptotic expansion of the form \be H(\rho)=-\log\left(\rho^{1/2}\sum_{j=0}^\infty a_j \rho^{2j}\right)\label{Hrhoexpression}\ee
in which $a_0\neq 0$. In particular, up to leading order $H(\rho)= \frac{}{}\log(\rho^{-1/2})$ so that $$ e^{H(\rho)}= \rho^{-1/2}+ O(1). $$
\item The function $f_\e(r)=\tfrac{1}{4} + \tfrac{1}{2}\rho \del_\rho H(\rho)$ vanishes to second order at $r=0$, and increases monotonically to its limiting value of $\tfrac{1}{4}$. Additionally, $|\tfrac{f_\e(r)}{r} |\leq C \e^{-2/3}$ for all $r$, and the difference of $f_\e(r)$ and $f_\e'(r)$ from their limiting values decreases exponentially, i.e. 
\bea
|f_\e(r)-\tfrac{1}{4}|&\leq& C\rho\text{Exp}(-c\rho^{3/2}) \\ 
|f_\e'(r)| & \leq & C \rho\text{Exp}(-c\rho^{3/2})
\eea  

\noindent and similarly for the higher derivatives.    

\end{enumerate}

\end{prop}

\begin{proof}
The first statement is immediate from the re-scaling above. The first through third items follow immediately from Equation (27) in \cite{MWWW}. The first two statement of the fourth bullet point follow from parts (a)-(c) of Lemma 3.3 in \cite{MWWW}. The exponential bound on $f'_\e(r)$ follows from that on $H(\rho)$ by the equation $\tfrac{2}{r}\del_r f_\e(r) =\tfrac{9}{4}\tfrac{K^2 r }{2\e^2}\sinh(2h_\e(r))$, and the one on $f'_\e(r)$ by the fundamental theorem of calculus. 

\end{proof}

The properties of the function $h_\e(r), f_\e(r)$ above translate into similar properties of the two-dimensional fiducial solutions $(\Phi^{h_\e}, A^{h_\e})$. We state these momentarily for the 3-dimensional case. The radial profiles of the two-dimensional fiducial solutions are plotted below with those of the limiting $\Z_2$-harmonic spinor:

\begin{figure}
\begin{center}
\begin{picture}(200,250)
\put(-120,-1.3){\includegraphics[scale=1.6]{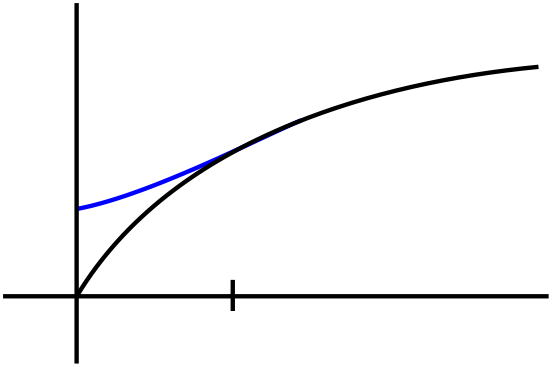}}
\put(115,0){\includegraphics[scale=1.6]{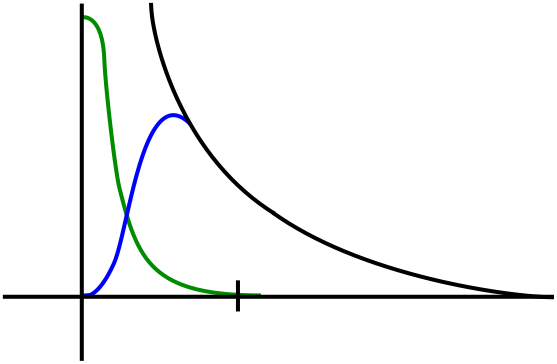}}
\put(20,85){\Large $|\Phi_0|$}
\put(215,69){\Large $|A_0|$}
\put(143,148){\color{OliveGreen}\large $|F_{A^{h_\e}}|$}
\put(170,60){\color{blue}\Large $|{A^{h_\e}}|$}
\put(-85,75){\color{blue}\Large $|\Phi^{h_\e}|$}
\put(-45,8){$\large O(\e^{2/3})$}
\put(190,8){$\large O(\e^{2/3})$}
\put(-125,58){$\large O(\e^{\tfrac{1}{3}})$}
\put(112,90){$\large O(\e^{-\tfrac{2}{3}})$}
\put(112,132){$\large O(\e^{-\tfrac{4}{3}})$}
\put(70,12){$r\to \infty$}
\put(300,12){$r\to \infty$}
\label{Fig1}
\end{picture}
\caption{The radial profiles of the de-singularized configurations and limiting $\Z_2$-harmonic spinor.}
\end{center}
\end{figure}
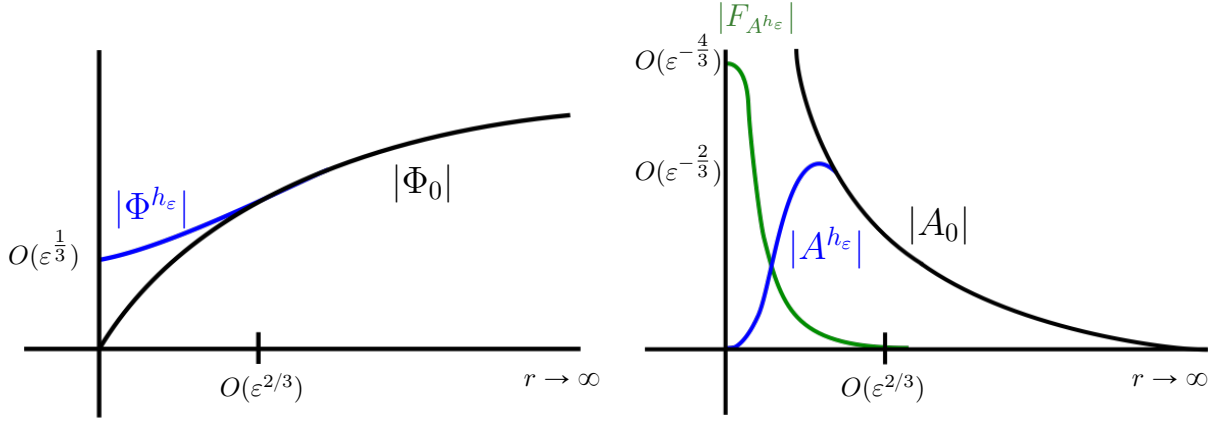
\bigskip 

\subsection{De-Singularization on $Y$}
This subsection introduces the de-singularized configurations on the closed manifold $Y$, which are a $t$-parameterized of the 2-dimensional fiducial solutions of the previous subsection on each plane normal to $\mathcal Z_0$. Here, we work in local coordinates on a tubular neighborhood $N_{\lambda}(\mathcal Z_0)$ of radius $\lambda$ possibly depending on $\e$. 

Returning to the case of full generality when $\Phi_0$ may have higher order terms, write 

\bea
\Phi_0&=&\begin{pmatrix} c(t)r^{1/2}  \ \  \ \  \ \\ d(t) r^{1/2}e^{-i\theta}\end{pmatrix}\otimes 1+ \begin{pmatrix} -\overline d(t)r^{1/2}  \ \  \ \ \  \ \\ \overline c(t) r^{1/2}e^{-i\theta}\end{pmatrix}\otimes j  \ + \ O(r^{3/2})\\
A_0 &=& \frac{i}{2}d\theta = \frac{1}{4}\left(\frac{dz}{z}-\frac{d\overline z}{\overline z}\right)
\eea  

\noindent as before, and let $$K^2(t):=\frac{2}{3}( |c(t)|^2+ |d(t)|^2).$$

\noindent Assumption \ref{assumption2} requires that $K(t)$ is bounded below by a constant greater than 0 depending only on $\Phi_0$.

\begin{defn}
Define the {\bf $t$-dependent  $\e$-invariant length} by $$\rho_t:= \left(\frac{K(t)}{\e}\right)^{2/3}r,$$
and the {\bf de-singularized configurations} by $$(\Phi^{h_\e}, A^{h_\e}):= e^{\chi_\e(r)h_\e(r,t)}\cdot (\Phi_0, A_0)$$

\noindent where $h_\e(r,t)= H(\rho_t)$ and where $\chi_\e(r)$ is a cutoff function equal to 1 on a neighborhood of $r\leq \lambda(\e)$ around $\mathcal Z_0$. In the right hand side, $\cdot$ still denotes the action of the complex gauge group on the normal planes. 
\label{scaleinvariantrho}
\end{defn}

Notice that  since for $r\geq \lambda$, the function $h_\e(r,t)$ is exponentially small in $\e$ (provided $\lambda\geq \e^{2/3}$), the cutoff function changes it in a very minor way. With the cutoff, the de-singularized configurations extend to all of $Y$ by setting them equal to $(\Phi_0, A_0)$ outside the tubular neighborhood $N_\lambda(\mathcal Z_0)$.

 The following properties are retained from the 2-dimensional version: 
 
\begin{prop}
The de-singularized configurations satisfy the following properties: 

\begin{itemize}
\item The configuration $(\Phi^{h_\e}, A^{h_\e})$ is smooth. 
\item The pair converges to the limiting configuration $(\Phi_0, A_0)$ in $C^\infty_{loc}(N_\lambda(\mathcal Z_0))$ exponentially quickly in the sense that there is a constant $c_0$ such that for $r\geq c_0 \e^{2/3}$, one has $$\|\Phi^{h_\e}-\Phi_0\|_{C^k} \leq C_k \e^{-2k/3}\text{Exp}(-\frac{cr^{3/2}}{\e}) \hspace{1cm}\| A^{h_\e}- A_0\|_{C^k}\leq C_k r\e^{-2(k+1)/3} \text{Exp}(-\frac{cr^{3/2}}{\e})$$  
\item There are pointwise bounds $$\tfrac{|\Phi^{h_\e}|}{\e}\geq c \e^{-2/3} \hspace{1cm}  |A^{h_\e}| \leq C \e^{-2/3} \hspace{1cm}|\nabla_{A_{h_\e}}\Phi^{h_\e}|\leq \frac{C}{r^{1/2}}.$$ 
and $|\Phi^{h_\e}|$ is monotonically increasing in $r$ for small $r$. 
\end{itemize}
\label{desingularizedproperties}
\end{prop}

\begin{proof} For smoothness in the normal directions, notice the expansion of (\ref{Hrhoexpression}) from Proposition \ref{Hrhoproperties} shows $H(\rho)=\log(a_0 \rho^{1/2}(1+O(\rho)^2))$ where $O(\rho^2)$ contains only even powers and is therefore smooth. It follows that $e^{h_\e}= r^{1/2}(1+O(r^2))$ where the $O(r^2)$ is also smooth, thus the leading order term of $\Phi^{h_\e}$ is $$\begin{pmatrix}e^{h_\e(r)} c(t) r^{1/2} \\ e^{-h_\e(r)}d(t)r^{1/2}e^{-i\theta}  \end{pmatrix}\otimes 1 + \begin{pmatrix}-e^{h_\e(r)} \overline d(t) r^{1/2} \\ e^{-h_\e(r)}\overline c(t)r^{1/2}e^{-i\theta}  \end{pmatrix}\otimes j $$ 

\noindent is smooth and constant at $r=0$ in the top component, and vanishes like $\overline z$ at $r=0$ in the second. The same applies to the higher order terms from Proposition \ref{asymptoticexpansion}  which only contain additional factors of $z^m\overline z^n$.  Similar considerations show that $\rho \sinh(2H(\rho))$ is smooth and vanishes to second order at the origin, which implies the same for $f_\e(r)$ thus $A^{h_\e}=f_\e(r) \left(\frac{dz}{z}- \frac{d\overline z}{\overline z}\right)$ is smooth and vanishes to first order. For smoothness in the $t$-directions, notice that   

\be \d{}{t} e^{H(\rho_t)}= e^{H(\rho_t)} \d{H}{\rho_t} \d{\rho_t}{t}= e^{H(\rho_t)} \d{H}{\rho_t} \frac{2K'(t)}{3K(t)}\rho_t \label{rhotderiv}\ee

\noindent is again smooth across the origin since $\rho_t\tfrac{\del H}{\del \rho_t}\sim \text{const}$ is also smooth across the origin. 

The second and third bullet points follow directly from rescaling the corresponding properties of $H(\rho)$ from Proposition \ref{Hrhoproperties}, and using the expression (\ref{rhotderiv}) to bound the $t$-derivatives.  
 
\end{proof}
\subsection{Calculation of Error}
Denote by $E^{(0)}_\e$ the error by which the de-singularized configurations fail to solve the (un-renormalized) Seiberg-Witten equations. That is, 

$$SW(\tfrac{\Phi^{h_\e}}{\e}, A^{h_\e})= E^{(0)}_\e.$$

\noindent The superscript is present to indicate that this is initial error in an eventual iteration process. 
\begin{lm}

Let $\gamma<<1$ be a small positive constant. There is an $\e_0$ such that for $\e<\e_0$, the error $E^{(0)}_\e$ satisfies $$\|E^{(0)}_\e\|_{L^2(Y)}\leq C\e^{-\gamma},$$

\noindent where $C$ is constant independent of $\e$. Moreover, the error is exponentially concentrated along $\mathcal Z_0$ in the sense that there is a constant $c_0$ such that for $r\geq c_0\e^{2/3}$ there is a pointwise bound $$\|E^{(0)}_\e\|_{C^0}\leq \frac{C}{\e^2}\text{Exp}(-\frac{cr^{3/2}}{\e}).$$
\label{errorcalculation}
\end{lm}

\begin{rem}
One should think of, say, $\gamma=10^{-6}$. The purpose of this small constant is so that in the region where $r\geq \e^{2/3-\gamma}$, the difference of the de-singularized configurations from the limiting $\Z_2$-harmonic spinor $(\Phi_0, A_0)$ is exponentially small. It could just as easily be a power of $\log(\e)$, and is in fact probably not necessary at all, but the proof is more intricate.  
\end{rem}
\begin{proof}
The second statement is immediate from the exponential decay properties in the second bullet point of (\refeq{desingularizedproperties}) and the fact that $(\Phi_0, A_0)$ solves the Seiberg-Witten equations in the region $r\geq c_0 \e^{2/3}$. 

Write $E_\e^{(0)}= (E_\e', E_\e'')$ for the spinor and form components of the error respectively, so that 

\bea
\slashed D_{A^{h_\e}} \tfrac{\Phi^{h_\e}}{\e} &=& E_\e' \\ 
\star F_{A^{h_\e}} + \frac{\mu(\Phi^{h_\e}, \Phi^{h_\e})}{\e^2}&=& E_\e''.
\eea

\noindent We calculate the error in two regions, the ``interior'' region $\text{Int}=\{r\leq \e^{2/3-\gamma'}\}$ where $\gamma'=\gamma/10$ for $\gamma$ as in the statement of the lemma, and the ``exterior'' region  $Y-\text{Int}$ where $r\geq \e^{2/3-\gamma'}$, so that $$\|E_{\e}^{(0)}\|_{L^2}\leq \|E_{\e}^{(0)}\|_{L^2(\text{Int})}+ \|E_{\e}^{(0)}\|_{L^2(Y-\text{Int})}.$$

The boundedness in the exterior region is immediate given the exponential decay from the first sentence of the proof. For the interior region, the triangle inequality implies that: 

\bea
\|E_\e'\|_{L^2(\text{Int})} & =& \|\slashed D_{A^{h_\e}}\tfrac{\Phi^{h_\e}}{\e}\|_{L^2(\text{Int})}\\ &\leq & \| \slashed D^0_{A^{h_\e}}\tfrac{\Phi^{h_\e}}{\e}\|_{L^2(\text{Int})}+  \| (\slashed D-\slashed D^0)\tfrac{\Phi^{h_\e}}{\e}\|_{L^2(\text{Int})}+\| \text{cl}(B_0)\tfrac{\Phi^{h_\e}}{\e}\|_{L^2(\text{Int})} +  \| (\text{cl}-\text{cl}^0)A^{h_\e}\tfrac{\Phi^{h_\e}}{\e}\|_{L^2(\text{Int})} \\ 
\|E_\e''\|_{L^2(\text{Int})} &=& \| \star F_{A^{h_\e}} + \tfrac{\mu(\Phi^{h_\e}, \Phi^{h_\e})}{\e^2}\|_{L^2(\text{Int})}  \\ & \leq & \| \star_0 F_{A^{h_\e}} + \tfrac{\mu^0(\Phi^{h_\e}, \Phi^{h_\e})}{\e^2}\|_{L^2(\text{Int})}+   \| (\star-\star_0) F_{A^{h_\e}} \|_{L^2(\text{Int})} +\|  \tfrac{(\mu-\mu^0)(\Phi^{h_\e}, \Phi^{h_\e})}{\e^2}\|_{L^2(\text{Int})}
\eea
\noindent where $\slashed D^0, \star_0, \mu^0, \text{cl}^0$ denote the Dirac operator, hodge star, moment map, and Clifford multiplication in the product metric.

Investing the first term, recall that by definition of the de-singularized solutions these solve the leading order term in the product metric for the 2-dimensional Dirac operator, hence  

\bea
 \| \slashed D^0_{A^{h_\e}}\tfrac{\Phi^{h_\e}}{\e}\|_{L^2(\text{Int})} &\leq & \|\sigma_t \del_t \tfrac{(\Phi^{h_\e})}{\e}\|_{L^2(\text{Int})}+  \|(\slashed D^0_{A^{h_\e}})^\C \tfrac{(\Phi^{h_\e})^{h.o.}}{\e}\|_{L^2(\text{Int})}\\\
 \eea
 where $(\Phi^{h_\e})^{h.o.}$ denotes the desingularization of the $O(r^{3/2})$ terms, and $\slashed D^\C$ denotes the two-dimensional Dirac operator (\refeq{1stSWdr}).  Since there are bounds $|\Phi|, |\del_t \Phi_0|< C r^{1/2}$,  the first of these terms is bounded by 
 \bea 
  &\lesssim & \frac{1}{\e}\left(\int_{S^1}\left(\int  (|\del_t e^{H(\rho_t)}| + |\del_t e^{-H(\rho_t)}|)^2 |\Phi_0|^2  + (e^{H(\rho_t)}+ e^{-H(\rho_t)})^2  |\del_t \Phi_0|^2 ) r dr d\theta \right) dt \right)^{1/2}\\
 &\lesssim & \frac{1}{\e}\left(\int_{S^1}\left(\int  (|\del_t e^{H(\rho_t)}| + |\del_t e^{-H(\rho_t)}|)^2r  + (e^{H(\rho_t)}+ e^{-H(\rho_t)})^2  r ) r dr d\theta \right) dt \right)^{1/2}\\
 &\lesssim &  \frac{1}{\e}\left(\int_{S^1} \left(\frac{\e}{K}\right)^{2}\left(\int_{\rho_t\leq C\e^{-\gamma'}}  (e^{H(\rho_t)}+ e^{-H(\rho_t)})^2   \Big | \d{H}{\rho_t}\frac{2\dot K(t)}{3K(t)}\rho_t\Big |^2\rho_t^2 d\rho_t d\theta \right) dt \right)^{1/2}\\
 &\lesssim& \text{Const}
\eea
since the integrand is a bounded function of $\rho_t$ at the origin (by the same reasoning as in the proof of Proposition \ref{desingularizedproperties}) and decays exponentially as $\rho_t\to \infty$. For the second term, using complex gauge-invariance, one has 

\bea
&\lesssim&\frac{1}{\e}\left( \int_{S^1} \left(\int_{r \leq \e^{2/3-\gamma'}} | e^{H(\rho_t)} +e^{-H(\rho_t)}|^2 |\nabla_{A_0} (\Phi_0)^{h.o.}|^2 rdr d\theta \right)dt \right)^{1/2} \\
& \lesssim & \frac{1}{\e}\left( \int_{S^1} \left(\frac{\e}{K(t)}\right)^{2} \left(\int_{\rho_t\leq \e^{-\gamma'}} | e^{H(\rho_t)} +e^{-H(\rho_t)}|^2 \rho_t^2  d \rho_t d\theta \right)dt \right)^{1/2} \\
&\lesssim& C \e^{-\gamma}
\eea
since $|\nabla_{A_0}(\Phi_0)^{h.o.}|\leq C r^{1/2}$.

The other terms are similar using the fact that $g-g^0=O(r)$ where $g^0$ is the product metric: for each, one rescales to the $\rho_t$ coordinates, collects powers of $\e$ and observes that the rescaled integrand is a fixed integrable function of $\rho_t$. 
\end{proof}
\section{Re-scaling}

\label{section5}

To prove Theorem \ref{main} we must correct the de-singularized configurations to fiducial solutions on a tubular neighborhood $N_{\lambda(\e)}(\mathcal Z_0)$ we must solve the non-linear equation  \be (\mathcal L_{(\Phi^{h_\e}, A^{h_\e}, \e)} + Q)(\ph_\e, a_\e) = E^{(0)}_\e.\label{non-lineartosolve}\ee

\noindent The solvability of this equation---and therefore the conclusion of Theorem \ref{main}---follows from a standard application of the Implicit Function Theorem after showing a version of the following statement: 

\bigskip 

\begin{minipage}{5.8in}
{\it Theorem: In the proper context, the linearization of the Seiberg-Witten Equations at the de-singularized configurations $\mathcal L_{(\Phi^{h_\e}, A^{h_\e}, \e)}$ is invertible, and the norm of the inverse is suitably controlled as $\e\to 0$.} 
\end{minipage} 

\bigskip 

\noindent The precise version of this theorem, wherein the meaning of ``proper context'' and ``suitably controlled'' are clarified, is given in Theorem \ref{invertibleL} in Section 7. 

The remainder of the paper is devoted to the set-up and proof of Theorem \ref{invertibleL}. In the present section, we define weighted Sobolev spaces which provide the functional-analytic setting for the orem, and begin the study of the  linearization in the model case that the metric on $Y$ is a product near $\mathcal Z_0$.  

\subsection{Function Spaces} 
\label{subsection5.1}
Let $N_{\lambda}(\mathcal Z_0)$ be the tubular neighborhood of $\mathcal Z_0$ of radius $\lambda$. Eventually, $\lambda=\lambda(\e)$ will depend on the parameter $\e$ as in the statements of Theorem \ref{main}-\ref{mainc}. .

We now define a family of weighted Sobolev spaces, which naturally arise from the Weitzenb\"ock formula in Section \ref{section2.4}. To this end, let $R_\e$ denote a weight function given by \be R_\e= \sqrt{\kappa^2 \e^{4/3} + r^2}\label{Redef}\ee

\noindent where $r=\text{dist}(-, \mathcal Z_0)$ and  $$\kappa= \frac{1}{\underset{t\in S^1}{\min}  \  K(t)^{2/3}}$$
 on a tubular neighborhood $N_{\lambda_0}(\mathcal Z_0)$ for some $\e$-independent constant $\lambda_0$, and smoothing off so that $R_\e= \text{constant }$ outside this neighborhood. This weight function is approximately equal to $r$ for $r>O(\e^{2/3})$, and for $r\sim \e^{2/3}$ it levels off so that it is globally bounded below by a constant times $\e^{2/3}$.   In the invariant scale, this leveling off occurs at $\rho_t=O(1)$. Furthermore, taking the minimum over $S^1$ ensures the weights satisfy $$\left(\frac{\e}{K(t)}\right)^{4/3}\frac{1}{R_\e^2}\leq \frac{1}{{1+\rho_t^2}}$$ for every $t$.

Consider the norms on sections of $(\ph,a)\in \Gamma(S_E) \oplus (\Omega^0\oplus \Omega^1)(i\R)$ given by 

\begin{eqnarray}
\|(\ph,a)\|_{H^1_{\e}}&:=& \left(\int_{N_\lambda(\mathcal Z_0)} | \nabla \ph|^2 + |\nabla a|^2 + \frac{|\ph|^2}{R_\e^2}+ \frac{|\mu(\ph, \Phi^{h_\e})|^2}{\e^2} + \frac{|a|^2 |\Phi^{h_\e}|^2}{\e^2} \ dV \right)^{1/2}\label{H1norm}\\
\|(\ph,a)\|_{L^2}&:=& \left(\int_{N_\lambda(\mathcal Z_0)} |\ph|^2 + |a|^2 \ dV \right)^{1/2}\label{L^2norm}
\end{eqnarray}
 
\noindent in which $\nabla$ and $dV$ denote the connection formed from the Levi-Civita connection, the spin connection, and the connection $B$ on $E$, and $dV$ the Riemannian volume form.

We then define 

\begin{defn}
Let \bea H^1_{\e}(N_\lambda(\mathcal Z_0))&=& \{(\ph,a) \ | \ \|(\ph,a)\|_{H^1_{\e}(N_\lambda(\mathcal Z_0))}<\infty \} \\ 
L^2(N_\lambda(\mathcal Z_0))&=& \{(\ph,a) \ | \ \|(\ph,a)\|_{L^2(N_\lambda(\mathcal Z_0))}<\infty \} \eea 
denote the Hilbert spaces of sections on which the above norms are finite, equipped with the inner products arising from the polarizations of the respective norms (\refeq{H1norm}-\refeq{L^2norm}). When it is clear from the context, we will omit the reference to the domain $N_\lambda(\mathcal Z_0)$ from the notation. 
\end{defn}

Note that the operator $\mathcal L_{(\Phi^{h_\e}, A^{h_\e},\e)}$ is uniformly bounded (in $\e$) on these spaces. These norms are natural in the sense all but the middle term of the norm arise from the positive terms of the Weitzenb\"ock formula for $\mathcal L_{(\Phi^{h_\e}, A^{h_\e}, \e)}$ by omitting the cross term $\tfrac{1}{\e}\mathfrak B$. Since $\Phi^{h_\e}\sim \Phi_0$ outside a $\rho_t\sim 1$ neighborhood of $\mathcal Z_0$, the fourth term gives the $\ph^{\text{Im}}$ component a stronger weight that $\ph^\text{Re}$.  

We also have the following weighted versions of the above spaces: let $\nu \in \R$ be a real number. Then we define the $(\e,\nu)$-weighted norms by

\bea
\|(\ph,a)\|_{H^1_{\e,\nu}}&:=& \left(\int_{N_\lambda(\mathcal Z_0)} \left(| \nabla \ph|^2 + |\nabla a|^2 + \frac{|\ph|^2}{R_\e^2}+ \frac{|\mu(\ph, \Phi^{h_\e})|^2}{\e^2} + \frac{|a|^2 |\Phi^{h_\e}|^2}{\e^2}\right) R_\e^{2\nu} \ dV \right)^{1/2}\\
\|(\ph,a)\|_{L^2_\nu}&:=& \left(\int_{N_\lambda(\mathcal Z_0)}\left( |\ph|^2 + |a|^2\right) R_\e^{2\nu} \ dV \right)^{1/2}
\eea
and 

\begin{defn}
 \bea H^1_{\e,\nu}(N_\lambda(\mathcal Z_0))&=& \{(\ph,a) \ | \ \|(\ph,a)\|_{H^1_{\e,\nu}(N_\lambda(\mathcal Z_0))}<\infty \} \\ 
L^2_\nu(N_\lambda(\mathcal Z_0))&=& \{(\ph,a) \ | \ \|(\ph,a)\|_{L^2_\nu(N_\lambda(\mathcal Z_0))}<\infty \} \eea 
to be the spaces of sections on on which these norms are finite. 
\end{defn}

\begin{rem}
Since $N_\lambda(\mathcal Z_0)$ is compact, there is an equivalence of norms showing $H^1_\e= L^{1,2}(N_\lambda(\mathcal Z_0))$ for every $\e$, just not uniformly in $\e$. 
\end{rem}
\subsection{The Model Operator}
The operator $\mathcal L_{(\Phi^{h_\e}, A^{h_\e},\e)}$ can be treated as a small perturbation of the  operator in the case that the metric near $\mathcal Z_0$ is a product and $\Phi_0$ has only leading order terms. In this section we begin the study of this model case. The general case is deduced easily from this one in Section  \ref{genmetricssection}.

Thus assume from here until Section \ref{genmetricssection}, that for some $\e$-independent constant $\lambda_0$, the pair $(g,B)$ and $\Z_2$-harmonic spinor $(\Phi_0, A_0)$ are given by 

$$g= dt^2 + dx^2 + dy^2 \hspace{2cm} B_0=\text{d} \text{ is the product connection on $\mathbb H$}$$\noindent and $$\Phi_0=\begin{pmatrix} c(t)r^{1/2} \\ d(t)r^{1/2}e^{-\theta}\end{pmatrix}\otimes 1+ \begin{pmatrix}- \overline d(t)r^{1/2} \\ \overline c(t)r^{1/2}e^{-\theta}\end{pmatrix}\otimes j \hspace{2.4cm} A_0 = \frac{i}{2}d\theta$$

\noindent in the geodesic normal coordinates and trivialization of $S_E$ on $N_{\lambda_0}(\mathcal Z_0)$. Moreover, we assume that $\chi_\e(r)h_\e(r)=h_\e(r)$ on $N_{\lambda_0}(\mathcal Z_0)$.

\begin{defn}
The operator

$$\mathcal L^{h_\e}=\begin{pmatrix} \slashed D_{A^{h_\e}}  & \mathcal \gamma(\_)\tfrac{\Phi^{h_\e}}{\e} \\ \tfrac{\mu (\_, \Phi^{h_\e})}{\e} & \bold d\end{pmatrix} $$
defined using the above data is referred to as the  {\bf model Linearization} at the de-singularized configurations. It may be written as $$\mathcal L^{h_\e}= \sigma(dt) \del_t + \mathcal N_t$$
where $\sigma$ is the principal symbol and $\mathcal N_t$ is the {\bf Normal Operator } defined by

\begin{equation}
\mathcal N_t=\begin{pmatrix} \slashed D^\C_{A^{h_\e}}  & \mathcal \gamma(\_)\tfrac{\Phi^{h_\e}}{\e} \\ \tfrac{\mu(\_, \Phi^{h_\e})}{\e} & \bold d^\C\end{pmatrix}
\label{normaloperator}
\end{equation} 
where $\slashed D^\C, \bold d^\C$ are the operators on the normal disks $\{t\}\times D_{\lambda_{0}}$. Explicitly, writing $a=(a_0, a_tdt + a_x dx + a_ydy)$,  

$$\slashed D_{A^{h_\e}}^\C=\begin{pmatrix} 0 & -2\del_{A^{h_\e}} \\ 2\delbar_{A^{h_\e}} & 0 \end{pmatrix}  \hspace{1cm} \bold d^{\C}=\begin{pmatrix} 0 &0 & -d^{\star}  \\ 0 &0 & d \\  -d & -\star d &  0 
\end{pmatrix}\begin{pmatrix} a_0 \\ a_t  \\   a_xdx+ a_ydy \end{pmatrix}\label{dsigma}$$

\noindent where $d, \star$ now denote the operators on $D_{\lambda_0}$. 
\label{modeldef}
\end{defn}

Without changing notation, we continue to use $H^1_{\e,\nu}$ and $L^2_\nu$ to refer to the Hilbert spaces defined using the data in the model case. 

\subsection{Re-Scaling}

Since the $\e$-paramterized family of de-singularizing complex gauge transformations $h_\e(r)=H(\rho)$ for fixed $t$ are all dilations of a single $\e$-invariant function, the $\e$-parameterized family of normal operators $\mathcal N_t$ are likewise all scalings of a single $\e$-independent operator. In this subsection, we rescale the normal disks to $\mathcal Z_0$ to express $\mathcal N_t$ in terms of the scale-invariant coordinate $\rho_t$ of Definition \ref{scaleinvariantrho}. Throughout this subsection, we omit the $t$-dependence from the notation, and use $\widehat{\mathcal N}$ to denote the rescaled version of $\mathcal N$ and similarly for other structures.  

Let \be\bold r:= \left( \frac{K}{\e}\right)^{2/3}\lambda\label{rescaledradius}\ee

\noindent so that the scaling \bea \Upsilon_{\e}: D_\bold r &\to &D_\lambda \\ \rho &\mapsto&  r =\left(\frac{\e}{K}\right)^{2/3}\rho\eea
defines an isomorphism between the disk in the rescaled coordinate $\rho$ and that in the original coordinate $r$. We consider the re-scaled disk $D_\bold r$ equipped with polar coordinates $(\rho,\theta)$, with accompanying euclidean and complex coordinates $(x', y')$ and $(w, \overline w)$ respectively. It is considered with the Euclidean metric $(dx')^2 + (dy')^2$ (as opposed to the pullback of the Euclidean metric on $D_\lambda$). Since only the normal coordinates are scaled, the $dx, dy$ components of $a$ scale as $1$-forms, while the remaining components and spinor scale as functions. Explicitly, the pullback of forms and spinors from $D_\lambda$ are related by

\bea
\Upsilon^*_\e(dx)= \left(\frac{\e}{K}\right)^{2/3}dx' \  \hspace{1cm}\Upsilon^*_\e(\ph(r))&=& \ph(\rho(r)).  \\  \ \Upsilon^*_\e(dy)= \left(\frac{\e}{K}\right)^{2/3} dy'  \hspace{1cm} \Upsilon^*_\e(a_0(r))&=& a_0(\rho(r))\\
  \hspace{1cm} \Upsilon^*_\e(a_t(r))&=& a_t(\rho(r)).
\eea

\begin{defn} We define the {\bf Scale-invariant Configurations} by

\bea
\Phi^H:= \frac{1}{K}  \left(\frac{K}{\e}\right)^{1/3} \Upsilon_\e^* \Phi^{h_\e}  &=&\frac{1}{K}\left(\begin{pmatrix} e^{H}c \rho^{1/2} \\ e^{-H}d \rho^{1/2}e^{-i\theta}\end{pmatrix}\otimes 1+\begin{pmatrix} -e^{H}\overline d \rho^{1/2} \\ e^{-H}\overline c \rho^{1/2}e^{-i\theta}\end{pmatrix}\otimes j\right)\\
A^{H}= \Upsilon^*(A^{h_\e}) &=& \left(\frac{1}{4}+\frac{1}{2}\rho\del_\rho H(\rho)\right)\left(\frac{dw}{w}-\frac{d\overline w}{\overline w}\right), 
\eea

\noindent And the {\bf Scale-invariant Normal Operator} by 

\begin{equation}
\widehat {\mathcal N}_t:= \begin{pmatrix}\begin{pmatrix} 0 & -2\del_{A^H} \\ 2\delbar_{A^H}&0 \end{pmatrix} & \gamma(\_)\Phi^H \\ \mu(\_, \Phi^H) &
\bold d^\C \end{pmatrix}
\end{equation}
where $\del_{A^H}, \gamma, \mu, \bold d^\C$ are defined using the Euclidean metric in $(x',y')$ coordinates. 
\end{defn}

The rescaling $\Upsilon_\e$ extends to a map of sections, which is weighted to be a pointwise isometry:   
$$\overline \Upsilon_\e: \Gamma( D_\lambda;  S_E \oplus (\Lambda^0 \oplus \Lambda^0\oplus \Lambda^1)(i\R))\to \Gamma( D_\bold r;  S_E \oplus (\Lambda^0 \oplus \Lambda^0\oplus \Lambda^1)(i\R))$$
by $$\overline \Upsilon_\e(\ph, a_0, a_t, a_1):= \left(\Upsilon_\e^* \ph, \Upsilon_\e^* (a_0, a_t), \left( \frac{K}{\e}\right)^{2/3} \Upsilon^*_\e a_1\right)$$

\noindent so that, explicitly 
\bea
\ph(r) &\mapsto& \ph(\rho)\\
(a_0(r), a_t(r)) &\mapsto& (a_0(\rho), a_t(\rho) ) \\ a_x dx + a_y dy &\mapsto & a_x(\rho) dx' + a_y(\rho) dy'.
\eea
This map preserves the pointwise norms since $dx, dx'$ are unit norm in the Euclidean metrics on $D_\lambda, D_\bold r$ respectively. 
\begin{remark}
It's equivalent to use $\Upsilon_\e^*$ in place of $\overline \Upsilon_\e^*$ and define the Scale-invariant Normal Operator using the pullback metric in place of the Euclidean metric $g_\rho$. Re-scaling the 1-form components by hand, however, makes the operator manifestly $\e$-independent. 
\end{remark}
\begin{prop}
The Normal operator and the Scale-invariant Normal operator are related by 
$$\overline \Upsilon_\e\mathcal N_t(\ph, a)= \left(\frac{K(t)}{\e}\right)^{2/3} \widehat{\mathcal N}_t(\overline \Upsilon_\e(\ph,a)).$$
\label{scalingprop}
\end{prop}

\begin{proof} Changing $\nabla_{r}$ derivatives to $\nabla_{\rho}$ one has $$\nabla_x= \left(\frac{K}{\e}\right)^{2/3}\nabla_{x'}$$
and similarly for $y$. Likewise, for the connection, Clifford multiplication by the unit forms $dw$ and $dz$ is identical in the two Euclidean metrics, and $$\frac{1}{z}= \left(\frac{K}{\e}\right)^{2/3} \frac{1}{w}$$ and identically for $\overline w$. Thus the diagonal terms scale as claimed. For the off-diagonal terms, one has 

\bea \frac{\Phi^{h_\e}}{\e}&=& \frac{1}{\e}\left(\begin{pmatrix} e^{H}c r^{1/2} \\ e^{-H}d r^{1/2}e^{-i\theta}\end{pmatrix}\otimes 1+\begin{pmatrix} -e^{H}\overline d r^{1/2} \\ e^{-H}\overline c r^{1/2}e^{-i\theta}\end{pmatrix}\otimes j\right) \\ & =& \frac{1}{\e} \left(\frac{\e}{K}\right)^{1/3}\left(\begin{pmatrix} e^{H}c \rho^{1/2} \\ e^{-H}d \rho^{1/2}e^{-i\theta}\end{pmatrix}\otimes 1+\begin{pmatrix} -e^{H}\overline d r^{1/2} \\ e^{-H}\overline c \rho^{1/2}e^{-i\theta}\end{pmatrix}\otimes j\right) \\
&=& \left(\frac{K}{\e}\right)^{2/3}\Phi^H \eea
by definition of $\Phi^H$, hence the off-diagonal terms scale in the same way.  
\end{proof}

\subsubsection*{Scale-Invariant Hilbert Spaces}

We also define a scale-invariant version of the Hilbert space $ H^1_{\e}$. Let $\bold r>>0$ continue denote the re-scaled radius as in (\refeq{rescaledradius}) (eventually, we will take $\bold r=O(\e^{-1/6}))$. Let $R=\sqrt{1+\rho^2}$ denote a weight function 

\begin{defn} The {\bf scale-invariant norm} on sections of $S_E\oplus (\Omega^0\oplus \Omega^1)$ on $D_\bold r$  is given by 
\medskip
\begin{equation}
\|(\ph, a)\|_{\widehat{ H}^1(D_\bold r)}:=\left( \int_{D_\bold r} |\nabla \ph|^2 +|\nabla a|^2 + \frac{|\ph|^2}{R^2}+ |\mu(\ph,\Phi^H)|^2 + |a|^2 |\Phi^H|^2  \ dV \right)^{1/2}
\label{Hhatnorm}
\end{equation}
and the {\bf scale-invariant $L^2$ norm} by 

\begin{equation}
\|(\ph, a)\|_{L^2(D_\bold r)}:=\left( \int_{D_\bold r} |\ph|^2 + |a|^2  \ dV \right)^{1/2}
\label{L2hatnorm}
\end{equation}
\medskip

\noindent where, $dV$ denotes the Euclidean volume form and $\nabla$ the product connection induced by the chosen trivialization defined using structures defined by the scale invariant coordinate $\rho$. As in the unscaled case, there is an equivalence of norms so that $H^1(D_\bold r)= L^{1,2}(D_\bold r)$ for every $\bold r$, just not uniformly so. 
\end{defn}

\medskip 
The scale-invariant norm is the two-dimensional and  scale-invariant version of the $\e$-weighted norm of $ H^1_{\e}$ in the following sense. The two-dimensional version of the $H^1_\e$-norm, i.e. assuming that all configurations and $\Phi_0$ are $t$-invariant, is given by the positive square root of 
	
	\begin{equation}\int_{D_\lambda}  |\nabla\ph|^2 + |\nabla a|^2 + \frac{|\ph|^2}{R_\e^2} + \frac{|\mu(\ph,\Phi^{h_\e})|^2}{\e^2} + \frac{|a|^2|\Phi^{h_\e}|^2}{\e^2} \ dV.\label{2dimhnorm}\end{equation} \noindent Then $\overline \Upsilon_\e$ provides an equivalence (of $t$-invariant configurations)

$$\|\overline \Upsilon_\e (\ph,a)\|_{\widehat { H}^1}\simeq \|(\ph,a)\|_{ { H}^1_\e}.$$
i.e. the two are bounded by a universal constant times the other once $\e<<1$. The equivalence is only not an isometry because $R_\e$ was defined in Equation (\refeq{Redef}) to be a minimum over $t$: if we had defined $R_\e$ to be a $t$-dependent weight then the above expression for the norms is an equality. 

The $L^2$ norm is not scale-invariant: instead one has $$\|\overline \Upsilon_\e(\ph,a)\|_{L^2}= \left(\tfrac{K(t)}{\e}\right)^{2/3} \|(\ph,a)\|_{L^2}. $$

\noindent Combining this with the relation 

$$ \overline{\Upsilon}_\e \mathcal N(\ph,a)=\left(\tfrac{K(t)}{\e}\right)^{2/3} \widehat{\mathcal N}(\overline\Upsilon_\e(\ph,a)) $$
from Proposition \ref{scalingprop}, we see the diagram 
\begin{center}
\tikzset{node distance=2.5cm, auto}
\begin{tikzpicture}
\node(A){$ H^1_\e(D_\lambda)$};
\node(B)[below of=A]{$\widehat{ H}^1(D_\bold r)$};
\node(C)[right of=A]{$L^2(D_\lambda, dV_r)$};
\node(D)[right of=B]{$L^2(D_\bold r, dV_\rho)$};
\draw[->] (A) to node {$\overline{\Upsilon}_\e$} (B);
\draw[->] (A) to node {$\mathcal N$} (C);
\draw[->] (B) to node {$\widehat{\mathcal N}$} (D);
\draw[->] (C) to node {$\overline{\Upsilon}_\e$} (D);
\label{diagramA}\end{tikzpicture}\end{center}
commutes. 

\section{The Normal Operator}
\label{section6}

\label{normaloperator}
This section analyzes the scale-invariant normal operator $\widehat{\mathcal N_t}$ as a boundary value problem on disks in $\R^2$. Explicitly, writing a pair $(\ph,a)$ as $\ph=(\alpha,\beta)$ and $a=(a_0, a_tdt + a_x d\hat x + a_y d\hat y)$, the scale invariant Normal operator is given by 

\begin{equation}\widehat {\mathcal N_t}(\alpha,\beta,a)= \begin{pmatrix} \begin{pmatrix} 0 & -2\del_{A^H} \\ 2\overline \del_{A^H} & 0 \end{pmatrix} & \gamma( \_ \ )\Phi^H  \\ \begin{pmatrix}\mu_\R( \_ , \Phi^H) \\ \mu_\C( \_ , \Phi^H)\end{pmatrix} & \begin{pmatrix} 0 &  -d^\star +d  \\  -d -\star d   & 0  \end{pmatrix}
\end{pmatrix} \begin{pmatrix}\alpha \\ \beta \\ a_0 + a_t dt \\ a_x d\widehat x + a_y d\widehat y \end{pmatrix}   \end{equation}
\smallskip

\noindent where all structures are defined using the Euclidean metric on $\R^2$.  
 The main result is the below Proposition~\ref{invertibleN}, which identifies the kernel of $\widehat{\mathcal N}$ and shows that on its orthogonal complement the inverse is bounded uniformly in the scale-invariant norms. The proof of Proposition \ref{invertibleN} requires several steps, and parts of the proof are somewhat subtle. 
 
The two key ingredients of the proof are a holomorphic description of $\widehat {\mathcal N}$ in the case that $a_0=a_t=0$, which gives control of the operator on disks of fixed radius, and the Weitzenb\"ock formula, which gives control of the operator for large radii. The subtlety of the proof lies in making these ideas work in congress. 
This Section is organized as follows. Section 6.1 provides some brief set-up and gives the precise statement for the properties of $\widehat{\mathcal N}$. Section 6.2 is devoted to a review of the relevant Fredholm theory for first order boundary value problems with Atiyah-Patodi-Singer boundary conditions and for polynomially weighted Sobolev spaces, and Section 6.3 gives the precise boundary conditions. Section 6.4 contains the proof in the case that $a_0=a_t=0$, and Section 6.5 completes the general case. Finally, in Section 6.6, different projection operators for the kernel of $\widehat{\mathcal N}$ are discussed to be used in the next section.

\subsection{Set-up}
\label{section6.1}

We will impose Atiyah-Patodi-Singer boundary conditions on pairs $(\ph,a)$ on $D_\bold r$ to make $\widehat{\mathcal N}$ a Fredholm operator. The boundary conditions are that pairs $(\ph,a)$ lie in the kernel of a certain boundary projection 

\begin{equation}
\widehat{\Pi}^{H^1}(D_\bold r): L^{1,2}(D_\bold r)\lre H^+\subseteq L^{1/2,2}(S^1_\bold r \ ; \ S_E \oplus (\Omega^0\oplus \Omega^1))
\end{equation}
where $H^+$ is a ``half-dimensional'' spectral subspace of boundary values. The precise definition of $\widehat{\Pi}^{H^1}$ will be given at the end of Section \ref{APSbdsection} in Definition \ref{NBV1}. Define the following Hilbert Spaces:  

\begin{eqnarray}
\widehat{H^1}(D_\bold r)&:=&\{(\ph,a) \ | \  \ \|(\ph,a)\|_{H^1(D_\bold r)}< \infty \ \text{ and } \ \widehat{\Pi}^{H^1} (\ph,a)=0\}\label{H^12ddef} \\
  L(D_\bold r) &:=& L^2 (D_\bold r \ ; \ S_E \oplus (\Omega^0\oplus \Omega^1))
\end{eqnarray}
 where the first is equipped with the inner product resulting arising from the polarization of the norm (\refeq{Hhatnorm}), and the latter with the standard $L^2$ product. We may now state the main result of Section 6.

\begin{prop}
The operator 
 \begin{equation}\widehat{\mathcal N_t}: \widehat{ H}^1(D_\bold r)\lre  L^2(D_\bold r)\end{equation}

\noindent is a bounded Fredholm operator of (real) Index 2. For $\bold r$ sufficiently large, it is surjective with a kernel of real dimension 2 and the inverse on the complement of the kernel is uniformly bounded. That is, 
 
\begin{equation}
\|(\ph,a)\|_{\widehat { H}^1} \leq C( \|\widehat{\mathcal N_t}(\ph,a)\|_{L^2}  + \|p^{\ker} (\ph,a)\|_2)
\end{equation}
\smallskip
\noindent holds for $C$ independent of $\bold r, t$ and $p^{\ker}$ is a projection operator to $\C$.  \begin{flushright}$\diamondsuit$\end{flushright}
\label{invertibleN}
\end{prop}

The presence of a non-trivial kernel merits explanation. Indeed, its appearance may, at first, be surprising since the Seiberg-Witten equations are self-adjoint in 3-dimensions, and thus in most contexts have index 0. However, since the boundary conditions imposed here are somewhat immaterial (when pasting the fiducial solutions onto a 3-manifold, they are cut off near the boundary), we are free to choose any boundary conditions we wish without affecting any eventual gluing construction, and we could have, of course, selected boundary conditions of index zero. The subtelty is that such a choice will {\it never} result in a uniform bound on the inverse.

This is an essential consequence of the geometry, and is the first manifestation of the convergence of the linearization to a non-Fredholm limit discussed in the introduction. The limiting operator in the normal planes $\slashed D_{A_0}^\C$ has two-dimensional kernel on $L^2$ which consists of elements that decay like $O(r^{-1/2})$ away from $\mathcal Z$. Since $\slashed D_{A^{h_\e}}\to \slashed D_{A_0}$ (in no precise sense, since the difference in not bounded in $L^2$), there is a two-dimensional space of configurations approaching this limiting kernel, which have similar asymptotics. But because these elements decay toward the boundary, they {\it cannot} be excluded by disallowing their boundary values; cuttong off these elements with $r^{-1/2}$ decay towards the boundary will necessarily lead to a violation of any uniform bound on the inverse. Thus we cannot use the naive index 0 boundary condition, and must instead allow boundary modes capturing these kernel elements, and project to their orthogonal complement in the correct norm. This problem becomes quite subtle in the 3-dimensional case, when the limiting operator has $\slashed D_{A_0}$ has an infinite-dimensional kernel in $L^2$ as is discussed in Section 7. The two dimensional kernel of $\widehat{\mathcal N}_t$ therefore plays an essential role in this and the following section. It is identified explicitly over the course of the proof.

\medskip 
Proposition \ref{invertibleN} combined with the relation \ref{diagramA} from Proposition \ref{scalingprop} immediately implies the following result for the un-rescaled operator $\mathcal N_t$: 

\begin{cor}
For every fixed $t_0\in \mathcal Z$, on smooth configurations $(\ph,a)$ satisfying the un-rescaled version of the boundary conditions $\Pi^{H_\e^1}(\ph,a)=0$,  the following estimate on the normal disk $\{t\}\times D_\lambda$ holds uniformly in $t,\e$. 
\bea  \|(\ph,a)\|_{H^1_\e(\{t\}\times D_\lambda)}\leq C  \|\mathcal N_{t} (\ph,a)\|_{L^2(\{t\}\times D_\lambda)}+ \|\pi^{\ker}_{t}(\ph,a)\|_{2}\eea 

\noindent Here, the left side denotes (\refeq{2dimhnorm}) (and does not include $\nabla_t$ terms) formed using the Euclidean norm.  \qed 
\end{cor}

\medskip 

The remainder of Section 6 is devoted to the proof of Proposition \ref{invertibleN}. Before beginning the proof,   it is convenient place the form components $( a_xd\widehat x+ a_yd \widehat y \ , \ a_0 + a_tdt)$ in a holomorphic context. There are isomorphisms 
\begin{eqnarray} 
 \Omega^{1}(i\R)&\overset{(1)}\to &\Omega^{0,1}(\C) \hspace{2.75cm}  \Omega^0\oplus \Omega^0(i\R)\overset{(3)} \to \Omega^{1,0}(\C)\\ 
i(a_x d\widehat x+ a_y d\widehat y)& \mapsto &   (a_y - ia_x)d\overline w\hspace{2cm} i(a_0+a_tdt)\mapsto (a_0+ia_t)dw \label{isomorphisms}\end{eqnarray}
on the domain, and 
\bea  
 \Omega^{0}\oplus \Omega^2(i\R)&\overset{(2)}\to &\Omega^{1,1}(\C) \hspace{3.7cm}  \Omega^1(i\R) \overset{(4)}\to \Omega^{1,1}(\C)\\ 
 (ih_2, ih_2 d\widehat x\wedge d\widehat y)& \mapsto &  (h_1-ih_2)dw\wedge d\overline w\hspace{.9cm} i(pd\widehat x + q\widehat dy)\mapsto (p+iq)dw\wedge d\overline w \eea
on the codomain. Setting 
\bea
\zeta&:=& (a_0+ia_t)dw \\
\omega&:=& (a_y -ia_x)d\overline w
\eea
 $\widehat{\mathcal N}_t$ may be considered as an operator $\Gamma( \C^2 \oplus \C^2 \oplus \Omega^{1,0}\oplus \Omega^{0,1}) \to \Gamma(\C^2 \oplus \C^2 \oplus \Omega^{1,1}\oplus \Omega^{1,1})$ 
now given by
$$\widehat {\mathcal N}_t(\alpha,\beta,\zeta  ,\omega )= \begin{pmatrix} \begin{pmatrix} 0 & -2\del_{A^H} \\ 2\overline \del_{A^H} &0 \end{pmatrix} & \gamma( \ \ )\Phi^H  \\ \begin{pmatrix}\mu_\R( \ , \Phi^H) \\ \mu_\C( \ , \Phi^H)\end{pmatrix} & \begin{pmatrix} 0 & 2\del  \\  -2\delbar & 0  \end{pmatrix}
\end{pmatrix} \begin{pmatrix}\alpha \\ \beta \\ \zeta \\ \omega \end{pmatrix}   $$

\noindent where Clifford multiplication becomes \begin{equation}\gamma(p dw,q d\overline w)= \begin{pmatrix} ip & -\overline q \\ -q &  i\overline p \end{pmatrix}\label{cliffordmult}\end{equation} and the moment map on $\psi=(\alpha,\beta)$ is \begin{eqnarray} \mu_\R(\psi,\Phi^H)&=& -\alpha_1\overline{\alpha}_1^H + \overline \beta_1 \beta_1^H -\alpha_2\overline{\alpha}_2^H + \overline \beta_2 \beta_2^H  \\ \mu_\C(\psi, \Phi^H)&=& -\overline \alpha_1 \beta^H_1 - \beta_1 \overline {\alpha}_1^H -\overline \alpha_2 \beta^H_2 - \beta_2 \overline {\alpha}_2^H \label{mumaps}  \end{eqnarray} its adjoint as before. We view $\widehat{\mathcal N}_t$ in this guise for the remainder of Section 6. We also leave the $t$-dependence implicit for the remainder of the section.

\subsection{Fredholm Theory}
\label{section6.2}
\label{APSbdsection}
This subsection discusses Fredholm theory for Dirac operators in two different contexts: 1) as boundary-value problems with Atiyah-Patodi-Singer boundary conditions, and 2) on non-compact domains with polynomially weighted Sobolev spaces. The proof of Proposition \ref{invertibleN} will require both of these perspectives, as the weighted spaces are needed for estimates to be uniform in the radius $\bold r$. Since it suffices for our purposes, the discussion here is limited to the relevant cases of the operators $\del,\delbar$ on the disk; the reader is referred to \cite{KM, MashallThesis,Toms155Notes} for more general discussions.

\subsubsection*{APS Boundary Conditions}
First, we consider boundary-value problems for $\del,\delbar$. Let $D\subseteq \C$ denote the unit disk, and $L^{k,2}(D; \C)$ the standard Sobolev spaces of complex-valued functions. The continuous restriction or trace map $$\Tr: L^{k,2}(D;\C)\to L^{k-1/2, 2}(\del D; \C)$$ 
gives functions well-defined boundary values for $k\geq 1$. Within the space of boundary values for $k=1$, we have the half-dimensional subspaces  
\bea
H^+_{[m]} &= & \{ u\in L^{1/2,2}  \ | \ u=\sum_{k \geq m} a_k e^{ik \theta} \} \subseteq L^{1/2,2}(\del D, \C)\\
H^-_{[m]} &= & \{ u\in L^{1/2,2}  \ | \ u=\sum_{k \leq m} a_k e^{ik \theta} \} \subseteq L^{1/2,2}(\del D, \C) 
\eea

\noindent of functions whose Fourier series have non-vanishing components only on the positive and negative sides of $m\in \Z$ (inclusive) respectively. We denote the projections to these spaces  by \be\Pi_{[m]}^\pm: L^{1/2,2}(\del D; \C)\to H^{\pm}_{[m]}\label{projections}\ee respectively. 

Now consider $\delbar$ on $D$. Its (infinite dimensional) kernel consists of holomorphic functions on the disk, whose boundary values lie in $H^+_{[0]}$. The following two propositions are standard results, whose proofs can be found in \cite[Pg. 85]{Toms155Notes}.

 \begin{prop}
The operator 

\begin{equation}
(\delbar, \Pi^+_{[0]}): L^{1,2}(D;\C)\to L^{2}(D;\C)\oplus H^{+}_{[0]}
\end{equation}
is invertible, and, {\it a fortiori}, Fredholm of Index 0.
\end{prop}
\qed 

More generally, 
\begin{prop} The operator
 
\begin{equation}
(\delbar, \Pi^+_{[m]}): L^{1,2}(D;\C)\to L^{2}(D;\C)\oplus H^{+}_{[m]}
\label{delbarm}
\end{equation}
has 
\begin{itemize}
\item (if $m>0$) empty cokernel and kernel of dimension $m$ spanned by $\{1, z,\ldots, z^{m-1}\}$. 
\item (if $m<0$) empty kernel and cokernel of dimension $-m$ spanned by $ \{(0, e^{-i\theta}), \ldots, (0, e^{-im\theta})\}. $
\end{itemize}
\label{fredholmdelbar}
\end{prop}
\qed

The corresponding statement holds for the anti-holomorphic case
\begin{equation}
(\del, \Pi^-_{[m]}): L^{1,2}(D;\C)\to L^{2}(D;\C)\oplus H^{-}_{[m]},
\label{delm}
\end{equation}
and for the Sobolev spaces $L^{k,2}$ for $k>1$.

Alternatively, one may consider restricting to the space of functions on which the boundary values are 0. Denote the kernel of the projection by $$L^{1,2}_{m,+}:= \{u\in L^{1,2} \ | \ \Pi^+_{[m]}(u)=0\}$$ and similarly for $L^{1,2}_{m,-}$. To keep the notation clear, the reader may find it helpful to read $L^{1,2}_{m,+}$ as ``$L^{1,2}$ functions whose restriction to the boundary has vanishing Fourier components on the $+$ side of $m$ (inclusive)''.   Proposition \ref{fredholmdelbar} becomes the following statement.

\begin{prop} The operator
\bea 
\delbar: L^{1,2}_{m,+}(D;\C)& \to& L^2(D;\C)\\
\eea
is Fredholm with
\begin{itemize}
\item (if $m\geq 0$) empty cokernel and kernel of dimension $m$ spanned by $\{1, z, \ldots, z^{m-1}\}$. 
\item (if $m<0$) empty kernel and cokernel of dimension $-m$ spanned by $\{1, \overline z,\ldots, \overline z^{m-1}\}$. 
\end{itemize}
and similarly for the anti-holomorphic case.
\label{invertibledelbar} 
\end{prop} 

\begin{proof}
The statement about the kernels follows directly from Proposition \ref{fredholmdelbar}. To see the cokernel is as stated, let $m<0$ and $\ph \in \text{coker}(\delbar)$. Then intergration by parts shows that $\forall u\in L^{1,2}_{m,+}$. 
\bea
0&=& \int_{D} \br \delbar u , \ph \kt dz\wedge d\overline z = -\int_{D} \br u,\del \ph \kt dz\wedge d\overline z - \int_{\del D} \br u, \ph\kt re^{i\theta}d\theta.
\eea
Varying $u$ over compactly supported functions, we see $\del \ph =0$ on the interior of $D$. On the boundary, $$u|_{\del D}=....+u_{m-2} e^{i(m-2)\theta}+ u_{m-1}e^{i(m-1)\theta}$$ 
hence varying $u$ over functions with such boundary values shows that $\ph$ satisfies 
$$
\begin{cases}
\del \ph=0 \\
\Pi^-_{[m]}(\ph)=0.\
\end{cases}$$
Note the $+1$ shift in the boundary values resulting from the $e^{i\theta}$ factor in the boundary integral. The form of the cokernel then follows from the statement about the kernel for $\del$ with the above boundary conditions. 
\end{proof}

\subsubsection*{Polynomial Weights}

When considered on all of $\R^2$, the operators $\del,\delbar: L^{1,2}(\R^2)\to L^2(\R^2)$ are not Fredholm, as they have dense spectrum at zero, and therefore fail to have closed range. The same phenomenon prevents the inverse on finite disks from being uniformly bounded in the size of the disk. In order to get a Fredholm problem on the entire plane, one must use polynomially weighted spaces. These same weights make the required estimates uniform in the radius of the disk. Here again, we content ourselves with an exposition within the scope of our purposes. The general theory is that of elliptic operators on manifolds with cylindrical ends, which can be found in \cite{KM} (Chapter 17), or \cite{LockhartMcOwen,MashallThesis}. 

Let $R: \C \to \R^{\geq 0}$ be a positive monotonically increasing weight function equal to 1 near the origin and equal to $r$ far from the origin. We define weighted norms 

\bea
\|u\|_{L^{k,2}_{\nu}} &:=&\left(\int_{\R^2}\left(R^{2k}|\nabla^k  u|^2 + \ldots  + |u|^2\right) R^{2\nu} \ dV\right)^{1/2}
\eea
and 
\begin{defn}
The {\bf Polynomially Weighted Sobolev Spaces} $$L^{k,2}_{\nu}(\R^2)= \{ u \ | \ \|u\|_{L^{k,2}_\nu}<\infty\}$$
to be the completion of compactly supported smooth functions with respect to these norms. 
\end{defn}

 It is easy to check, using $0\leq \d{R}{r} \leq 1$ that:  

\begin{lm}
The map $$f\mapsto R^{-\mu + \nu} f$$
is an isomorphism $$L^{k,2}_{\nu}(\R^2) \to L^{k,2}_{\mu}(\R^2)$$inducing an equivalence of norms. 
\label{changingweights}
\end{lm}

The following result summarizes the Fredholm theory for $\del,\delbar$ in the non-compact setting \cite{LockhartMcOwen}: 

\begin{prop}
The operators $$\delbar,\del: L^{1,2}_{\nu-1}(\R^2;\C)\to L^2_{\nu}(\R^2;\C)$$ are Fredholm for $\nu \notin \Z$. Specifically, 

\begin{itemize}
\item for $\nu \in (0,1)$ they are isomorphisms. 
\item for $\nu \in (-n-1, -n)$ they are surjective with kernel of dimension $n$ spanned by $\{1, z, \ldots, z^{n-1}\}$ (resp. $\overline z$).
\item for $\nu \in (n,n+1)$ they are injective with cokernel of dimension $n$ spanned by $\{1,\overline z, \ldots,\overline z^{n-1}\}$ (resp. $z$).
\end{itemize}
\label{FredholmonR^2}
\end{prop}
\qed

The next proposition is the appropriate version of the first bullet point for disks of finite radius. Combining these weighted spaces with the boundary conditions as in the previous subsection, we have the spaces  $\{u \in L^{1,2}_\nu \ | \ \Pi^{\pm}_{[m]}u=0\}$. Likewise for $L^{k,2}_{\nu}$.  

\begin{prop}
For $\nu=1/2$ the $\delbar$-operator subject to the boundary conditions $\Pi^+_{[0]}=0$ $$ \delbar: L^{1,2}_{ \nu-1}(D_\bold r; \C)\to L^2_{\nu}(D_\bold r; \C)$$
is invertible, and there is a constant $C$ such that $$\|u\|_{L^{1,2}_{\nu-1}} \leq C\, \|\delbar u\|_{L^2_{\nu}}$$ holds uniformly in $\bold r$. The corresponding statement holds for $\del$. In fact, both statements hold for any $\nu\in (0,1)$ where the constant $C$ may depend on $\nu$.   
\label{delbaruniformestimate}
\end{prop}

\begin{proof}
The fact that $\delbar$ is an isomorphism with these boundary conditions follows from Proposition \ref{invertibledelbar}, so it suffices to show the uniform estimate here. In fact, by Lemma \ref{changingweights} it suffices to show it on the spaces $\{u \in L^{1,2}_{-1} \ | \ \Pi^{+}_{[0]}u=0\}\to L^2$ with the operator  \bea R^{1/2} \circ \delbar \circ R^{-1/2}&=& \frac{1}{2}e^{i\theta}\left(\del_r + \frac{i}{r}\del_\theta -\frac{1}{2r}\frac{r}{R}\d{R}{r}\right)\\
&=& \frac{1}{2}e^{i\theta}\left(\del_r + \frac{1}{r}\left(i\del_\theta -\frac{\chi}{2}\right)\right)
 \eea
 where $\chi$ is a function smoothly rising from $0$ at the origin and equal to 1 once $R=r$. To show the estimate, we integrate by parts: for $f\in\{u \in L^{1,2}_{-1} \ | \ \Pi^{+}_{[0]}u=0\}$ one has
 
 \bea
 \int_{\R^2}|R^{1/2}\circ \del \circ R^{-1/2} f|^2 R^{-1}dV& =&\frac{1}{4}\int_{\R^2} \br\del_r + \frac{1}{r}\left(i\del_\theta -\frac{\chi}{2}\right) \ , \ \del_r + \frac{1}{r}\left(i\del_\theta -\frac{\chi}{2}\right) \kt  \ dV \\
 &=& \int_{\R^2} |\del_r f|^2 + \frac{1}{r^2}|\left(i\del_\theta -\frac{\chi}{2}\right)f|^2 \\ & & \  \ \ + \ \br \del_r f \ , \  \frac{1}{r}\left(i\del_\theta -\frac{\chi}{2}\right)f\kt
+ \br  \frac{1}{r}\left(i\del_\theta -\frac{\chi}{2}\right)f \ , \  \del_r f  \kt \ r dr d\theta \\
&=& \int_{\R^2} |\del_r f|^2 + \frac{1}{r^2}|\left(i\del_\theta -\frac{\chi}{2}\right)f|^2 \  + \br  f, \d{\chi}{r} f\kt  \ dV + \int_{\del D_\bold r} \br f, i\del_\theta - \frac{\chi}{2} f \kt d\theta\\
& \geq & \int_{\R^2} |\del_r f|^2 + \frac{1}{r^2}|\left(i\del_\theta -\frac{\chi}{2}\right)f|^2 \ dV  + \sum_{k<0}  (-k-\tfrac{\chi}{2})|f_k|^2\\
&\geq &\int_{\R^2}|\nabla f|^2 + \frac{\chi^2}{r^2} |f|^2 \ dV
 \eea
 where we have integrated by parts in r, and observed that the boundary term is strictly positive as a result of our boundary conditions, and used that $\d{\chi}{r}>0$. The last line follows because $i\del_\theta-\tfrac{\chi}{2}$ is an invertible operator with lowest eigenvalue equal to $\tfrac{\chi}{2}$ on every circle of fixed radius. This is the desired estimate except for the fact that the second term is supported away from the origin. To remedy this, we apply the Poincar\'e inequality to $f$ times a large cutoff function equal to 1 where $\chi\neq 1$.
 
 More generally, for $\nu \in (0,1)$ the proof is identical replacing $\tfrac{\chi}{2}$ by  $\chi \nu $.   \end{proof}
 
 We also note the following specific corollary in the case when $\nu=-1/2$ and the boundary conditions for which the kernel is the constant functions, i.e. on the space $\{ u \in L^{1,2}_{-3/2} \ | \ \Pi^{-}_{[-1]}u=0\}$: 
 
 \begin{prop}
 For $\nu=-1/2$ and subject to the boundary $\Pi^{-}_{[-1]}$ the operator $$\del: L^{1,2}_{\nu -1}(D_\bold r; \C)\to L^{2}_{\nu}(D_\bold r; \C)$$
 is surjective with kernel equal to the constants. Moreover, the estimate \be\|u\|_{L^{1,2}_{\nu-1}} \leq C \left(\|\del u\|_{L^2_{\nu}} + \|\pi^\text{const} u \|_{L^{1,2}_{\nu-1}}\right)\label{elliptic6.10}\ee
 holds uniformly in $\bold r$ once $\bold r>>0$ where the projection is that arising from the $L^{1,2}_{-3/2}$-inner product.  The equivalent also holds for $\nu\in (-1,0)$ with the constant being allowed to depend on $\nu$. 
 \label{deluniformestimate}
 \end{prop}
 \begin{proof}
 The statement about surjectivity and the form of the kernel follows again from Proposition \ref{invertibledelbar}. By conjugation, it suffices to prove the statement for $\delbar$. For $\nu=-1/2$, switching weights to $L^{1,2}_{-1}$ as in the previous proposition, and integrating by parts again yields the same result with $\chi$ replaced by $-\chi$ and the boundary sum replaced by $k\leq  0$. The boundary term is therefore still positive, but the $\d{\chi}{r}$ term is negative. It is compactly supported, hence we obtain an estimate for a compactly supported operator $K$. 
 
 $$\|u\|_{L^{1,2}_{-3/2}}\leq C_1 \left( \|\del u \|_{L^{2}_{-1/2}} + \|Ku\|_{L^2_{-1/2}}\right). $$
 
 \noindent Now we proceed by contradiction: assume there were a sequence $u_n$ on disks of radius $r_n$ having unit $L^{1,2}_{-3/2}$ norm, and violating the inequality to prove with constant $1/n$. The above estimate shows one must have $\|Ku_n\|> \tfrac{1}{C_1}-\tfrac{1}{n}$ and so $u_n$ must have non-zero portion of its norm on the compact support of $K$. Cutting off $u_n$ with increasingly large logarithmic cutoff functions $\chi_n$ shows that $\chi_n u_n$ eventually violates the equivalent inequality with projection to the kernel on all of $\R^2$, contradicting Proposition \ref{FredholmonR^2}. 
 
 The general case of $\nu\in(-1,0)$ is follows similarly. 
 \end{proof}

 \begin{rem}\label{rem6.11}
 Notice that if we consider the space $\{u \in L^{1,2}_{-3/2} \ | \ \Pi^{-}_{[0]}u=0\}$ with the index 0 boundary conditions the elliptic estimate (\refeq{elliptic6.10}) {\it cannot} be made uniform in the radius. Indeed, letting $\bold r_n=n$, taking a logarithmic cutoff function $\chi_n$ equal to $1$ for $r< n/2$ and vanishing on the boundary with $|d\chi_n|\leq c|\log n|^{-1}/r$, choosing constants $c_n$ so that $\|c_n\chi_n\|_{L^{1,2}_{-3/2}}=1$, one can see that the sequence of functions $c_n\chi_n$ satisfy the boundary conditions yet $\|\del (c_n\chi_n)\|_{L^2_{-1/2}} \to 0$, contradicting \ref{elliptic6.10}. 
 \end{rem}
 
 The relevance of the two above specific cases to our situation is that the connection $A^{H}$ implicitly adds a $\nu=+1/2$ weight to $\delbar$ and a $\nu=-1/2$ weight to $\del$. Indeed, recalling from Section \ref{section4.1} we wrote $$\del_{A^H}=\frac{1}{2}e^{i\theta}\left(\del_\rho + \frac{1}{\rho}\left(i\del_\theta - \frac{\chi_H}{2} \right)\right)  \hspace{1cm} \del_{A^H}=\frac{1}{2}e^{i\theta}\left(\del_\rho - \frac{1}{\rho}\left(i\del_\theta - \frac{\chi_H}{2} \right)\right)$$
 
 \noindent where $\tfrac{\chi_H}{2}= \tfrac{1}{2} + \rho \del_\rho H$.

 The above two propositions translate into the following statements about these operators. In it, we use the space
 $L^{1,2}_{-1}$ whose norm is $$\|u\|_{L^{1,2}_{-1}}=\int |\nabla u|^2 + \frac{|u|^2}{R^2} dV$$ 
 
 \begin{prop} The operators $\delbar_{A^H}, \del_{A^H}$ satisfy the following respectively:

\noindent {\bf (1)} With the Index 0 boundary conditions $\Pi^+_{[0]}=0$, the operator $$\delbar_{A^H}: L^{1,2}_{-1}(D_\bold r;\C)\to L^2(D_\bold r; \C)$$
 is invertible and $$\|\alpha\|_{L^{1,2}_{-1}}\leq C \|\delbar_{A^H}\alpha\|_{L^2}$$
 holds uniformly in $\bold r$. 

\noindent{\bf (2)} With the Index 2 boundary conditions, $\Pi^+_{[-1]}=0$ the operator $$\del_{A^H}: L^{1,2}_{-1}(D_\bold r ; \C) \to L^2(D_\bold r ; \C)$$
is surjective with kernel of dimension 2, and $$\|\beta\|_{L^{1,2}_{-1}}\leq C(\|\del_{A^{H}}\beta\|_{L^2} + \|p (\beta)\|_{L^{1,2}_{-1}})$$
holds uniformly in $\bold r$ for $\bold r>>0$, where $p$ is the orthogonal projection to the kernel in the $L^{1,2}_1$-inner product. 

\noindent {\bf (3)} More generally, the same statements hold for $\nu \in (-\frac12, \frac12)$, i.e. 

$$\|\alpha\|_{L^{1,2}_{\nu-1}}\leq C(\nu) \|\delbar_{A^H}\alpha\|_{L^2_\nu}\hspace{1cm}\|\beta\|_{L^{1,2}_{\nu -1}}\leq C(\nu)(\|\del_{A^{H}}\beta\|_{L^2_\nu} + \|p (\beta)\|_{L^{1,2}_{\nu-1}})$$
 \label{delbarAHprop}
 \end{prop}  
 
 \begin{proof}
 The operator $\delbar_{A^H}$ has the same form of the operator $R^{1/2}\circ \delbar \circ R^{-1/2}$, and as in the proof of Proposition \ref{delbaruniformestimate}, this operator acting on $L^{1,2}_{-1}$ is equivalent to $\delbar$ acting on $L^{1,2}_{-1/2}$. Item {\bf(1)} therefore follows directly from Proposition \ref{delbaruniformestimate}. The only minor caveat is that the effective weight function for $\delbar_{A^H}$ is asymptotically exponentially close to $r$ not equal to it outside a compact region, but this is of no consequence in the proof as one can easily check. 
 
 Likewise, $\del_{A^H}$ on $L^{1,2}_{-1}$ has the same form as the operator $R^{-1/2}\circ \del \circ R^{1/2}$ acting on the $L^{1,2}_{-3/2}$, and is thus equivalent to the situation of Proposition \ref{deluniformestimate} with the same minor caveat, and item ({\bf 2}) follows.
 
More generally, $\delbar_{A^H}$ with weight $\nu$ is equivalent to $\delbar$ with the weight $\nu + 1/2$ and $\del_{A^H}$ is equivalent for $\del$ with weight $\nu-1/2$. {\bf (3)} therefore follows from Proposition \ref{delbaruniformestimate} and Proposition  \ref{deluniformestimate} in the cases for $\nu \in (0,1)$ and $\nu \in (-1,0)$ respectively.   
 \end{proof}

\subsection{Boundary Conditions for $\widehat{\mathcal N}$}
\label{bdconditionsN}
\label{section6.3}

We now give boundary conditions for $\widehat{\mathcal N}$ on $D_\bold r$. Given that $\widehat{\mathcal N}$ consists of the operators $\del,\delbar $ and lower order terms, it would be natural to consider $\widehat{\mathcal N}$ acting on the following Sobolev spaces: 

\begin{eqnarray}
 \begin{matrix} L^{1,2}_{0,+} \ (D_\bold r; \C^2) \ \ \\ \oplus \\   L^{1,2}_{0,-}\ (D_\bold r ; \C^2)  \  \ \ \\   \oplus  \\ L^{1,2}_{0,+} \  (D_\bold r;\Omega^{1,0}) \ \\  \oplus \\  L^{1,2}_{0, -} \ (D_\bold r; \Omega^{0,1})    \end{matrix}  \ \ \ \ \overset{\widehat{\mathcal N}}\lre \ \ \ \  \begin{matrix} L^2  (D_\bold r; \C^2) \ \ \\ \oplus \\   L^{2}(D_\bold r ; \C^2)  \  \ \ \\   \oplus  \\ L^{2}   (D_\bold r;\Omega^{1,1}) \   \\ \oplus \\    L^{2}   (D_\bold r;\Omega^{1,1}). \end{matrix} 
 \label{index0setting}
\end{eqnarray}
Indeed, the above discussion of APS boundary conditions for the $\del,\delbar$ operators shows that when the zeroth-order terms are omitted from $\widehat{\mathcal N}$, the resulting operator is invertible on the above spaces, and thus $\widehat{\mathcal N}$ is index 0 (since the off diagonal terms are compact on the compact domain $D_\bold r$). Explicitly, this space is comprised of tuples $(\alpha,\beta,\zeta, \omega)$ having Fourier expansions on the boundary in which $\alpha,\zeta$ have only negative Fourier modes, and $\beta,\omega$ have only positive ones. 

The actual boundary conditions we will take are a slight modification of the above. We will expand the above space by allowing the $\beta$ component to have a constant Fourier mode $\beta_0$ on the boundary, and restrict it by disallowing a particular linear combination of the  $\alpha_{-1}$ and $\beta_0$ modes. As in Remark  \ref{rem6.11}, the index 0 boundary conditions allows the space to contain kernel elements that decay towards the boundary necessarily violating any uniform estimates. Notationally, this shift in the boundary values of $\beta$ is also necessitated by the $e^{-i\theta}$ on the $\beta$-component of $\Phi^H$. 

We now define these boundary conditions in terms of projection operators. Let 
$\Pi^{\pm}_{[m]}$ be the boundary projections defined by (\refeq{projections}) in Section \ref{APSbdsection}. We also define a two-dimensional projection $$\mu^\del_\C: L^{1,2}(D_\bold r; S_E\otimes (\Omega^0\oplus \Omega^1)) \lre \C$$
 given on a spinor  $(\alpha,\beta, \zeta , \omega )$ as follows. Let $a_1, a_2$ be the components of the $e^{-i\theta}$ boundary mode of $\alpha$, and $b_1, b_2$ be the components of the constant boundary mode of $\beta$ so that 

\smallskip 
 \be  \alpha_{-1}  =a_1 \otimes 1 + a_2 \otimes j  \hspace{1cm} \beta_{0} =b_1 \otimes 1 + b_2 \otimes j \label{muboundarysetup}\ee

\smallskip

\noindent where the subscript on the left hand sides denotes the Fourier mode. Then  
\smallskip 
  \be \mu^\del_\C\left(\alpha, \beta , \zeta , \omega \right)= b_1 \overline \alpha_1^H + \overline a_1 \beta_1^H + b_2 \overline \alpha_2^H + \overline a_2 \beta_2^H.   \label{mucboundarydef2d}\ee

\noindent Here, $\alpha_i^H$ and $\beta_i^H$ are the components of $\Phi^H$ restricted to the boundary (the subscripts on these denote the $\otimes 1$ and $\otimes j$ components, not the Fourier modes).  

\begin{defn}\label{boundarymodes}
We define the {\bf twisted boundary conditions} for $\widehat{\mathcal N}$ by the requirement $$\widehat{\Pi}^{ H^1}=0$$ where

$$\widehat{\Pi}^{ H^1}: L^{1,2}(D_\bold r; S_E\otimes (\Omega^0\oplus \Omega^1)) \lre H^+_{[0]}\oplus H^-_{[-1]}\oplus H^{+}_{[0]}\oplus H^-_{[0]}\oplus \C$$
is given by $$\widehat{\Pi}^{ H^1}:= \Pi^+_{[0]}\oplus  \Pi^-_{[-1]}\oplus  \Pi^+_{[0]}\oplus  \Pi^-_{[0]}\oplus \mu_\C^\del.$$

\end{defn}

Explicitly, the boundary conditions require that tuples $(\alpha,\beta,\zeta,\omega)$ have boundary Fourier expansions of the following form: 
 \begin{eqnarray}
 \text{Fourier mode}   & &  \ldots {\underline{k=-2} } \hspace{.5cm}     {\underline{k=-1} }  \hspace{.8cm}{\underline{k=0} } \  \hspace{.30cm} {\underline{k=1 }}   \ \ \  \hspace{.25cm} {\underline{k=2 }}  \ \ldots \hspace{1.7cm}\\
 \alpha|_{\del D_\bold r} &=& \ldots \alpha_{-2}e^{-2i\theta} + \alpha_{-1}e^{-i\theta} \\
\beta|_{\del D_\bold r} &=& \hspace{4.3cm}  \beta_0 \   +  \   \beta_1 e^{i\theta} \ +  \ \beta_2 e^{2i\theta} \  +\ldots \\
\zeta|_{\del D_\bold r} &=& \ldots \zeta_{-2}e^{-2i\theta}  +  \zeta_{-1}e^{-i\theta} \\
\omega|_{\del D_\bold r} &=&  \hspace{4.3cm}  0  \  \ + \  \  \omega_1 e^{i\theta}  \  +  \    \omega_2 e^{2i\theta} +\ldots.
\end{eqnarray}
\label{NBV1}
\noindent such that $\alpha_{-1}, \beta_0$ are constrained to linear combinations which satisfy \be \mu_\C^\del(\alpha_{-1},\beta_0)=0 \ee  with the notation of (\refeq{muboundarysetup}) and (\refeq{mucboundarydef2d}). 
  
This completes the definition of the domain $\widehat{ H}^1$ of the operator $\widehat{\mathcal N}$ advertised in (\refeq{H^12ddef}). We can also immediately conclude the first statement of Proposition \ref{invertibleN} which claimed that $\widehat{\mathcal N}$ with these boundary conditions is a Fredholm operator of real index 2. \\

\noindent {\it Proof of the Index statement in Proposition \ref{invertibleN}}.   On the compact domain $D_\bold r$, the zeroth order terms of $\widehat{\mathcal N}$ are compact, so it suffices to show the statement for the first order terms. Relative to the Index 0 setting of \ref{index0setting}, we have allowed a zeroeth order mode in $\beta$, which one complex dimension for each of the two copies of $\C$ in the domain of $\beta$, hence four real dimensions. Since the map $\mu_\C^\del: L^{1,2}(D_\bold r)\to \C$ has full rank (which is a consequence of $|c(t)|^2 + |d(t)|^2>0$), adding this condition subtracts two from the real index.   
\qed

\subsection{The Holomorphic Interpretation}
\label{holomorphicinterpretation}

This subsection proves Proposition \ref{invertibleN} in the case that $\zeta=0$. In this context, we can interpret the form $\omega$ as endowing the vector bundle $\C^2\oplus \C^2$ with a particular holomorphic structure, which is necessarily complex gauge equivalent to the standard one on the disk. Specifically, in this subsection we consider the reduced ``holomorphic'' operator 

$$\widehat {\mathcal N}^\C(\alpha,\beta,\omega   )= \begin{pmatrix} \begin{pmatrix} 0 & -2\del_{A^H} \\ 2\overline \del_{A^H} &0 \end{pmatrix} & \gamma( \ \ )\Phi^H  \\ \mu_\R( \ , \Phi^H) &  2\del 
\end{pmatrix} \begin{pmatrix}\alpha \\ \beta  \\ \omega \end{pmatrix},   $$

\noindent with the reduced boundary conditions given by 

\begin{eqnarray}
\alpha|_{\del D_\bold r} &=& \ldots \alpha_{-2}e^{-2i\theta} + \alpha_{-1}e^{-i\theta} \\
\beta|_{\del D_\bold r} &=& \hspace{4.5cm}  \beta_0 + \beta_1 e^{i\theta}+ \beta_2 e^{2i\theta} \  +\ldots \\
\omega|_{\del D_\bold r} &=&  \hspace{4.5cm}  0  \ + \  \omega_1 e^{i\theta}+ \omega_2 e^{2i\theta} +\ldots.
\end{eqnarray}

 \noindent obtained by omitting the requirements on $\zeta$ and $\mu_\C^\del$ from Equations \ref{NBV1}. Let $\widehat{H}^1_\C(D_\bold r)$ and $L^2_\C(D_\bold r)$ denote the Hilbert Spaces omitting the $\zeta$ component and the $\mu_\C$ term in the first, and the fourth factor in $L^2$. We will often abbreviate them $\widehat{ H}^1_\bold r$ and $ L^2_\bold r$. The norm is now given by 
 \be \|(\alpha,\beta,\omega)\|_{\widehat H^1_\C}:= \left(\int_{D_\bold r} |\nabla(\alpha,\beta,\omega)|^2 + \frac{|(\alpha,\beta)|^2}{R^2}+ | \mu_\R((\alpha,\beta), \Phi^H)|^2 + |\omega|^2 |\Phi^H|^2 dV\right)^{1/2}\label{H1Cnorm}\ee

\begin{prop}
The operator $$\widehat{\mathcal N}^\C: \widehat  {{H}}^1_\C(D_\bold r)\lre L^2_\C(D_\bold r)$$
is Fredholm of real Index 4. For $\bold r$ sufficiently large, it is surjective with kernel of dimension 4, and there is a projection $\pi^\psi:\widehat{ H}^1_{\bold r}\to \C^2$ such that the estimate 

$$\|(\alpha,\beta,\omega)\|_{\widehat{H}^1}\leq C (\|\widehat{\mathcal N}^\C(\alpha,\beta,\omega)\|_{L^2}+ \|\pi^\psi(\alpha,\beta)\|)$$
holds uniformly in $\bold r$, and $t$. 
\label{holomorphicinvertible}
\end{prop}

\begin{proof} The index statement is immediate from the above discussion of boundary conditions, since we have added 4 real dimensions in the $\beta_0$ component compared to the index 0 boundary conditions. The remainder of the proof consists of three steps, each of which requires several lemmas. 

\bigskip 

\noindent {\bf Step 1 (Complex Gauge Action):} To begin, we decompose the domain into a slice of the complex gauge action and its complement.

Define 

\be T\mathcal G^\C_\bold r:= \{h \in L^{2,2}(D_\bold r ; \Omega^0(\C)) \ | \ \Pi^+_{[0]}(h)=0 \text{ and }\Pi^-_{[0]}(\delbar h)=0\}. \label{doubleaps}\ee

\noindent to be the $L^{2,2}$ configurations with {\bf double APS boundary conditions}. Here $h|_{\del D_\bold r}$ is understood via the restriction map $L^{2,2}(D_\bold r)\to L^{3/2, 2}(\del D_\bold r)$ and $\delbar h$ via the same with one lower regularity. There is the linearized action at $(\Phi^H, A^H)$ 
$$\bold d_{(\Phi^H, A^H)}: T \mathcal G^\C_\bold r \to \widehat{ H}^1_\bold r$$ 
given by $$h\mapsto ( h\alpha^H, -\overline h \beta^H ,2\delbar h).$$

\noindent Since we are interested only in a holomorphic description of the linearized operator here, rather than the moduli space of solutions to the non-linear equation, it's not necessary to introduce the complex gauge group itself. The decomposition of $\widehat{ H}^1_{\bold r}$ is philosophically decomposing into a slice of the complex gauge action and its complement, but our approach here only retains this philosophy (and suggestive notation) and we do not need to explicitly check the above space is the Lie algebra of a well-defined Hilbert Lie Group. 
 
 \begin{rem}
 A few remarks are in order: 

(1) The ``double APS'' boundary conditions are rather non-standard in the ory of second-order elliptic PDE, but are the natural boundary conditions for the  square of a Dirac operator, as they require the boundary term to vanish when integrating by parts. In our case, explicitly, (\refeq{doubleaps}) requires
 
 $$h|_{\del D}=\sum_{\ell <0} h_\ell e^{i\ell \theta} \hspace{1cm} \delbar h|_{\del D}=\sum_{\ell > 0} a_\ell e^{i\ell \theta}$$
 so that the boundary term $
\br e^{\pm i \theta} h, \delbar h \kt_{L^2(\del D)} =0$ vanishes (see Lemma \ref{intbyparts1}). 

More generally, on a manifold with boundary $(X,\del X)$, one could split $L^2(\del X)=H^+\oplus H^-$ where $H^\pm$ are respectively the positive and negative eigenspaces of $\slashed D|_{\del X}$ and require $\gamma(\vec n) \Phi \in H^-$ and  $\slashed D \Phi \in H^+$  
so that the integration by parts formula (\cite{KM}, Lemma 4.5.1)
$$\int_{X} \br  \Phi, \slashed D \slashed D\Phi \kt =\int_{X} | \slashed D\Phi|^2 - \int_{\del X}\br \gamma(\vec n) \Phi, \slashed D \Phi \kt$$
has vanishing boundary term. 
 
(2) Note that the Index 0 boundary conditions  (\refeq{index0setting}) do not allow an action of the complex gauge group in the desired way. Writing  $h|_{\del D}=\sum_{\ell <0} h_\ell e^{i\ell \theta}$ as required by (\refeq{doubleaps}), and using  $$\Phi^H = \begin{pmatrix} e^H cr^{1/2} \\ e^{-H}dr^{1/2} e^{-i\theta}
 \end{pmatrix}\otimes 1 + ...\otimes j \hspace{1cm}\Rightarrow \hspace{1cm} h\cdot  \Phi^H= \begin{pmatrix} f_1  e^{-i\theta}+ f_2 e^{-2i \theta}+ ...\\ g_0 + g_1 e^{i\theta} + ...
 \end{pmatrix}\otimes 1 + ...\otimes j$$
 so that the $\beta$ component of $h\cdot \Phi^H$ may have a non-zero constant component on the boundary for $h\in T\mathcal G_\bold r^\C$. This is another reason for introducing the twisted boundary conditions.

 \end{rem} 
 
 \medskip
  \begin{lm}
 The Linearized action $$\bold d_{(\Phi^H, A^H)}: T \mathcal G^\C_\bold r \to {\widehat{ H}^1_\bold r}$$ is an isomorphism onto its image, which is a graph over the form component $\omega\in L^{1,2}_{-,0}(D_\bold r; \Omega^{0,1}(\C))$.
 \end{lm} 
  \begin{proof}
 It suffices to show that the projection of $\bold d_{(\Phi^H, A^H)}$ to the third component of triples $(\alpha,\beta,\omega)$ is an isomorphism. We have that $$\delbar: L^{2,2}_{+,0}(D_\bold r; \C) \to L^{1,2}(D_\bold r;\Omega^{0,1}(\C)) $$
 is an isomorphism, by the discussion following Proposition \ref{fredholmdelbar}, (see \ref{delm}). Here again, the domain denotes the space of functions $h$ on which $\Pi^+_{[0]}h=0$. Thus all that needs to be shown is that adding the second boundary condition to the domain restricts the image to those $L^{1,2}$ configurations satisfying the first-order boundary condition, i.e. that  
 
 $$ \Pi^+_{[0]}h=0  \ \text{ and }\Pi^-_{[0]}(\delbar h)=0\Leftrightarrow h\in L^{2,2}_{0,+} \text{ and }\delbar h\in L^{1,2}_{0,-}$$ 
 but the left side is exactly the definition of the spaces on the right. 
 \end{proof}
 
  As a consequence of the Lemma, there is a splitting into the tangent directions of the complex gauge action and a horizontal slice complementing it. Explicitly, there is an isomorphism 
 \be \begin{matrix}\mathcal H_\bold r \\ \oplus \\ T \mathcal G^\C_\bold r \end{matrix} \ \ \overset{\simeq }  \lre \ \  {\widehat{ H}^1_\bold r}\label{graphiso}\ee
 where
 $$ \mathcal H_\bold r = L^{1,2}_{0,+}(D_\bold r;\C^2)\oplus L^{1,2}_{-1,-}(D_\bold r;\C^2)$$ are the ``horizontal'' components of ${ \widehat H^1_\bold r}$. Explicitly, the isomorphism is given by $(Id, \bold d_{(\Phi^H, A^H)})$ i.e.

 $$(\psi, h)\mapsto (\psi, 0)+ (h \alpha^H , -\overline h \beta^H ,  \delbar h).  $$
 
 \medskip

  \noindent Conversely, any configuration can be written uniquely  $(\alpha,\beta , \omega)=(\psi, 0)+ (h\cdot \Phi^H ,  \delbar h) $.

\begin{lm}
The operator $\widehat{\mathcal N}^\C$ acting on triples $(\alpha,\beta , \omega )= (\psi,0)+ ( h \cdot \Phi^H,  2\delbar h)$ is given by the mixed-order operator

$$\Box(\psi, h)=\begin{pmatrix} \slashed D^\C_{A^H}   & 0 \\ \mu_\R(-, \Phi^H)  & -\Delta -|\Phi^H|^2 \end{pmatrix} \begin{pmatrix} \psi  \\ h \end{pmatrix}$$ where $$\slashed D^\C_{A^H}=   \begin{pmatrix} 0 & -2\del_{A^H} \\ 2\overline \del_{A^H} & 0 \end{pmatrix}.  $$ 

\end{lm}

\begin{proof}
The Lemma is a direct computation of  $$\widehat {\mathcal N}^\C= \begin{pmatrix} \begin{pmatrix} 0 & -2\del_{A^H} \\ 2\overline \del_{A^H}\end{pmatrix} & \gamma( \ \ )\Phi^H  \\ \mu_\R( \ , \Phi^H) & 2\del 
\end{pmatrix}\begin{pmatrix} \psi + \begin{pmatrix} h \alpha^H \\ -\overline h \beta^H \end{pmatrix} \\ 2\delbar h \end{pmatrix}.$$

\noindent For the spinor component, 
\bea \begin{pmatrix} 0 & -2\del_{A^H} \\ 2\overline \del_{A^H} & 0 \end{pmatrix} \begin{pmatrix}\psi + \begin{pmatrix}h \alpha^H \\ -\overline h \beta^H\end{pmatrix}\end{pmatrix}&=& \slashed D^\C_{A^H}\psi -2\del(-\overline h)\beta^H +2 \delbar(h)\alpha^H\\
&=& \slashed D^\C_{A^H}\psi + 2 \del \overline h \beta^H + 2\delbar h \alpha^H \\
&=& \slashed D^\C_{A^H}\psi - \gamma(2\del \overline h dw - 2 \delbar h d\overline w) \Phi^H  \\
&=& \slashed D^\C_{A^H}\psi -\gamma(2\delbar h)\Phi^H
 \eea
where we've expressed Clifford multiplication on $\Omega^1$ in terms of $\Omega^{0,1}$ via \ref{cliffordmult}. 
For the form component, 
$$\mu_\R(\psi+ h\cdot \Phi^H) + 4\del\delbar h=- \Delta h + \mu_\R(\psi, \Phi^H) + \mu_\R(h\cdot\Phi, \Phi).$$ 
And using (\refeq{mumaps}), $$\mu_\R(h\cdot \Phi^H, \Phi^H)= -h |\alpha_1^H|^2 - h|\beta_1^H|^2 -h |\alpha_2^H|^2 - h|\beta_2^H|^2=-|\Phi^H|^2.$$ 

\end{proof}

\noindent {\bf Step 2: (The Diagonal Terms)} The splitting $\widehat H^1_\bold r= \mathcal H_\bold r\oplus  T\mathcal G^\C_\bold r$  does not respect the norm. The norm  on the  ${\widehat{H}^1_\bold r}$ side is $$\|(\psi + h\cdot \Phi^H, 2\delbar h)\|_{\widehat{ H}^1}$$ while the natural norm on $\mathcal H_\bold r\oplus T\mathcal G_\bold r^\C$ is   
$$\left(\|\psi \|_{L^{1,2}_{-1}}^2+ \|h\|_{L^{2,2}}^2\right)^{1/2}.$$ 

\noindent These two norms are not uniformly equivalent in $\bold r$. The norm on ${\widehat{ H}^1_\bold r}$ is ``larger'' in the sense that it contains the $|\mu_\R(\ph,\Phi^H)|^2$ term, while it is ``smaller'' in the sense that for some configurations $\psi + h\cdot \Phi^H$ is small, while $\psi, h$ are individually large but nearly cancel. This problem becomes more pronounced for as $\bold r\to \infty$: in regions where $\Phi^H$ is large, then $h\Phi^H$ \--- hence the $\widehat H^1$ norm \--- is large when $h$ is of unit size. Viewing $T\mathcal G_\C$ as a graph over the $\omega$-component again, this behavior means the slope of the graph diverges for such configurations. To keep track of this we define the following norms on $ \mathcal H_\bold r,T\mathcal G_\bold r^\C$ respectively:  

\begin{eqnarray}
\|\psi \|_{L^{1,2}_{-1}}:&=& \left(\int_{D_\bold r} |\nabla \psi|^2 +  \frac{|\psi|^2}{R^2} \ dV \right)^{1/2}\\
\|h\|_{T\mathcal G^\C}:&=& \left(\int_{D_\bold r} |\nabla^2 h|^2 + |\Phi^H|^2  |\nabla h|^2+  |\Phi^H|^4|h|^2 \ dV \right)^{1/2}
\label{norms}
\end{eqnarray} 
and the {\bf Graph Norm} on $\widehat{ H}^1_\bold r= \mathcal H_\bold r\oplus T\mathcal G_\bold r^\C$ by 

\begin{equation}
\|(\psi,h)\|_{Gr}=\left( \|\psi\|^2_{L^{1,2}_{-1}}+ \|h\|^2_{T\mathcal G^\C}\right)^{1/2}.
\label{graphnorm}
\end{equation} 

The proof of Proposition \ref{holomorphicinvertible}, rests on the following abstract lemma which identifies the kernel of $\Box$ and provides uniform bounds on the inverse on the complement of the kernel. The lemma references two norms on the domain, $\|-\|$ and $\|-\|'$, which will be taken to be the Graph norm and $\widehat{H}^1_\bold r$ norm respectively.

 \begin{lm}  \label{abstractlem}
 Suppose that $( X^i_\bold r, \|-  \|_{ X_i,\bold r})$ and $(Y^i_\bold r, \| - \|_{Y_i,\bold r})$ for $i=1,2$ are families of Banach spaces parameterized by $\bold r\in (0, \infty)$. Set $X_\bold r= X_\bold r ^1 \oplus X_\bold r^2$ and $Y_\bold r= Y^1_\bold r\oplus Y^2_\bold r$ and suppose $N: X_\bold r \to Y_\bold r$ is a linear operator bounded for each $\bold r$ and admitting a block lower-triangular decomposition as $$N=\begin{pmatrix} A  & 0 \\ B & C  \end{pmatrix}. $$
Then 
 \begin{enumerate}

 \item[(1)] Assume that $A: X^1_\bold r \to Y^1_\bold r$ and $C: X^2_\bold r \to Y^2_\bold r$ are invertible, then $N$ is invertible for every $\bold r$.  If instead, $A: X^1_\bold r \to Y^1_\bold r$ is surjective with kernel of some finite dimension independent of $\bold r$, then $N$ is surjective, and $\dim \ker(N)=\dim \ker(A)$. If $x_i\in X_\bold r$ for $i=1,...,m$ are a basis for $\ker (A)$ then  $$\begin{pmatrix}  x_i \\ -C^{-1}B x_i\end{pmatrix} \hspace{1cm}\text{for }i=1,...,m$$
form a basis of $\ker(N)$. Additionally, if $p: X_\bold r\to V$ is a projection to a finite dimensional space restricting to an isomorphism on $\ker(N)$, then $$N\oplus p: X_\bold r\to Y_\bold r\oplus V$$ is an isomorphism.

  \item[(2)]
 Moreover, assume there exists a norm $\|-\|'$ on $X_\bold r$ equivalent for each $\bold r$ to the norm induced by the direct sum.  Suppose additionally that there is a family of operators and projections  
  $$K: X_\bold r\to Y_\bold r \hspace{1cm} p: X_\bold r\to V$$ respectively, where $p$ is as above, satisfying the following estimates: 
  \begin{enumerate}\item There is a constant $\kappa_1$ such that $$\|K x\|_{Y_\bold r} \leq \kappa_1 (\|Nx\|_{\bold r} + \|p x\|_V )$$
  \item There is a constant $\kappa_2$ such that: 

 $$\|x\|'\leq \kappa_2(  \|N x\|_{Y_\bold r}+\|K x\|_{Y_\bold r} ) $$
 where $\kappa_i$ are uniform in $\bold r$. 
  \end{enumerate} 
 
\noindent Then, denoting $X'_\bold r=(X_\bold r, \|-\|')$,  the operator $N: X_\bold r'\to Y_\bold r\oplus V$ is uniformly invertible, i.e. there is a constant $\kappa$ independent of $\bold r$ such that $$\|x \|' \leq \kappa (\|N x\|_{Y_\bold r} + \|px\|).$$  
  \end{enumerate}
 \end{lm}
 \begin{proof}
The first statement of (1) follows directly from $A, C$ being invertible. The inverse is given explicitly by $$N^{-1}=\begin{pmatrix} A^{-1} & 0 \\ -C^{-1}B A^{-1} & C^{-1}
 \end{pmatrix}.$$
 If $A$ has kernel, but $C$ is invertible, the form of the kernel follows directly from the form of $N$. The statement involving $N\oplus p$ is immediate. 
 \noindent For assertion (2), the conclusion follows directly from applying the estimate {\bf (ii)} then {\bf (i)} successively. 
 \end{proof}
 
 This lemma will be applied in that case that $N=\Box$ with $p$ the projection to the kernel of $\slashed D_{A^H}^\C$, as suggested by the notation. The remainder of Step 2 focuses on the diagonal terms $A=\slashed D_{A^H}$ and $C=-\Delta -|\Phi^H|^2$ to verify the hypotheses of part 1. of the lemma. The subsequent Step 3 addresses the hypotheses of part 2 of Lemma  \ref{abstractlem}.

The following Integration by parts identities are needed: 

\begin{lm}
For $u,v\in T\mathcal G^\C_\bold r$, the following integration by parts formulas hold: 

 \begin{itemize}\item  $\int_{D_\bold r} \br \Delta u, v \kt \ dV = \int_{D_\bold r}\br 2\delbar u, 2\delbar v\kt \ dV.$  
 \item  $\int_{D_\bold r} \br \Delta u, v \kt \ dV=\int_{D_\bold r} \br \nabla u, \nabla v \kt \ dV+ \int_{\del D_\bold r} \br i \del_\theta u , v \kt \ d\theta$
  
 \end{itemize}
\label{intbyparts1}
\end{lm}
\begin{proof}
One has the following integration by parts formlae for $\del, \delbar$: 
\begin{eqnarray}
\int_{D_\bold r} \br 2\del u, v\kt + \br u, 2\delbar v\kt  \ dV &=& \int_{\del D_\bold r} \br u,v\kt \rho e^{i\theta}d\theta\\
\int_{D_\bold r} \br 2\delbar u, v\kt + \br u, 2\del v\kt  \ dV &=& \int_{\del D_\bold r} \br u,v\kt \rho e^{-i\theta}d\theta
\label{deldelbarbyparts}
\end{eqnarray} 
\noindent Since for $u\in T\mathcal G^\C$,  $$u= u_{-1} e^{-i\theta} + u_{-2}e^{-2i\theta}+\ldots\hspace{1cm}\delbar u= f_1 e^{i\theta}
+ f_2 e^{2i\theta}+\ldots$$ and likewise for $v$, the boundary term$$\int_{\del D_\bold r} \br  \delbar u, v \kt \rho e^{i\theta}d\theta=0$$
vanishes. Consequently, 
\bea
\int_{D_\e}\br \Delta u ,v \kt dV &=& \int_{D_\bold r} \br -4\del \delbar u, v\kt \ dV= \int_{D_\bold r} \br 2\delbar u, 2\delbar v\kt \ dV
\eea
yielding the first equality. For the second recall in polar coordinates $\Delta=-\tfrac{1}{\rho}\del_\rho (\rho\del \rho)- \tfrac{1}{\rho^2}\del_\theta^2$. Then the equality $$d( \br \del_\rho u, v \kt \rho d\theta)= (\br  \del_\rho u, \del_\rho v \kt + \br  \tfrac{1}{\rho}\del_\rho (\rho \del_\rho u), v\kt ) \rho d\rho d\theta$$ implies 
\bea
\int_{D_\bold r} \br \Delta u, v\kt \ dV&=& \int_{D_\bold r} \br \del_\rho u, \del_\rho v\kt+ \tfrac{1}{\rho^2}\br \del_\theta u , \del_\theta v\kt \ dV - \int_{\del D_\bold r} \br \del_\rho u, v \kt \rho d\theta\\
&=& \int_{D_\bold r} \br \nabla u, \nabla v \kt -\int_{\del D_\bold r} \br  \delbar u, v\kt e^{-i\theta} \rho d\theta +\int_{\del D_\bold r} \br i\del_\theta u, v \kt  d\theta \\
&=& \int_{D_\bold r} \br \nabla u, \nabla v \kt +\int_{\del D_\bold r} \br i\del_\theta u, v \kt  d\theta
\eea
where we have used $\delbar = e^{i\theta}(\del_\rho + \tfrac{i}{\rho}\del_\theta)$ and observed the boundary term involving $\delbar$ vanishes for the same reason as in the first bullet point.  
\end{proof}  

This next lemma verifies the necessary hypotheses for the operator $C=-\Delta-|\Phi^H|^2$. The lemma after it verifies the same for $A=\slashed D_{A^H}^\C$. 

\begin{lm} Consider $(T\mathcal G^\C_\bold r, \|-\|_{T\mathcal G^\C})$ equipped the norm described in \ref{norms}. Then $$-\Delta -|\Phi^H|^2 : T\mathcal G^\C_\bold r\to L^2(D_\bold r \ ; \ \Omega^{1,1})$$
is uniformly invertible. 
\label{CLemma}
\end{lm}
\begin{proof}First observe that the first bullet point of Lemma \ref{intbyparts} shows $\br \Delta h, h\kt$ is positive, hence $$\int_{D_\bold r} \br (\Delta +|\Phi^H|^2)h, h\kt \geq  \int_{D_\bold r} |\Phi^H|^2 |h|^2 \ dV \geq c \| h\|_{L^2}$$
 since $|\Phi^H|$ is bounded below uniformly. This operator is therefore positive with a uniform lower bound on the lowest eigenvalue. Consequently, there is a uniform estimate: 
 \begin{equation}\|(\Delta +|\Phi^H|^2) h\|_{L^2}  \geq c \|h\|_{L^2}.\label{eigenvaluebound}\end{equation}
 
 Next, expanding and using the second integration by parts formula from Lemma \ref{intbyparts1}, 
 
 \bea
 \int_{D_\bold r}|(\Delta+|\Phi^H|^2) h|^2 \ dV&=&\int_{D_\bold r} |\Delta h|^2 + |\Phi|^4 |h|^2 + 2 \br \Delta h, h |\Phi^H|^2\kt \ dV \\
 &=& \int_{D_\bold r} |\Delta h|^2 + |\Phi^H|^4 |h|^2 + 2 \br \nabla  h,  \nabla (h |\Phi^H|^2)\kt \ dV + \int_{\del D_\bold r}2\br  i\del_\theta h, |\Phi^H|^2 h\kt d\theta\\
 &\geq &  \int_{D_\bold r} |\Delta h|^2 + |\Phi^H|^4 |h|^2 + |\nabla h|^2 |\Phi^H|^2 \ dV + \int_{\del D_\bold r}2\br  i\del_\theta h, |\Phi^H|^2 h\kt d\theta\\
 & & -\int_{D_\bold r} \Big | 2 \br \nabla h , h\kt  (\nabla |\Phi^H|^2) \Big |  \ dV
 \eea
 and writing $h=\sum_{\ell>0}h_\ell e^{-i\ell \theta}$ shows
$$ \int_{\del D}2\br  i\del_\theta h, |\Phi^H|^2 h\kt d\theta= \int_{D_\bold r} \sum_{\ell>0} \ell |h_\ell|^2 |\Phi^H|^2 d\theta\geq 0.$$
Young's inequality shows \bea \br \nabla h , h\kt (\nabla |\Phi^H|^2) &\leq&  \tfrac{\epsilon}{2} |\nabla h|^2 + \tfrac{1}{2\epsilon} |h|^2 (\nabla |\Phi^H|^2 )^2 \\
&\leq & \tfrac{1}{2}|\nabla h|^2 |\Phi^H|^2 + C |h|^2 
\eea 
for $\epsilon$ sufficiently small, since $|\Phi^H|$ is bounded below uniformly and $|\Phi^H|^2\sim \rho$ so $\nabla |\Phi^H|^2$ is uniformly bounded. Absorbing the first term and moving the second to the  other side yields 
\begin{eqnarray}
  \int_{D_\bold r} |\Delta h|^2 + |\Phi^H|^4 |h|^2 + |\nabla h|^2 |\Phi^H|^2 \ dV &\leq& \tfrac{1}{2}  \int_{D_\bold r}|(\Delta+|\Phi^H|^2) h|^2 \ dV+ C \int_{D_\bold r}|h|^2 \ dV \\
  &\leq &  C'  \int_{D_\e}|(\Delta+|\Phi^H|^2) h|^2 \ dV
  \label{eigenvalueabsorbing}
\end{eqnarray}
after applying (\refeq{eigenvaluebound}). To conclude, we note that the estimate 

$$\int_{D_\bold r} |\nabla^2 h |^2 \ dV \leq C \int_{D_\bold r} |\Delta h |^2 \ dV$$

\noindent holds uniformly in $\bold r$. It is trivial on $\R^2$ via integration by parts, and if it were not true uniformly in $\bold r$ then on a sequence $h_n$ of unit norm in $T\mathcal G^\C_{r_n}$ violating the inequality, $\chi_n h_n$ would violate the inequality on $\R^2$ for a sequence of large cutoffs.  
\end{proof}

\bigskip 
\begin{lm}\label{ALemma}
Consider $(\mathcal H_\bold r, \|-\|_{L^{1,2}_{-1}})$ equipped with the norm described in \ref{norms}. Then $$\slashed D_{A^H}^\C: \mathcal H_\bold r \lre L^2(D_\bold r ; \C^4)$$ is surjective with kernel of real dimension 4 given by the complex span of 
$$ \beta_1 =\begin{pmatrix}  0 \\ e^{-H} \cdot \rho^{-1/2} \end{pmatrix} \otimes 1  \hspace{1cm} \beta_2 =  \begin{pmatrix}  0 \\ e^{-H} \cdot \rho^{-1/2} \end{pmatrix} \otimes j. $$
Moreover, if $p: \mathcal H_{\bold r} \to \ker(\slashed D^\C_{A^H})$ is the orthogonal projection to the kernel (with respect to the  $\|-\|_{L^{1,2}_{-1}}$ norm), then $$\|\psi\|_{L^{1,2}_{-1}}\leq C (\|\slashed D_{A^H}^\C \psi\|_{L^2}+ \|p(\psi)\|)$$
\noindent holds uniformly in $\bold r$.  
\end{lm}

\begin{proof}
First, we identify the kernel. It follows from the discussion of APS boundary conditions that 
\bea
D^\C_{A^H}=\begin{pmatrix} 0 & -2\del_{A^H} \\ 2\delbar_{A^H}& 0 \end{pmatrix} : \begin{matrix} L^{1,2}_{0,+}(D_\bold r; \C^2 )\\ \oplus \\ L^{1,2}_{-,-1}(D_\bold r ; \C^2)\end{matrix}\lre \begin{matrix} L^2(D_\bold r; \C^2) \\ \oplus \\ L^2(D_\bold r; \C^2) \end{matrix}
\eea
is a bounded Fredholm operator of (real) Index 4, since $\gamma(A^H)$ is a compact perturbation. To see the kernel is as claimed, we (complex) gauge transform to the standard complex structure. Let $U=\log(\rho^{1/2})$, and $A_0= \tfrac{1}{4}\left(\tfrac{dw}{w}-\tfrac{d\overline w}{\overline w}\right)$ be the singular connection in $\rho$-coordinates. Then we have the relations
\bea
e^U\cdot \Gamma_0& =&  A_0\\
e^H \cdot A_0&=& A^H
\eea  
where $\Gamma_0$ is the trivial connection (the second equality is the definition of $A^H$). Thus letting $G=H+U$, the gauge transformation $e^G~r^{1/2}$ acts by $$e^G \cdot \Gamma_0\mapsto A_H.$$
Notice that $G$ is non-singular at the origin, since $U,H$ have the same leading order term with opposite signs. Since $G$ is rotational invariant, it preserves the boundary values and thus the property of lying in $\mathcal H_\bold r$.  Consequently $$\delbar_{A^H}(e^G u)=e^G \delbar u \hspace{1cm} \del_{A^H} (e^{-G}v)=e^{-G}\del v$$  
where $\delbar,\del$ are the standard operators on $\C$. Since $e^{\pm G}\neq 0$ we see an element $(\alpha,\beta)$ is in the kernel if and only if $(e^{-G}\alpha, e^G \beta)$ are holomorphic and anti-holomorphic respectively in the standard complex structure. Since with the boundary conditions allow no holomorphic functions, and only the constant anti-holomorphic functions, we find the kernel consists of elements $\beta$ such that $e^G\beta=\text{const}$. The assertion for the form of the kernel follows, and the surjectivity from the index computation.    

The uniform estimate is given by Proposition \ref{delbarAHprop}, with the inconsequential caveat that we are free to instead take the weight $R$ used in that Proposition   \ref{delbarAHprop} (which is equal to $\rho$ outside a compact region) to be $\sqrt{1+\rho^2}$. 

 \end{proof}

 \noindent {\bf Step 3: (Uniform Invertibility)}
 
 To obtain the uniform estimate of Proposition \ref{holomorphicinvertible}, we apply part 2. of Lemma \ref{abstractlem}. In this step, we verify the estimates {\bf (ii)} and then {\bf (i)}. The first, {\bf (i)} is obtained from the solvability of $\Box$, and {\bf (ii)} follows from the Weitzenb\"ock formula. The next two Lemmas establish first {\bf (ii)} and then {\bf (i)}. 
 
 We define the projection operator $$K: \widehat{ H}^1_\bold r\to L^2_\bold r$$ by \begin{equation}K(\ph,a):= \frac{\ph}{R}.\label{Kdef}\end{equation}
Since $\|(\ph,a)\|_{\widehat{H}^1_\bold r}$ contains the term $\int_{D_\bold r} \tfrac{|\ph|^2}{R^2} \ dV$ ,   $K$ is well-defined and bounded. 

\begin{lm}There is a constant $\kappa_2$ such that the estimate  

\begin{equation}
\|(\ph,a)\|_{\widehat{ H}^1_\bold r}\leq \kappa_2 (\|\widehat{\mathcal N}^\C(\ph,a)\|_{L^2} +  \|K(\ph,a)\|_{L^2})
\end{equation}
holds uniformly in $\bold r$. 
\label{hypothesisii}
\end{lm}

\begin{proof}
This proposition is Weitzenb\"ock formula (Proposition \ref{Weitzenb\"ock}) combined with the observation that the cross-term is bounded by $K$. We begin by showing that the chosen boundary conditions have no boundary terms when integrating by parts. 

\begin{claim} For $(\ph,\omega)\in \widehat{H}^1_\bold r$ one has $$\int_{D_\bold r}\br \slashed D_{A^H}^\C \ph, \gamma(\omega)\Phi^H\kt \ dV = \int_{D_\bold r}\br \ph, \slashed D_{A^H}^\C\gamma(\omega)\Phi^H\kt  \ dV. $$ 
\label{integrationbypartsA}
\end{claim}
\begin{proof}
\noindent  Explicitly (recalling the expression \ref{cliffordmult} for Clifford multiplication) the left side in terms of the $\alpha,\beta$ component of the spinor is $$\text{Re}\int_{D_\bold r} \br  -2\del_{A^H} \beta, -\overline \omega \beta^H \kt+  \br  2\delbar_{A^H}\alpha,  -\omega \alpha^H \kt \ dV. $$

\noindent Using the integration by parts formulae 
\begin{eqnarray}
\int_{D_\bold r} \br 2\del u, v\kt + \br u, 2\delbar v\kt  \ dV &=& \int_{\del D_\bold r} \br u,v\kt \rho e^{i\theta}d\theta\\
\int_{D_\bold r} \br 2\delbar u, v\kt + \br u, 2\del v\kt  \ dV &=& \int_{\del D_\bold r} \br u,v\kt \rho e^{-i\theta}d\theta
\label{deldelbarbyparts1}
\end{eqnarray} 
we see the boundary term is $$\text{Re}\int_{\del D_\bold r} \br -\beta, -\overline\omega \beta^H\kt  e^{i\theta}+ \br \alpha, -\omega \alpha^H \kt  e^{-i\theta}  \ \rho d\theta $$

\noindent Since $\alpha^H, \beta^H$ have only constant and $e^{-i\theta}$ boundary modes respectively, and 
\bea
\alpha|_{\del D} &=& \ldots \alpha_{-2}e^{-2i\theta} + \alpha_{-1}e^{-i\theta} \\
\beta|_{\del D} &=& \hspace{4.5cm}  \beta_0 + \beta_1 e^{i\theta}+ \beta_2 e^{2i\theta} +\ldots \\
\omega|_{\del D} &=&  \hspace{4.5cm}    \ \ \ \ \  \ \ \omega_1 e^{i\theta}+ \omega_2 e^{2i\theta} +\ldots  
\eea
there are no overlapping Fourier modes on the boundary. This completes the claim. 
\end{proof}

We may write $$\widehat{\mathcal N}^\C(\ph,a)= \begin{pmatrix}\slashed D_{A^H}^\C \ph + \gamma(\omega)\Phi^H \\ \mu_\R(\ph, \Phi^H) + 2\del \omega \end{pmatrix}.$$

The above claim combined with the cancellation of the first-order terms as in Proposition [Weitzenb\"ock] shows that

\begin{eqnarray}
\int_{D_\bold r} |\widehat{\mathcal N}^\C (\ph)|^2  \ dV&=&\int_{D_\bold r} |\slashed D_{A^H}^\C \ph|^2 + |2 \del \omega|^2 + |\mu_\R (\ph,\Phi^H)|^2 + |\omega|^2 |\Phi^H|^2  \\   & &\ +  \  \br (\ph,\omega), \mathfrak B(\ph,\omega)\kt_{L^2}  \ dV\end{eqnarray}
where the cross-term $\mathfrak B(\ph,\omega)$ are as in Proposition \ref{Weitzenb\"ock}, i.e.  $$\mathfrak B\begin{pmatrix} \ph \\ \omega \end{pmatrix} = \begin{pmatrix}  -2 \omega\cdot \nabla_{A^H}\Phi^H \\2i \br i\ph, \nabla_{A^H}\Phi^H\kt  \end{pmatrix}$$
where $\cdot$ denotes the contraction of form components (where $\omega=(-\text{Im}(\omega)id\hat x+ \text{Re}(\omega)id\hat y)$ as in \ref{isomorphisms}), and the bottom component is a 1-form. We now claim the following two estimates:  

\begin{claim}
The estimate 

\bea \left(\int_{D_\bold r} \frac{|\ph|^2}{R^2}+ |\nabla \ph|^2 \ dV\right)^{1/2} &\leq&  C_1  (\|\slashed D^\C_{A^H}\ph\|_{L^2} + \|\tfrac{\ph}{R}\|_{L^2}) \\ &=&C_1  (\|\slashed D^\C_{A^H}\ph\|_{L^2} + \|K(\ph,a)\|_{L^2})  \eea 
holds uniformly in $\bold r$. 
\label{lockmcowenkernel}
\end{claim}

\begin{claim}
There is a constant $C_2$ such that 

$$ \br \begin{pmatrix}\ph \\ a \end{pmatrix}, \mathfrak B\begin{pmatrix}\ph \\ a \end{pmatrix}  \kt_{L^2}\leq C_2\|K(\ph,a)\|_{L^2}+\frac{1}{2}\int_{D_\bold r} |\mu_\R(\ph,\Phi^H)|^2+|\omega|^2|\Phi^H|^2  \ dV $$
holds uniformly in $\bold r$. 
\label{crossterm}
\end{claim}
 
\noindent  To conclude the proof of Lemma \ref{hypothesisii}, move the $\mathfrak B$ term to the right side and apply Claim \ref{crossterm}, then absorb the $\mu_\R$ and $|\omega|^2|\Phi^H|^2$ terms. Possibly increasing the constant on the $K(\ph,a)$ term, applying claim \ref{lockmcowenkernel} makes the left side into the $\widehat{ H}^1_\bold r$ norm.
\end{proof}

We now prove the two claims.  
\begin{proof} {\it (of Claim \ref{lockmcowenkernel})} For the $\alpha$ component this follows directly from Proposition \ref{delbaruniformestimate}. For the $\beta$ component, integration by parts as in Proposition \ref{deluniformestimate} shows such an estimate with an operator $K_1$ being a term arising from the exponentially decaying curvature $F_{A^H}$ (which is positive but acts by $-i$ on the $\beta$ component). This term is dominated by $K$ for a sufficiently large $C_1$.   
\end{proof} 

\begin{proof} {\it (of Claim \ref{crossterm})}
Recall $R=\sqrt{(1+\rho^2)}$. Observe that there are constants $C_1, c_1$ such that $$|\nabla_{A^H}\Phi^H|^2 \leq C_1 R^{-1} \hspace{1.5cm}  c_1 R \leq |\Phi^H|^2.$$

\noindent The first of these follows since $A^H, \Phi^H$ as exponentially close to $A_0$ and $\Phi_0$ respectively, and $\nabla_{A_0}\Phi_0\sim \rho^{-1/2}$ since $\Phi_0\sim \rho^{1/2}$. Additionally, $\nabla_{A^H}\Phi^H$ is bounded across the origin. Likewise, the second holds since $|\Phi^H|$ is non-zero and increasing, and exponentially close to $\Phi_0\sim \rho^{1/2}$. Combining these, there is a positive constant $c_2<<1$ such that 

\begin{equation}
c_2 |\nabla_{A^H}\Phi^H|^2 R^2 \leq \tfrac{1}{4} |\Phi^H|^2. 
\end{equation}

\noindent Then 
\bea
\br \ph, -2\omega \cdot \nabla_{A^H}\Phi^H\kt\leq \frac{|\ph|^2}{2c_2R^2} +  \frac{|\omega|^2}{2}c_2  |\nabla_{A^H}\Phi^H|^2 R^2\leq c_3 \frac{|\ph|^2}{R^2}+\frac{1}{4} |\omega|^2 |\Phi^H|^2
\eea
and identically for the second component. The claim follows.  
\end{proof} 
As a consequence, hypothesis {\bf (ii)} in the abstract Lemma \ref{abstractlem} is satisfied. The following final lemma establishes hypothesis {\bf (i)} in the abstract Lemma \ref{abstractlem}: 

\begin{lm}
There is a constant $\kappa_1$ such that the estimate

\begin{equation}
\|K(\ph,a)\|_{L^2}\leq \kappa_1 (\|\widehat{\mathcal N}^\C(\ph,a)\|_{L^2}+ \|p(\ph,a)\|_{L^{1,2}_{-1}})
\end{equation}
holds uniformly in $\bold r$. 
\label{hypothesisi}
\end{lm}
\begin{proof}
Suppose that $(y_1,y_2)$ satisfy $\widehat{\mathcal N}^\C (\ph,a)=(y_1, y_2)$. Then writing $(\ph,a)=(\psi + h\cdot \Phi^H, \delbar h)$ one has 

\bea
\slashed D^\C_{A^H} \psi &=& y_1 \\ 
(-\Delta - |\Phi^H|^2)h + \mu_\R(\psi, \Phi^H)&=&y_2
\eea
Since $K\ph=\tfrac{\ph}{R}$ times a constant, it suffices to show 

$$\int_{D_\bold r} \frac{|\psi + h\cdot \Phi^H|^2}{R^2}\  dV \leq C(\|(y_1,y_2)\|_{L^2}+\|p(\psi)\|_{L^{1,2}_{-1}}).$$

\noindent One has 

\bea
\int_{D_\bold r} \frac{|\psi + h\cdot \Phi^H|^2}{R^2} \ dV &\leq & \int_{D_\bold r} \frac{|\psi|^2}{R^2}+ \frac{|h|^2|\Phi|^2}{R^2} \ dV \\
& \leq &  \|\psi\|^2_{L^{1,2}_{-1}}+ \int_{D_\bold r} \frac{|h|^2 |\Phi^H|^2}{R^2} \ dV \\
&\leq &   \|y_1\|^2_{2} + \|p(\psi)\|_{L^{1,2}_{-1}}^2+ \int_{D_\bold r} \frac{|h|^2 |\Phi^H|^2}{R^2} \ dV. 
\eea   

\noindent where we have applied Proposition \ref{ALemma}. We wish to show a uniform bound on the second term when 
$$(\Delta+|\Phi^H|^2) h = \mu(\psi, \Phi^H) - y_2.$$

To see this, note that the right hand side may be written as $R^{3/2}f$  for a function $f\in L^2$ with $\|f\|_{L^2}\leq C(\|(y_1,y_2)\|_{L^2}+ \|p(\psi)\|_{L^{1,2}_{-1}})$, since $$\int_{D_\bold r} \Big |\frac{\mu(\psi, \Phi^H)-y_2}{R^{3/2}}\Big |^2 \ dV \leq C \int_{D_\bold r} \frac{|\psi|^2}{R^2}+ |y_2|^2 \ dV\leq C(\|(y_1,y_2)\|^2_{L^2}+\|p(\psi)\|^2_{L^{1,2}_{-1}})$$ where $C$ is such that $|\Phi^H|\leq Cr^{1/2}$. The second inequality follows from applying the uniform estimate on $\slashed D_{A^H}^\C$ from Proposition  \ref{ALemma} again.

Now let $h_1 \in T \mathcal G^\C_\bold r$ be the unique solution of $(\Delta +|\Phi^H|^2 )h_1 = f$. By the definition of the norm on $T\mathcal G^\C_\bold r$ one has $\|R^{1/2}\nabla h_1\|_{L^2} \leq C \| |\Phi^H|^2 \nabla h_1\|_{L^2}$ so $R^{1/2} \nabla h_1 \in L^2$, and clearly $|R^{-1/2}h_1| \in L^2$.  Thus $$ \Delta R^{3/2} h_1+ \nabla R^{3/2} \cdot \nabla h_1 \in L^2(D_\bold r; \C)$$ and its $T\mathcal G^\C$ norm is bounded by a constant time $\|(y_1,y_2)\|_{2} + \|p(\psi)\|_{L^{1,2}_{-1}}$. Now let $g\in T\mathcal G^\C_\bold r$ be the unique solution of  
$$(\Delta +|\Phi^H|^2)g =  \Delta R^{3/2} h_1+ \nabla R^{3/2} \cdot \nabla h_1 $$
satisfying the given boundary conditions, which also satisfies $\|g\|_{T\mathcal G^\C}\leq C( \|(y_1,y_2)\|_{2} + \|p(\psi)\|_{L^{1,2}_{-1}})$ and define $$h:= R^{3/2}h_1 -g.$$

\noindent Then  
\bea
(\Delta+|\Phi^H|^2) h&=& R^{3/2}(\Delta h_1 +|\Phi^H|^2 h_1) + \Delta R^{3/2} h_1+ \nabla R^{3/2} \cdot \nabla h_1 - (\Delta+\|\Phi^H|^2)g \\
&=& R^{3/2} f
\eea
is the unique solution sought. And it now follows that 
\bea
\int_{D_\bold r} \frac{|h\cdot \Phi^H|^2}{R^2} \leq C\int_{D_\bold r} \frac{|h|^2}{R}  \ dV\leq  C_1 \int_{D_\bold r} R^2 |h_1|^2 + \frac{|g|^2}{R} \ dV &\leq&  C_2 \|h_1\|^2_{T\mathcal G^\C} + \| g\|^2_{T\mathcal G^\C}\\  &\leq &C_3(\|(y_1,y_2)\|_{L^2}+\|p(\psi)\|_{L^{1,2}_{-1}}).
\eea

\end{proof}

We can now conclude the proof of Proposition \ref{holomorphicinvertible}. Indeed, Lemmas \ref{CLemma} and \ref{ALemma} show that the hypothesis of part (1) of Lemma \ref{abstractlem} are satisfied for $N=\widehat{\mathcal N}^\C=\Box$. Subsequently, Lemmas \ref{hypothesisi} and \ref{hypothesisii} show that the two hypothesis of part (2) are satisfied for $K$ as defined in Equation (\ref{Kdef}) and $p$ as in Lemma \ref{ALemma}. The proposition now follows from applying Lemma \ref{abstractlem}.

\end{proof}

\subsection{The General Case}
\label{generalcaseN}
\label{section6.5}

This subsection completes the proof of Proposition \ref{invertibleN} by deducing the general case from the case that $\zeta=0$ studied in the previous subsection. This involves two steps: first an integration by parts that shows the terms arising from $\zeta$ are strictly positive, and second, replacing the projection to the 4-dimensional kernel of $\widehat{\mathcal N}^\C$ with the 2-dimensional one of $\widehat{\mathcal N}$.

{\bf{Step 1:}}
We have the following integration by parts: 
\begin{lm}
$$ \int_{D_\bold r}|\widehat{\mathcal N}(\ph, \omega, \zeta)|^2  \ dV= \int_{D_\bold r} |\widehat{\mathcal N}^\C(\ph, \omega, 0)|^2+ |\mu_\C (\ph,\Phi^H)|^2 + |\nabla \zeta|^2 + |\zeta|^2|\Phi^H|^2  \ dV$$
\label{withb} 
\end{lm}

\begin{proof}
We may write
\bea \widehat{\mathcal N}(\ph, \omega, \zeta) &=&\widehat{\mathcal N}^\C(\ph, \omega)+ \begin{pmatrix}\gamma(\zeta)\Phi^H \\ 0 \\ -2\delbar \zeta + \mu_\C(\ph,\Phi^H)\end{pmatrix}\\
&=& \begin{pmatrix}\slashed D_{A^H}^\C & \gamma( \ \ )\Phi^H  \\ \mu_\R( \ , \Phi^H) &  2\del  \\ 0 & 0
\end{pmatrix} \begin{pmatrix}\ph  \\ \omega \\ \zeta  \end{pmatrix}+  \begin{pmatrix}\gamma(\zeta)\Phi^H \\ 0 \\ -2\delbar \zeta + \mu_\C(\ph,\Phi^H)\end{pmatrix}. \eea

Next, we integrate by parts to show the cross terms cancel, as in the Weitzenb\"ock formula Proposition \ref{Weitzenb\"ock}. The cross terms are:

\be
2\text{Re}\br -2\del_{A^H}\beta \ , \  i\zeta \alpha^H\kt \hspace{1cm}  2\text{Re}\br  2\delbar_{A^H}\alpha \ , \  i\overline \zeta \beta^H\kt \hspace{1cm}
2\text{Re} \br -2 \delbar \zeta \ , \ \mu_\C(\ph,\Phi^H) \kt.
\label{crossterms}
\ee

\noindent Provided we may integrate by parts with no boundary terms, the lemma follows from the same cancellation that occurs is Proposition \ref{Weitzenb\"ock} after it is pushed through the appropriate isomorphisms with the original form expressions via \ref{isomorphisms} . In order to show the boundary term vanishes, note that the allowed boundary modes are illustrated by

\begin{eqnarray} \int_{D_\bold r}\br -2\del_{A^H}\beta \ , \ i\zeta \alpha^H  \kt  + \br \beta \ , \ - 2\delbar_{A^H} i \zeta \alpha^H \kt &=&- \int_{\del D_\bold r} \br \beta, i\zeta \alpha^H  \kt \rho e^{-i\theta }d\theta  \label{partsintA}\\
\int_{D_\bold r}\br 2\delbar_{A^H}\alpha\ , \ i\overline \zeta \beta^H  \kt  + \br \alpha \ , \  2\del_{A^H} i \overline \zeta \beta^H \kt &=& \int_{\del D_\bold r} \br \alpha, i\overline\zeta \beta^H  \kt \rho e^{i\theta }d\theta
\label{partsintB}
\end{eqnarray}

\noindent and recall the boundary conditions require

\begin{eqnarray}
\alpha|_{\del D_\bold r} &=& \ldots \alpha_{-2}e^{-2i\theta} + \alpha_{-1}e^{-i\theta} \label{boundarymodesA} \\
\beta|_{\del D_\bold r} &=& \hspace{4.5cm}  \beta_0 + \beta_1 e^{i\theta}+ \beta_2 e^{2i\theta} \  +\ldots \\
\zeta|_{\del D_\bold r} &=& \ldots \zeta_{-2}e^{-2i\theta}  + \zeta_{-1}e^{-i\theta} \\
\omega|_{\del D_\bold r} &=&  \hspace{4.5cm}  0  \ + \  \omega_1 e^{i\theta}+ \omega_2 e^{2i\theta} +\ldots.
\label{boundarymodes}
\end{eqnarray}
\label{NBV2}

\noindent and, writing
 $$  \alpha_{-1}  =a_1 \otimes 1 + a_2 \otimes j  \hspace{1cm} \beta_{0} =b_1 \otimes 1 + b_2 \otimes j $$

\smallskip

\noindent where the subscript on the left hand sides denotes the Fourier mode, it is additionally required that  
\smallskip 
  $$ 0=\mu_\C^\del(\alpha,\beta)= b_1 \overline \alpha_1^H + \overline a_1 \beta_1^H + b_2 \overline \alpha_2^H + \overline a_2 \beta_2^H. $$
  
\noindent   Using (\refeq{boundarymodesA})-(\refeq{boundarymodes}), most modes on right hand side of (\ref{partsintA}) and (\ref{partsintB}) vanish. The real part of the boundary term becomes 

\bea
&=& \text{Re} \int_{D_\bold r}  -b_1 \overline{(i\zeta_{-1})} \overline \alpha_1^H -b_2 \overline{(i\zeta_{-1}) }\overline \alpha_2^H +a_1 \overline{ (i\overline \zeta_{-1})} \overline \beta_1^H+a_2 \overline {(i\overline \zeta_{-1})} \overline \beta_2^H \rho d\theta  \\
&=& \int_{D_\bold r} \text{Re}(  b_1\overline\alpha_1^H  (i\overline \zeta_{-1})+ b_2\overline\alpha_2^H  (i\overline \zeta_{-1}) + a_1 \overline \beta_1^H (-i \zeta_{-1} + a_2 \overline \beta_2^H (-i \zeta_{-1} ))\\
&=& \int_{D_\bold r} \text{Re}( \mu_\C^\del(\alpha,\beta)\cdot (i\overline \zeta_{-1}))=0
\eea 
where we have conjugated the second two terms. Thus the boundary term is 0 when integrating by parts. 

Then since $\delbar_{A^H}\alpha^H= \del_{A^H}\beta^H=0$, the cross term vanishes by the cancellation in the Weitzenb\"ock formula. Indeed, in this context, one has 
\bea
\text{Re}(\br -2\del_{A^H}\beta \ , \  i\zeta \alpha^H\kt +\br  2\delbar_{A^H}\alpha \ , \  i\overline \zeta \beta^H\kt)&=& \text{Re} (  \br \beta, 2\del_{A^H} (i \overline \zeta \alpha^H) \kt+\br \beta, -2\delbar_{A^H} (i \zeta \alpha^H) \kt )\\ &=& \text{Re}\Big \langle \begin{pmatrix} - \beta \\ \alpha \end{pmatrix} \ , \ \begin{pmatrix} -i2\delbar\zeta & 0 \\ 0 & -i2 \overline {\delbar \zeta} \end{pmatrix} \begin{pmatrix} \alpha^H \\ \beta^H\end{pmatrix} \Big \rangle\\
&=& \text{Re}\Big \langle  \mu_\R \left(\begin{pmatrix} - \beta \\ \alpha \end{pmatrix}, \begin{pmatrix} \alpha^H \\ \beta^H \end{pmatrix} \right) \ , \ -2\delbar \zeta\Big \rangle\\
&=& -\text{Re}\br  -2\delbar \zeta, \mu_\C(\ph,\Phi^H) \kt
\eea
since 

$$\mu_\R \left(\begin{pmatrix} - \beta \\ \alpha \end{pmatrix}, \begin{pmatrix} \alpha^H \\ \beta^H \end{pmatrix} \right)=  \beta_1 \overline\alpha_1^H + \overline \alpha_1 \beta_1^H +\beta_2 \overline \alpha_2^H + \overline \alpha_2 \beta_2^H=- \mu_\C(\ph, \Phi^H) $$

\noindent thus the cross terms (\ref{crossterms}) cancel after integrating by parts.  
\end{proof}

It now follows from the above identity of Lemma \ref{withb} in conjunction with the result for $\widehat{\mathcal N}^\C$,  Proposition \ref{holomorphicinvertible} that the estimate 

\begin{equation}\|(\ph,a)\|_{\widehat{ H}^1} \leq C (\|\widehat{\mathcal N}(\ph,a)\|_{L^2} + \|p(\psi)\|_{L^{1,2}_{-1}}) 
\label{otherprojection}\end{equation}

\noindent holds for $C$ independent of $\bold r$.

{\bf Step 2:}  The final step is to adjust the projection $p: \widehat H^1_\bold r\to \C^2$ to one valued in $\C$ (denoted by the same letter).  The point here is simply that when adding the $\mu_\C$ term, only two dimensions of the four (real) dimension are still kernel elements. 

Recall that the elements of the four (real) dimensional kernel of $\widehat{\mathcal N}$ can be written as the complex span of $$\beta_i=\beta_i^\circ + h_i \cdot \Phi^H$$

\noindent where \begin{equation} \beta_1^\circ =\begin{pmatrix}  0 \\ e^{-H} \cdot \rho^{-1/2} \end{pmatrix} \otimes 1  \hspace{1cm} \beta^\circ_2 =  \begin{pmatrix}  0 \\ e^{-H} \cdot \rho^{-1/2} \end{pmatrix} \otimes j. \label{ker2} \end{equation}
and $h_i$ is the unique solution of $$(-\Delta - |\Phi^H|^2)h_i = -\mu_\R(\beta_i^\circ,\Phi^H) $$
where $h$ satisfies the boundary conditions of (\refeq{doubleaps}).

Since $\mu_\C$ is complex gauge invariant, one has $\mu_\C(h\cdot \Phi^H, \Phi^H)=0$, hence $$\mu_\C(\beta_i, \Phi^H)= \mu_\C(\beta_i^\circ, \Phi^H)$$

\noindent The expressions (\refeq{ker2}) and the form of $\Phi^H$ show that for an kernel element  $k_1 \beta_1 + k_2 \beta_2$, one has  $$\mu_\C(k_1 \beta_1 + k_2 \beta_2, \Phi^H)= -k_1 e^{-H} \rho^{-1/2}\cdot \overline {c} e^H \rho^{1/2} - k_2 e^{-H}\rho^{-1/2}\cdot (-d)e^{H}\rho^{1/2}= -k_1 \overline {c} + k_2d  $$ 
is constant on $D_\bold r$. 

Assumption \ref{assumption2} implies $|c(t)|^2 + |d(t)|^2>0$ which shows that $$\mu_\C : \ker (\widehat{\mathcal N}^\C)\to \C$$ has full rank, and it is complex linear on the span of $\beta_i^\circ$. Let $\beta_t$ be an element whose complex span is the subset $\{\beta \in \ker(\widehat{\mathcal N}^\C) \ | \ \mu_\C(\beta)=0\}$. By construction $\beta_t |_{\del D_\bold r}$ satisfies $\mu_\C^\del(\beta_t)=0$, and so satisfies the boundary conditions. It is then not hard to show (argue by contradiction), that the four dimensional projection $p(\psi)$ can be replaced by the two dimensional one $p^\text{ker}(\psi)$. The details are omitted since in the next subsection we replace $p(\psi)$ with a projection that is more natural for the $H^1_\e$-norm, rather than $p(\psi)$ which is natural in the graph norm. This concludes the proof of Proposition \ref{invertibleN}.  

\qed

\subsection{The $L^2$-orthogonal projection} 
\label{section6.6}
The final detail to consider is switching the projection $p^{\ker}$, which is natural in the ``graph'' decomposition $(\psi,h)$ from (\refeq{graphiso}), to a projection which is more natural in the pair $(\ph,a)$. From here on we fix the size of the neighborhood of $\mathcal Z_0$ to have radius $$\boxed {\lambda= \e^{1/2}.}$$ so that $\bold r= (K(t))^{2/3} \e^{-1/6}$. 

Let $\beta_t$ continue to denote the kernel of $\mathcal N_{t}$ for every $t\in S^1$.

\begin{defn}
The (normalized) $L^2$- {kernel projection} $\pi_{t,\nu}^\text{ker}: {H}^1_{\e,}(D_\lambda)\to \C$ is defined to be  $$\pi_t^\text{ker}(\ph,a)= \int_{\{t\}\times D_\lambda} \frac{\br  (\ph,a),  \beta_t \kt}{\| \beta_t\|^2_{L^2(D_\lambda)}} dV$$
\end{defn}

\noindent The denominator normalizes it so that $\pi_t^{\ker}(\beta_t)=1$, since $\beta_t$ is normalized in the $H^1_{\e}$ norm rather than the $L^2$ norm.  

The version of Proposition \ref{invertibleN} which we will employ in the next section is then the following. 

\begin{cor}
For any fixed $t\in \mathcal Z$, assume that $(\ph,a)\in H^1_{\e}(\{t\}\times D_\lambda)$ is a configuration  satisfying the boundary conditions given in (\refeq{NBV1}). Then the estimate on the normal disk $\{t\}\times D_\lambda$ 

\bea  \|(\ph,a)\|_{H^1_{\e}(\{t\}\times D_\lambda)}\leq \frac{C}{\e^{1/12}} \left( \|\mathcal N_{t} (\ph,a)\|_{L^2(\{t\}\times D_\lambda)}+ \|\pi^{\ker}_{t}(\ph,a)\|_{2}\right)\eea 
holds for a constant $C$ independent of $t,\e$. 

\label{l2projectionNbound}
\end{cor}

\begin{rem} The above estimate is (obviously) not uniform in $\e$. By using the natural orthogonal projection in $H^1_\e(\{t\}\times D_\lambda)$, it is straighforward to obtain a uniform estimate. Using the $L^2$-norm is essential in the next section, however, where the projection must be controlled by the $t$-derivatives $\|\nabla_t (\ph,a)\|_{L^2}$. Finally, we remark that the constant $\e^{-1/12}$ depends on our choice of $\lambda=\e^{1/2}$; for  $\lambda= \e^{\alpha}$, the power of $\e$ that appears in the estimate is $(\tfrac{\alpha}{2}-\tfrac{1}{3})$ so the estimate becomes uniform as the radius approaches the invariant radius $O(\e^{2/3})$. However, it can never be uniform and also allow an intermediate region where the exponential decay estimates from Corollary \ref{mainb} to apply. Our choice of $\alpha=1/2$ is purely aesthetic, and any $\tfrac23>\alpha>\tfrac13$ would work.     
\end{rem}
\begin{proof}
By the scaling invariance in  \ref{diagramA} and Proposition \ref{scalingprop}, it suffices to show the estimate 
\bea  \|(\ph,a)\|_{\widehat H^1_{}( D_\bold r)}\leq \frac{C}{\e^{1/12}} \left( \|\widehat{\mathcal N}_{t} (\ph,a)\|_{L^2(D_\bold r)}+ \|\pi^{\ker}_{t}(\ph,a)\|_{2}\right)\eea 

\noindent in the invariant scale instead. Here we have not scaled the projection, and in a slight abuse of notation we have abbreviated $\overline \Upsilon_\e (\ph,a)$ by still denoting it $(\ph,a)$.

 We proceed now by contradiction using \ref{otherprojection}. Suppose there is no such $C$ satisfying the conclusion. Then for every $n\in \N$ there is an $\e_n$ and an element $(\ph_n, a_n)$ of unit $\widehat H^1_\bold r$ norm on the disk of radius $r_n= K(t)^{2/3} \e_n^{-1/6}$ such that 

$$\frac{\e^{1/12}}{n}\geq \|\widehat{\mathcal N}(\ph_n,a_n)\|_{L^2} + \|\pi^\text{ker}(\ph_n)\|.$$

We may write $(\ph_n, a_n)=( \psi_n + h_n\cdot \Phi^H, \omega_n, \zeta_n)$ where $\omega_n=2\delbar h_n$. Since $p(\psi)\neq 0$ else \ref{otherprojection} would be violated, $\psi$ must have some component in the kernel elements. Write $\psi_n= \psi_n^\text{ker}+\xi_n$ the orthogonal decomposition in $\mathcal H_\bold r$ so that $p(\xi_n)=0$. Writing $h_n= h_n^\text{ker} + h_n'$, the element can be expressed   

$$(\ph_n, a_n)=( \psi_n^\text{ker} + h^\text{ker}_n \cdot \Phi^H, 0,0)+(\xi_n + h_n'\cdot \Phi^H \  , \  \omega_n \ , \   \zeta_n).$$

\noindent where the first is in the kernel of $\widehat{\mathcal N}^\C$.

The equality 
\be \|\widehat{\mathcal N}(\ph_n, \omega_n, \zeta_n)\|_{L^2}^2=\|\widehat{\mathcal N}^\C(\ph_n, \omega_n, 0)\|^2_{L^2}+ \|\mu_\C (\ph_n,\Phi^H)\|^2_{L^2} + \|\nabla \zeta_n\|_{L^2}^2 + \|\zeta_n |\Phi^H|\|_{L^2}^2\label{fourterms}\ee 

\noindent from Lemma \ref{withb}, shows that the $L^2$ norm of each term on the right hand side must be individually less than $\tfrac{\e^{1/12}}{n}$.    In particular, from the first term, $$\|\widehat{\mathcal N}^\C(\xi_n + h_n'\cdot \Phi^H, \omega_n, 0 )\|=\|\widehat{\mathcal N}(\ph_n, \omega_n,0)\|\leq \frac{\e^{1/12}}{n}$$

\noindent and since $p(\xi_n)=0$ vanishes on $\xi_n$ by construction, and trivially on $(h_n'\cdot \Phi^H,\omega_n, \zeta_n)$ since $p$ does not see those components, the result for $\widehat{\mathcal N}^\C$ (Proposition \ref{holomorphicinvertible}), combined with the bounds on the third and fourth terms implies a bound on all the parts of the $\widehat{H}^1_{\bold r}$-norm except the $\mu_\C$ part. That is, letting $\phi_n=\xi_n + h_n'\cdot \Phi^H  $

\be  \|(\phi_n, \omega_n)\|_{\widehat H^1_\C} + \|\nabla \zeta_n\|_{L^2}+ \|\zeta_n |\Phi^H|\|_{L^2}\leq \frac{C\e^{1/12}}{n} \label{mucomitted}\ee  

\noindent for $C$ independent of $\bold r$. Recall in this that the $\widehat H^1_\C $-norm is given by omitting the $\mu_\C$ component, as in (\ref{H1Cnorm}). 

Now write $$\beta^\text{ker}_n= \psi^\text{ker}_n + h^\text{ker}_n \cdot \Phi^H$$ for the kernel element, and decompose it \be \beta_n^\text{ker}= k_n\beta_t + j_n \beta_t^\perp\label{kerndecomp}\ee
where $\beta_t$ is the true kernel element and $\beta_t^\perp$ is the element in the kernel of $\widehat{\mathcal N}^\C$ not satisfying $\mu_\C=0$ normalized in the  $\widehat{H}^1_\C$-norm.  

Next, we claim $j_n$ is small. On the unit disk, we have 

$$\int_{D_1} |\mu_\C(j_n \beta_t^\perp + \phi_n, \Phi^H)|^2 dV \leq  \int_{D_\bold r} |\mu_\C(j_n \beta_t^\perp + \phi_n, \Phi^H)|^2 dV\leq \frac{\e^{2/12}}{n^2}$$

\noindent but on the unit disk, $\Phi^H, R$ are universally bounded, hence $|\mu_\C(\phi_n,\Phi^H)|\leq C | \phi_n| \leq C\tfrac{|\phi_n|}{R}$ on $D_1$, for universal constants. Therefore (\refeq{mucomitted}) implies $$\|\mu_\C(\phi_n,\Phi^H)\|_{L^2(D_1)}\leq \frac{C\e^{1/12}}{n}$$ on the unit disk, thus since $\mu_\C(\beta^\perp_t,\Phi^H)$ is constant on the unit disk, we must have $j_n\leq \tfrac{C\e^{1/12}}{n}$, and thus $\|j_n \beta_t^\perp\|_{\widehat{H}^1_\C}\leq \frac{C\e^{1/12}}{n}$.  

Combining this with (\ref{mucomitted}) and the inequality on the second term of (\ref{fourterms}), we have 

\begin{eqnarray} \|(j_n\beta_t^\perp + (\phi_n, \omega_n,\zeta_n))\|_{\widehat{H}^1}&\leq&  \|j_n \beta_t^\perp + (\phi_n, \omega_n)\|_{\widehat{H}^1_\C} + \|\nabla \zeta_n\|_{L^2}\\ & & \   + \  \|\zeta_n |\Phi^H|\|_{L^2} + \|\mu_\C(\ph_n, \Phi^H)\|_{L^2} \\ & \leq &  \frac{C\e^{1/12}}{n}  + \frac{C\e^{1/12}}{n} + \frac{\e^{1/12}}{n}\leq \frac{C\e^{1/12}}{n}.\label{qnest}\end{eqnarray}

\noindent In the last term we have used that $\mu_\C(\ph_n, \Phi^H)= \mu_\C(j_n\beta_t^\perp+\phi_n, \Phi^H)$ since $\mu_\C$ vanishes on $\beta_t$. 

Finally, we are able to conclude that $$\|k_n \beta_t\|_{\widehat{H}^1} \geq 1-\tfrac{C\e^{1/12}}{n}$$
\noindent and so $|k_n|\geq 1-\tfrac{C\e^{1/12}}{n}$ as well. This now yields a contradiction to $\|\pi_t^\text{ker}(\ph_n, a_n) \|\leq \tfrac{\e^{1/12}}{n}$. Indeed, writing $(\ph_n,a_n)= k_n \beta_t + q_n$ then (\refeq{qnest}) and re-scaling back shows that $\|q_n\|_{\widehat{H}^1}= \|\overline \Upsilon_\e^{-1}q_n\|_{H^1_\e}\leq \tfrac{C\e^{1/12}}{n}$. Then by Cauchy-Schwartz, working now on the disk of radius $\e^{1/2}_n$,   

\bea
\frac{\e^{1/12}}{n}\geq |\pi_t^\text{ker}(\ph_n, a_n)|&=&\Big | \int_{D_\bold r} \frac{\br k_n \beta_t + \overline \Upsilon_\e^{-1} q_n, \beta_t\kt }{|\beta_t|_{L^2}^2} dV\Big | \\ & \geq& |k_n| -\frac{\|\overline \Upsilon_\e^{-1} q_n\|_{L^2} }{\|\beta_t\|_{L^2}}\cdot \frac{\|\beta_t\|_{L^2} }{\|\beta_t\|_{L^2}}\\
&\geq & |k_n| - \|\tfrac{\overline \Upsilon_\e^{-1} q_n}{R_\e}\|_{L^2} \cdot \e^{1/2} \tfrac{1}{\|\beta_t\|_{L^2}}\\
&\geq & |k_n| - \|\overline \Upsilon_\e^{-1}q_n\|_{ H^1_\e} \cdot \tfrac{\e^{1/2}}{\|\beta_t\|_{L^2}}
\eea
but, $\beta_t \sim \rho^{-1/2}$ (the detailed proof of this is given in the next section in Lemma \ref{betaasymptotics} ), so it $\|\beta_t\|_{L^2}\geq c \e^{1/2+1/12}$. Therefore, the latter term is bounded by a constant times $\tfrac{1}{n}$, while $k_n\sim 1$, yielding a contradiction. This completes the proof.

\end{proof}

The last order of business is to show the bounds $\|\beta_t\|_{L^2}\geq c \e^{1/2+1/12}$ used in the final two sentences of the above proof of Corollary \ref{l2projectionNbound}, and also that similar bounds hold for the $t$-derivative $\dot \beta_t$ (which are used in Section \ref{section7}). These follow from basic integration if one knows that $\beta_t\sim \rho^{-1/2}$ and similarly for the derivative $\dot \beta_t$. 

\begin{lm}
The elements $\beta_t$ have non-vanishing leading order term so that  $$\beta_t \sim \rho_t^{-1/2} $$
for $\rho_t>\!>1$. As a consequence, we have the following bounds where the constants $C,c, \kappa_1$ are independent of $\e,t$  
\begin{enumerate}
\item[1)]  $c\e^{1/2+1/12}\leq \|\beta_t\|_{L^2(D_\lambda)} \leq C \e^{1/2+1/12}$
\item[2)] For $\rho_t>\!> 1$ sufficiently large, $|\dot \beta_t| \leq \kappa_1 |\beta_t|$ pointwise. 
\item[3)] ${\|\dot \beta_t\|_{L^2(D_\lambda)}}\leq \kappa_1{\|\beta_t\|_{L^2(D_\lambda)}}$  \ and \ ${\|\dot \beta_t\|_{L^2(\del D_\lambda)}}\leq \kappa_1{\|\beta_t\|_{L^2(\del D_\lambda)}}$
\end{enumerate}
\label{betaasymptotics}  
\end{lm}

As in equation (\refeq{kerndecomp}) in the proof of Corollary \ref{l2projectionNbound}, the kernel element may be decomposed $$\beta_t = \psi_t + h_t \cdot \Phi_0$$
where $h_t\cdot \Phi_0$ is the component of the kernel tangent to the complex gauge orbits. It was shown in Lemma \ref{ALemma} that $\psi_t \sim \rho^{-1/2}$, thus the above lemma simply asserts that $h_t$ has asymptotics that does not disrupt this. The proof of this is straightforward but slightly intricate, and is given in Appendix \ref{appendix3}.  

\section{The Linearization}
\label{section7}

This section proves that in the proper context, the linearization $\mathcal L^{h_\e}$ at the de-singularized configuration $(\Phi^{h_\e}, A^{h_\e})$ is invertible. The precise statement is given in the below Theorem \ref{invertibleL}, which is the main technical result of this article. It proves a precise version of the ``proto-theorem'' stated in Section 5 below Equation (\refeq{non-lineartosolve}), and the main results Theorem \ref{main} and Theorem \ref{mainc} follow directly from Theorem \ref{invertibleL} as is shown in Section 8. 

The statement of Theorem \ref{invertibleL} is almost identical to the statement of Theorem \ref{mainc}, but replaces the linearization at the approximate solutions $\mathcal L_\e$ with the linearization at the de-singularized configurations $\mathcal L^{h_\e}$. To review the statement briefly, the operator $\mathcal L^{h_\e}$ is viewed as a first-order boundary value problem on the tubular neighborhood $N_{\lambda}(\mathcal Z_0)$ where $\lambda=c\e^{1/2}$. This is done by introducing a Hilbert space $H$ and a projection \be \Pi^{\mathcal L}: L^{1,2}(N_\lambda(\mathcal Z_0))\to H \label{bdorthogonalityconstraints7.1}\ee
    
\noindent so that $\ker(\Pi^\mathcal L)$ is the subspace of sections satisfying certain boundary and orthogonality conditions. The statement also relies on the weighted norms $\| -\|_{H^1_{\e,\nu}}$ and $\|-\|_{L^2_{\e, \nu}}$ defined in Section \ref{subsection5.1}.  
 \begin{thm} {\bf (Invertibility of $\mathcal L^{h_\e})$} 
\label{invertibleL}
Subject to the boundary and orthogonality conditions defined by (\refeq{bdorthogonalityconstraints7.1}), the extended gauge-fixed linearization at the de-singularized configurations $(\tfrac{\Phi^{h_\e}}{\e}, A^{h_\e})$

\be \mathcal L^{h_\e}: \ker(\Pi^\mathcal L)\subseteq L^{1,2}(N_\lambda(\mathcal Z_0)) \lre L^2(N_\lambda(\mathcal Z_0))\ee

\noindent is Fredholm of Index 0. Additionally, there is an $\e_0>0$ such that for $\e<\e_0$ it is invertible, and there are positive constants $C,\gamma^\text{in}\!<\!<1$ independent of $\e$ such that  the bounds

\smallskip
\be \|(\ph,a)\|_{H^1_{\e,\nu}} \leq  \frac{C}{\e^{1/12+\gamma^\text{in}}} \  \|\mathcal L^{h_\e}(\ph,a)\|_{L^2_{\e,\nu}}\label{Linvertiblebound}\ee

\be\|(\ph,a)\|_{H^1_{\e,\nu}} \leq   C\e^{1/12-\gamma^\text{in}} \|\mathcal L^{h_\e}(\ph,a)\|_{L^2}\label{linvertiblemixedweight}\ee
\smallskip
\noindent hold for weights $0\leq \nu<\tfrac14$. Notice the distinction is that there is no weight on the codomain in the latter bound. 
\end{thm}

Specifically, when $\gamma'$ is the small constant that was used to define the interior region in the proof of Lemma \ref{errorcalculation}, then $$\gamma^\text{in}= \tfrac{2}{3}\left(\tfrac{1}{4}-\nu\right)+ \nu \gamma'$$

\noindent hence $\gamma^\text{in}\!<\!<1 $ when $\gamma'$ is chosen suitably small and $\nu$ suitably close to $\tfrac{1}{4}$. 
 
 The remainder of this Section is devoted to the proof of Theorem \ref{invertibleL}, and is organized as follows. Akin to Section 6.2, Section 7.1 develops the Fredholm theory of Atiyah-Patodi-SInger boundary value problems in for the Dirac operator on a 3-manifold though in a slightly non-standard context more suitable to the problem at hand. Section 7.2 describes a distinguished subspace of configurations--section of the ``kernel subbundle''-- which play a prominent role in the proof. Section 7.3 gives the precise definition $H$ and $\Pi^\mathcal L$ in the statement of \ref{invertibleL}, and Section 7.4 concludes the proof via an integration by parts argument. Section 7.5 then deduces the case of a general metric from the model case. We continue to assume, until that section, that the assumptions of the model case (definition \ref{modeldef}) hold.

 \subsection{APS Boundary Conditions in 3d}
\label{aps3dsubsection}

\subsubsection{Untwisted Boundary Conditions}\label{aps3dsubsection7.1.1}

Consider $Y=S^1\times D^2$ equipped with the product metric. Let $(t,r,\theta)$ be cylindrical coordinates. As in  Section \ref{APSbdsection} there is a restriction (or trace) map $$\Tr: L^{1,2}(S^1\times D^2 ; \C^2)\to L^{1/2,2}(T^2;\C^2)$$
to the boundary values and we will choose a ``half-dimensional'' subspace $H_0\subseteq L^{1/2,2}(T^2;\C^2)$. Typically, one choose the negative eigenspace of the induced Dirac operator on the boundary (see \cite{KM} Section 17), which leads to an Index 0 problem. For our purposes, an alternative choice of a $t$-independent space $H_0$ is more suitable. The restriction to the boundary torus of a spinor $\ph=(\alpha,\beta)$ can be decomposed in Fourier series $$\begin{pmatrix} \alpha \\ \beta \end{pmatrix}\Big |_{S^1\times \del D}=\sum_{k,\ell \in \Z} \begin{pmatrix} \alpha_{k\ell} \\ \beta_{k\ell} \end{pmatrix} e^{i\ell t} e^{ik\theta}.$$

We define the subspace $H_0\subseteq L^{1/2,2}(T^2;\C^2)$ by $$H_0 = \Big\{(\alpha,\beta)\in L^{2} \ | \ \alpha = \sum_{k<0, \ell \in \Z} a_{k\ell}e^{i\ell t}e^{ik\theta} \ , \ \beta=\sum_{k>0, \ell \in \Z} \beta_{k\ell}e^{i\ell t}e^{ik\theta} \Big\}\cap L^{1/2,2}(T^2;\C^2).$$
\noindent Equivalently, in the notation of Section \ref{APSbdsection} (recall Equation \refeq{projections}) we require $\alpha_\ell(\theta)$ has vanishing $H^+_{[0]}$ component and $\beta_\ell(\theta)$ has vanishing $H^-_{[0]}$ component for every $\ell \in \Z$. Let
$$\Pi_0: L^{1,2}(S^1\times D^2;\C^2)\to H_0^\perp$$
be the projection to the orthogonal complement of $H_0$, so that $$\ph\in \ker(\Pi_0) \  \ \Leftrightarrow  \ \ \ph| _{S^1\times \del D} \in H_0.$$

\noindent Pictorially, associating the boundary Fourier modes with the lattice $\Z^2$ where $\ell $ is the vertical index and $k$ the horizontal, the condition to lie in $\ker(\Pi_0)$ says that $\alpha$ has non-zero boundary modes on the left half-lattice  
\begin{eqnarray*}
\text{Fourier mode}   & &  \ldots {\underline{k=\!-\!2} } \hspace{.3cm}     {\underline{k=\!-\!1} }  \hspace{.3cm}{\underline{k=0} } \  \hspace{.30cm} {\underline{k>0 }}  \ \ \ldots \\
 &&  \ \ \   \ \vdots \   \  \  \ \ \ \ \ \  \ \  \ \vdots \   \   \ \ \ \ \  \ \ \ \hspace{.05cm}   0   \hspace{2.5cm} \\
\underline{\ell=2}   \ \ &&\ldots \alpha_{-2,2}  \  \ \  \ \ \alpha_{-1,2}  \  \  \ \ \  0  \hspace{2.5cm} \\
\underline{\ell=1} \ \ &&\ldots \alpha_{-2,1}  \  \  \ \ \ \alpha_{-1,1}  \  \ \ \ \  0  \hspace{2.5cm} \\ 
\underline{\ell=0}  \ \ && \ldots \alpha_{-2,0}  \  \ \ \ \ \alpha_{-1,0}  \ \ \  \ \   0  \ \  \ \ \ \  \ \ \ldots 0 \ldots    \\
\underline{\ell=\!-\!1} \ \  && \ldots \alpha_{-2,0}  \  \ \ \  \ \alpha_{-1,0}  \  \ \  \ \  0\hspace{2.5cm}     \\
\underline{\ell=\!-\!2} \ \ &&  \ldots \alpha_{-2,0}  \  \  \ \ \ \alpha_{-1,0}  \  \ \  \ \  0\hspace{2.5cm}   \\
 &&  \ \ \   \ \vdots \   \  \  \ \ \ \ \ \  \ \  \ \vdots \   \   \ \ \ \ \  \ \ \   \hspace{.05cm} 0    \hspace{2.5cm}
\end{eqnarray*}
\label{NBV3d}

\noindent while $\beta$ has non-zero modes in the right half-lattice. Equivalently, since the boundary conditions only restrict the $\theta$ Fourier modes, we can express the condition $(\alpha,\beta)\in \ker(\Pi_0)$ as 
\begin{eqnarray*}
 \hspace{1.75cm}  \underline{k=-1}  \  \   \ &\underline{ k =0}& \ \ \ \underline{k=1} \medskip  \\  
\ldots  \alpha_{-2}(t)  \  \ \   \ \ \ \ \alpha_{-1}(t)     \   \ \ \ &0&   \ \ \ \  \  \ 0  \  \   \ \  \ \ \ \ \ \ \ 0 \ldots \\
\ldots  0  \   \ \ \  \ \ \ \ \   \ \ \ \ 0 \ \ \ \  \     \   \ \ \ &0 & \ \ \ \  \beta_{1}(t)  \  \  \ \ \ \  \beta_{2}(t) \ldots \\ 
\end{eqnarray*}
\label{NBV3d}

\noindent The next proposition shows that the Dirac operator with these boundary conditions has Index 0. Although this result is quite standard, it is beneficial to give a proof here that is suggestive of the eventual proof of the invertibility of $\mathcal L^{h_\e}$. The key point is that an estimate on the operator defined on slices of constant $t$ is applied for each $t$ and then integrated over $t\in S^1$ 
\begin{prop}
The boundary value problem  \medskip $$(\slashed D,\Pi_0):L^{1,2}(S^1\times D^2 ; S_E)\lre L^2(S^1\times D^2 ; S_E)\oplus H_0^\perp$$ 
\noindent  is invertible, and {\it a fortiori} Fredholm of Index 0. Equivalently, the same holds for the  operator $$\slashed D: \ker(\Pi_0)\lre L^2(S^1\times D^2; S_E).$$
\label{untwisteddiracbd}\qed 
\end{prop}

The proof is a standard application of integration by parts. 

\subsubsection{Twisted Boundary Conditions}
\label{twistedbdsubsection1}
The boundary conditions we will impose on $\mathcal L^{h_\e}$ are based on a twisted variation of the boundary conditions given in the previous Subsection \ref{aps3dsubsection7.1.1}. In this subsection, define an abstract version of the twisted boundary conditions and calculate the Fredholm index of the resulting Dirac operator. This result will be employed later to calculate the index of $\mathcal L^{h_\e}$ in the context of Theorem \ref{invertibleL}. 

The idea of the twisted boundary conditions is that for each fixed $t$, we allow four new modes by allowing $\beta_0$ to be non-vanishing and then impose four constraints on linear combinations of $\beta_0$ and $\alpha_{-1}$. An example of such a linear constraint is provided by the condition $\mu_\C^\del$ that appear in the boundary conditions for $\widehat{\mathcal N}_t$ in Definition \ref{boundarymodes}.

To make this precise, let $E_{-1,0} \to S^1$ be the trivial vector bundle with fiber $\R^8$, where the fiber is thought of as the complex span \be  (E_{-1,0})_{t}= \text{Span}_\C\left \{\begin{pmatrix} e^{-i\theta} \\ 0 \end{pmatrix}\otimes 1 \ \ , \ \ \begin{pmatrix} e^{-i\theta} \\ 0 \end{pmatrix}\otimes j\  \ , \ \ \begin{pmatrix}  0 \\ 1  \end{pmatrix}\otimes 1\ \ , \ \  \begin{pmatrix} 0 \\ 1 \end{pmatrix}\otimes j \right \}\ee

\noindent so that the space of sections $\Gamma(E_{-1,0})\subseteq L^2(\del (S^1\times D); S_E)$ is the closed subspace consisting of boundary configurations
$$\begin{pmatrix}\alpha_{-1}(t) e^{-i\theta}\\ 0 \end{pmatrix}  \hspace{1cm} \begin{pmatrix}0 \\ \beta_0(t) \end{pmatrix}. $$

\noindent In terms of the previous diagram, it is the subspace spanned by the boxed modes. 
\begin{eqnarray*}
 \hspace{1.75cm}  \underline{k=-1}  \  \   \ &\underline{ k =0}& \ \ \ \underline{k=1} \medskip  \\  
\ldots  \alpha_{-2}(t)  \  \ \   \ \  \ \ \ \boxed{\alpha_{-1}(t)}    \     \ \ &\alpha_0(t)&  \ \ \ \  \alpha_{1}(t)  \  \  \ \ \ \  \alpha_{2}(t) \ldots \\
\ldots  \beta_{-2}(t)  \   \  \ \ \ \   \ \ \ {\beta_{-1}(t)}    \   \   \ \ &\boxed{\beta_0(t)} & \ \ \ \  \beta_{1}(t)  \  \  \ \ \ \  \beta_{2}(t) \ldots \\ 
\end{eqnarray*}

Let $$V_t\subseteq E_{-1,0}$$ denote a real 4-plane distribution, and set \be H_1:=  \{(\alpha,\beta)\in L^{2} \ | \ \alpha = \sum_{k<-1, \ell \in \Z} a_{k\ell}e^{i\ell t}e^{ik\theta} \ , \ \beta=\sum_{k\geq 0, \ell \in \Z} \beta_{k\ell}e^{i\ell t}e^{ik\theta} \} \label{H1def}\ee 
to be the previously allowed modes omitting the $\alpha_{-1}$ and $\beta_0$ modes. Then consider
 $$H_{\text{Tw}}= (H_1 \oplus L^2(S^1;V_t)) \cap L^{1/2,2} \hspace{1.5cm} H_{\text{Tw}}^\perp= (H_1 \oplus L^2(S^1;V_t))^\perp  \cap  L^{1/2,2}  $$ 

\noindent and denote $$\Pi_{\text{Tw}}: L^{1,2} \to H_{\text{Tw}}^\perp$$ the projection to the orthogonal complement. 
\begin{defn}
The {\bf $V_t$-Twisted Boundary Conditions} are given by the requirement that 

$$\ph|_{\del(S^1\times D^2)} \in  H_{\text{Tw}} \   \Leftrightarrow \  \Pi_{\text{Tw}}(\ph)=0$$
\noindent so that the allowed modes are illustrated by 
\begin{eqnarray*}
 \hspace{1.75cm}  \underline{k=-1}  \  \   \ &\underline{ k =0}& \ \ \ \underline{k=1} \medskip  \\  
\ldots  \alpha_{-2}(t)  \  \ \   \ \ \  \boxed{ \alpha_{-1}(t)}       \ \ \ &0&   \ \ \ \  \  \ 0  \  \   \ \  \ \ \ \ \ \ \ 0 \ldots \\
\ldots  0  \   \ \ \  \ \ \ \ \   \ \ \ \ 0 \ \ \ \  \     \   \ \ \ &\boxed{\beta_0(t)} & \ \ \ \  \beta_{1}(t)  \  \  \ \ \ \  \beta_{2}(t) \ldots \\ 
\end{eqnarray*}
subject to the constraint that $$\hspace{.9cm }\boxed{\alpha_{-1}(t)}+ \boxed{\beta_0(t)}\in V_t \ \ \forall t\in S^1.$$
\begin{flushright}$\diamondsuit$\end{flushright}
\end{defn}

As an example, the untwisted case considered in the previous subsection is the special case that $V_t= \text{span}\{\alpha_{-1}\}$, i.e. $\beta_0=0$ for all $t$.  We constrain the distribution $V_t$ in two ways. These constraints are expressed in terms of two anti-involutions, which we now define. Write the fiber of $E_{-1,0}$ as $\C^2\otimes_\C \mathbb H$ so elements may be written $$\ph=\begin{pmatrix} \alpha_{-1} \\ \beta_0\end{pmatrix}$$ 
where the two components are $\mathbb H$-valued. Let $$J: \C^2\otimes \mathbb H \to \C^2\otimes \mathbb H \hspace{1.5cm}  \sigma_1 : \C^2\otimes \mathbb H \to \C^2\otimes \mathbb H $$ denote, respectively, the involutions 

$$J\begin{pmatrix} \alpha_{-1} \\ \beta_0\end{pmatrix}= \begin{pmatrix} - \beta_0 \\ \alpha_{-1}\end{pmatrix} \hspace{2cm} \sigma_1=\begin{pmatrix}i & 0  \\ 0 & -i\end{pmatrix} $$ 
(Note that $J$ is complex linear, and is not the charge conjugation map often denoted by the same letter). 
\begin{lm} \label{lagrangiansubspacelem}
The following hold. Throughout, we use the real inner product on $\C^2 \otimes \mathbb H$. 

\begin{enumerate}
\item[(1)] $(\sigma_1J)^{2}=-Id$, hence $\sigma_1J$ is an almost-complex structure. In particular, $(\sigma_1J)v \perp v$ for any $v \in \C^2\otimes \mathbb H$. 

\item[(2)] For spinors $\ph=(\alpha,\beta) \in \ker(\Pi_{\text{Tw}})$, the operator $$\slashed D^\C=\begin{pmatrix} 0 & -2\del \\ 2\delbar & 0\end{pmatrix}$$ satisfies the integration by parts formula 
$$\int_{S^1\times D^2} \br \slashed D^\C \ph    ,  \ph   \kt   \ - \   \br  \ph   ,  \slashed D^\C  \ph \kt \ dV \ = \int_{T^2} \Big \br  J\begin{pmatrix}\alpha_{-1} \\ \beta_0 \end{pmatrix}  \ , \  \begin{pmatrix}\alpha_{-1} \\ \beta_0 \end{pmatrix} \Big \kt  \ dA.$$ 
\end{enumerate}
\end{lm}

\begin{proof}By definition $$(\sigma_1J)^2 \begin{pmatrix}\alpha \\ \beta \end{pmatrix}=\sigma_1 J \sigma_1 \begin{pmatrix}-\beta \\ \alpha \end{pmatrix}=\sigma_1 J \begin{pmatrix}-i\beta \\ -i\alpha \end{pmatrix} =\sigma_1 \begin{pmatrix}i\alpha \\ -i\beta \end{pmatrix}=-\begin{pmatrix}\alpha \\ \beta \end{pmatrix}    $$

\noindent and taking the real inner-product and conjugating the bottom term, $$\text{Re}\Big \br\sigma_1 J\begin{pmatrix}\alpha \\ \beta \end{pmatrix}, \begin{pmatrix}\alpha \\ \beta \end{pmatrix}  \Big \kt = \text{Re}( \overline{-i\beta}\alpha) + \text{Re}(\overline{-i\alpha}\beta)=\text{Re}(i \overline{\beta}\alpha) + \text{Re}(-i\alpha\overline \beta)=0.$$

\noindent which completes item (1). 

Item (2) follows immediately from the previously used integration by parts formulas

\begin{eqnarray}
\int_{D^2} \br -2\del \beta, \alpha \kt + \br \beta, -2\delbar \alpha\kt  \ dV &=& \int_{\del D} \br -\beta,\alpha \kt  e^{i\theta}d\theta\\
\int_{D^2} \br 2\delbar \alpha, \beta\kt + \br \alpha, 2\del \beta\kt  \ dV &=& \int_{\del D} \br \alpha,\beta\kt e^{-i\theta}d\theta
\end{eqnarray} 
and the observation that the condition $\Pi_{\text{Tw}}(\alpha,\beta)=0$ implies the only non-zero inner product for the boundary modes occurs in the $\alpha_{-1}$ and $\beta_0$ modes. 
\end{proof}

As a consequence of the above Lemma \ref{lagrangiansubspacelem}, there is a complex-linear isomorphism $(E_{-1,0}, \sigma_1 J)\simeq (\C^4, i)$ where the latter is given the standard almost-complex structure. This endows the former with a symplectic structure for which $\sigma_1 J$ is a compatible almost-complex structure given by the pullback of the standard symplectic structure on $\C^4$. We impose the following hypotheses on the distribution $V_t$: 
\bigskip

\noindent  \ \ \ {\bf Hypothesis (I):} Assume that $V_t\subseteq E_{-1,0}$ is a bundle of Lagrangian subspaces with respect to the

  \hspace{2.2cm}  \ \ \ symplectic structure compatible with the almost-complex $\sigma_1 J$.  In particular, this 
  
  \hspace{2.2cm} \ \ \   requires  $(\sigma_1 J)V_t \perp V_t$ for all $t\in S^1$. 
  \bigskip 

\noindent  \ \ \ {\bf Hypothesis (II):} $V_t$ is homotopic through distributions satisfying {\bf (I)} to a constant distribution. 

\bigskip 

The twisted analogue of Proposition \ref{untwisteddiracbd} is the following: 

\begin{prop}
Suppose that the 4-plane distribution $V_t$ satisfies hypotheses {\bf (I)} and {\bf (II)}. Then the operator \medskip  \be (\slashed D, \Pi_{\text{Tw}}): L^{1,2}(S^1\times D^2; S_E) \lre L^2(S^1\times D^2; S_E)\oplus H_{\text{Tw}}^\perp\label{fullbdvaluesdirac}\ee
\medskip

\noindent is Fredholm of Index 0. 

\label{twistedbdindex}
\end{prop}

\begin{proof}

First, the hypothesis {\bf (I)} implies that this operator is Fredholm. Integrating by parts and using Young's inequality and {\bf {(I)}} shows the boundary is bounded above by $\tfrac{\epsilon}{2}\|(\Pi_\text{Tw})^\perp\ph\|^2+ \tfrac{1}{2\epsilon}\|\Pi_\text{Tw}\ph\|^2$. Choosing $\epsilon$ sufficiently small and absorbing the first into the left-hand side shows   

\be \|\ph\|^2_{L^{1,2}}\leq C ( \|\slashed D\ph\|^2_{L^2} + \|\Pi_{\text{Tw}}\ph\|^2_{L^{1/2,2}} + \|K\ph\|^2_{L^2})\label{injclosedrange}\ee

\noindent where $K: L^{1,2}\to L^2$ is a compact operator. Using this, it follows from standard theory that $(\slashed D, \Pi_{\text{Tw}})$ has closed range and finite dimensional kernel. Integrating by parts on $\br \slashed D \ph, \psi\kt$ shows an element of the complement of the range must solve $\slashed D\psi=0$ subject to the twisted boundary conditions for the distribution $W_t=\sigma_1 V_t$, which also satisfies {\bf (I)} hence (\refeq{injclosedrange}) applies to show the cokernel is finite dimensional.  

Hypotheses {\bf (II)} implies that $(\slashed D, \Pi_{\text{Tw}})$ is homotopic through Fredholm operators to one for which $V_t$ is constant. Since the space of 4-planes in $\R^8$ satisfying hypothesis {\bf (I)} is homeomorphic to the Lagrangian Grassmannian, it is connected, and $V_t$ it is therefore homotopic to the distribution $$V_0 = \text{span}\begin{pmatrix}\alpha_{-1}\\ 0\end{pmatrix}$$
which obviously satisfies hypothesis {\bf (I)}. The twisted boundary condition for $V_0$ is the untwisted boundary condition of Proposition \ref{NBV3d}, which has index 0. 

\end{proof}

\subsubsection{The Degenerating Family}
\label{motfortwisted1}

The boundary and orthogonality conditions $\Pi^\mathcal L$ for $\mathcal L^{h_\e}$ used in Theorem \ref{invertibleL} are more intricate than a simple choice of twisted Lagrangian distribution $V_t$. Before proceeding, we describe the geometric intuition motivating their definition. 

In order to identify a proper context in which $\mathcal L^{h_\e}$ is invertible, one must understand more precisely how the family of Fredholm operators   \begin{equation}\slashed D_{A^{h_\e}}\to \slashed D_{A_0}\label{diraclimit}\end{equation} degenerate to the singular semi-Fredholm operator in the limit, and in particular how the infinite-dimensional cokernel of Proposition \ref{infcokernel} arises. (Here, the limit of the operators should be interpreted only in an imprecise sense, as the difference is not bounded in $L^2$). One might expect that there is an infinite-dimensional family of eigenfunctions with small eigenvalues approaching 0 for which the ratio of the $L^{1,2}$ to the $L^2$ norm becomes infinite, which gives rise to the infinite-dimensional cokernel in the limit (which consists of kernel elements that are $L^2$ but not $L^{1,2}$).  

Indeed, this occurs in the two-dimensional case. Here, assuming the metric on $D_\lambda$ is a product, the cokernel of the limiting operator $\slashed D_{A_0}^\C$ is finite-dimensional and spanned by $$k_1=\begin{pmatrix}0 \\ r^{-1/2}\end{pmatrix}\otimes 1  \ \ \ \text{and } \ \ k_j=\begin{pmatrix}0 \\ r^{-1/2}\end{pmatrix}\otimes j $$
\noindent on $D_\lambda$. The kernel of the de-singularized operator $\slashed D_{A^{h_\e}}$ in $L^{1,2}$ is spanned by the configurations \be \beta_1= \begin{pmatrix} 0 \\ e^{-H}\rho^{-1/2}\end{pmatrix}\otimes 1=\left(\frac{K(t)}{\e}\right)^{-\tfrac{1}{3}}\begin{pmatrix} 0 \\ e^{-h_\e}r^{-1/2}\end{pmatrix}\otimes 1\label{smoothedcokernel}\ee  

\noindent  (and likewise for the $\otimes j$ component) that were described in Lemma \ref{ALemma}. For every $\e>0$, these are smoothed off versions of the function $r^{-1/2}$ where the smoothing occurs closer and closer to the origin. Thus these elements converge in $L^2$ (after renormalizing in $L^2$) to the limiting cokernel element $k_1$ which fails to be in $L^{1,2}$. 

Counterintuitively, in the three-dimensional case on $S^1\times D_\lambda$ this picture is only correct for the constant Fourier mode in the $t$-direction. The reason for this is that for the $\ell \neq 0$ modes, the two spinor components $\alpha,\beta$ are coupled. Again assuming the metric on $S^1\times D_\lambda$ is Euclidean, the infinite-dimensional cokernel is spanned (over $\C$) in $L^2$ by the elements 

\be \psi^1_\ell = \sqrt{|\ell|} e^{i\ell t} \begin{pmatrix} e^{-i\theta} \ \  \tfrac{e^{-|\ell | r}}{\sqrt{r}} \\ \text{sgn}(\ell)\tfrac{e^{-|\ell | r}}{\sqrt{r}} \end{pmatrix}\otimes 1 \hspace{1cm}\psi^j_\ell = \sqrt{|\ell|} e^{i\ell t} \begin{pmatrix} e^{-i\theta} \ \  \tfrac{e^{-|\ell | r}}{\sqrt{r}} \\ \text{sgn}(\ell)\tfrac{e^{-|\ell | r}}{\sqrt{r}} \end{pmatrix}\otimes j.\label{psilimits}\ee
\noindent for $\ell \in \Z$. Recall that the complex gauge transformation acts by $e^{-h_\e}\sim r^{1/2}$ in the bottom component, but by $e^{h_\e}\sim r^{-1/2}$ in the top component. For $\ell \neq 0$, one sees that after smoothing the bottom component behaves as in the two-dimensional case, but the top component becomes {\it more} singular. Consequently, in this case the smoothed cokernel elements analogous to (\refeq{smoothedcokernel}) are now {\it neither} $L^{1,2}$ \text{nor} $L^2$ along $\mathcal Z_0$. This suggests that the infinite-dimensional cokernel that appears in the limit $\e\to 0$ does not arise from a family of $L^{1,2}$- eigenspaces; instead each $\psi_\ell$ appears to arise from an $\e$-parameterized family of elements that are not even in $L^2$ for $\e>0$, but which limit to an element of $L^2$. 

One can confirm this picture with the following basic calculation. Consider replacing $A^{h_\e}$ with the nearby non-smooth connection given as follows. Let $\rho_0\sim \e^{2/3}$ be a fixed constant, and define a connection $A_1$ piecewise by setting it to be the product connection for $\rho<\rho_0$ and setting it equal to $A_0$ for $\rho \geq \rho_0$. Writing the Dirac operator in Fourier series leads to ODEs in both regions, and it is straightforward to check by matching boundary conditions at $r=\rho_0$ that there are no solutions with exponential decay away from $\mathcal Z_0$ that are locally $L^2$ along $\mathcal Z_0$. This property should persist under the minor smoothing that corrects $A_1$ to $A^{h_\e}$. 

A more accurate picture of what occurs for the boundary-value problem is as follows. Rather surprisingly, the infinite-dimensional kernel that appears in the limit $\e\to 0$ arises from elements with exponential growth away from $\mathcal Z_0$. For the discontinuous connection $A_1$ from the previous paragraph, there is an finite-dimensional family of kernel elements which are $L^2$ along $\mathcal Z_0$, but which look like \be\sim e^{i\ell t} \frac{e^{+|\ell| r}}{\sqrt{ r}}\label{expgrowth}\ee for large $r$. But on the disk $D_\lambda$ of radius $\lambda=c\sqrt{\e}
$, the elements (\refeq{expgrowth}) are still monotonically decreasing toward the boundary for values of $|\ell| \leq O(\e^{-1/2})$. Cutting these elements off near the boundary leads to a family of spinors (satisfying the boundary conditions) with unit $H^1_\e$-norm for which applying $\slashed D_{A^{h_\e}}$ result in very small elements in $L^2$.  For precisely the same reason as in Remark \ref{rem6.11}, such elements violate any uniform bound on the inverse. Indeed, as one can check, if (and only if!) $|\ell|\leq O(\e^{1/2})$, the configurations $e^{i\ell t}\beta_t$ where $\beta_t$ is a kernel element of $\mathcal N_t$ (as in \refeq{smoothedcokernel} above) results approximate kernel elements on which $\mathcal L^{h_\e}$ is very small. 

 Given this picture, to find a setting in which there is a uniform bound on the inverse of $\mathcal L^{h_\e}$, we must choose boundary conditions that allow these approximate kernel elements and  then project orthogonal to them, just as we did for $\mathcal N_t$. Once $\ell$ is sufficiently large, the $t$-derivative becomes violent enough that these elements are no longer almost in the kernel and projection orthogonal to them is no longer necessary. Thus {\it the boundary conditions on $\mathcal L^{h_\e}$ are taken to be combination of boundary conditions and orthogonal projections on the interior}. Specifically, we choose boundary conditions that allow the boundary values of $e^{i\ell t}\beta_t$ for low $\ell$ but disallow them for high $\ell$, while at the same time we impose orthogonal projections that disallow the elements $e^{i\ell t}\beta_t$ in the interior for low $\ell$ while leaving the projections unrestricted for high $\ell$. The key point of the proof of Theorem \ref{invertibleL} is to show that these can be done simultaneously without imposing so many constraints as to result in a large negative index. 

The upshot of this intuitive dicussion is that it is the exponentially growing elements of the form (\refeq{expgrowth}) that are the impediment to obtaining uniform elliptic estimates for $\mathcal L^{h_\e}$, rather than exponentially decaying ones limiting to (\refeq{psilimits}). This observation identifies the correct setup for Theorem \ref{invertibleL}. Using this setup, the proof proceeds in the next four subsections without reference to the  above intuitive picture (in particular, no claims about exponentially growth are explicitly made).  
Although this intuitive geometric picture guides our setup for Theorem \ref{invertibleL}, it is not necessary to make it precise. The justification of this picture's correctness lies in the fact that a setup designed with it in mind actually yields a proof of Theorem \ref{invertibleL} , while attempts to prove Theorem \ref{invertibleL} envisioning other pictures (such as one in which every mode is analogous to the zeroth mode) are completely confounding.

\subsection{The Kernel Bundle}
\label{kernelbundlesection}

As explained in the previous subsection \ref{motfortwisted1}, the constraints $\Pi^\mathcal L$ for the operator $\mathcal L^{h_\e}$ used in Theorem \ref{invertibleL} are a mixture of boundary conditions and orthogonality constraints. The boundary portion of these constraints are a specific case of the twisted boundary conditions discussed in the previous subsection \ref{aps3dsubsection}. In this subsection, we define the accompanying orthogonality constraints. These project orthogonal to (a subspace) of the configurations that lie in the kernel of $\mathcal N_t$ for every $t\in S^1$.

Let $D_\lambda$ continue to denote the disk of radius $\lambda=c\e^{1/2}$. For each $t\in S^1$, recall that $\beta_t$ is the element whose complex span is $\ker(\mathcal N_t)\subseteq H^1_{\e}(\{t\}\times D_\e)$ such that it is normalized in the ${H^1_{\e}( D_\lambda)}$-norm.  

\begin{defn}
Define the {\bf Kernel Subbundle} as $$K(\mathcal N_t) \subseteq S^1 \times H^1_{\e}(\{t\}\times D_\lambda)$$
where the latter is viewed as the trivial vector bundle over $S^1$ having fiber $H^1_{\e}(D_\lambda)$. Thus its sections are $$\Gamma(K(\mathcal N_t))= \{\eta(t)\beta_t   \ | \  \eta: S^1 \to \C\}.$$ 
\end{defn}

\medskip 

Before proceeding, let us make a brief remark on function spaces. We have versions of the space $ H^1_{\e}$ in both two and three dimensions. To distinguish we rename them $H^1_{slice}$ and $H^1_{\e}$  respectively. Explicitly, where $\nabla$ denotes only the derivatives in the $D_\lambda$-directions, the norms are given by

\bea
\|(\ph,a)\|_{ H^1_{slice}} &=&\left(\int_{D_\lambda}  |\nabla\ph|^2 + |\nabla a|^2 + \frac{|\ph|^2}{R_\e^2} + \frac{|\mu(\ph,\Phi^{h_\e})|^2}{\e^2} + \frac{|a|^2|\Phi^{h_\e}|^2}{\e^2} \ rdrd\theta\right)^{1/2} \\  
 \|\ph,a\|_{ H^1_{\e}}&=&\left(\int_{S^1}\int_{D_\lambda}  |\del_t \ph|^2 + |\del_t a|^2 +  |\nabla\ph|^2 + |\nabla a|^2 + \frac{|\ph|^2}{R_\e^2} + \frac{|\mu(\ph,\Phi^{h_\e})|^2}{\e^2} + \frac{|a|^2|\Phi^{h_\e}|^2}{\e^2}  \ rdrd\theta dt\right)^{1/2} \\
 &=&  \left(\int_{S^1} \left(\int_{ D_\lambda} |\del_t (\ph,a)|^2  \ rdrd\theta + \|u\|^2_{H^1_{slice}} \right)dt\right)^{1/2} 
\eea

Configurations $\eta(t)\beta_t \in \Gamma(K(\mathcal N_t))$, play a distinguished role because for $\eta$ with sufficiently small derivative, these form the  $O(\e^{-1/2})$-dimensional space approximate kernel elements described in the previous subsection \ref{motfortwisted1}.

 The sections of the kernel bundle $K(\mathcal N_t)$ are, by design, distinguished by being precisely the configurations on which $\mathcal N_t$ vanishes for every $t$. More precisely,  
 
 \begin{lm}
The inclusion \bea L^{1,2}(S^1; \ker (\mathcal N_t)) &\to& H^1_{\e} \\ \eta & \mapsto& \eta(t)\beta_t \eea
of sections of the kernel bundle has image characterized by  being precisely the sections $\xi$ such that $$\mathcal \|\mathcal N_t \xi \|_{L^2(S^1\times D_\e)} =0.$$ 
\label{kernelsubbundlecharacterization}
\end{lm}
\begin{proof}
Notice first that $\eta(t)\beta_t \in  H^1_{\e}$ since 
\bea \|\eta(t)\beta_t\|_{ H^1_{\e}}&=&  \int_{S^1} \left(\int_{ D_\lambda} |\dot \eta \beta_t + \eta \dot \beta |^2 \  \ rdrd\theta + \|\eta(t)\beta_t\|_{H^1_{slice}} \right)dt \\
&\leq & \int_{S^1} |\dot \eta|^2 \|\beta_t\|^2_{L^2(D_\lambda)} +   | \eta|^2 \|\dot \beta_t\|^2_{L^2(D_\lambda)}  + |\eta(t)|^2 \|\beta_t\|^2_{ H^1_{slice}} dt  \\
&\leq &  \int_{S^1} |\dot \eta|^2 \|\beta_t\|^2_{L^2(D_\lambda)} +   | \eta|^2 \kappa_1 \|\beta_t\|^2_{L^2(D_\lambda)}  + |\eta(t)|^2 dt\lesssim_\e \|\eta\|_{L^{1,2}}
\eea 
where in the third line we have invoked item (3) of Lemma \ref{betaasymptotics}  .On such configurations, $\mathcal N_t (\eta(t)\beta_t)=0$ by definition. Supposing conversely that a smooth element $\xi$ had $\mathcal N_t\xi=0$. Then, by definition of the kernel subbundle, we may write $\xi = \eta(t)\beta_t$ for a function $\eta(t)$. Clearly $\eta\in L^2(S^1;\C)$ if and only if $\xi \in L^2(S^1\times D_\lambda)$, and then reversing the above inequality shows such a configuration is in $ H^1_{\e}$ only if $\eta \in L^{1,2}$. 
\end{proof}

The projection $\pi^\text{ker}_{t}: H^1_{slice}\to \C$ now becomes a $t$-parameterized family of projections resulting in a function in $L^{1,2}(S^1;\C)$. 

\begin{defn}
the {\bf normalized projection to the kernel bundle} $$P:  H^1_{\e} \to L^{1,2}(S^1;\C)$$
 defined by
$$[P(\xi)](t)= \int_{\{t\}\times D_\lambda} \frac{\br \xi(t) , \beta _{t} \kt}{\|\beta_t\|^2_{L^2(\{t\}\times D_\lambda)}} \ rdrd\theta$$
\label{tdependentprojectionsdef}
\end{defn}

\noindent so that for each fixed $t$, the value is the value of the slice projection $\pi^\text{ker}_{t}$ for $\mathcal N_t$. Notice that since $\del_t \xi  , \del_t \beta_t  \in L^2$ this is a bounded map into $L^{1,2}$ by Cauchy-Schwartz, and since $\eta(t)\beta_t$ projects to $\eta(t)$, the previous Lemma \ref{kernelsubbundlecharacterization} shows its image is all of $L^{1,2}$.

We can also view the projection as a $\Z$-parameterized family of projections to $$P^\ell: H^1_{\e}\lre  \C$$ giving the Fourier modes:  \be P^\ell (\xi)= \int_{S^1} \int_{\{t\} \times D_\lambda} \frac{\br \xi , e^{i\ell t}\beta_{t} \kt}{\|\beta_t\|^2_{L^2(\{t\}\times D_\lambda)}} \  rdrd\theta dt\label{lthmodeprojection}\ee

\noindent so that if $\xi= \eta(t)\beta_t$ then $P^\ell(\xi)=\br \eta, e^{i\ell t}\kt_{L^2(S^1;\C)}$ is the $\ell^{th}$ Fourier coefficient, and the original projection is given by $$P(\xi)=\sum_{\ell \in \Z} P^\ell(\xi)e^{i\ell t}.$$

\noindent Clearly since $P(\xi)\in L^{1,2}$ the sequence satisfies  $\{P^\ell(\xi)\}_{\ell \in \Z}\in l^{1,2}(\Z)$.  

Additionally, we can split this family of projections into two regimes: the high and low Fourier modes. Let $L_0\in \N$ denote a large constant to be chosen later. In a slight abuse of notation, we write $\frac{1}{L_0}\e^{-1/2}$ to mean the  smallest integer greater than $\tfrac{1}{L_0}\e^{-1/2}$ if $\e^{-1/2}\notin \Z$.

\begin{eqnarray}
P^{\text{low}}: H^1_{\e} &\to&  \C^{1+ 2   \e^{-1/2}/L_0} \hspace{5.6cm} \xi \to \bigoplus_{|\ell| \leq \frac{1}{L_0}\e^{-1/2}} P^\ell(\xi) \label{Plowdef} \\
P^{\text{high}}: H^1_{\e} &\to &\{ \eta \in  L^{1,2}(S^1;\C) \ | \ \eta_\ell=0 \text{ for }|\ell| \leq  \tfrac{1}{L_0}\e^{-1/2}\}\hspace{.9cm} \xi \to \sum_{|\ell | \geq \frac{1}{L_0}\e^{-1/2}} P^\ell(\xi)e^{i\ell t} 
\end{eqnarray}

\subsection{Boundary and Projection Conditions}
\label{twistedbdsubsection}
In this section we define the precise constraints $\Pi^\mathcal L$imposed  on the operator $\mathcal L^{h_\e}$ in Theorem \ref{invertibleL}. As explained in subsection \ref{motfortwisted1}, these are a combination of boundary conditions and interior orthogonal projections using $P_\ell$. First, we cover the boundary conditions, which are a particular case of the twisted boundary conditions which appeared in subsection \ref{twistedbdsubsection}, and subsequently define the projection conditions.

\subsubsection{Pure Boundary Conditions}
 We define a twisted boundary condition by specifying a 4-dimensional distribution $V_t\subseteq E_{-1,0}$ as in subsection \ref{twistedbdsubsection1}.

 Let $a_1, a_2$ be the components of the $e^{-i\theta}$ boundary mode of $\alpha$, and $b_1, b_2$ be the components of the constant boundary mode of $\beta$ so that 

\smallskip 
 \be  \begin{pmatrix}\alpha_{-1}(t)  \\ \beta_0(t)\end{pmatrix}  =\begin{pmatrix} a_1(t) \\ b_1(t) \end{pmatrix}\otimes 1+\begin{pmatrix} a_2(t) \\ b_2(t) \end{pmatrix} \otimes j.\label{alphabetathing}\ee

\smallskip

\noindent Recall that  
\smallskip 
  \be \mu^\del_\C\left(\alpha, \beta , \zeta , \omega \right)= b_1 \overline \alpha_1^{H} + \overline a_1 \beta_1^{H} + b_2 \overline \alpha_2^{H} + \overline a_2 \beta_2^{H}.   \label{mucboundarydef}\ee
was the Index 2 boundary condition imposed on $\mathcal N_t$, where, $\alpha_i^{H}$ and $\beta_i^{H}$ are the components of $\Phi^H$ restricted to the boundary (the subscripts on these denote the $\otimes 1$ and $\otimes j$ components, not the Fourier modes). Here, we split this condition to give an index zero one.

 \begin{defn} Define the 4-dimensional Lagrangian distribution $V_t^{\Phi_0}$ determined by $\Phi_0$ by setting 
 
 \be
 \mu_\C^{\alpha}:= \overline a_1 \beta_1^H + \overline a_2 \beta_2^H \hspace{2cm} \mu_\C^\beta:= b_1 \overline \alpha_1^H + b_2 \overline \alpha_2^H
 \label{mualphabetadef}
 \ee
 and taking $$ V^{\Phi_0}_t := \{ (\alpha_{-1}, \beta_0)\in E_{-1,0} \ | \ \mu_\C^\alpha = \mu_\C^\beta=0\}$$\label{VPhidef}
 \end{defn}

 \begin{lm} \label{mufiltration}The distribution $V_t^{\Phi_0}$ fits into the proper complex filtration 
 $$\{0\} \ \subsetneq  \ \ker(\mathcal N_t) |_{\del D_\e} \ \subsetneq  \  V_t^{\Phi_0} \ \subsetneq \  ( \mu_\C^\del)^{-1}(0) \ \subsetneq \   E_{-1,0}.$$
 
\noindent in particular, configurations with boundary values in $V_t^{\Phi_0}$ satisfy the boundary conditions for $\mathcal N_t$ and the kernel elements $\beta_t$ have boundary values in $V_t^{\Phi_0}$.

 \end{lm}
 
 \begin{proof}
 The inclusion $V_t^{\Phi_0}\subsetneq (\mu_\C^{\del})^{-1}(0)$ is obvious since $$\mu_\C^{\del}= \mu_\C^{\alpha} + \mu_\C^\beta$$ hence vanishes if they vanish individually. 
 
 Next, recalling that the coefficients of the leading order term of $\Phi_0$ are denoted $c(t), d(t)$, notice that up a factors of $r^{-1/2}$ and $e^{-i\theta}$ on the boundary, we have $$\mu_\C^\alpha= \overline a_1 d(t) +\overline a_2 \overline c(t) \hspace{2cm} \mu_\C^\beta= b_1 \overline c(t) + b_2(-d(t)) $$
 
\noindent thus these see only the $a_i$ and $b_i$ components respectively. For each $t$, $(\mu_\C^{\alpha})^{-1}(0) \cap \{\beta_0=0\}$ and $(\mu_\C^{\beta})^{-1}(0)\cap \{\alpha_{-1}=0\}$ are spanned over $\C$ by 

\be v_1:=\begin{pmatrix}c(t) \\ 0 \end{pmatrix}\otimes 1 + \begin{pmatrix}-\overline d(t) \\ 0 \end{pmatrix}\otimes j \hspace{1.5cm}v_2:= \begin{pmatrix} 0 \\ d(t) \end{pmatrix}\otimes 1 + \begin{pmatrix} 0 \\ \overline c(t) \end{pmatrix}\otimes j\label{formofmualphabeta}\ee 
\noindent respectively. Thus $V_t^{\Phi_0}$ is the complex span the above elements. The condition $|c(t)|^2 + |d(t)|^2>0$ shows it is indeed $4$-dimensional over $\R$.  

To observe that $\beta_t |_{\del D_\e}\in V_t^{\Phi_0}$, recall that $\beta_t$ is the linear combination of 
\bea \beta_1&=& \begin{pmatrix} 0 \\ e^{-H}\rho^{-1/2} \end{pmatrix} \otimes 1+ h_1\cdot  \Phi^H\\ \beta_j&=& \begin{pmatrix} 0 \\ e^{-H}\rho^{-1/2} \end{pmatrix} \otimes j+ h_j\cdot \Phi^H \eea 

\noindent satisfying $\mu_\C=0$ with $h_1, h_j$ as in (the proof of) Lemma \ref{betaasymptotics}. Since $$\Phi^H=\begin{pmatrix} e^{H}c(t)\rho^{1/2} \\  e^{-H}d(t)\rho^{1/2} e^{-i\theta} \end{pmatrix}\otimes 1 + \begin{pmatrix}-e^H\overline d(t)\rho^{1/2} \\ e^{-H}\overline c(t) \rho^{1/2}e^{-i\theta}\end{pmatrix}\otimes j $$
and $h_1, h_j$ have only negative Fourier modes, it is immediate that the $h_1\cdot \Phi^H$ and  $h_j\cdot \Phi^H$ contribute boundary terms that are complex multiples of (\refeq{formofmualphabeta}) in the relevant Fourier modes. Thus \bea \mu_\C^{\alpha}(\beta_t)&=&0 \\ 
\mu_\C^{\beta}(\beta_t)&=& \mu^\del_\C(\beta_t)=0\eea
since the linear combination of the first terms of $\beta_1$ and $\beta_j$ is exactly the one satisfying $\mu_\C=0$ (notice that $h_j=\tfrac{\overline c(t)}{d(t)})h_1$ in a projective sense so the lower order terms contribute the same linear combinations as the leading order). This shows $\beta_t\in V_t^{\Phi_0}$.  

For completeness, we note that the expression for $\beta_t$ restricted to the boundary is
\bea
\beta_t|_{T^2}=\left[\begin{pmatrix} 0 \\ d(t) \end{pmatrix}\otimes 1 + \begin{pmatrix} 0 \\ \overline c(t) \end{pmatrix}\otimes j  + \frac{3K(t)}{4}\left(\begin{pmatrix} - c(t) e^{-i\theta} \\ d(t)\end{pmatrix}\otimes 1+ \begin{pmatrix} \overline d(t)e^{-i\theta} \\  \overline c(t)\end{pmatrix}\otimes j \right)\right] \rho_t^{-1/2} + O(\rho_t^{-1}).
\eea

\end{proof}

Next, we show this subbundle $V_t^{\Phi_0}$ satisfies the hypotheses of Proposition \ref{twistedbdindex} for the Dirac operator with these twisted boundary conditions to be Fredholm.  

\begin{lm}
The distribution $V_t^{\Phi_0}$ satisfies the hypotheses

\begin{enumerate}
\item[{\bf (I)}] $(\sigma_1 J)V_t^{\Phi_0} \perp V_t^{\Phi_0}$ for all $t\in S^1$. 
\item[{\bf (II)}] $V_t^{\Phi_0}$ is homotopic through distributions satisfying {\bf (I)} to a constant one.
 \end{enumerate}
 
 \noindent of Proposition \ref{twistedbdindex}.
 \label{hypothesisIII}
\end{lm}

\begin{proof}

Let $v_1, v_2$ denote the elements spanning $V_t^{\Phi_0}$ as in (\refeq{formofmualphabeta}). For {\bf (I)} it suffices to show the three hermitian inner products 

$$\br  \sigma_1 J v_1, v_1\kt=0 \hspace{1.5cm}\br  \sigma_1 J v_2, v_2\kt=0 \hspace{1.5cm} \br  \sigma_1 J v_1, v_2\kt=0 $$
all vanish. Notice the third implies the same reversing $v_1$ and $v_2$ since $\sigma_1J$ is orthogonal and squares to $-I$. 

The first two are obvious, since $v_1$ has only $\alpha$-components, while $Jv_1$ has only $\beta$-components. Likewise for $v_2$. For the third, 

\bea
\br  \sigma_1 J v_1, v_2\kt&=& \Big \br  \sigma_1 J \begin{pmatrix}c(t) \\ 0 \end{pmatrix}\otimes 1 +\sigma_1J \begin{pmatrix}-\overline d(t) \\ 0 \end{pmatrix}\otimes j  \ , \  \begin{pmatrix} 0 \\ d(t) \end{pmatrix}\otimes 1 + \begin{pmatrix}-\overline 0 \\ \overline c(t) \end{pmatrix}\otimes j\Big \kt \\
&=& \Big \br  \sigma_1  \begin{pmatrix}0 \\ c(t)  \end{pmatrix}\otimes 1 +\sigma_1 \begin{pmatrix}0 \\  -\overline d(t)  \end{pmatrix}\otimes j  \ , \  \begin{pmatrix} 0 \\ d(t) \end{pmatrix}\otimes 1 + \begin{pmatrix}-\overline 0 \\ \overline c(t) \end{pmatrix}\otimes j\Big \kt \\
&=& i (\overline cd +(\overline {-\overline d}) \overline c)= i (\overline cd- d\overline c)=0.
\eea

For {\bf (II)}, observe that the subspace $V_t^{\Phi_0}$ depends only on the functions $c(t)$ and $d(t)$ which are required to satisfy $|c(t)|^2 + |d(t)|^2>0$ by Assumption \ref{assumption2}. Thus normalizing, we view them as map \bea S^1 &\to& S^3 \subseteq \C^2\\ t&\mapsto &\frac{(c(t), d(t))}{|c(t)|^2 + |d(t)|^2} \eea 
and since $S^3$ is simply connected, there is a homotopy through pairs satisfying $|c(t)|^2 + |d(t)|^2>0$ connecting them to the constant pair $(1,0)$. 
\end{proof}

We now define 

\begin{defn} \label{purebddef} Write $(\ph,a)= (\alpha,\beta,\zeta,\omega)$ as in Section \ref{bdconditionsN}. Then $(\ph,a)$ satisfies the {\bf Pure Boundary Conditions} on $\mathcal L^{h_\e}$ if
\bea
(\alpha,\beta)|_{\del(N_\lambda(\mathcal Z_0))} \in H_\text{Tw}& \Leftrightarrow &\Pi_{\text{Tw}}(\alpha,\beta)=0 \\
(\zeta,\omega)|_{\del(N_\lambda(\mathcal Z_0))} \in H_0 \ \ & \Leftrightarrow &\Pi_{0}(\zeta,\omega)=0, \\
\eea
where $H_{\text{Tw}}$ uses the distribution $V_t^{\Phi_0}$ defined above in Definition \ref{VPhidef}, and $H_0$ is the untwisted version from Section \ref{aps3dsubsection}. The allowed boundary modes are illustrated by 
\begin{eqnarray*}
 \hspace{1.75cm}  \underline{k=-1}  \  \   \ &\underline{ k =0}& \ \ \ \underline{k=1} \medskip  \\  
\ldots  \alpha_{-2}(t)  \  \ \   \ \ \  \boxed{ \alpha_{-1}(t)}       \ \ \ &0&   \ \ \ \  \  \ 0  \  \   \ \  \ \ \ \ \ \ \ 0 \ldots \\
\ldots  0  \   \ \ \  \ \ \ \ \   \ \ \ \ 0 \ \ \ \  \     \   \ \ \ &\boxed{\beta_0(t)} & \ \ \ \  \beta_{1}(t)  \  \  \ \ \ \  \beta_{2}(t) \ldots \\ 
\ldots  \zeta_{-2}(t)  \  \ \   \ \ \  \ { \zeta_{-1}(t)}     \  \ \ \ &0&   \ \ \ \  \  \ 0  \  \   \ \  \ \ \ \ \ \ \ 0 \ldots \\
\ldots  0  \   \ \ \  \ \ \ \ \   \ \ \ \ 0 \ \ \ \  \     \   \ \ \ &  0  & \ \ \  \ \omega_{1}(t)  \  \  \ \ \ \  \omega_{2}(t) \ldots \\ 
\end{eqnarray*}
\noindent where the boxed modes are constrained so that $$\hspace{.9cm }\boxed{\alpha_{-1}(t)}+ \boxed{\beta_0(t)}\in V_t^{\Phi_0} \ \ \forall t\in S^1.$$

\end{defn}

\begin{cor}
Subject to the pure boundary conditions, the boundary-value problem
\medskip  \be (\mathcal L^{h_\e}, \Pi_{\text{Tw}} \oplus \Pi_0) : H^1_{\e}(N_\lambda(\mathcal Z_0)) \lre L^2(N_\lambda(\mathcal Z_0))\oplus H_\text{Tw}^\perp \oplus H_0^\perp\ee
\medskip 
is Fredholm of Index 0. \label{corpurebdindex}
\end{cor}
\begin{proof}
The spaces in question are equivalent to the space $L^{1,2}$ and $L^2$ respectively since the domain is compact, so we may disregard the weighted norms. The compactness of the domain also implies that, (although the off-diagonal and connection terms are large),

$$\begin{pmatrix}
\slashed D_{A^{h_\e}} & \gamma(\_)\frac{\Phi^{h_\e}}{\e} \\ \frac{\mu(\_,\Phi^{h_\e})}{\e}  & \bold d 
\end{pmatrix}  \ = \begin{pmatrix} \slashed D_{ }  & \frac{}{}&0 \medskip  \\  0  &  & \bold d 
\end{pmatrix} \mod \mathcal K $$
where $\mathcal K$ is the space of compact operators. The previous Lemma  \ref{hypothesisIII} shows that the conditions of Proposition  \ref{twistedbdindex} are satisfied for the twisted boundary conditions defined by $V_t^{\Phi_0}$, hence Proposition  \ref{twistedbdindex} applies to show that the top block is Fredholm of Index 0. Under the association $(a_0 + a_t dt, a_x dx + a_ydy)\sim (\zeta,\omega)$ the operator $\bold d$ is the Dirac operator up to a sign, hence the bottom block is also Fredholm of Index 0 by the untwisted case of Proposition \ref{untwisteddiracbd}. 
\end{proof}

\subsubsection{Mixed Boundary and Projection Conditions}

The eventual proof to proving the invertibility of $\mathcal L^{h_\e}$ follows from an integration by parts argument. Holistically, it has the following form. Let $\xi=(\ph,a)$, then 

\begin{eqnarray}
\int_{N_\lambda} |\mathcal L^{h_\e}\xi|^2 \ dV &=& \int_{N_\lambda} |\del_t \xi |^2 + |\mathcal N_t \xi |^2 + \br  \sigma_t \del_t \xi \ , \ \mathcal N_t\xi \kt + \br\mathcal N_t\xi  \ , \sigma_t \del_t\xi \kt  \ dV  \\
 &=&  \int_{N_\lambda} |\del_t \xi |^2 + |\mathcal N_t \xi |^2 +  \br \xi, \{\sigma_t \del_t, \mathcal N_t \} \xi \kt \ dV + \int_{\del(N_\lambda)} \br -\sigma_t J \xi, \del_t \xi \kt\label{intbypartsholistic}
\end{eqnarray}
where $\sigma_t$ is the symbol of $\mathcal L^{h_\e}$. The cross term $\{\sigma_t\del_n, \mathcal N_t\}$ is comparatively small, and can be absorbed. We would wish to impose constraints so that

\medskip 
{\bf i)} the boundary term vanishes 

\smallskip 
{\bf ii)} $\xi$ is orthogonal to the subspace consisting of sections of $\Gamma(K(\mathcal N_t))$. 

\medskip 
\noindent Given both of these, one could then apply the estimate for $\mathcal N_t$ on each slice of fixed $t$ to conclude the orem, as in the proof of Proposition \ref{untwisteddiracbd}. The problem is that imposing both these constraints does not lead to a Fredholm problem (it is `` $\text{Ind}=-L^2(S^1;\C)$'').  The solution is to observe that for sufficiently low Fourier modes, the boundary term can be absorbed. On the other hand, for sufficiently high Fourier modes in $K(\mathcal N_t)$, the $\del_t \xi$ term becomes sufficiently large to dominate the norm of these configurations, rendering the projections unnecessary. The actual conditions we impose therefore allow the low Fourier modes for an extra boundary component, and also allow non-zero projections to $K(\mathcal N_t)$ for the high fourier modes.

We define the mixed boundary and orthogonality conditions as a direct sum $$\Pi^\mathcal L:= (\Pi_{\text{Tw}}^\circ \oplus \Pi_0) \oplus P^\text{low}$$

\noindent wherein $P^\text{low}$ is the orthogonal projection defined in Equation (\refeq{Plowdef}), $\Pi_0$ is the untwisted boundary projection on the form components $(\zeta,\omega)$  identical to that in Definition \ref{purebddef}, and   $\Pi_\text{Tw}^\circ $ is a boundary condition which has $1+2\e^{-1/2}/L_0$ fewer constraints than $\Pi^\text{Tw}$ obtained by removing the boundary conditions in certain low modes.

\bigskip 

Let us now explain these more precisely.  Just as we split the kernel projection $P$ into a family of projections $P^\ell$ parameterized by Fourier modes, we can do the same for the twisted boundary projection $\Pi_\text{Tw}$ to obtain a family of projections indexed by $\Z^4$. As in Equation (\refeq{formofmualphabeta}) let $v_1(t)$ and $v_2(t)$ be two vectors whose complex span is $V^{\Phi_0}_t\subseteq E_{-1,0}$ for each $t\in S^1$. Similarly let $w_1(t), w_2(t)$ be two vectors whose complex span is $(V^{\Phi_0}_t)^\perp$ for each $t\in S^1$. By Lemma \ref{mufiltration}, we may choose $w_1(t)$ so that \be \text{span}_\C\{v_1(t), v_2(t), w_1(t)\}=(\mu_\C^\del)^{-1}(0)\subseteq E_{-1,0} \label{w1def}.\ee 
  
\noindent Writing the $E_{-1,0}$ components of the boundary values as $$\begin{pmatrix}\alpha_{-1}(t) \\ \beta_0(t)\end{pmatrix}= a_1(t)v_1(t)+ a_2(t)v_2(t) + b_1(t)w_1(t) + b_2(t)w_2(t)$$

\noindent where $a_i(t), b_i(t)\in L^{1/2,2}(S^1;\C)$ then the condition that the boundary values lie in $V_t^{\Phi_0}$ can be expressed as $$b_1(t)= b_2(t)=0.$$

\begin{defn} 
The {\bf Mixed Boundary and Projection Constraints} are defined by the condition that 
 
 \be (\ph, a)|_{\del(N_\lambda(\mathcal Z_0))} \in \ker(\Pi^{\mathcal L})
 \ee
 
 \noindent where $$\Pi^{\mathcal L}:=(\Pi_{\text{Tw}}^\circ \oplus \Pi_0) \oplus P^\text{low}$$ 
 is given by 
 
 \begin{itemize}
 
 \item $P^\text{low}: H^1_\e(N_\lambda(\mathcal Z_0))\to \C^{1+2\e^{-1/2}/L_0}$ is the projection to the low modes of the kernel bundle defined in Equation (\refeq{Plowdef}). 
 \item $\Pi_0: H^1_\e(N_\lambda(\mathcal Z_0))\to H_0^\perp$ is the untwisted boundary condition of Subsection \ref{aps3dsubsection7.1.1} on the form components $(\zeta,\omega)$. 
 \item  $\Pi_\text{Tw}^\circ: H^1_\e(N_\lambda(\mathcal Z_0))\to (H_\text{Tw} \ \oplus \  \C^{1+2\e^{-1/2}/L_0})^\perp$ is defined (using the notation above) by the constraints that $$\pi^\text{high}(b_1(t))=0 \hspace{2cm} b_2(t)=0$$
 
\noindent where $\pi^\text{high}$ denotes the projection to Fourier modes $|\ell|\geq \tfrac{1}{L_0}\e^{-1/2}$, and $\pi^\text{low}=1-\pi^\text{high}$ so that $\Pi_\text{Tw}= \Pi_\text{Tw}^\circ \oplus \pi^\text{low}$.   
 \end{itemize}

\noindent The allowed modes on the boundary are illustrated by

\begin{eqnarray*}
 \hspace{1.75cm}  \underline{k=-1}  \  \   \ &\underline{ k =0}& \ \ \ \underline{k=1} \medskip  \\  
\ldots  \alpha_{-2}(t)  \  \ \   \ \ \  \boxed{ \alpha_{-1}(t)}       \ \ \ &0&   \ \ \ \  \  \ 0  \  \   \ \  \ \ \ \ \ \ \ 0 \ldots \\
\ldots  0  \   \ \ \  \ \ \ \ \   \ \ \ \ 0 \ \ \ \  \     \   \ \ \ &\boxed{\beta_0(t)} & \ \ \ \  \beta_{1}(t)  \  \  \ \ \ \  \beta_{2}(t) \ldots \\ 
\ldots  \zeta_{-2}(t)  \  \ \   \ \ \  \ { \zeta_{-1}(t)}     \  \ \ \ &0&   \ \ \ \  \  \ 0  \  \   \ \  \ \ \ \ \ \ \ 0 \ldots \\
\ldots  0  \   \ \ \  \ \ \ \ \   \ \ \ \ 0 \ \ \ \  \     \   \ \ \ &  0  & \ \ \  \ \omega_{1}(t)  \  \  \ \ \ \  \omega_{2}(t) \ldots \\ 
\end{eqnarray*}
\noindent where the boxed modes are constrained so that $$\hspace{.9cm }\boxed{\alpha_{-1}(t)}+ \boxed{\beta_0(t)}\in V_t^{\Phi_0}   \ \bigoplus \ \Big\{ \sum_{|\ell | \leq \frac{1}{L_0}\e^{-1/2}} b_\ell e^{i\ell t}w_1(t) \Big\} \ \ \forall t\in S^1.$$
and configurations are further constrained by the requirement that $$P^\text{low}(\ph,a)=0.$$

\noindent A visualization of these conditions in comparison to the pure boundary condition is given in Figure 2 below.

\medskip 
\noindent {\bf Notice:} In addition to the above, configurations lying in $\ker(\Pi^\mathcal L)$ lie in the space of $(\mu_\C^\del)^{-1}(0)$ for each fixed $t$, which was the boundary condition imposed on $\mathcal N_t$.   
\label{mixedbddef}

\bigskip 
\end{defn}

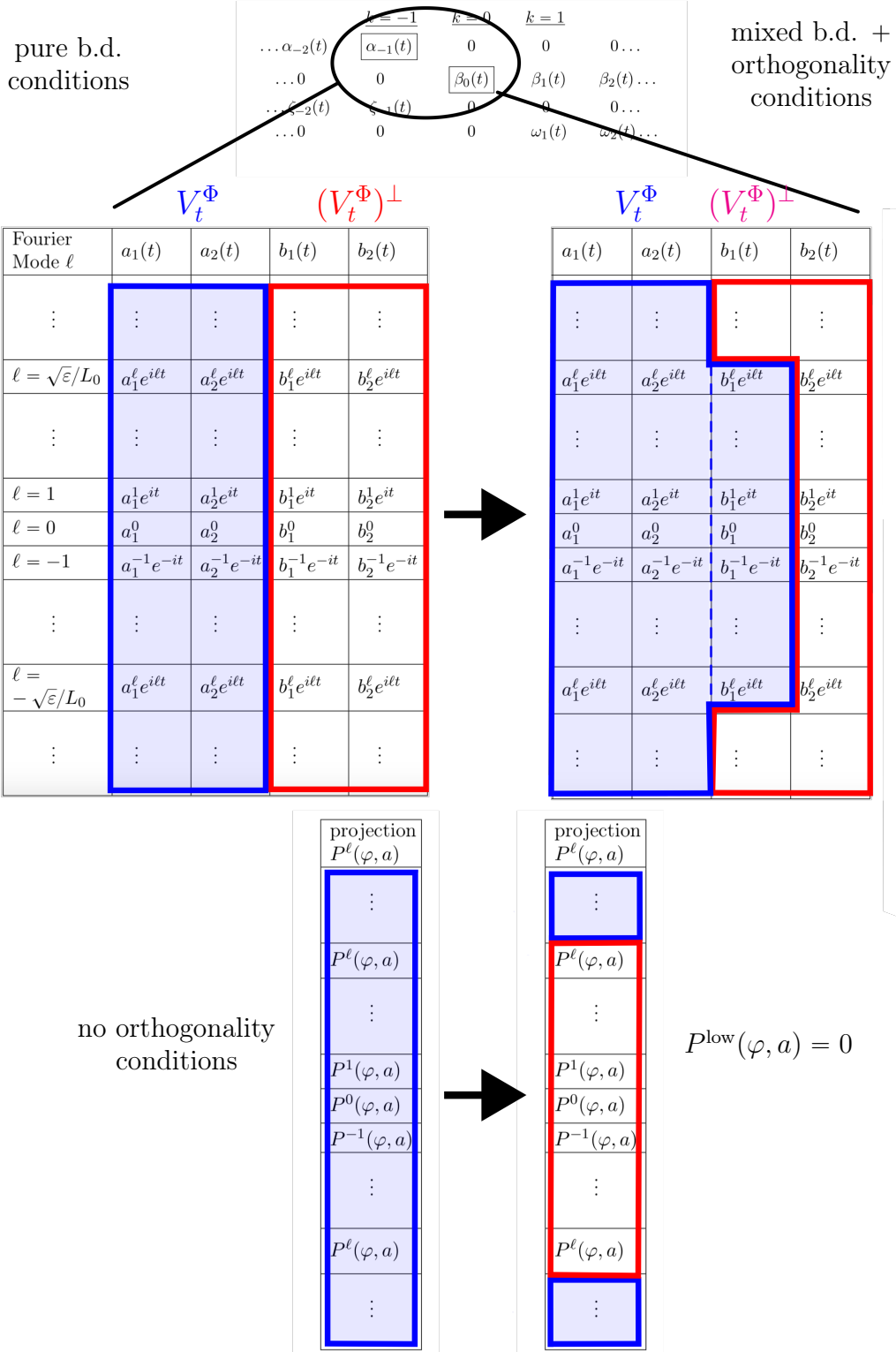
\begin{figure}[h!]
\begin{center}
\begin{picture}(500,515)
\put(25,-25){\includegraphics[scale=.50]{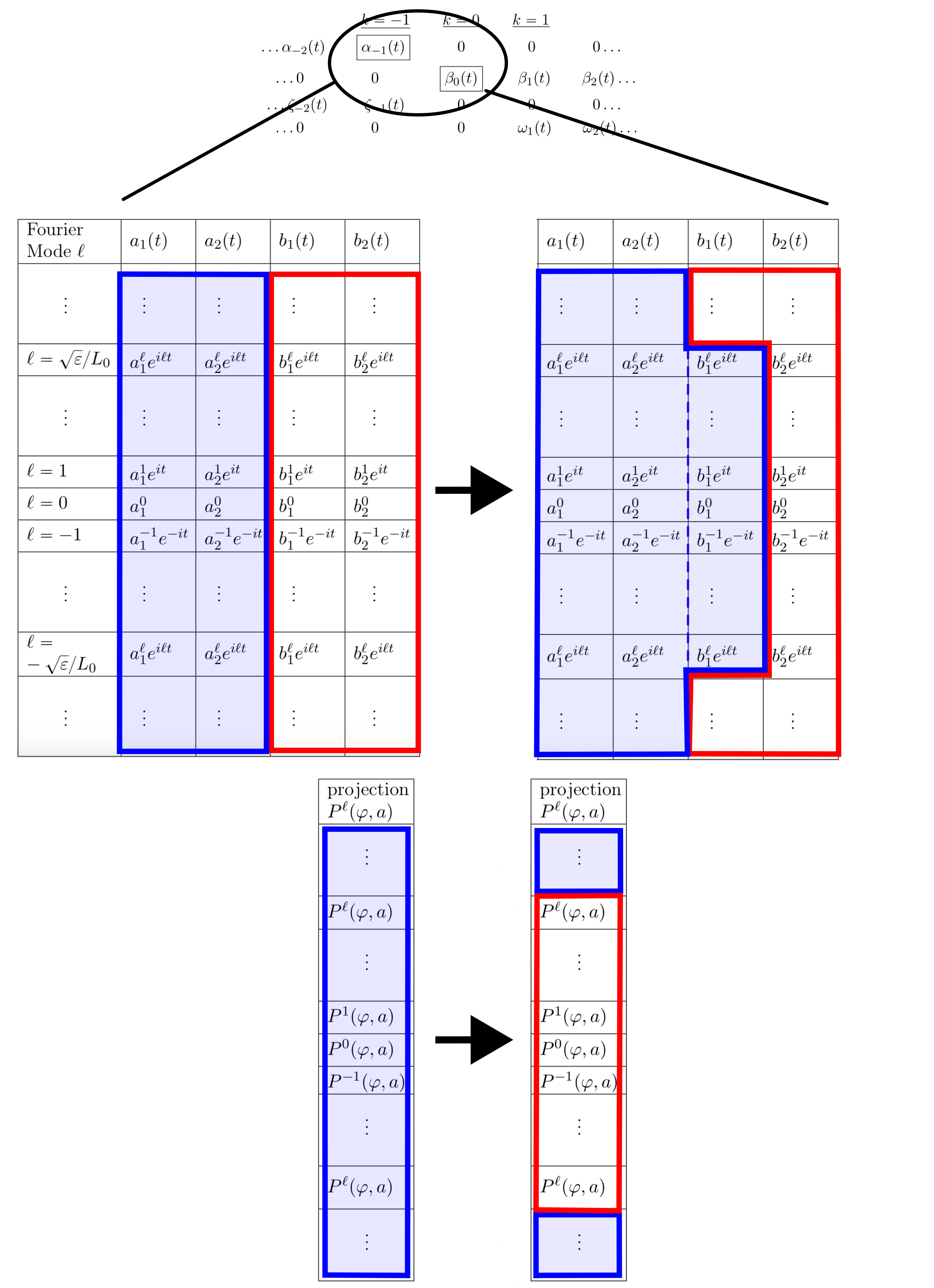}}
\put(107,465){\color{blue}\Large $V_t^{\Phi}$}
\put(167,465){\color{red}\Large $(V_t^{\Phi})^\perp$}
\put(295,465){\color{blue}\Large $V_t^{\Phi}$}
\put(335,465){\color{magenta}\Large $(V_t^{\Phi})^\perp$}
\put(35,525){\color{black}\large $\begin{matrix}\text{pure b.d.} \\ \text{conditions}\end{matrix}$}
\put(345,525){\color{black}\large $\begin{matrix}\text{mixed b.d. $+$} \\ \text{orthogonality} \\ \text{conditions}\end{matrix}$}
\put(65,105){\color{black}\large $\begin{matrix}\text{no orthogonality} \\ \text{conditions}\end{matrix}$}
\put(325,105){\color{black}\large $\begin{matrix}P^\text{low}(\ph,a)=0\end{matrix}$}
\label{Fig2}
\end{picture}

\bigskip 
\vspace{.5mm}

\caption{Illustration of the pure boundary conditions (left) versus the mixed boundary and orthogonality conditions (right). Allowed modes are indicated by blue boxes, and modes constrained to be 0 are indicated by red boxes. Compared to the pure boundary conditions, the mixed conditions remove $1+2\e^{-1/2}/L_0$ boundary constraints from $b_1$ modes, and impose the same number of orthogonality constraints on low modes in the kernel bundle.}
\end{center}

\end{figure}

 \begin{prop}
 The mixed boundary value and projection problem 
 
 \medskip 
 \be
 (\mathcal L^{h_\e}, \Pi^\mathcal L): H^1_{\e,\nu}\lre L^2_\nu  \  \oplus \  (H_\text{Tw}\oplus \C^{1+2\e^{-1/2}/L_0})^\perp \ \oplus \ H_0^\perp  \ \oplus \ \C^{1+2\e^{-1/2}/L_0}
 \ee
 
 \medskip
 \noindent is Fredholm of Index 0. 
 \label{indexmixedbd}
 \end{prop}
 \begin{proof}
 Compared to the pure boundary conditions in Corollary  \ref{corpurebdindex}, we removed $2(1+2\e^{-1/2}/L_0)$ real dimensions of constraints on the boundary modes, and added the same number via the interior projection $P^\text{low}$, thus the property of being Fredholm and the index are unchanged. The spaces with difference weights are equivalent, hence they do not change the Fredholmness property. 
 \end{proof}

\subsection{Cross-terms, Boundary Terms, and Weights}

This section proves several  technical lemmas used in the proof of Theorem \ref{invertibleL}. As explained in the previous subsection (recall Equation \refeq{intbypartsholistic}), the approach to Theorem \ref{invertibleL} inspired by the observation that the cross term is comparatively mild since configurations only concentrate in the directions of the normal disks. the  next three subsections give precise estimates on this cross term, the boundary term in this integration by parts, and improved weighted estimates for the Dirac operator:

\subsubsection{Cross-Terms} 

The first technical lemma states that the cross term when integrating by parts is small in the weighted norms. 
 Denote the $dt\wedge dx$ and $dt\wedge dy$ components of the curvature $F_{A^{h_\e}}$ by $F_{A^{h_\e}}^\perp$. Additionally, we let $\frak B_t$ denote the $dt$-components of the cross term $\frak B$ from the Weitzenb\"ock formula  \ref{Weitzenb\"ock}. Explicitly,  \be\mathfrak B_t \begin{pmatrix}\ph \\ a \end{pmatrix}= \begin{pmatrix}  \gamma((-1)^\text{deg} a)\sigma_1 \del_t \Phi^{h_\e} - 2a_t \del_t\Phi^{h_\e} \\ (-1)^{\text{deg}} \mu(\ph, \sigma_1 \del_t \Phi^{h_\e}) + 2i \br i\ph, \del_t \Phi^{h_\e}\kt dt   \end{pmatrix}. \label{frakBt}\ee

\medskip 

\begin{lm}\label{anticommutatorbd}
The anti-commutator $ \{\sigma_t \del_t, \mathcal N_t\}=\sigma_t \del_t \mathcal N_t + \mathcal N_t \sigma_t \del_t $ is given by

$$\{\sigma_t \del_t, \mathcal N_t\}\begin{pmatrix}\ph \\  a\end{pmatrix}= \gamma(F_{A^{h_\e}}^\perp). \ph + \frac{1}{\e} \mathfrak B_t \begin{pmatrix}\ph \\ a \end{pmatrix}$$

\noindent It follows that for configurations $\frak q, \frak p \in H^1_{\e,\nu}$ there is a constant $C$ independent of $\e$ such that 
\bea
\br \frak q \ , \ \{\sigma_t \del_t, \mathcal N_t\}  \frak p \kt_{L^2_\nu} \leq C\e^{1/2} \left( \| \frak q\|^2_{H^1_{\e,\nu}} + \| \frak p\|^2_{H^1_{\e,\nu}} \right).
\eea 
\end{lm}

\begin{proof}
This is an easy consequence of the Weitzenb\"ock formula. Recall 

\be
\mathcal L^{h_\e}\mathcal L^{h_\e} \begin{pmatrix}\ph \\ a \end{pmatrix}= \begin{pmatrix} \slashed D_{A^{h_\e}}\slashed D_{A^{h_\e}} \ph \\   \bold d\bold d a \end{pmatrix} +\frac{1}{\e^2} \begin{pmatrix}  \gamma(\mu(\ph, \Phi^{h_\e}) \Phi^{h_\e}) \\   \mu(\gamma(a)\Phi^{h_\e}, \Phi^{h_\e}) \end{pmatrix}+ \frac{1}{\e}\mathfrak B \begin{pmatrix}\ph \\ a \end{pmatrix}\label{LWeitzenb\"ock}
\ee
where 

$$\mathfrak B \begin{pmatrix}\ph \\ a \end{pmatrix}= \begin{pmatrix}  \gamma((-1)^\text{deg} a)\slashed D_{A^{h_\e}} \Phi^{h_\e} - 2a.\nabla \Phi^{h_\e} \\ (-1)^{\text{deg}} \mu(\ph,\slashed D_{A^{h_\e}} \Phi^{h_\e}) + 2i \br i\ph,\nabla \Phi^{h_\e}\kt    \end{pmatrix}.$$

On the other hand, since $\mathcal L^{h_\e}= \sigma_t \del_t + \mathcal N_t$, we have 

\be \mathcal L^{h_\e} \mathcal L^{h_\e}  \begin{pmatrix}\ph \\ a \end{pmatrix} = -\del_t^2\begin{pmatrix}\ph \\ a \end{pmatrix}  + \mathcal N_t \mathcal N_t \begin{pmatrix}\ph \\ a \end{pmatrix} + \{\sigma_t \del_t, \mathcal N_t\}\begin{pmatrix}\ph \\  a\end{pmatrix}\label{LWeitzenb\"ock2}\ee
and the Wietzenbock formula for the linearization at $\Phi^{h_\e}(t)$ for each fixed $t$ applied to $t$-independent configurations shows  

\bea
\mathcal N_t\mathcal N_t \begin{pmatrix}\ph \\ a \end{pmatrix}= \begin{pmatrix} \slashed D^\C_{A^{h_\e}}\slashed D^\C_{A^{h_\e}} \ph \\  \Delta^\C a \end{pmatrix} +\frac{1}{\e^2} \begin{pmatrix}  \gamma(\mu(\ph, \Phi^{h_\e}) \Phi^{h_\e}) \\   \mu(\gamma(a)\Phi^{h_\e}, \Phi^{h_\e}) \end{pmatrix}+ \frac{1}{\e}\mathfrak B^\C \begin{pmatrix}\ph \\ a \end{pmatrix}
\eea
where $\Delta^\C$ is the Laplacian on normal planes, and $$\mathfrak B^\C \begin{pmatrix}\ph \\ a \end{pmatrix}= \begin{pmatrix}  \gamma((-1)^\text{deg} a)\slashed D^\C_{A^{h_\e}} \Phi^{h_\e} - 2a.\nabla_{x,y} \Phi^{h_\e} \\ (-1)^{\text{deg}} \mu(\ph,\slashed D^\C_{A^{h_\e}} \Phi^{h_\e}) + 2i \br i\ph,\nabla_{x,y} \Phi^{h_\e}\kt    \end{pmatrix}.$$ 

\noindent Comparing (\refeq{LWeitzenb\"ock}) and (\refeq{LWeitzenb\"ock2}) and using this expression and \bea \slashed D_{A^{h_\e}}\slashed D_{A^{h_\e}}  &=&  -\del_t^2 + \slashed D^\C_{A^{h_\e}}\slashed D^\C_{A^{h_\e}}+\gamma(F_{A^{h_\e}}^\perp)\\\bold d \bold d  &=& -\del_t^2 + \Delta^\C  \eea
  yields the formula for $\{\sigma_t \del_t , {\mathcal N}_t\}$.

  We now proceed to show the bound in the second assertion. To begin, we claim there is a pointwise bound 
  \be |\del_t \Phi^{h_\e} |\leq C | \Phi^{h_\e}|\label{pointwisedtbound}.\ee 
  To verify this claim, first recall that $\Phi^{h_\e}$ is given by $$\Phi^{h_\e}:=\begin{pmatrix} e^{H(\rho_t)} c(t)r^{1/2} \ \ \  \\ e^{-H(\rho_t)} d(t) r^{1/2} e^{-i\theta}\end{pmatrix}\otimes 1 + \begin{pmatrix} -e^{H(\rho_t)}\overline d (t)r^{1/2} \ \ \  \\ e^{-H(\rho_t)} \overline c(t) r^{1/2} e^{-i\theta}\end{pmatrix}\otimes j.$$

\noindent Differentiating the top left component shows 

\bea
\del_te^{H(\rho_t)} c(t)r^{1/2} &=& e^{H(\rho_t)} \dot c(t)r^{1/2}+ e^{H(\rho_t)} c(t)r^{1/2}\p{H}{\rho_t} \p{\rho_t}{t} \\
  &=&  e^{H(\rho_t)} \dot c(t)r^{1/2}+ e^{H(\rho_t)} c(t)r^{1/2}   \ \cdot  \  \p{H}{\rho_t} \ \cdot \  \frac{2 \dot K(t)}{3K(t)}\rho_t
\eea
\noindent where we have used the expression $\rho_t=\left({K(t)}/\e\right)^{2/3} r$ to calculate $\tfrac{\del \rho_t}{\del t}$. By Assumption \ref{assumption2}, $3|K(t)|^2= 2|c(t)|^2 + |d(t)|^2>0$ is bounded independent of $\e$ the ratio ${\dot K(t)}/{K(t)}$ is bounded by a constant. Next, recall that $H(\rho_t)\sim -\log(\rho_t)^{-1/2}$, hence $\rho_t \del_{\rho_t}H$ is bounded at $\rho=0$ and decays exponentially hence is bounded by a universal constant. Using Assumption \ref{assumption2} again, the above is bounded by $C|\Phi^{h_\e}|$. The other components are identical, yielding the bound (\refeq{pointwisedtbound}) as claimed.

Using this bound yields a bound on the $\mathfrak B_t$ portion of the cross term. Write $\frak q=(\ph,a)$ and $\frak q=(\psi,b)$. Combining the pointwise (\refeq{pointwisedtbound}) with the expression in (\refeq{frakBt}) and using Young's inequality yields 

\bea
\br \begin{pmatrix} \ph \\ a \end{pmatrix},\frac{1}{\e}\mathfrak B_t \begin{pmatrix} \psi \\ b \end{pmatrix} \kt R^{2\nu}_\e &\leq&  C \ \left(|\ph|  \frac{|\Phi^{h_\e}|}{\e} |b| + |\psi|  \frac{|\Phi^{h_\e}|}{\e} |a|\right )R^{2\nu}_\e\\
&=& C \left( \frac{|\ph|^2}{2\e^{1/2}} + \frac{\e^{1/2}}{2} \frac{|b|^2 |\Phi^{h_\e}|^2}{\e^2}  +\frac{|\psi|^2}{2\e^{1/2}} + \frac{\e^{1/2}}{2} \frac{|a|^2 |\Phi^{h_\e}|^2}{\e^2} \right) R_\e^{2\nu}\\
&\leq & C \e^{1/2} \left(\frac{|\ph|^2}{R_\e^2} + \frac{|\psi|^2}{R_\e^2} +  \frac{|a|^2 |\Phi^{h_\e}|^2}{\e^2}+  \frac{|b|^2 |\Phi^{h_\e}|^2}{\e^2}\right)R_{\e}^{2\nu}\\
&\leq & C \e^{1/2} \  (\|\frak q\|^2_{H_{\e,\nu}^1} + \|\frak p\|^2_{H_{\e,\nu}^1} )
\eea 
where in passing to the third line we have used $R_\e\leq c\e^{1/2}$ on $N_\lambda(\mathcal Z_0)$.  

For the curvature term, recall that $${A^{h_\e}}= f(\rho_t) \left(\frac{dz}{z} - \frac{d\overline z}{\overline z}\right)$$ where $f_\e(\rho_t)$ is the function from Proposition (\refeq{Hrhoproperties}). Using the chain rule and the expression $\del_t\rho_t$ from above shows that $$\Big |(F_{A_{h_\e}})_{tz} \Big |=\Big | f'(\rho_t) \p{\rho_t}{t} \frac{1}{z}\Big |\leq \frac{2\dot K(t)}{3K(t)^{1/3}} \Big | f'(\rho_t)\frac{\rho_t}{\rho_t} \frac{1}{\e^{2/3}}\Big |\leq \frac{C}{\e^{2/3}} |f'(\rho_t)|. $$

\noindent Moreover, since $f'(\rho_t)$ decay exponentially  in $\rho_t$, (recalling the weight $R\sim \sqrt{1+\rho^2}\sim \e^{-2/3} R_\e$) we have $$\Big | (F_{A^{h_\e})_{tz}}\Big |\leq  \frac{C}{\e^{2/3}} \Big | \frac{f'(\rho_t)R^2}{R^2}\Big |\leq  C \e^{2/3}\frac{1}{R_\e^2} $$

\noindent and identically for the $t\overline z$ component. 
Then there's a pointwise bound, 
\bea
\br \ph, F_{A^{h_\e}}^\perp \psi\kt R_\e^{2\nu} \leq C \e^{2/3} \left( \frac{|\ph|^2}{R_\e^2} + \frac{|\psi|^2}{R_\e^2} \right) R_\e^{2\nu}
\eea
and integrating yields the result. 

  \end{proof}

  \bigskip 
  The other cross-term that arises comes from the $t$-derivative when decomposing a configuration $$\frak q= \frak q^\perp + \eta(t)\beta_t$$ as a section of $K(\mathcal N_t)$ and a section that is slicewise $L^2$-orthogonal to it.  Since the $t$-derivatives of $\beta_t$ depend only on $\Phi_0$ thus is bounded independent of $\e$, the condition that $\frak q^\perp$ is slicewise orthogonal to $\beta$ implies that $\del_t\frak q^\perp, \del_t(\eta\beta_t)$ are {\it almost} orthogonal. The next lemma gives a precise bound. 
  
  The fact that this lemma holds is the key reason we used a normalized $L^2$-projection to define the projection $P^\ell$ in (\refeq{lthmodeprojection}), rather than an $H^1_{\e}$ where nothing similar is true. 

\begin{lm}
Suppose that a configuration $\frak q$ is written $$\frak q= \frak q^\perp + \eta(t)\beta_t$$

\noindent where $\frak q^\perp$ is slice-wise $L^2$-orthogonal to $\beta_t$ and $P^\text{low}(\frak q)=0$, i.e. $\eta(t)$ has only Fourier modes in the high range. Then 

$$\frac{1}{2}\Big(\|\del_t(\eta(t) \beta_t)\|_{L^2}+  \|\del_t \frak q^\perp \|_{L^2}^2 \Big)  \ \leq \  \|  \del_t\frak q\|^2_{L^2} \ + \ \e^{5/6}\Big \|\frac{\frak q^\perp}{R_\e}\Big\|^2_{L^2}.$$\label{almostorthogonality}
\end{lm}

\begin{proof}
Throughout the proof, we denote $t$-derivatives by $\del_t \frak q= \dot {\frak q}$. Slicewise orthogonality implies

\be 0 \ = \ \del_t \br  \frak q^\perp,  \beta_t\kt_{L^2(\{t\}\times D_\lambda)}  \ = \ \br {\dot {\frak q}}^\perp ,  \beta_t\kt_{L^2(\{t\}\times D_\lambda)} + \br \xi, \dot \beta_t\kt_{L^2(\{t\}\times D_\lambda)} .\label{sliceorthogonality1}\ee

\bigskip

 \noindent Then expanding, and with the understanding that we use the $L^2$ norm and inner product throughout,  
\bea
\|\del_t(\frak q^\perp + \eta(t)\beta_t)\|^2 &=& \| \dot {\frak q}^\perp\|^2  + \|\del_t (\eta(t)\beta_t)\|^2 + 2\br \dot {\frak q}^\perp , \del_t(\eta(t)\beta_t) \kt\\ 
&=&  \| \dot {\frak q}^\perp\|^2  + \frac{1}{2}\|\del_t (\eta(t)\beta_t)\|^2+  \frac{1}{2}\| \dot \eta(t) \beta_t + \eta(t)\dot \beta_t \|^2 + 2\br \dot {\frak q}^\perp , \dot \eta(t)\beta_t+ \eta(t)\dot \beta_t\kt.
\eea

\noindent Focusing on the third term momentarily, we have the following. Recall that by Lemma  \ref{betaasymptotics}  we have the bound $\|\dot \beta_t\|_{L^2(\{t\}\times D_\lambda)}\leq C
\| \beta_t\|_{L^2(\{t\}\times D_\lambda)}$. Applying this, 

\bea
\| \dot \eta(t) \beta_t + \eta(t)\dot \beta_t \|^2& =& \|\dot \eta(t)\beta_t\|^2 + \| \eta \dot \beta_t\|^2 + 2\br \dot \eta(t) \beta_t, \eta(t) \dot \beta_t \kt\\
&\geq &  \|\dot \eta(t)\beta_t\|^2 + \| \eta \dot \beta_t\|^2 + 2 \int_{S^1}\br \dot \eta(t), \eta_t\kt \int_{\{t\}\times D_\lambda} \br \beta_t, \dot \beta_t\kt \\
&\geq &   \|\dot \eta(t)\beta_t\|^2  - 2\Big | \int_{S^1}\br \dot \eta(t), \eta_t\kt \int_{\{t\}\times D_\lambda}\underbrace{ \br \beta_t, \dot \beta_t\kt}_{\leq \tfrac{1}{2}(|\beta_t|^2 + |\dot\beta_t|^2)} dA dt \Big | \\
&\geq &  \|\dot \eta(t)\beta_t\|^2  - 2C\Big | \int_{S^1}\br \dot \eta(t), \eta_t\kt  \|\beta_t\|^2_{L^2(D_\lambda)}  dt\Big | \\
&\geq &   \|\dot \eta(t)\beta_t\|^2  - \frac{2C}{L_0} \e^{1/2}\Big | \int_{S^1} |\dot \eta(t)|^2  \|\beta_t\|^2_{L^2(D_\lambda)}  dt\Big | \\
&\geq & \frac{1}{2}  \|\dot \eta(t)\beta_t\|^2 \eea

\noindent once $\e$ is sufficiently small. 

Substituting this expression for the third term yields  

\bea
\|\del_t\frak q \|^2  &\geq & \| \dot {\frak q}^\perp\|^2  + \frac{1}{2}\|\del_t (\eta(t)\beta_t)\|^2+  \frac{1}{4}\| \dot \eta(t) \beta_t \|^2 + 2\br \dot {\frak q}^\perp , \dot \eta(t)\beta_t+ \eta(t)\dot \beta_t\kt  \\ 
&\geq &  \| \dot {\frak q}^\perp\|^2  + \frac{1}{2}\|\del_t (\eta(t)\beta_t)\|^2+  \frac{1}{4}\| \dot \eta(t) \beta_t \|^2 + 2\br \dot {\frak q}^\perp , \dot \eta(t)\beta_t \kt  - \Big | \frac{\e^{1/6}\|\dot{\frak q}^\perp\|^2}{2}  + \frac{\|\eta(t) \dot \beta_t\|^2 }{2\e^{1/6}} \Big |\\
&\geq &  \frac{1}{2}\| \dot \xi\|^2  + \frac{1}{2}\|\del_t (\eta(t)\beta_t)\|^2+  \frac{1}{8}\| \dot \eta(t) \beta_t \|^2 + 2\br \dot {\frak q}^\perp , \dot \eta(t)\beta_t \kt  
\eea
where we have again used $\|\dot \beta_t\|_{L^2_\nu(\{t\}\times D_\lambda)}\leq 
\| \beta_t\|_{L^2_\nu(\{t\}\times D_\lambda)}$ by Lemma \ref{betaasymptotics}, and that $\|\eta(t)\|^2_{L^2} \leq L_0^2 \e \|\dot \eta\|_{L^2}^2$ by the assumption of only high Fourier modes, hence the final term can be absorbed into $\tfrac{1}{4}\| \dot \eta(t) \beta_t \|^2 $ up to replacing it by $\tfrac{1}{8}$. For the remaining inner product, using the relation 
(\refeq{sliceorthogonality1})

\bea
|\br \dot {\frak q}^\perp , \dot \eta(t)\beta_t \kt |& =&\Big | \int_{S^1}\dot \eta(t)\int_{\{t\}\times D_\lambda} \br \dot {\frak q}^\perp, \beta_t \kt dA dt \Big | \\
&\leq & \Big | \int_{S^1}\dot \eta(t)\int_{\{t\}\times D_\lambda} \br {\frak q}^\perp,\dot \beta_t \kt dA dt   \Big |\\
&\leq &  \frac{\| {\frak q}^\perp\|^2}{2\e^{1/6}} + \e^{1/6} \frac{ \|\dot \eta \dot \beta_t\|^2}{2}\\ &\leq&  \e^{5/6} {\Big\|\frac{{\frak q}^\perp}{R_\e}\Big\|^2} + \e^{1/6} \frac{ \|\dot \eta \beta_t\|^2}{2} 
\eea
and absorbing the second of these into $\tfrac{1}{8}\| \dot \eta(t) \beta_t \|^2$ and moving the first to the other side yields the result. 

\end{proof}

\subsubsection{Weighted Estimates for Dirac Operator}
The next two lemmas required for the proof of \ref{invertibleL} are weighted estimates for the standard Dirac operator and the de-singularized $\Z_2$-Dirac operator $\slashed D_{A^{h_\e}}$. These estimates reference a compact operator $K$ similar to the one used in (\refeq{Kdef}), which we now define.

For $\gamma'<<1$ as in the proof of Lemma \ref{errorcalculation}, define the operator $K_\e$ by \be K_\e\ph = \frac{\ph}{R_\e} \bold 1_{\{r< \e^{2/3-\gamma'}\}}\label{Kedef}\ee

\noindent where $\bold 1_{\{r< \e^{2/3-\gamma'}\}}$ denotes the indicator function of the ball of radius $r\leq \e^{2/3-\gamma'}$.  

Denote by $L^{1,2}_\nu, L^2_\nu$ as the completion of compactly supported smooth functions in $Y$ with respect to the norms 
 
 \begin{eqnarray}
 \|u\|_{L^{1,2}_\nu} &:=& \left({\int_{N_\lambda(\mathcal Z_0)}\left( |\nabla u|^2 + \frac{|u|^2}{R_\e^2}\right) R_\e^{2\nu}} \ dV\right)^{1/2}\label{L12nunorm}\\
  \|u\|_{L^{2}_\nu} &:=& \left({\int_{N_\lambda(\mathcal Z_0)} {|u|^2}R_\e^{2\nu}} \ dV\right)^{1/2}\\
 \end{eqnarray}
 
\noindent  In these expression $\nabla, dV$ denote the structures arising from the product metric on $Y$. 

The first lemma is a basic estimate for the weighted Dirac operator (with the trivial connection). The subsequent lemma does the trickier case of the almost-singular connection $A^{h_\e}$. In both, one should have in mind that $\nu \in (0, \tfrac{1}{4})$ is chosen very close to the upper limit, say, $\nu=\tfrac{1}{4}-10^{-6}$.

\begin{lm} Let $\slashed D$ denote the standard Dirac operator with the trivial connection, and fix a weight $\nu\in (0,\tfrac{1}{4})$. If $u$ is a configuration on $N_\lambda(\mathcal Z_0)$ satisfying the Index 0 boundary conditions of Lemma \ref{untwisteddiracbd}, then  $$ \|u\|_{L^{1,2}_\nu}\leq C_\nu\left(\|\slashed Du\|_{L^2_\nu} + \|K_\e u\|_{L^2_\nu} \right)$$
\noindent where $K_\e$ is the compact operator defined above in (\refeq{Kedef}). 
 \label{weightedestimatesdirac1}
\end{lm}

The proof is a standard application of the idea that the weight shifts the spectrum of the operator restricted to slices of constant $r$. This lemma actually holds for $\nu  \in (0, \tfrac{1}{2})$. The upcoming estimate for the de-singularized operator, however, restricts to $\nu  \in (0, \tfrac{1}{4})$. In this second case, the estimate is almost certainly true for the same range $\nu  \in (0, \tfrac{1}{2})$, but the proof in the more general case appears to require more sophisticated parametrix methods, and is not needed here (see \cite{MazzeoEdgeOperators, FangyunThesis}).

  \begin{prop} Let $\slashed D_{A^{h_\e}}$ denote the de-singularized $\Z_2$-Dirac operator, and fix a weight $\nu \in (0,\tfrac{1}{4})$. If $u$ is a configuration satisfying  the boundary constraint portion of the mixed boundary and constraint conditions (Definition \ref{mixedbddef}), i.e. 
  
  $$\Pi_\text{Tw}^\circ(\ph)=0$$
   then once $L_0$ in the definition of $\Pi_\text{Tw}^\circ$ is chosen sufficiently large,  $$ \|\ph\|_{L^{1,2}_\nu}\leq C_\nu \left( \|\slashed D_{A^{h_\e}}\ph\|_{L^2_\nu} + \|K_\e\ph\|_{L^2_\nu}\right)$$
 where $K_\e$ is the compact operator defined in (\refeq{Kedef}).
  \label{weightedestimatesdirac2}
 \end{prop} 
 
 \begin{proof}

The proof consists of three steps: an interior estimate where $K_\e\neq 0$, an outside estimate where $K_\e=0$, and parametrix patching combining them.

 \medskip 
 
 \noindent {\bf Step 1: Interior Estimate.}
 
 The following estimate holds on the interior domain $I_r=\{r\leq \e^{2/3-\gamma'}\}$ for configurations $\ph$ vanishing on the boundary $r=\e^{2/3-\gamma'}$. 
 
 $$ \int_{I_r} \left(|\nabla \ph|^2 + \frac{|\ph|^2}{R_\e^2}\right) R_\e^{2\nu} \ dV \leq C   \left(\int_{I_r}|\slashed D_{A^{h_\e}}\ph|^2 R_\e^{2\nu} \ dV+ \int_{I_r} |K_\e\ph|^2 R_\e^{2\nu} \ dV\right).$$

\noindent This is obvious: integrate by parts and one obtains the first derivative squared and error terms given by $F_{A^{h_\e}}, \tfrac{\nu}{R_\e}\tfrac{dR_\e}{dr}$. These and the $L^2$ term on the left hand side are pointwise bounded by multiples of $K_\e$.

\bigskip 

\noindent {\bf Step 2: Outside Estimate.}

Let $\slashed D_{A_0}$ denote the limiting Dirac operator. Recall that $|A^{h_\e}-A_0|$ is exponentially small in the region $N_{\lambda}(\mathcal Z_0)-I_r$. In this step, we show the estimate for the connection $A_0$ on all of $N_{\lambda}(\mathcal Z_0)$, and in the next step apply it to configurations supported on $N_{\lambda}(\mathcal Z_0)-I_r$.

Assume $\nu < 1/4$ as before, and the $\Pi_\text{Tw}^\circ(\ph)=0$. Additionally, assume $\ph$ vanishes on a small neighborhood of $\mathcal Z_0$ (say $r<\e$). Then 

\smallskip

 $$ \int_{N_\lambda(\mathcal Z_0)} \left(|\nabla \ph|^2 + \frac{|\ph|^2}{r^2}\right)r^{2\nu} \ dV \leq C_\nu   \left(\int_{N_\lambda(\mathcal Z_0)}|\slashed D_{A_0}\ph|^2 r^{2\nu} \ dV\right) + C_\nu \Big | \int_{\del N_\lambda(\mathcal Z_0)} \br -\sigma_t J \ph, \del_t \ph \kt r^{2\nu} r d\theta dt\Big |.$$

 \noindent Notice the weight function here is the genuine radial function $r$ rather than the smoothed off version $R_\e$. 
 \medskip 
 
 To begin, write $$\slashed D_{A_0}= \sigma_t \del_t +\slashed D_{A_0}^\C$$ where $$\sigma_t = \begin{pmatrix} i &0 \\ 0 & -i \end{pmatrix} \hspace{1cm} \slashed D_{A_0}=\begin{pmatrix} 0 & -2\del_{A_0} \\ 2\delbar_{A_0} & 0 \end{pmatrix}$$

\begin{claim} The following hold: 
\begin{enumerate} \item[(1)] For $\nu<1/4$, $$\int_{D_\lambda}\left( \frac{(\nu-\frac{1}{2})^2}{r^2}|u|^2 \right) r^{2\nu} \ dV \leq \int_{D_\lambda} |\slashed D_{A_0}^\C u |^2 r^{2\nu } \ dV$$
\item[(2)] There is a constant $c_\nu$ such that $$ \int_{D_\lambda}\left(|\nabla^\C \ph|^2 +\frac{|\ph|^2}{r^2}\right)r^{2\nu} \ dV \leq c_\nu \int_{D_\lambda} |\slashed D_{A_0}^\C u |^2 r^{2\nu } \ dV.$$
\end{enumerate}
Here $\nabla^\C$ denotes the derivatives in the $D_\lambda$-directions. Notice also that the first estimate is asserted without a constant $c_\nu$.  
\label{lm74}
 \end{claim}

\begin{proof}
 
 Write $u=\begin{pmatrix}\alpha \\ \beta \end{pmatrix}$. Since the components decouple, it suffices to show the result for each. First, consider the $\alpha$ component. The recall the polar coordinate expression $2\delbar=e^{i\theta}(\del_r + \tfrac{i}{r}\del_\theta)$, and write $$\alpha = a r^{-\nu}$$ for $a$ in the space defined by the $\nu=0$ version of the norm on the left hand side of the statement of the proposition. Then $$2\delbar_{A_0}\alpha = e^{i\theta} (\del_r a + \tfrac{i}{r}\del_\theta +\tfrac{-\nu}{r} - \tfrac{1}{2r}) r^{-\nu}$$

 \bea
 \|2\delbar_{A_0}\alpha\|_{L^2_\nu}^2 &=& \int_{D_\lambda} \br \del_r a + \tfrac{1}{r}(i\del_\theta -(\nu+\tfrac{1}{2}))a \ , \ \del_r a + \tfrac{1}{r}(i\del_\theta -(\nu-\tfrac{1}{2}))a \kt \ r drd\theta  \\
 &=& \int_{D_\lambda} |\del_r a |^2 + \tfrac{1}{r^2} | (i\del_\theta -( \nu+\tfrac{1}{2}))a|^2 \ dV \\ &  &  +\int_{D_\lambda} \br a , -\del_r(i\del_\theta -( \nu+\tfrac{1}{2}))a   + (i\del_\theta -( \nu+\tfrac{1}{2}))\del_r a\kt drd\theta   + \int_{\del D_\lambda} \br  a, i\del_\theta - (\nu-\tfrac{1}{2}) a \kt d\theta\\
 &\geq &  \int_{D_\lambda} |\del_r a |^2 + \tfrac{1}{r^2} | (i\del_\theta -( \nu+\tfrac{1}{2}))a|^2 \ dV. 
 \eea
 
 \noindent since $F_{A_0}=0$ and the restriction of $\alpha$ to the boundary has only Fourier modes in $\theta$ with, hence the boundary term is positive since $|\nu-\tfrac{1}{2}|<1$. A similar integration by parts holds for $b=\beta r^{-\nu}$, except $i\del_\theta -( \nu+\tfrac{1}{2})$ is replaced by $- i\del_\theta -( \nu-\tfrac{1}{2})$. Since $\nu-\tfrac{1}{2}<0$ the boundary term is again positive since the allowed Fourier modes are $k\geq 0$.  Both  $i\del_\theta -( \nu+\tfrac{1}{2})$ and $- i\del_\theta -( \nu-\tfrac{1}{2})$ have lowest eigenvalue $(\nu-\tfrac{1}{2})$ on the circle. The first bullet point of the claim follows.

    For the second bullet point, notice that $i\del_\theta -(\nu+\tfrac{1}{2})$ and $-i\del_\theta - (\nu-\tfrac{1}{2})$ are invertible on the circle hence there are estimates $$\|a\|_{L^{1,2}(S^1)}\leq c_\nu \|(i\del_\theta -(\nu+\tfrac{1}{2})a\|_{L^2(S^1)}$$
    and likewise for $-i\del_\theta - (\nu-\tfrac{1}{2})$. Applying this instead of the $L^2$ estimate from the eigenvalues shows 
    
    $$\int_{D_\lambda} |\nabla (r^\nu \ph)|^2 + \frac{|r^{\nu}\ph|^2}{r^2} \ dV \leq C_\nu \int_{D_\lambda} |\slashed D_{A_0}^\C \ph|^2 r^{2\nu} \ dV$$
and the second bullet point follows. 
\end{proof}

With the claim established, we integrate by parts: 

\bea
\int_{N_\lambda(\mathcal Z_0)} |\slashed D_{A_0}\ph|^2 r^{2\nu} \ dV &=& \int_{N_\lambda(\mathcal Z_0)} |\del_t \ph|^2 r^{2\nu} + |\slashed D_{A_0}^\C \ph|^2 r^{2\nu} + \br \ph, \sigma_t \del_t \slashed D^\C_{A_0}\ph+ \slashed D_{A_0}^\C \sigma_t \del_t  \ph\kt r^{2\nu}   \\
& &+ \br \ph, \tfrac{2\nu}{r} \sigma_r \sigma_t \del_t \ph \kt r^{2\nu} \ dV  +\int_{\del N_\lambda(\mathcal Z_0)}\br -\sigma_t J \ph, \del_t \ph \kt r^{2\nu} rd\theta dt  \\
&\geq & \int_{N_\lambda(\mathcal Z_0)} |\slashed D_{A_0}^\C \ph|^2 r^{2\nu}  - \frac{\nu^2}{r^2}|\ph|^2 r^{2\nu} \ dV+\int_{\del N_\lambda(\mathcal Z_0)}\br -\sigma_t J \ph, \del_t \ph \kt r^{2\nu} rd\theta dt.
\eea
Now apply the first bullet point from the above claim. Since $\nu< \frac{1}{4}$ implies $(\nu-\tfrac{1}{2})^2 > \nu^2$, hence we find 

\be  \phantom{\hspace{.7cm}}\geq  C_\nu \int_{N_\lambda(\mathcal Z_0)} \frac{|\ph|^2}{r^2} r^{2\nu}+\int_{\del N_\lambda(\mathcal Z_0)}\br -\sigma_t J \ph, \del_t \ph \kt r^{2\nu} rd\theta dt. \label{736}\ee

Next, we integrate by parts again and substitute this inequality: 

\bea
\int_{N_\lambda(\mathcal Z_0)} |\slashed D_{A_0}\ph|^2 r^{2\nu} \ dV & =& \int_{N_\lambda(\mathcal Z_0)}  |\del_t \ph|^2 r^{2\nu} +|\slashed D^\C_{A_0}\ph|^2 r^{2\nu}  \\ & & + \br  \ph, \tfrac{2\nu}{r}\sigma_r \sigma_t \del_t  \ph \kt r^{2\nu} + \int_{\del N_\lambda(\mathcal Z_0)}\br -\sigma_t J \ph, \del_t \ph \kt r^{2\nu}rd\theta dt \\
 &\geq & \int_{N_\lambda(\mathcal Z_0)}\tfrac{1}{2}|\del_t \ph|^2r^{2\nu} + |\slashed D_{A_0}^\C \ph|^2  r^{2\nu}-C_\nu\int_{N_\lambda(\mathcal Z_0)}\frac{|\ph|^2}{r^2}r^{2\nu} \ dV  \\ & & +\int_{\del N_\lambda(\mathcal Z_0)}\br -\sigma_t J \ph, \del_t \ph\kt r^{2\nu} rd\theta dt \\
 &\geq &  c_\nu \int_{N_\lambda(\mathcal Z_0)} \left(|\nabla \ph|^2 + \frac{|\ph|^2}{r^2}\right)r^{2\nu} \ dV -C_\nu\int_{N_\lambda(\mathcal Z_0)}\frac{|\ph|^2}{r^2}r^{2\nu} \ dV  \\ & &+\int_{\del N_\lambda(\mathcal Z_0)}\br -\sigma_t J \ph, \del_t \ph\kt r^{2\nu} rd\theta dt 
\eea
 where we have now used the second bullet point in Lemma \ref{lm74}. Moving the negative term to the other side and applying (\refeq{736}) yields 
 
  \be \int_{N_{\lambda}} \left(|\nabla \ph|^2 + \frac{|\ph|^2}{r^2}\right)r^{2\nu} \ dV \leq C_\nu   \left(\int_{N_{\lambda}}|\slashed D_{A_0}\ph|^2 r^{2\nu} \ dV\right) + C_\nu \Big | \int_{\del N_{\lambda}} \br -\sigma_t J \ph, \del_t \ph \kt r^{2\nu} r d\theta dt\Big |\label{eq74}\ee

\noindent completing step 2.

 \medskip

{\bf Step 3: Parametrix Patching.} Let $\eta$ denote a cutoff equal to 1 at the origin and supported in the region $I_r=\{r< \e^{2/3-\gamma'}\}$ such that $d\eta$ has support in $r\in [\tfrac{1}{4}\e^{2/3-\gamma'}, \tfrac{1}{2}\e^{2/3-\gamma'}]$ and satisfies $$|d\eta|\leq \frac{c}{R_\e}.$$

\noindent We now complete the proof:  let $\ph$ be a spinor satisfying $\Pi^\circ_\text{Tw}(\ph)=0$ and having finite $L^{1,2}_\nu$ (as in \refeq{L12nunorm}). Applying the estimates from Step 1 and Step 2 to $\eta u$ and $(1-\eta)u$ respectively, and using the fact that $A_{h_\e}$ is exponentially close to $A_0$ in the ``outside'' region, 

\bea
\|u\|^2_{L^{1,2}_\nu} &= & \|\eta u + (1-\eta)u\|^2_{L^{1,2}_\nu} \\
&\leq& C_\nu \left( \|\eta u \|^2_{L^{1,2}_\nu} +\| (1-\eta)u\|^2_{L^{1,2}_\nu}\right) \\
&\leq& C_\nu  \left(\|\slashed D_{A^{h_\e}} (\eta u)\|^2_{L^2_\nu}  + \|\slashed D_{A_0} ((1-\eta) u)\|^2_{L^2_\nu} + \|K_\e u\|^2_{L^2_\nu}+ \text{b.d. term} \right)\\ 
&\leq& C_\nu \left( \|\slashed D_{A^{h_\e}} (\eta u)\|^2_{L^2_\nu}  + \|\slashed D_{A^{h_\e}} ((1-\eta) u)\|^2_{L^2_\nu}  +\|K_\e u\|^2_{L^2_\nu} + O(\text{Exp}(-\tfrac{1}{\e^\gamma})) + \text{b.d. term}\right)\\ 
\eea
where the boundary term is as in (\refeq{eq74}). Then, $$\slashed D_{A^{h_\e}}(\eta u)= \eta \slashed D_{A^{h_\e}}u + \gamma(d\eta)u$$ and likewise for $(1-\eta)$. Substituting this shows the above is bounded by

\bea
&\leq& C_\nu \left(   \|\eta \slashed D_{A^{h_\e}} u \|^2_{L^2_\nu}  + \|(1-\eta)\slashed D_{A_0} u\|^2_{L^2_\nu} +2 \|\gamma(d\chi)u\|_{L^2_\nu}^2  +\|K_\e u\|^2_{L^2_\nu} + O(\text{Exp}(-\tfrac{1}{\e^\gamma})) + \text{b.d. term}\right)\\ 
&\leq& 2C_\nu \|\slashed D_{A^{h_\e}}u\|^2_{L^2_\nu} + 4c \|K_\e u\|^2_{L^2_\nu} + \text{b.d. term}
\eea
where we have used the definition of $K_\e$ and to bound the derivative of the cutoff. The exponentially small term is easy to absorb into $\|u\|_{L^{1,2}|_\nu}$ once $\e$ is sufficiently small.

The final step is to absorb the boundary term. This a consequence of the lemma in the following subsection, combined with the fact that the twisted boundary conditions allow only pairings between boundary Fourier modes with $|\ell|\leq \frac{1}{L_0} \e^{-1/2}$, which gives an estimate  

$$ \Big | \int_{\del N_\lambda}\br  -\sigma_t J \ph, \del_t \ph\kt \ dA  \Big | \leq  \frac{C}{\e^{1/2}L_0} \|\ph\|^2_{L^2(\del N_\lambda)}$$

\noindent which is proved precisely in Claim \ref{claim751} during the proof of Theorem \ref{invertibleL}. Given this, combining this estimate with the next lemma and choosing $L_0$ sufficiently large completes Step 3 and the proof of Proposition   \ref{weightedestimatesdirac2}. 
 \end{proof}

 \subsubsection{Boundary Terms}

Since the radius of $N_{\lambda}(\mathcal Z_0)$ is very small, scaling leads to a strong estimate on the (weighted) $L^2$-norm of the boundary values. This is one of the key reasons the size of the neighborhood must shrink as $\e\to 0$.

\begin{lm}{\bf (boundary absorption Lemma)}
 There exists a constant $C_\nu$ such that on $N_\lambda(\mathcal Z_0)$, $$\int_{\del {N_\lambda(\mathcal Z_0)}} |\ph|^2 R_\e^{2\nu} \ rd\theta dt \leq C_\nu \e^{1/2} \int_{N_\lambda(\mathcal Z_0)} \left(|\nabla \ph|^2 +\frac{|\ph|^2 }{R_\e^2}\right) R_\e^{2\nu} \ dV$$

\label{bdabsorption}
\end{lm}

\begin{proof}
This follows from scaling the trace inequality from the disk of radius $r=1$. First we prove the inequality in the case that $\nu=0$. Let $C_0$ be the constant for which the two-dimensional trace inequality $$\|\psi\|^2_{L^2(\del D)}\leq C_0 \int_{D} |\nabla \psi|^2 + |\psi|^2 \ dV$$

\noindent holds on the disk of radius $r=1$. Apply this to $\psi(y)=\ph(\e^{-1/2}y)$ to see 

$$\e^{-1/2}\|\ph\|^2_{L^2(D_\lambda)}  = \|\psi\|^2_{L^2(D_1)}\leq C_0 \int_{D}|\nabla \psi|^2 + |\psi|^2 \ dV \leq C_0\int_{D_\lambda} |\nabla \ph|^2 + \frac{|\ph|^2}{\e} \ dV \leq C_1 \int_{D_\lambda} |\nabla \ph|^2 + \frac{|\ph|^2}{R_\e^{2}} \ dV.$$

\noindent Integrating with respect to $t$ yields the inequality in the case that $\nu=0$.

For a general $\nu\in (0,1/4)$, apply the above to $\widetilde \ph= R_\e^{\nu}\ph$ and combine this with the equivalence of norms as in Lemma \ref{changingweights}. 

$$\|\ph R_\e^{\nu}\|^2_{L^{1,2}_0}\leq C_\nu \|\ph\|^2_{L^{1,2}_\nu}.$$

\end{proof}

\subsection{Integration by Parts}
\label{intbyparts}

This subsection carries out the proof of Theorem \ref{invertibleL} in the model case. The case of a general metric is treated in the subsequent section by a perturbation argument. The proof in the model case combines the holistic 
integration by parts argument described in Equation (\refeq{intbypartsholistic}) with estimates reminiscent of the proof of the uniform invertibility of $\widehat{\mathcal N}_t$ in Section \ref{section6} (recall Item (2) of Lemma  \ref{abstractlem}). 

First we show an estimate \be 
\|(\ph,a)\|_{H^1_\e}\leq C \ \|\mathcal L^{h_\e}(\ph,a)\|_{L^2} + \text{b.d. term}\label{7.33unweighted}
\ee
for the weight $\nu=0$. Next, a weighted version of the Weitzenb\"ock formula shows that for a weight $\nu<1/4$ 

 \be 
\|(\ph,a)\|_{H^1_{\e,\nu}}\leq C \ \left(\|\mathcal L^{h_\e}(\ph,a)\|_{L^2_\nu} + \|K_\e(\ph,a)\|_{L^2_\nu}\right)
\label{7.34weighted}
\ee

\noindent also holds where  $K_\e$ is as defined in $(\refeq{Kedef})$. But in turn, we also have 

\be  \|K_\e(\ph,a)\|_{L^2_\nu}\lesssim_\e \ \|(\ph,a)\|_{H^1_\e} \label{step3}\ee

\noindent where $\lesssim_\e$ denotes a bound by a constant times an appropriate power of $\e$. Applying \ref{7.33unweighted} again to the right hand side of  \ref{step3} and showing the boundary term can be absorbed yields the result after the appropriate bookkeeping of powers of $\e$.

\begin{proof} (of Theorem \ref{invertibleL}) The index statement was proved in Proposition \ref{indexmixedbd}, and it therefore suffices to show injectivity, for which it is enough to prove the second estimate in the statement of the orem. The first estimate follows immediately from the second using $\tfrac{1}{R_\e^\nu}\leq C\e^{-2\nu/3}$. We therefore prove 
\be \|(\ph,a)\|_{H^1_{\e,\nu}} \leq   C\e^{1/12-\gamma_2} \|\mathcal L^{h_\e}(\ph,a)\|_{L^2}. \label{toprove}\ee
\noindent By taking limits, it suffices to prove the estimate for smooth configurations. Thus let $\frak q=(\ph,a)$ be a smooth configuration satisfying the mixed boundary and projection conditions. In particular, with such a configuration it makes sense to reference $\frak q|_{\{t\}\times D_\lambda}$ for any $t\in S^1$ and integrate with respect to $t$ at the end. The proof now consists of three steps corresponding to the bounds (\refeq{7.33unweighted})- (\refeq{step3}) as described above respectively.

\medskip

\noindent {\bf Step 1:}
The following estimate holds: \be \| \frak q \|^2_{H^1_\e}\leq \frac{C}{\e^{1/6}}\|\mathcal L \frak q\|^2_{L^2} + \frac{C}{\e^{1/2}L_0} \|\mathfrak \ph\|^2_{L^2(\del N_\lambda(\mathcal Z_0))}. \ee

Omitting the superscript $h_\e$ from the proof, we may write $$\mathcal L=\sigma_t \del_t + \mathcal N_t.$$
\noindent Expanding and integrating by parts (this was alluded to in \refeq{intbypartsholistic}) yields
\bea
\|\mathcal L\frak q\|_{L^2}^2 & =& \int_{N_\lambda} |\del_t \frak q|^2 + |\mathcal N_t \frak q|^2 + \br \sigma_t \del_t \frak q, \mathcal N_t \frak q \kt+  \br  \mathcal N_t \frak q , \sigma_t \del_t \frak q \kt \ dV \\
 &\geq & \int_{N_\lambda} |\del_t \frak q|^2 + |\mathcal N_t \frak q|^2 + \br\frak q,  \sigma_t \del_t  \mathcal N_t \frak q \kt+  \br   \frak q , \mathcal N_t \sigma_t \del_t \frak q \kt \ dV  + \int_{\del N_\lambda}\br  -\sigma_t J \ph, \del_t \ph\kt \ dA\\
  &\geq & \int_{N_\lambda} |\del_t \frak q|^2 + |\mathcal N_t \frak q|^2 + \br\frak q, \{ \sigma_t \del_t ,   \mathcal N_t \} \frak q \kt + \int_{\del N_\lambda}\br  -\sigma_t J \ph, \del_t \ph\kt \ dA.
\eea

\noindent Now, $\frak q$ may be decomposed into the component in the kernel subbundle and its slicewise $L^2$ orthogonal complement as in Lemma \ref{almostorthogonality} so that  $$\frak q= \frak q^\perp+ \eta(t)\beta_t$$ where $\frak q^\perp$ is $L^2$-orthogonal to $\beta_t$ on $\{t\}\times D_\lambda$ for every $t$.  Lemma \ref{almostorthogonality} applies to show that  

\be\frac{1}{2}\Big(\|\del_t(\eta \beta_t)\|^2_{L^2}+  \|\del_t \frak q^\perp \|_{L^2}^2 \Big) \  \leq \  \|  \del_t\frak q\|^2_{L^2}+ \e^{5/6}\Big\|\frac{\frak q^\perp}{R_\e}\Big\|^2_{L^2}.\label{substitution1}\ee

\noindent In addition, by definition of $\beta_t$ as the span of $\ker(\mathcal N_t)$ we have 

$$\mathcal N_t \frak q = \mathcal N_t (\frak q^\perp+\eta \beta_t)=\mathcal N_t(\frak q^\perp),$$ 
\noindent so by slicewise-orthogonality, which implies $\pi^\text{ker}_t(\frak q^\perp)=0$ the main result of Section \ref{section6} from Corollary \ref{l2projectionNbound} shows 

\be \|\frak q^\perp\|^2_{H^1_{slice}}\leq \frac{C}{\e^{1/6}} \|\mathcal N_t \frak q^\perp\|_{L^2}^2=\frac{C}{\e^{1/6}} \|\mathcal N_t \frak q\|_{L^2}^2.\label{substitution2}\ee
\noindent Substituting (\refeq{substitution1}) and (\refeq{substitution2}) into the integration by parts yields 

\bea
\|\mathcal L\frak q\|_{L^2}^2 
&\geq &\frac{1}{2} \| \del_t \frak q\|_{L^2}^2  + \frac{1}{4} \| \del_t (\eta \beta_t)\|_{L^2}^2 +\frac{1}{4} \|\del_t \frak q^\perp\|_{L^2}^2 - \e^{5/6}\Big\|\frac{\frak q^\perp}{R_\e}\Big\|^2_{L^2} + \frac{\e^{1/6}}{C}\|\frak q^\perp\|_{H^1_{slice}}^2 +\int_{N_\lambda(\mathcal Z_0)} \br\frak q, \{ \sigma_t \del_t ,   \mathcal N_t \} \frak q \kt  \\ & & + \int_{\del N_\lambda(\mathcal Z_0)}\br  -\sigma_t J \ph, \del_t \ph\kt \ dA. 
\eea

\noindent and combining the slice norm with the $\del_t\frak q^\perp$ and absorbing the $\e^{5/6}$ term yields 

\begin{eqnarray}
\|\mathcal L\frak q\|_{L^2}^2 
&\geq &\frac{1}{2} \| \del_t \frak q\|_{L^2}^2  + \frac{1}{4} \| \del_t (\eta \beta_t)\|_{L^2}^2 + \frac{\e^{1/6}}{C}\|\frak q^\perp\|_{H^1_{\e}}^2 +\int_{N_\lambda(\mathcal Z_0)} \br\frak q, \{ \sigma_t \del_t ,   \mathcal N_t \} \frak q \kt  \label{beforeL2sub}\\ & & + \int_{\del N_\lambda(\mathcal Z_0)}\br  -\sigma_t J \ph, \del_t \ph\kt \ dA. 
\end{eqnarray}

What remains is to show that the $L^2$ norm of the $\eta\beta_t$ components dominates the $H^1_\e$ norm of these. This is effectively a consequence of restricting to the high range of Fourier modes. Since $\beta_t$ is normalized in the $H^1_{slice}$-norm, one has 
 \be\|\eta(t)\beta_t\|_{H^1_{slice}}= \|\eta(t)\|_{L^2(S^1)}^2 \label{commensuratetoS1}\ee

\noindent thus we will show that \be\|\del_t(\eta \beta_t)\|^2_{L^2}\geq \frac{\e^{1/6}}{C}\|\eta\|^2_{L^2}\label{dominatedbyL2}.\ee

\noindent Using the basic relation that $|a|^2 =|(a+b)-b|^2 \leq 2|a+b|^2 + 2|b|^2$ shows

\bea
 \int_{N_\lambda(\mathcal Z_0)}| \del_t (\eta \beta_t)|^2 \ dV  &=& \int_{N_\lambda(\mathcal Z_0)}|  \dot \eta \beta_t + \eta(t)\dot \beta_t|^2 \ dV \\
 &\geq &  \frac{1}{2}\int_{S^1} |\dot \eta|^2 \int_{D_\lambda}|\beta_t|^2 \ dV -  \int_{S^1} | \eta|^2 \int_{D_\lambda}|\dot \beta_t|^2 \ dV. 
\eea 

\noindent Next, applying the bounds from Lemma \ref{betaasymptotics}  that $$\|\beta_t\|_{L^2(D_\lambda)}^2 \geq c \e^{7/6} \hspace{1cm} \|\dot \beta_t\|_{L^2(D_\lambda)}^2\leq C \e^{7/6}$$
uniformly in $t$, and use the fact that $\eta(t)$ has only Fourier modes for $|\ell| \geq \frac{1}{\e^{1/2}L_0}$ so that $$\|\dot \eta\|^2_{L^2(S^1)} \geq \frac{1}{\e L_0^2} \|\eta\|_{L^2(S^1)}^2$$ shows
\bea
 \| \del_t (\eta \beta_t)\|_{L^2}^2 &\geq &  \frac{ c \e^{7/6}}{2\e L_0^2} \|\eta\|_{L^2(S^1)}^2 - C \e^{7/6} \|\eta\|_{L^2(S^1)}^2  \\ &\geq& \frac{ \e^{1/6}}{C} \|\eta(t)\|_{L^2(S^1)}^2 
\eea
\noindent which is (\refeq{dominatedbyL2}). By (\refeq{commensuratetoS1}) we conclude  

\be  \| \del_t (\eta \beta_t)\|_{L^2}^2  \geq \frac{1}{2}\| \del_t (\eta \beta_t)\|_{L^2}^2  + \frac{\e^{1/6}}{C} \|\eta \beta_t\|^2_{H^1_{slice}}\geq \frac{\e^{1/6}}{C}\|\eta\beta_t\|_{H^1_\e}^2. \label{tosubstitute749}\ee

\medskip 
Substituting the above (\refeq{tosubstitute749}) into the integration by parts formula (\refeq{beforeL2sub}) then shows 

\bea
\|\mathcal L\frak q\|_{L^2}^2 
&\geq &\frac{\e^{1/6}}{C}\left( \|\eta \beta_t\|^2_{H^1_\e} + \|\frak q^\perp\|^2_{H^1_\e}  \right)  +\int_{N_\lambda(\mathcal Z_0)} \br\frak q, \{ \sigma_t \del_t ,   \mathcal N_t \} \frak q \kt  + \int_{\del N_\lambda(\mathcal Z_0)}\br  -\sigma_t J \ph, \del_t \ph\kt \ dA.\\ 
&\geq & \frac{\e^{1/6}}{4C} \|\frak q\|^2_{H^1_\e} -  2 \e^{1/2}\|\frak q\|_{H^1_\e}^2 + \int_{\del N_\lambda(\mathcal Z_0)}\br  -\sigma_t J \ph, \del_t \ph\kt \ dA.\\ 
\eea
where we have used Lemma \ref{anticommutatorbd} to bound the anti-commutator. Once $\e$ is sufficiently small, we conclude the bound 

\be   \|\frak q\|_{H^1_\e}^2 \leq \frac{C}{\e^{1/6}} \|\mathcal L \frak q\|_{L^2}^2 +\Big | \int_{\del N_\lambda(\mathcal Z_0)}\br  -\sigma_t J \ph, \del_t \ph\kt \ dA  \Big |. \ee

The following assertion therefore finishes Step 1: 

\begin{claim}\be  \Big | \int_{\del N_\lambda(\mathcal Z_0)}\br  -\sigma_t J \ph, \del_t \ph\kt \ dA  \Big | \leq  \frac{C}{\e^{1/2}L_0} \|\ph\|^2_{L^2(\del N_\lambda(\mathcal Z_0))}\label{claim751eq}\ee\label{claim751}\end{claim}

\begin{proof}
Recall the vector $w_1 \in E_{-1,0}$  such that $ w_1 \perp V_t^{\Phi_0} $ and $w_1\in (\mu_\C^\del)^{-1}(0)$ from the definition of the mixed boundary and projection constraints (equation \ref{w1def}). Let $u_1=\sigma_1 J w_1$ and $u_2$ its orthogonal complement in $V_t^{\Phi_0}$. 

Thus we have that $u_1, u_2, w_1$ are pairwise orthogonal in the Hermitian inner product, with the first two spanning $V_t^{\Phi_0}$ and the relation that $ \sigma_t J w_1=u_1$ while $\br  \sigma_t J w_1, u_2\kt=0$ in the Hermitian inner product.  Now write the $E_{-1,0}$-component of $\ph|_{\del N_\lambda(\mathcal Z_0)}$ as $$\pi^{(-1,0)}(\ph |_{\del N_\e})= a_1(t)u_1(t)+ a_2(t)u_2(t) + b_1(t)w_1(t)$$

\noindent so that \be |a_1(t)|^2 + |a_2(t)|_{}^2 + |b_1(t)|^2\leq |\ph(t)|^2 \label{phbound}\ee

\noindent The twisted boundary conditions dictate that $b_1(t)$ has only $t$-Fourier modes with $|\ell|\leq \frac{1}{\e^{1/2}L_0}$. Additionally, we have the time derivative 

$$\pi^{(-1,0)}(\del_t \ph |_{\del N_\e})= \dot a_1u_1+ a_1\dot u_1 \ + \ \dot a_2u_2 + a_2\dot u_2 \ + \ \dot b_1w_1+ b_1\dot w_1.$$

\noindent where the dependence of each on $t$ is implicit. Evaluating the inner product on the left hand side of the expression (\refeq{claim751eq}), we have 

\bea
\br  -\sigma_t J \ph, \del_t \ph\kt \ dA  & =& \br \sigma_1 J (a_1u_1 \ + \  a_2u_2  \ + \  b_1w_1)  \ , \   \dot a_1u_1 \ + \  \dot a_2u_2  \ + \  \dot b_1w_1 \kt \\ 
& & \  + \  \br \sigma_1 J ( a_1u_1 \ + \  a_2u_2  \ + \  b_1w_1)  \ , \  a_1\dot u_1 \ + \ a_2\dot u_2  \ + \  b_1\dot w_1 \kt. \eea
\noindent Since the Lagriangian property implies $\sigma J w_1 \perp w_1$ and likewise for $u_i$ this reduces to 

\bea
& \leq &  \br \sigma_1 J (a_1u_1+ a_2u_2) , \dot b_1w_1 \kt + \br \sigma_1 J (b_1w_1) ,  \dot a_1u_1+ \dot a_2u_2 \kt  \\ & &  + C |\ph(t)|^2 \Big(|\dot u_1|   +   |\dot u_2|   +  |\dot w_1|\Big )
\eea

\noindent where we have used (\refeq{phbound}) on all the terms where the derivative hits the basis vectors. In fact, using the orthogonality conditions for the chosen basis, and the fact that their time derivatives are bounded by a constant depending only on $\Phi_0$, the above reduces to 
\bea 
& \leq &  \br a_1, \dot b_1\kt  + \br b_1, \dot a_1\kt + C |\ph(t)|^2\\ 
\eea

 Next, since $\dot b_1$ has only Fourier modes with$ |\ell|\leq \frac{1}{\e^{1/2}L_0}$, integrating the above yields we have 
\bea\int_{\del N_\lambda(\mathcal Z_0)} \br a_1, \dot b_1\kt  + \br b_1, \dot a_1\kt \ dA &\leq&C \int_{S^1_\theta}  \Big(\sum_{ |\ell|\leq \frac{1}{\e^{1/2}L_0}} |\ell| |(a_1)_\ell| |(b_1)_\ell| \Big) rd\theta \\ &\leq&  \frac{C}{\e^{1/2}L_0}\int_{S_\theta^1} \Big(\sum_{ |\ell|\leq \frac{1}{\e^{1/2}L_0}} |(a_1)_\ell|^2+ |(b_1)_\ell|^2 \Big) rd\theta \\ &\leq & \frac{C}{\e^{1/2} L_0}  \|\ph\|^2_{L^2(\del N_\lambda(\mathcal Z_0))}\eea 
and once $\e$ is sufficiently small, the additional factor of $C|\ph(t)|^2$ can be absorbed. 
\end{proof}

\noindent {\bf Step 2:} Let $K_\e$ denote the compact operator defined in (\refeq{Kedef}). The following estimate holds for $\nu\in (0,1/4)$ and in particular for, say, $|\nu-\tfrac{1}{4}|<<1$, say $\nu=\tfrac{1}{4}-10^{-6}$. 

\be
\|\mathfrak q\|^2_{H^1_{\e,\nu}}\leq C \|\mathcal L \frak q\|^2_{L^2_{\nu}} +C \|K_\e\frak q\|^2_{L^2_\nu}.
\ee

This follows readily from the Weitzenb\"ock formula and the weighted estimates of Lemma \ref{weightedestimatesdirac1} and Proposition  \ref{weightedestimatesdirac2} in the previous subsection. Expanding as in \ref{L2Weitzenb\"ock} with the cross-term kept explicit, 

\bea
\|\mathcal L\frak (\ph,a)\|^2_{L^2_\nu} &=& \|\slashed D_{A^{h_\e}}\ph\|^2_{L^2_\nu} + \|\bold d a\|^2_{L^2_\nu} + \frac{1}{\e^2} \|\gamma(a)\Phi^{h_\e}\|^2_{L^2_\nu}+\frac{1}{\e^2} \|\mu(\ph,\Phi^{h_\e})\|^2_{L^2_\nu}   \\ & &+  \int_{N_\lambda(\mathcal Z_0)}  \Big \br \begin{pmatrix} \slashed D_{A^{h_\e}}\ph  \\ \bold d a\end{pmatrix} \ , \ \begin{pmatrix}\gamma(a)\tfrac{\Phi^{h_\e}}{\e} \\ \tfrac{\mu(\ph,\Phi^{h_\e})}{\e} \end{pmatrix} \Big \kt R_\e^{2\nu}+\Big \br \begin{pmatrix}\gamma(a)\tfrac{\Phi^{h_\e}}{\e} \\ \tfrac{\mu(\ph,\Phi^{h_\e})}{\e} \end{pmatrix}   \ , \ \begin{pmatrix} \slashed D_{A^{h_\e}}\ph  \\ \bold d a\end{pmatrix} \Big \kt R_\e^{2\nu} \ dV..  \eea

\noindent Since $\bold d=-\slashed D$ up to viewing $a=(\zeta,\omega)$ under the isomorphisms of Section 7.1, Lemma  \ref{weightedestimatesdirac1} is applicable. Applying this lemma to $\bold d$ and Proposition  \ref{weightedestimatesdirac2} to $\slashed D_{A^{h_\e}}$ shows 

\bea
C\left(\|K_\e (\ph,a)\|_{L^2_\nu}+\|\mathcal L\frak (\ph,a)\|^2_{L^2_\nu}\right) &\geq& \|\nabla \ph\|^2_{L^2_\nu} + \|\nabla a\|^2_{L^2_\nu}+\Big \| \frac{\ph}{R_\e}\Big \|_{L^2_\nu} + \frac{1}{\e^2} \|\gamma(a)\Phi^{h_\e}\|^2_{L^2_\nu}+\frac{1}{\e^2} \|\mu(\ph,\Phi^{h_\e})\|^2_{L^2_\nu}   \\ &+ &  \int_{N_\lambda(\mathcal Z_0)}  \Big \br \begin{pmatrix} \slashed D_{A^{h_\e}}\ph  \\ \bold d a\end{pmatrix}  ,  \begin{pmatrix}\gamma(a)\tfrac{\Phi^{h_\e}}{\e} \\ \tfrac{\mu(\ph,\Phi^{h_\e})}{\e} \end{pmatrix} \Big \kt R_\e^{2\nu}+\Big \br \begin{pmatrix}a  \\ \ph  \end{pmatrix}    ,  \begin{pmatrix}\tfrac{\mu( \slashed D\ph,\Phi^{h_\e})}{\e}  \\ \gamma(\bold d a)\tfrac{\Phi^{h_\e}}{\e}\end{pmatrix} \Big \kt R_\e^{2\nu} \ dV \\ \\ 
&=& \|(\ph,a)\|^2_{H^1_{\e,\nu}}  +  \int_{N_\lambda(\mathcal Z_0)}  \Big \br \begin{pmatrix} \slashed D_{A^{h_\e}}\ph  \\ \bold d a\end{pmatrix}  ,  \begin{pmatrix}\gamma(a)\tfrac{\Phi^{h_\e}}{\e} \\ \tfrac{\mu(\ph,\Phi^{h_\e})}{\e} \end{pmatrix} \Big \kt R_\e^{2\nu}  dV  \\ & & + \ \int_{N_\lambda(\mathcal Z_0)}\Big \br \begin{pmatrix}a  \\ \ph  \end{pmatrix}    ,  \begin{pmatrix}\tfrac{\mu( \slashed D\ph,\Phi^{h_\e})}{\e}  \\ \gamma(\bold d a)\tfrac{\Phi^{h_\e}}{\e}\end{pmatrix} \Big \kt R_\e^{2\nu} \ dV. \\
  \eea

Integrating by parts on the first cross term, and noting that the boundary conditions imply the boundary term vanishes (up to rewriting $a=(\zeta,\omega)$ this is the same boundary term that vanishes in \ref{integrationbypartsA} and \ref{partsintB} in Section 6), and the expressions \ref{appendix1identities} from the proof of the Weitzenb\"ock formula yield 

\bea
\left(\|K_\e (\ph,a)\|_{L^2_\nu}+\|\mathcal L\frak (\ph,a)\|^2_{L^2_\nu}\right) &\geq& \frac{1}{C}\|(\ph,a)\|^2_{H^1_{\e,\nu}} + \frac{1}{\e}\br (\ph,a), \mathfrak B (\ph,a)\kt_{L^2_\nu}  \\ & & +  \int_{N_\lambda(\mathcal Z_0)}\Big \br  \begin{pmatrix} \ph \\ a \end{pmatrix} \ , \  \sigma(dr) \frac{2\nu}{R_\e}\d{R_\e}{r} \begin{pmatrix}\gamma(a)\tfrac{\Phi^{h_\e}}{\e} \\ \tfrac{\mu(\ph,\Phi^{h_\e})}{\e} \end{pmatrix} \Big \kt R_\e^{2\nu} \ dV.
\eea

Recall from (\refeq{Kedef}) that $K_\e$ is supported in the region $r\leq \e^{2/3-\gamma'}$. Restricting to this region,  Young's inequality shows  

\bea
 \int_{N_\lambda}\Big \br  \begin{pmatrix} \ph \\ a \end{pmatrix} ,  \sigma(dr) \tfrac{2\nu}{R_\e}\tfrac{d R_\e}{dr} \begin{pmatrix}\gamma(a)\tfrac{\Phi^{h_\e}}{\e} \\ \tfrac{\mu(\ph,\Phi^{h_\e})}{\e} \end{pmatrix} \Big \kt R_\e^{2\nu} \ dV &  \leq  &  4C \Big \|\frac{(\ph,a)}{R_\e} \Big \|^2_{L^2_\nu}+  \frac{1}{4C\e^2}\left(\|\gamma(a)\Phi^{h_\e}\|_{L^2_\nu}^2 + \|\mu(\ph,\Phi^{h_\e})\|^2_{L^2_\nu}\right)  \\
 &  \leq  & C\|K_\e(\ph,a)\|^2_{L^2_\nu} + \frac{1}{4C}\|(\ph,a)\|^2_{H^1_{\e,\nu}}
 \eea 
 \noindent and, 
 \bea 
  \frac{1}{\e}\br (\ph,a), \mathfrak B (\ph,a)\kt_{L^2_\nu}  &\leq & \frac{1}{4C\e^2}\|\gamma(a)\Phi^{h_\e}\|_{L^2_\nu}^2 + 4C \int_{r \leq \e^{2/3-\gamma'}} |\ph|^2 \frac{|\nabla_{A^{h_\e}}\Phi^{h_\e}|^2}{|\Phi^{h_\e}|^2} \ dV \\
    &  \leq  & \frac{1}{4C}\|(\ph,a)\|^2_{H^1_{\e,\nu}}+ C\|K_\e(\ph,a)\|^2_{L^2_\nu}. 
\eea

\noindent The last inequality follows from the inequality following inequality for the re-scaled quantities, which implies the quantity to the right of it. 

$$ \frac{|\nabla_{A^H}\Phi^H|^2}{|\Phi^H|^2}  \leq \frac{1}{R^2}  \hspace{1cm} \Rightarrow \hspace{1cm} \frac{|\nabla_{A^{h_\e}}\Phi^{h_\e}|^2}{|\Phi^{h_\e}|^2}  \leq \frac{C}{R^2_\e} $$

\noindent Indeed, $\Phi^{H}$ to $\Phi^{h_\e}$ introduces the same factor on the top and the bottom, while rescaling the covariant derivative $\nabla_{A^H}$ in $\rho$ coordinates to $\nabla_{A^{h_\e}}$ in $r$ coordinates introduces a factor of $\left(\frac{K}{\e}\right)^{4/3} \leq C \frac{R^2}{R_\e^2}$.  This shows the desired estimate for the $r\leq \e^{2/3-\gamma'}$ region. 

Proceeding to the region where $r\geq \e^{2/3-\gamma'}$, we claim that   

 \be
  \frac{1}{\e}\br (\ph,a), \mathfrak B (\ph,a)\kt_{L^2_\nu}  \leq  C \e^{3\gamma'/2} \left( \frac{1}{\e^2}\|\gamma(a)\Phi^{h_\e}\|_{L^2_\nu}^2+  \frac{1}{\e^2}\|\mu(\ph,\Phi^{h_\e})\|_{L^2_\nu}^2 + \Big \|\frac{\ph}{R_\e}\Big \|^2_{L^2_\nu}\right) \label{outsideregionbound}
\ee
\noindent here. To see this note the following things. 

First, since up to exponentially small factors, $\mathfrak B$ only sees the imaginary components of the spinor in this region, and $$\frac{1}{\e}|\nabla_{A^{h_\e}}\Phi^{h_\e}|\sim \frac{1}{r^{1/2}\e} \leq \e^{3\gamma'/2} \frac{r}{\e^2}\leq C \e^{3\gamma'/2}\frac{|\Phi^{h_\e}|^2}{\e^2}$$

\noindent where $\sim $ denotes a bound up to an exponentially small error (which are easily absorbed by the norm). 

Likewise, since $\sigma(dr)$ is a real form, $\sigma(dr)\gamma(a)$ is a purely-imaginary form, and the term arising from the derivative of the weights similarly only sees the $\ker(\mu(\_, \Phi^{h_\e}))^\perp$ components. Thus

 \bea 
 \Big \br  \begin{pmatrix} \ph \\ a \end{pmatrix} \ , \  \sigma(dr) \frac{2\nu}{R_\e}\d{R_\e}{r} \begin{pmatrix}\gamma(a)\tfrac{\Phi^{h_\e}}{\e} \\ \tfrac{\mu(\ph,\Phi^{h_\e})}{\e} \end{pmatrix} \Big \kt  & \leq & \frac{ C}{\e^{3\gamma'/2}} \frac{|(\ph^{\perp}, a)|^2}{R_\e^2}   + \e^{3\gamma'/2} \left(\frac{1}{\e^2}|\gamma(a)\Phi^{h_\e}|^2+  \frac{1}{\e^2}|\mu(\ph,\Phi^{h_\e})|^2 \right) \\  &\leq & C \e^{3\gamma'/2} \left(\frac{1}{\e^2}|\gamma(a)\Phi^{h_\e}|^2+  \frac{1}{\e^2}|\mu(\ph,\Phi^{h_\e})|^2 \right) \\
\eea
since $\frac{1}{R_\e^2}\leq \frac{\e^{2\gamma}}{\e^{4/3}}\leq \e^{3\gamma} \frac{r}{\e^2}\leq C \e^{3\gamma}\frac{|\Phi^{h_\e}|^2}{\e^2}$. Combining these yields (\refeq{outsideregionbound}). Since $3\gamma'/2>0$, it follows from (\refeq{outsideregionbound}) that for $\e$ sufficiently small, the cross-terms can therefore be absorbed in the outside region as on the inside. 

Putting the regions together again,  

$$\left(\|K_\e (\ph,a)\|_{L^2_\nu}+\|\mathcal L\frak (\ph,a)\|^2_{L^2_\nu}\right) \geq \frac{1}{C}\|(\ph,a)\|^2_{H^1_{\e,\nu}} $$
up to increasing $C$ by a constant factor. This completes Step $2$. 

\noindent {\bf Step 3:} The following estimate holds, again for $|\tfrac{1}{4}-\nu|<<1$: 

\be
\|\frak q\|^2_{H^1_{\e,\nu}} \leq C \|\mathcal L \frak q\|^2_{L^2_\nu}+ C\e^{1/6-2\gamma_2} \|\mathcal L\frak q\|^2_{L^2}.\label{toprovestep3}
\ee

\noindent where $\gamma_2<<1$ again. 

This follows from the previous steps and the boundary absorption Lemma \ref{bdabsorption}. Combining the inequalities from the previous steps. Beginning with the inequality from Step 2:  

\bea
\|\frak q\|^2_{H^1_{\e,\nu}} &\leq& C \|\mathcal L \frak q\|^2_{L^2_\nu}+ C \|K_\e\frak q\|^2_{L^2_\nu} \\
&\leq & C \|\mathcal L \frak q\|^2_{L^2_\nu}+ C\int_{r\leq \e^{2/3-\gamma'}} \frac{|\frak q|^2 }{R_\e^2} R_\e^{2\nu} \ dV \\
&\leq & C \|\mathcal L \frak q\|^2_{L^2_\nu}+ C \e^{(2/3-\gamma' )2\nu} \int_{r\leq \e^{2/3-\gamma'}} \frac{|\frak q|^2 }{R_\e^2}\ dV\\ 
&\leq &C \|\mathcal L \frak q\|^2_{L^2_\nu}+ C \e^{1/3} \e^{-\tfrac{2}{3}(\nu - \tfrac{1}{4}) - 2\gamma' \nu}  \|\frak q\|^2_{H^1_\e}\\
&\leq & C \|\mathcal L \frak q\|^2_{L^2_\nu}+ C \e^{1/3} \e^{-2\gamma_2}  \|\frak q\|^2_{H^1_\e}
\eea
where $\gamma_2 ={\tfrac{2}{3}(\nu - \tfrac{1}{4}) + \gamma' \nu}<<1$. In this, we have used the lower bound that $\frac{|\Phi^{h_\e}|^2}{\e^2}\geq \frac{C}{R_\e^2}$ (which follows from the third bullet point of Lemma \ref{desingularizedproperties}) on the form component to absorb $K_\e$ into the $H^1_\e$-norm. 

Next applying the estimate from Step 1, the above shows

\bea
\|\frak q\|^2_{H^1_{\e,\nu}}  &\leq & C \|\mathcal L \frak q\|^2_{L^2_\nu}+ C \e^{1/3} \e^{-\gamma_2}  \left( \frac{C}{\e^{1/6}} \|\mathcal L \frak q\|^2_{L^2}  + \frac{C}{\e^{1/2}L_0} \|\mathfrak q\|^2_{L^2(\del N_\lambda(\mathcal Z_0))} \right)\\
&\leq & C \left( \|\mathcal L \frak q\|^2_{L^2_\nu} + \e^{1/6-\gamma_2} \|\mathcal L \frak q\|^2_{L^2} \right) +   \frac{C\e^{1/3-\gamma_2}}{ \e^{1/2}L_0} \int_{\del N_\lambda(\mathcal Z_0)} |\frak q|^2 dA \\
&\leq & C \left( \|\mathcal L \frak q\|^2_{L^2_\nu} + \e^{1/6-\gamma_2} \|\mathcal L \frak q\|^2_{L^2} \right) +   \frac{C\e^{1/3-\gamma_2} (\e^{1/2})^{-2\nu}}{ \e^{1/2}L_0} \int_{\del N_\lambda(\mathcal Z_0)} |\frak q|^2 R_\e^{2\nu} dA \\ 
&\leq & C \left( \|\mathcal L \frak q\|^2_{L^2_\nu} + \e^{1/6-\gamma_2} \|\mathcal L \frak q\|^2_{L^2} \right) +   \frac{C\e^{1/12-\gamma_2}}{ \e^{1/2}L_0} \int_{\del N_\lambda(\mathcal Z_0)} |\frak q|^2 R_\e^{2\nu} dA \\ 
&\leq & C \left( \|\mathcal L \frak q\|^2_{L^2_\nu} + \e^{1/6-\gamma_2} \|\mathcal L \frak q\|^2_{L^2} \right) +   {C\e^{1/12-\gamma_2}} \|\frak q\|_{H^1_\e}^2 \\ 
\eea
where in the last line we have applied the boundary absorption Lemma  \ref{bdabsorption} to show

$$\int_{\del N_\lambda(\mathcal Z_0)} |\ph|^2 R_\e^{2\nu} \ rd\theta dt \leq C \e^{1/2} \int_{N_\lambda(\mathcal Z_0)} \left(|\nabla \ph|^2 +\frac{|\ph|^2 }{R_\e^2}\right) R_\e^{2\nu} \ dV$$
and recalled that $L_0$ is a universal constant independent of $\e$. Up to increasing $C$ (by a factor of 2, say), the last term may be absorbed on the left hand side once $\e$ is sufficiently small, yielding the desired estimate (\refeq{toprovestep3}) and completing Step 3.

 Since $R_\e^{2\nu}\leq \e^{\nu}< \e^{1/6-2\gamma}$ on $N_{\lambda}(\mathcal Z_0)$, the second term $\e^{1/6-\gamma_2} \|\mathcal L \frak q\|^2_{L^2}$ in the parentheses dominates the first. Taking the square root yields the desired estimate (\refeq{toprove}), completing the proof of Theorem \ref{invertibleL} in the model case. 
\end{proof}

\subsection{General Metric}
\label{genmetricssection}
To complete the proof, we extend the above result from the model case to the general case. Thus we now assume, in full generality, that in geodesic normal coordinates and a trivialization on $N_\lambda(\mathcal Z_0$), we have 

\bea
g= dt^2 + dx^2 + dy^2 + h \hspace{1cm} B_0= \text{ is a fixed smooth $SU(2)$-connection}
\eea
where $h$ is described in Definition \ref{geodesicnormal}, and that 
$$ \Phi_0= \begin{pmatrix} c(t) r^{1/2}  \ \ \ \ \  \\ d(t)r^{1/2}e^{-i\theta} \end{pmatrix}\otimes 1+\begin{pmatrix} -\overline d(t) r^{1/2}  \ \ \ \ \  \\ \overline c(t)r^{1/2}e^{-i\theta} \end{pmatrix} \otimes j+\Phi^{h.o.} \hspace{2cm}A_0=\frac{i}{2}d\theta+ \epsilon_j \tfrac{i}{2}dt$$
where $\Phi^{h.o.}$ is the higher order terms given in Proposition \ref{asymptoticexpansion}. Thus we have bounds 

\bea
\|h\|_{C^0}&\leq& C r\\
\|\nabla h\|_{C^0}&\leq& C \\ 
\|B_0+\epsilon_j \tfrac{i}{2}dt\|_{C^0}&\leq& C \\   
\|\Phi^{h.o.}\|_{C^0} & \leq & Cr^{3/2}\\
\|\nabla \Phi^{h.o.}\|_{C^0} & \leq & Cr^{1/2} 
\eea 
for constants $C$ independent of $\e$. We also re-introduce the cut-off function as in Definition \ref{scaleinvariantrho}to replace $h_\e(r)$ with $\chi_\e(r)h_\e(r)$ where $\chi_\e(r)$ is equal to $1$ on a neighborhood of radius $c \e^{1/2}$ for $c<1$ and supported in the neighborhood of radius $\e^{1/2}$.

\begin{proof}
(of Theorem \ref{invertibleL} in the case of the above). Let $\mathcal L^{\text{Euc}}$ now denote the model operator. The result of Section \ref{intbyparts} show that for $\nu< 1/4$,

$$\|(\ph,a)\|_{H^1_{\e,\nu}} \leq   C\e^{1/12-\gamma_2} \|\mathcal L^\text{Euc}(\ph,a)\|_{L^2}.$$
Thus it suffices to show that \be \|(\mathcal L^{h_\e}-\mathcal L^{\text{Euc}})(\ph,a)\|_{L^2}\leq C \e^{1/16} \|(\ph,a)\|_{H^1_{\e,\nu}}\label{smalldiff}\ee for $\e$ sufficiently small and $\nu$ sufficiently close to $1/4$. We may write 

\bea \|(\mathcal L^{h_\e}-\mathcal L^\text{Euc})(\ph,a)\|_{L^2} &\leq & \|(\slashed D_{A^{h_\e}}- \slashed D^\text{Euc}_{A^{h_\e}})\ph\|_{L^2 } +\|(\bold d-\bold d^\text{Euc})a\|_{L^2 }\\ & & + \|(\gamma-\gamma^\text{Euc})(a) \tfrac{\Phi^{h_\e}}{\e}\|_{L^2} + \|\gamma(a) \tfrac{(\Phi^{h_\e})^{h.o.}}{\e}\|_{L^2}  \\  & & \|(\mu-\mu^\text{Euc})(\ph, \tfrac{\Phi^{h_\e}}{\e})\|_{L^2}+ \|\mu(\ph, \tfrac{(\Phi^{h_\e})^{h.o.}}{\e})\|_{L^2}.  \eea

Bounding each term individually, one has 
\bea
\|(\bold d-\bold d^\text{Euc})a\|^2_{L^2 } &\leq& \int_{N_\lambda(\mathcal Z_0)} Cr^2 |\nabla a|^2 + C |a|^2 \ dV \\
&\leq & C \e  \int_{N_\lambda(\mathcal Z_0)}  |\nabla a|^2 + \frac{ |a|^2|\Phi^{h_\e}|^2}{\e^2} \ dV\\
&\leq & C \e^{2/3}  \int_{N_\lambda(\mathcal Z_0)}  \left(|\nabla a|^2 + \frac{ |a|^2|\Phi^{h_\e}|^2}{\e^2}\right) R_\e^{2\nu} \ dV \\
&\leq &C \left(\e^{1/3}\|(\ph,a)\|_{H^1_{\e,\nu}}\right)^2
\eea
and identically, for the Dirac operator with trivial connection
\bea 
\|(\slashed D-\slashed D^\text{Euc})\ph\|^2_{L^2 }& \leq & \int_{N_\lambda(\mathcal Z_0)} Cr^2 |\nabla a|^2 + C |a|^2 \ dV \\
&\leq & C \e  \int_{N_\lambda(\mathcal Z_0)}  |\nabla a|^2 + \frac{ |\ph|^2}{R_\e^2} \ dV\\
&\leq & C \left(\e^{1/3}\|(\ph,a)\|_{H^1_{\e,\nu}}\right)^2
\eea
while for the connection term
\bea
\|(\gamma-\gamma^\text{Euc})(A^{h_\e})\ph\|^2_{L^2}&\leq& \int_{N_\lambda(\mathcal Z_0)} Cr^2 |A^{h_\e}|^2 |\ph|^2 \ dV  \\
&\leq & C \int_{N_\lambda(\mathcal Z_0)} r^{3/2} \frac{|\ph|^2}{R_\e^2} R_\e^{2\nu} \ dV\\
&\leq & C \left(\e^{3/8} \|(\ph,a)\|_{H^1_{\e,\nu}}\right)^2\\
\|B_0 + \epsilon_j \tfrac{i}{2}dt\ph\|_{L^2}^2 &\leq& C(\e^{3/4}\|(\ph,a)\|_{H^1_{\e,\nu}})^2.
\eea
since $r\leq \e^{1/2}$. This completes the two diagonal terms.  

For the off-diagonal terms, we have

 \bea \|(\gamma-\gamma^\text{Euc})(a) \tfrac{\Phi^{h_\e}}{\e}\|^2_{L^2}+  \|\gamma(a) \tfrac{(\Phi^{h_\e})^{h.o.}}{\e}\|^2_{L^2}&\leq &\int_{N_\lambda(\mathcal Z_0)} Cr^2 \frac{ |a|^2 |\Phi^{h_\e}|^2}{\e^2}  + Cr^2 \frac{ |a|^2 |\Phi^{h_\e}|^2}{\e^2}  +  \ dV \\ &  \leq& C \int_{N_\lambda(\mathcal Z_0)} Cr^{3/2} \frac{ |a|^2 |\Phi^{h_\e}|^2}{\e^2}  R_\e^{2\nu}\ dV \leq C \left( \e^{3/8} \|(\ph,a)\|_{H^1_{\e,\nu}}\right)^{2}
 \eea 
\noindent  and since $|(\Phi^{h_\e})^{h.o.}| \leq Cr |\Phi^{h_\e}|\leq Cr |(\Phi^{h_\e})^{l.o.}| $. And 

\bea
 \|(\mu-\mu^\text{Euc})(\ph,  \tfrac{\Phi^{h_\e}}{\e})\|^2_{L^2}+ \|\mu(\ph,  \tfrac{(\Phi^{h_\e})^{h.o.}}{\e}\|^2_{L^2}&\leq & \int_{N_{\lambda(\mathcal Z_0)}} C r^2 |\ph|^2 \frac{|\Phi^{h_\e}|^2}{\e^2} \ dV \\
 &\leq &  \int_{N_{\lambda(\mathcal Z_0)}} C \frac{r^2 \e^{1/2}}{ {\e^2}} R_\e^{2-2\nu} \frac{|\ph|^2}{R_\e^2} R_\e^{2\nu} \ dV\\
 &\leq& C \frac{(\e^{1/2})^{4.5}}{\e^2} \|(\ph,a)\|^2_{H^1_{\e,\nu}} \leq C(\e^{1/8} \|(\ph,a)\|^2_{H^1_{\e,\nu}})^2. 
\eea 
since $r\leq \e^{1/2}$. 

Combining these estimates gives (\refeq{smalldiff}). This shows that $$\mathcal L^{h_\e}: H^1_{\e,\nu}\to L^2$$ is invertible, and the same bound$$\|(\ph,a)\|_{H^1_{\e,\nu}} \leq   C\e^{1/12-\gamma_2} \|\mathcal L^{h_\e}(\ph,a)\|_{L^2}$$
holds on the inverse as in the model case, where $H^1_{\e,\nu}$ still denotes the norm formed using the Euclidean structures and the model case. Switching the norm to the one formed using the non-model structures is essentially the same estimates, but we now only need them to be bounded by a uniform constant. Re-introducing the cut-off function $\chi_\e(r)$ clearly introduces only an exponentially small change, which is of no consequence. This completes the proof of Theorem \ref{invertibleL} in the general case. 

\end{proof}

\section{Implicit Function Theorem}
\label{section8}

In this final section we conclude the proofs of the main results Theorems \ref{main}-\ref{mainc}. The existence of the fiducial solutions advertised in Theorem \ref{main} is concluded by applying the standard Inverse Function Theorem to solve the non-linear equation (\refeq{non-lineartosolve}), which was 

\be (\mathcal L^{h_\e} + Q)(\ph_\e, a_\e) = E^{(0)}_\e.\label{nonlineartosolve2}\ee

\noindent up to decreasing the size of tubular neighborhood by a factor of $1/2$.

The following quantitative version of the Inverse Function Theorem is taken from \cite{KM} (Theorem 18.3.6). 

\begin{thm} {\bf (Inverse Function Theorem)}\label{IFT}
Let $H_1, H_2$ be Hilbert spaces, and $S: H_1\to H_2$ a continuous map between them satisfying $S(0)=0$. Suppose that $S$ has the form $$S=\mathcal L+ Q$$
\noindent where $\mathcal L$ is linear and invertible, and $Q$ is uniformly Lipschitz on the $\eta_1$ radius ball $B_{\eta_1}(H_1)\subset H_1$ with Lipschitz constant $M$, i.e. 
$$\|x_1\|_1, \|x_2\|_1 \leq \eta_1 \hspace{1cm}\Rightarrow \hspace{1cm} \|Q(x_1)- Q(x_2)\|_2 \leq M \|x_1-x_2\|_1.$$
If $M\leq 1/\|\mathcal L^{-1}\|$, then $S$ is injective on $B_{\eta_1}(H_1)$ and the image contains the ball $B_{\eta_2}(H_2)\subset H_2$ where $$\eta_2=\frac{\eta_1(1-M\|\mathcal L^{-1}\|)}{\|\mathcal L^{-1}\|}. $$
In particular, for every $y\in B_{\eta_2}(H_2)$ there is a unique $x \in B_{\eta_1}(H_1)$ satisfying $$S(x)=y.$$\qed 
\end{thm}

Let $N_{\lambda/2}(\mathcal Z_0)$ denote the tubular neighborhood of radius $\tfrac{3}{4}\lambda=\tfrac{3}{4}\e^{1/2}$, and let $\chi_1$ denote a logarithmic cut-off function equal to $1$ on $N_{\lambda/2}(\mathcal Z_0)$ and supported in $N_{\lambda}$ such that \be |d\chi_1|\leq \frac{C}{r}.\label{cutoffderiv}\ee

\noindent To solve Equation (\refeq{nonlineartosolve2}) on $N_{3\lambda/4}(\mathcal Z_0)$ it suffices to solve

\be (\mathcal L^{h_\e} + \chi_1^2Q)(\ph_\e, a_\e) = E^{(0)}_\e.\label{nonlineartosolvecutoff}\ee 

\noindent on $N_{\lambda}(\mathcal Z_0)$, since $E_\e^{(0)}$ is supported on the inner neighborhood as $\chi_\e(r)h_\e(r)=0$ for $r\geq c\e^{1/2}$. The introduction of $\chi_1$ allows us to apply Sobolev inequalities on the closed manifold with $\e$-independent constant, rather than scaling them to $N_{\lambda}(\mathcal Z_0)$. 

We have the following interpolation bound for configurations $(\ph,a)\in H^1_{\e,\nu}$: 

\begin{lm} For $0<\nu < 1/4$, 
$$\| \chi_1 (\ph,a)\|_{L^4(N_{\lambda}(\mathcal Z_0))} \leq C\e^{-\nu/6} \|(\ph,a)\|_{H^1_{\e,\nu}(N_{\lambda}(\mathcal Z_0))}$$
\label{quadraticinterpolation}
\end{lm} 

\begin{proof}
The Gagliardo-Nirenberg Interpolation inequality on $Y$ (see Equation (1.4) of \cite{NonintegerGNInterpolation} and apply this using a partition of unity) states 

$$\|u\|_{L^4(Y)}^2 \leq C\left( \|u\|_{L^2}^{1/2} \|\nabla u \|_{L^2}^{3/2} + \|u\|^2_{L^2}\right).$$

Applying this to the configuration $\chi_1 \frak q= \chi_1 (\ph,a)$ yields 
\bea\| \chi_1 \frak q \|_{L^4}^2& \lesssim&  \|\chi_1 \frak q\|_{L^2}^{1/2}  \|\nabla(\chi_1  \frak q)  \|^{3/2}_{L^2}  + \|\chi_1 \frak q\|^2_{L^2}\\
&\lesssim &  \e^{-4\nu/3}\left( \| R_\e^{\nu} \chi_1\frak q\|_{L^2}^{1/2} \  \cdot \  \| R_\e^{\nu}\nabla(\chi_1  \frak q)  \|^{3/2}_{L^2}  + \|R_\e^{\nu}\chi_1 \frak q\|^2_{L^2}\right)
\eea
since $R_\e\lesssim \e^{2/3}$. Then, as $1/R_\e \geq c\e^{-1/2}$ on $N_{\lambda}(\mathcal Z_0)$, 
\bea 
&\lesssim &  \e^{-4\nu/3} \e^{1/4 }\left( \| \tfrac{R_\e^{\nu}}{R_\e} \chi_1\frak q\|_{L^2}^{1/2} \  \cdot \  \| R_\e^{\nu}\nabla(\chi_1  \frak q)  \|^{3/2}_{L^2}  + \|\tfrac{R_\e^{\nu}}{R_\e}\chi_1 \frak q\|^2_{L^2}\right) \\
&\lesssim &  \e^{-\nu/3}\left( \| \frak q\|^{1/2}_{H^1_{\e,\nu}(N_\lambda)} \  \cdot \ \left(  \| R_\e^{\nu}\chi_1 \nabla  \frak q)  \|_{L^2} + \| R_\e^\nu d\chi_1 \frak q \|  \right)^{3/2}  + \| \frak q \|^2_{H^1_{\e,\nu}(N_\lambda)}\right) \\
\eea
and by, (\refeq{cutoffderiv}), $|d\chi_1|\leq \tfrac{C}{R_\e}$. Hence, also using that $\tfrac{1}{R_\e^2}\lesssim \tfrac{|\Phi^{h_\e}|^2}{\e^2}$ for the connection component $a$, 

\bea 
&\lesssim &   \e^{-\nu/3}\left( \| \frak q\|^{1/2}_{H^1_{\e,\nu}(N_\lambda)} \  \cdot \ \left(  \| \frak q\|_{H^1_{\e,\nu}(N_\lambda)}  +  \| \frak q\|_{H^1_{\e,\nu}(N_\lambda)}  \right)^{3/2}  + \| \frak q \|^2_{H^1_{\e,\nu}(N_\lambda)}\right)\\ 
&\lesssim &  \e^{-\nu/3} \| \frak q\|^{2}_{H^1_{\e,\nu}(N_\lambda)}.
\eea
and taking the square root completes the Lemma. 
\end{proof}

Using this to bound the quadratic term we now conclude the proofs of Theorem \ref{main}-\ref{mainc} and Corollary \ref{mainb}. The statements given in the introduction follow from the statements here after replacing $\lambda$ by $\lambda/2$. First, we apply the Inverse Function Theorem \ref{IFT} to solve Equation (\refeq{nonlineartosolve2}) and thus conclude the proof of Theorem \ref{main}. 

\begin{proof}(of Theorem \ref{main}). In the notation of the statement of the Inverse Function Theorem \ref{IFT}, set $$H_1=H^1_{\e,\nu}(N_\lambda(\mathcal Z_0)) \hspace{2cm} H_2=L^2 (N_\lambda(\mathcal Z_0))$$ with $|\tfrac14-\nu|<<1$ as before. Theorem \ref{invertibleL} shows that as a map $H_1\to H_2$ we have $$\|\mathcal L^{-1}_{(\Phi^{h_\e}, A^{h_\e}, \e)}\|\leq C_{\mathcal L} \e^{1/12-\gamma_2}$$

Set $\eta_1 = C_0(\e) \|\mathcal L^{-1}_{(\Phi^{h_\e}, A^{h_\e}, \e)}\|$ where $C_0(\e)= C_0 \e^{-\gamma}$ and $C_0$ is a fixed constant so that the error of Lemma \ref{errorcalculation} obeys $\|E^{(0)}_\e\|_{L^2}\leq \tfrac{1}{10}C_0(\e)$. For two configurations $\frak q_1=(\ph_1,a_1)$ and $\frak q_2=(\ph_2, a_2)$ we may write $$\chi_1^2 Q(\frak q_1) - \chi_1^2 Q(\frak q_2)= \chi_1 (\frak q_1 + \frak q_2) \# \chi_1(\frak q_1 -\frak q_2)$$ 
 where $\#$ denotes a pointwise quadratic map. Then using Lemma \ref{quadraticinterpolation},  $\frak q_1, \frak q_2 \in B_{\eta_1}(0)\subseteq H^1_{\e,\nu}$ implies 
 
 \bea
 \| \chi_1^2 Q(\frak q_1) - \chi_1^2 Q(\frak q_2)\|_{L^2}&\leq& \|\chi_1 (\frak q_1 + \frak q_2)\|_{L^4(N_\lambda)} \cdot  \|\chi_1 (\frak q_1 - \frak q_2)\|_{L^4(N_\lambda)} \\ 
 &\leq& \|\chi_1 (\frak q_1 + \frak q_2)\|_{L^4(N_\lambda)} \cdot  \|\chi_1 (\frak q_1 - \frak q_2)\|_{L^4(N_\lambda)}\\
 &\leq & \e^{-\nu/3} \|\frak q_1 + \frak q_2\|_{H^1_{\e,\nu}(N_\lambda)} \cdot  \|\frak q_1 - \frak q_2\|_{H^1_{\e,\nu}(N_\lambda)} \\
 &\leq & \e^{-\nu/3} 2C_0(\e) C\e^{1/12-\gamma_2} \|\frak q_1 - \frak q_2\|_{H^1_{\e,\nu}(N_\lambda)} \\
 &\leq & 2C_0(\e) C \e^{-\gamma_3}  \|\frak q_1 - \frak q_2\|_{H^1_{\e,\nu}(N_\lambda)}
 \eea
 for some $\gamma_3<<1$. Thus the Lipschitz bound is satisfied with $M= 2C_0(\e) C_{\mathcal L} \e^{-\gamma_3} \leq \tfrac{1}{C_\mathcal L \e^{1/12-\gamma_2}} = \frac{1}{\|\mathcal L^{-1}\|}$, and $M\|\mathcal L^{-1}\|< \e^{1/12-\gamma_2 -\gamma_3}<\frac{1}{2}$ once $\e$ is sufficiently small. The Inverse Function Theorem applies with $$\eta_2 = \frac{\eta_1 (1-M\|\mathcal L^{-1}\|)}{\|\mathcal L^{-1}\|}\geq  C_0(\e)(1-\tfrac{1}{2}) > \tfrac{C_0(\e)}{2},$$
 and by our choice of $C_0(\e)$, the equation 
$$ (\mathcal L_{(\Phi^{h_\e}, A^{h_\e}, \e)} + \chi_1^2Q)(\ph_\e, a_\e) = E^{(0)}_\e $$
therefore admits a unique solution $(\ph_\e, a_\e) \in H^1_{\e,\nu}(N_\lambda(\mathcal Z_0))$ which then solves the Seiberg-Witten equation on $N_{3\lambda/4}(\mathcal Z_0)$, 
 such that \be \|(\ph_\e, a_\e)\|_{H^{1}_{\e,\nu}(N_{\lambda}(\mathcal Z_0))} \leq C \e^{1/12-\gamma_2}.\label{finalperturbationsize}\ee
 
 \noindent The configurations \be (\Phi_\e, A_\e):=\Big(\frac{\Phi^{h_\e}}{\e}, A^{h_\e}\Big) + (\ph_\e, a_\e)\label{fidsoltheoremaproof}\ee are then the desired family of model solutions. This construction is local, so the proof applies independently on every component of $\mathcal Z_0$ once $\e$ is sufficiently small.   \end{proof}

Corollary \ref{mainb} is deduced directly from this using the main results of \cite{ConcentratingDirac}. The details are given in Appendix \ref{appendix2}.

\begin{proof} (of Theorem \ref{mainc}) Continuing to denote the cut-off function defined in \ref{cutoffderiv} by $\chi_1$, denote by $\mathcal L^\text{App}$ the linearization at the approximate solutions \be(\Phi^\text{App}_\e, A^\text{app}_\e):=\Big(\frac{\Phi^{h_\e}}{\e}, A^{h_\e}\Big) + \chi_1^2(\ph_\e, a_\e)\label{cutofffid}\ee (so that in the definition (\refeq{approx}) preceding the statement of Theorem \ref{mainc} one takes $\chi=\chi_1^2$). Similarly denote the linearization at the de-singularized configurations by $\mathcal L^{h_\e}$. Lemma (\ref{quadraticinterpolation}) and (\refeq{finalperturbationsize}) imply  
 
 $$\|(\mathcal L^\text{App}- \mathcal L^{h_\e})(\ph,a)\|_{L^2(N_\lambda(\mathcal Z_0))}\leq \e^{-\gamma_4} \|(\ph,a)\|_{H^1_{\e,\nu}(N_\lambda(\mathcal Z_0))} $$
 
\noindent for $\gamma_4<<1$, and as above, $\|(\mathcal L^{h_\e})^{-1}\|_{L^2 \to H^1_\nu}\leq  C_{\mathcal L}\e^{1/12-\gamma_2}$ so $$(\mathcal L^{h_\e})^{-1} \Big(\mathcal L^{h_\e} + (\mathcal L^\text{App}- \mathcal L^{h_\e})\Big)= Id + O(\e^{1/12-\gamma_2-\gamma_4}).$$

\noindent It follows that the same invertibility statement given in Theorem \ref{invertibleL} for $\mathcal L^{h_\e}$ holds for $\mathcal L^{\text{App}}$ up to possibly increasing the constants by a factor of 2. 
\end{proof}
\bigskip 

\appendix

\section{Appendix I: Bootstrapping Convergence} \label{appendix2}

In this section we use the results of \cite{ConcentratingDirac} to deduce Corollary \ref{mainb} by a straightforward (but slightly detailed) bootstrapping argument. Notice that the bounds on $\ph^\text{Re}$ stated in Corollary \ref{mainb} follow directly from Theorem \ref{invertibleL} and the definition of the $H^1_{\e,\nu}$-norm. For the $(\ph^\text{Im},a)$ components, we apply the following result found in (Appendix I of) \cite{ConcentratingDirac}. In it, we assume that \be \left(\tfrac{\Phi_0}{\e} , A_0\right) + (\ph,a)\label{perturbphi0}\ee is a solution of the two-spinor Seiberg-Witten equations near a $\Z_2$-harmonic spinor $(\mathcal Z_0,A_0, \Phi_0)$ satisfying Assumptions (\ref{assumption1})-(\ref{assumption2}).

\begin{prop}\label{concentratingmain}
There exists a $c_1>0$ such that if $K_\e \Subset N_{\lambda/2}(\mathcal Z_0)-\mathcal Z_0$ are an $\e$-parameterized family of compact subsets satisfying $\text{dist}(K_\e,\mathcal Z_0)\geq c_1 \e^{2/3-\gamma_1}$ for any $\gamma_1>0$, and one has 

\be 
\e \|\ph^\text{Re}\|_{L^{1,6}(K_\e')}\to 0 \hspace{2cm} \e^{2/3}\|\ph^\text{Re}\|_{C^0(K'_\e)}\to 0. \label{A2} \ee

\noindent  for some $K'_\e\supset K_\e$ with $\text{dist}(Y-K_\e', K_\e)\geq \tfrac{c_1}{2}\e^{2/3-\gamma_1}$,  then 

$$\|(\ph^\text{Im}_\e, a_\e)\|_{C^0(K_\e)}\leq \frac{C}{|\text{dist}(K_\e, \mathcal Z)|^{3/2}\e}\text{Exp}\left(-\frac{c}{\e}\text{dist}(K_\e,\mathcal Z)^{3/2}\right).$$
\qed 
\end{prop}

\bigskip

Suppose that $K_\e$ are a family of compact subsets satisfying $\text{dist}(K_\e,\mathcal Z_0)\geq c_1 \e^{2/3-\gamma_1}$ as in the statement of the proposition. The conclusion of the proof of Theorem \ref{main} and the estimates of Theorem \ref{invertibleL} is that the correction to $(\tfrac{\Phi^{h_\e}}{\e}, A^{h_\e})$ that yields the model solutions satisfies $\|(\ph_\e^\text{mod},a_\e^\text{mod})\|_{H^1_{\e}}\leq C\e^{-1/10}$. Since $(\Phi^{h_\e}, A^{h_\e})-(\Phi_0, A_0)$ is exponentially small on $K_\e$, then writing the model solutions in the form  (\refeq{perturbphi0}), one has 

\be \|(\ph,a)\|_{H^1_{\e}}\leq C\e^{-1/10} \label{initialbound}\ee 

\noindent as well. Thus it suffices to show this implies the bounds (\refeq{A2}) hold. Once this is shown, applying Proposition (\refeq{concentratingmain}) then yields Corollary \ref{mainb}.

The bound (\refeq{initialbound}) is shown by a standard bootstrapping argument, though a rather intricate one as if one is not careful, several applications of the elliptic estimates will pick up powers of $\e^{-1}$ too large for the desired bounds to hold.

Let $$K_\e^{(1)}= \{ \tfrac{c_1}{2}\e^{2/3-\gamma_1}\leq r \leq \tfrac{3\lambda}{4}\}$$ 
\noindent denote the closed annulus of the indicated radii. By assumption, it contains $K_\e$. In addition, let $$K_\e \Subset K_\e^{(m)} \Subset \ldots \Subset K_\e^{(1)}$$
be a nested sequence of finitely many (specifically 7) nested closed annuli. Choose them so that $$K_\e':= K_\e^{(m)}=  \{ \tfrac{3c_1}{4}\e^{2/3-\gamma_1}\leq r \leq \tfrac{5\lambda}{8}\}.$$
Additionally, for each $n=1,...,m$ let $\chi_\e^{(n)}$ denote a logarithmic cut-off function supported on $K_\e^{(n)}$ and equal to 1 on $K_\e^{(n+1)}$.  They may be chosen so that $$|d\chi_\e^{(n)}|\leq \frac{C}{r}$$ uniformly. Finally, we fix a smooth background connection extending $\text{d}$ in the chosen trivialization of $S_E$ on $N_{\lambda}(\mathcal Z_0)$ (given in Lemma \ref{localtrivialization}) with respect to which the Sobolev norms are taken.

\begin{claim} \label{claimA11}The bound  (\refeq{initialbound}) implies the following. 
\begin{itemize}
\item $\|(\ph,a)\|_{L^{6}(K_\e^{(2)})} \leq C \e^{-1/10}$
\item $ \|\gamma(a)\tfrac{\Phi_0}{\e}\|_{L^p(K_\e^{(2)})} + \|\tfrac{\mu(\ph^\text{Im},\Phi_0)}{\e}\|_{L^p(K_\e^{(2)})} \leq C \e^{-1/10}\e^{-(\tfrac{3}{2}-\tfrac{3}{p})}.$
\end{itemize}
\end{claim}

\begin{proof} Both bullet points follow from applying the global (on $Y$) Sobolev and interpolation inequalities to $\chi_\e^{(1)}(\ph,a)$. Indeed, for the first bullet point one has 

$$\|(\ph,a)\|_{L^{6}(K_\e^{(2)})} \leq  \|\chi^{(1)}_\e(\ph,a)\|_{L^{6}(K_\e^{(1)})}=  \|\chi^{(1)}_\e(\ph,a)\|_{L^{6}(Y)}\leq C \|\chi^{(1)}_\e(\ph,a)\|_{L^{1,2}(Y)} \leq C\|(\ph,a)\|_{H^1_\e(K_\e^{(1)})} $$

\noindent where we have used that $|\nabla(\chi_\e^{(1)}\ph,a)|\leq \chi_\e^{(0)}|\nabla (\ph,a)| + \tfrac{\chi_\e^{(1)}|(\ph,a)|}{R_\e}.$

For the second bullet point, we apply the Gagliardo-Nirenberg interpolation inequality on $Y$, which states that for $2<p<6$, 
\be \| u\|_{L^p(Y)}\leq C \|u\|_{L^2}^{1-\alpha} \|\nabla u\|_{L^2}^{\alpha} + \|u\|_{L^2(Y)}\hspace{1cm} 0<\alpha =\tfrac{3}{2}-\tfrac{3}{p}<1.\label{interpolation1}\ee

\noindent To derive this from the standard version for scalar functions on bounded domains in $\R^N$ (see, for instance, Equation (1.4) in \cite{NonintegerGNInterpolation}) use a partition of unity and apply the standard result to $|u|$, then invoke Kato's inequality. Applying (\refeq{interpolation1}) to $\gamma(\chi_\e^{(1)}a)\tfrac{\Phi_0}{\e}$ shows that 

\bea
\|\gamma(a)\tfrac{\Phi_0}{\e}\|_{L^p(K_\e^{(2)})}&\lesssim&  \|\gamma(a)\tfrac{\Phi_0}{\e}\|_{L^2}^{1-\alpha} \|\nabla (\gamma(\chi_\e^{(1)}a)\tfrac{\Phi_0}{\e})\|_{L^2}^{\alpha}  \\
&\leq & \left(\e^{-1/10}\right)^{1-\alpha}\left[   \|\tfrac{1}{\e}\nabla a\|_{L^2}^{\alpha}  + \|\tfrac{\nabla(\chi_\e^{(1)}\Phi_0)}{\e} a\|^{\alpha}_{L^2}\right] + \e^{-1/10}
\eea 
and $|\nabla(\chi_\e^{(1)})\Phi_0|=|( \nabla \chi_\e^{(1)})\Phi_0 + \chi_\e^{(1)}\nabla \Phi_0|\lesssim | \tfrac{1}{r}\Phi_0|$ on the support of $\chi_\e^{(1)}$ where $\tfrac{1}{r}\leq c\e^{-2/3}$, so the above is bounded by

\bea
&\leq & C\left(\e^{-1/10}\right)^{1-\alpha}\left(    \e^{-\alpha}  \e^{-\alpha/10} +  \e^{-2\alpha/3} \e^{-\alpha/10}\right).
\eea 

\noindent The same applies to $\tfrac{\mu(\ph^{\text{Im}},\Phi_0)}{\e}$. Consequently, 
\begin{eqnarray}
\|\gamma(a)\tfrac{\Phi_0}{\e}\|_{L^p(K_\e^{(2)})} &\leq& C \e^{-1/10}\e^{-\alpha}=C \e^{-1/10}\e^{-(\tfrac{3}{2}-\tfrac{3}{p})}\\
\|\tfrac{\mu(\ph^\text{Im},\Phi_0)}{\e}\|_{L^p(K_\e^{(2)})} &\leq& C \e^{-1/10}\e^{-\alpha}=C \e^{-1/10}\e^{-(\tfrac{3}{2}-\tfrac{3}{p})}.
\label{intbounds}
\end{eqnarray}

\end{proof}

With this in hand, we claim:  

\begin{claim}
The following bounds are satisfied on $K'_\e= K_\e^{(m)}$: 

\be
 \|\ph^\text{Re}\|_{L^{1,6}(K_\e)} \leq \frac{C}{\e}\hspace{3cm} \|\ph^\text{Re}\|_{C^0(K'_\e)}\leq \frac{C}{\e^{2/3}}.\label{tobootstrap} 
\ee
In particular, the hypotheses of Proposition \ref{concentratingmain} are satisfied. 
\end{claim}

\begin{proof}
 Let $\widetilde A$ denote a smooth background connection which extends over $\mathcal Z_0$. We have the elliptic estimates \bea \|\ph\|_{L^{k,p}(Y)}&\leq&  C_{k,p}\left(\|\slashed D_{\widetilde A} \ph\|_{L^{k-1,p}(Y)} + \|\ph\|_{L^2}\right) \\ \|a\|_{L^{k,p}(Y)}&\leq &C_{k,p} \left( \|\bold d a\|_{L^{k-1,p}(Y)} + \|a\|_{L^2(Y)}\right) \eea
for for Sobolev norms on the closed manifold $Y$.

Next, for $n=1,...,6$, applying the above estimates to $\chi_\e^{(n)}(\ph,a)$ yields 

\bea \|\ph\|_{L^{k,p}(K_\e^{(n+1)})}&\leq&  C_{k,p}\left(\|\slashed D_{\widetilde A} \ph\|_{L^{k-1,p}(K_\e^{(n)})}+ \|\gamma(d\chi^{(n)}_\e) \ph\|_{L^{k-1,p}(K_\e^{(n)})} + \|\ph\|_{L^2(K_\e^{(n)})}\right) \\ \|a\|_{L^{k,p}(K_\e^{(n+1)})}&\leq &C_{k,p} \left( \|\bold d a\|_{L^{k-1,p}(K_\e^{(n)})} +\|\sigma(d\chi^{(n)}_\e)a\|_{L^{k-1,p}(K_\e^{(n)})} +\|a\|_{L^2(K_\e^{(n)})}\right) \eea

\noindent on the nested annuli $K_\e^{(n)}$. We will apply these estimates using the fact that $(\ph^\text{Re}, \ph^\text{Im}, a)$ solve the following non-linear equations on $\Yminus\mathcal Z_0$: 

\begin{eqnarray}
\slashed D^\text{Re}_{A_0}\ph^{\text{Re}}     \    \ \hspace{1cm} \  \ + \gamma(a)\ph^{\text{Im}} &=& 0 \\ 
\slashed D^\text{Re}_{A_0}\ph^{\text{Im}} + \gamma(a)\tfrac{\Phi_0}{\e}+\gamma(a)\ph^{\text{Re}} &=&0  \\
\bold d a  + \tfrac{\mu(\ph^\text{Im},\Phi_0)}{\e}+ \mu(\ph^{\text{Im}},\ph^{\text{Re}})&=& 0.
\label{3equationversion}
\end{eqnarray}

Now we bootstrap:  apply the elliptic estimate of $\slashed D_{\widetilde A}$ for $(k,p)=(1,\tfrac{12}{5})$ to $\chi_\e^{(2)}\ph^{\text{Im}}$: 

\begin{eqnarray*} \|\ph^{\text{Im}}\|_{L^{1,12/5}(K_\e^{(3)})} & \leq &C_{1,12/5} \Big(  \|\gamma(a)\tfrac{\Phi_0}{\e}\|_{L^{12/5}(K_\e^{(2)})}+ \|(\widetilde A-A_0)\ph^{\text{Im}}\|_{L^{12/5} (K_\e^{(2)})} +  \|\gamma(a)\ph^{\text{Re}}\|_{L^{12/5}(K_\e^{(2)})} \\ & &  + \|\gamma(d\chi_\e^{(5)})\ph^{\text{Im}}\|_{L^{12/5}(K_\e^{(2)})} + \|\ph^{\text{Im}}\|_{L^2(K_\e^{(2)})} \Big)\\
&\lesssim & \left( \e^{-1/10} \e^{-(\tfrac{3}{2}-\tfrac{3}{12/5})}+  \e^{-1/3}\e^{-1/10} + \e^{-1/10}\e^{-1/10} + \e^{-1/3}\e^{-1/10}+ \e^{-1/10}\right) 
\end{eqnarray*}

\noindent where we have used the interpolation bound from the second item of  \ref{claimA11} on the first term. For the second term, we apply H\"older's inequality with exponents $p=\tfrac{5}{2}$ and $q=\tfrac{5}{3}$ to bound this by $\|\tfrac{1}{r}\|_{L^4} \|\ph\|_{L^6}$. Since $\|\tfrac{1}{r}\|_{L^4(K_\e^{(n)})}\lesssim \e^{-1/3} $, the bound on the second term above follows again using the first item of \ref{claimA11} to bound $\|\ph\|_{L^6(K_\e^{(2)})}$. For the third term, $\gamma(a)\ph^\text{Re}$, we have simply applied Cauchy-Schwartz and the fact that $L^{24/5}\hookrightarrow L^6$ then used the first bullet point of Claim \ref{claimA11} once again. The fourth term is identical to the second since $d\chi_\e^{(5)}\sim \tfrac{1}{r}$ as well, and the $L^2$-term is much smaller. 

The exact same argument applies to $a$ using the elliptic estimate for $\bold da$, but without the $\widetilde A-A_0$ term. Together, this gives

\be
 \|(\ph^{\text{Im}},a)\|_{L^{1,12/5}(K_\e^{(3)})} \leq C \e^{-1/3-1/10}\label{bootstrap2}.  
\ee

\noindent This implies, via the Sobolev embedding $L^{1,12/5}\hookrightarrow L^{12}$ that 

\be
 \|(\ph^{\text{Im}},a)\|_{L^{12}(K_\e^{(4)})} \leq C \e^{-1/3-1/10}\label{bootstrap3}.  
\ee
as well. Indeed, applying the Sobolev inequality on $Y$ to $\chi_\e^{(3)}(\ph^\text{Im},a)$ we find  $$  \|(\ph^{\text{Im}},a)\|_{L^{12}(K_\e^{(4)})} \leq \|(\ph^{\text{Im}},a)\|_{L^{1,12/5}(K_\e^{(3)})}+ \| d\chi_\e^{(3)}(\ph^{\text{Im}},a)\|_{L^{12/5}(K_\e^{(3)})}$$
and since $|d\chi_\e^{(3)}|\lesssim \tfrac{1}{r}\lesssim \tfrac{|\Phi_0|}{\e}$ on $K_\e^{(0)}$, the second bullet point of \ref{claimA11} shows the second term is strictly smaller than the first, giving (\refeq{bootstrap3}). 

\medskip

With the above bound in hand, we now apply two final elliptic estimates to yield the two inequalities asserted. First, apply the $\slashed D_{\widetilde A}$ estimate for $(k,p)=(1,6)$ to $\chi_\e^{(4)}\ph^{\text{Re}}$. Similarly to before,  

\bea \|\ph^{\text{Re}}\|_{L^{1,6}(K_\e^{(5)})} & \leq &C_{1,6} \left( \|(\widetilde A-A_0)\ph^{\text{Re}}\|_{L^{6}(K_\e^{(4)})} + \|\gamma(a)\ph^{\text{Im}}\|_{L^6(K_\e^{(4)})} + \|\gamma(d\chi_\e^{(1)})\ph^{\text{Re}}\|_{L^{6}(K_\e^{(4)})} + \|\ph^{\text{Re}}\|_{L^2(K_\e^{(4)})} \right)\\
&\leq & C_{1,3} \Big ( \|(\widetilde A-A_0) \|_{C^0(K_\e^{(4)})} \| \ph^{\text{Re}}\|_{L^{6}(K_\e^{(4)})} +\|a \|_{L^{12}(K_\e^{(4)})} \| \ph^{\text{Im}}\|_{L^{12}(K_\e^{(4)})}  \\ & &+ \|d\chi_\e^{(n)} \|_{C^0(K_\e^{(4)})} \|\ph^{\text{Re}}\|_{L^{6}(K_\e^{(4)})}  + \|\ph^{\text{Re}}\|_{L^2(K_\e^{(4)})} \Big)  \\
&\lesssim & \left(\e^{-2/3}\e^{-1/10} + (\e^{-1/10}\e^{-1/3})^2 + \e^{-2/3}\e^{-1/10}\right) \leq C \e^{-2/3-2/10}
\eea
as both $(A-A_0)$ and $|d\chi_\e^{(n)}|$ are bounded by constant multiples $\tfrac{1}{r}$ and $r\geq c\e^{2/3}$, so $\|\tfrac{1}{r}\|_{C^0(K_\e^{(1)})}\lesssim \e^{-2/3}$. In addition, we have used (\refeq{bootstrap3}) to bound the $L^{12}$ norm. Since $\e^{-2/3-2/10}\leq \e^{-1}$, the first bound asserted in (\refeq{tobootstrap}) follows.

For the second bound, we first apply the elliptic estimate to $\chi_\e^{(5)}\ph^\text{Re}$ for $(k,p)=(1, 3+\delta)$ for $\delta<<1$ on $K_\e^{(5)}$. This shows  

\begin{eqnarray} \|\ph^\text{Re}\|_{L^{1,3+\delta}(K_\e^{(6)})} & \leq &C_{1,3+\delta} \Big( \|(\widetilde A-A_0) \ph^\text{Re}\|_{L^{3+\delta}(K_\e^{(6)})}  + \|\gamma(a)\ph^\text{Im}\|_{L^{3+\delta}(K_\e^{(6)})} \label{bootstrap31} \\ & &  \ \ \  + \ \ \  \|\gamma(d\chi_\e^{(4)})\ph^\text{Re}\|_{L^{3+\delta}(K_\e^{(4)})} + \|\ph^\text{Re}\|_{L^2} \Big)\\
&\leq & C \Big ( \| \tfrac{1}{r} \|_{L^s(K_\e^{(4)})} \| \ph^\text{Re}\|_{L^{6}(K_\e^{(4)})}\label{A15} \\ & &  +\|a \|_{L^s(K_\e^{(4)})} \| \ph^\text{Im}\|_{L^{6}(K_\e^{(4)})} + \|\ph^\text{Re}\|_{L^2}\Big ).\label{bootstrapping3}\end{eqnarray}

\noindent Here, on the first and third terms we have applied H\"older's inequality with exponents $p'$ and $q'$ defiend by $6=(3+\delta)p'$ and $q'$ the H\"older conjugate so that $\tfrac{1}{q'}=1-\tfrac{3+\delta}{6}$. Because $\delta<<1$ it follows that $p',q'$ are both very close to $2$. Thus we have $s= (3+\delta)q'= 6 + \delta'$ for $\delta'<<1$. We use that $\|\tfrac{1}{r}\|_{L^s(K_\e^{(4)})}\leq \e^{-4/9+\delta''}$ and the $L^6$ bound from Claim \ref{claimA11} to bound the first and third product. For the second term $\gamma(a)\ph^\text{Im}$, the $L^{s}$-norm of $a$ for $s=6+\delta'$ is easily bounded by the $L^{12}$ norm, hence using (\refeq{bootstrap3}) again for these terms leads to the following: 

 $$ \|\ph^\text{Re}\|_{L^{1,3+\delta}(K_\e^{(6)})} \leq C (\e^{-4/9+\delta''} \e^{(-1/10)} + \e^{-1/3-1/10}\e^{(-1/10)} \leq C\e^{-2/3}.$$

\noindent Applying the Sobolev Embedding $C^{0,\beta}\hookrightarrow L^{1,3+\delta}$ on $Y$ for some $\beta<<1$ shows that 

$$ \|\ph^\text{Re}\|_{C^0(K_\e^{(7)})} \leq  \|\ph^\text{Re}\|_{L^{1,3+\delta}(K_\e^{(6)})}  +  \| (d\chi_\e^{(6)})\ph^\text{Re}\|_{L^{3+\delta}(K_\e^{(6)})} \leq C\e^{-2/3} $$ 

\noindent since final term was already bounded by $C\e^{-2/3}$ in (\refeq{A15}) above (the different cut-off function is immaterial for this estimate). This concludes the claim.  
\end{proof}

\section{Appendix II: Kernel Asymptotics} \label{appendix3}

This appendix contains the proof of Lemma \ref{betaasymptotics} that was deferred in Section 6.6. It is re-stated here for convenience. Recall that $\beta_t$ denoted the $\widehat{H}^1_\C$-normalized element whose complex span was $\ker(\widehat{\mathcal N}_t)$.

\begin{lm}
The elements $\beta_t$ have non-vanishing leading order term so that  $$\beta_t \sim \rho_t^{-1/2} $$
for $\rho_t>\!>1$. As a consequence, we have the following bounds where the constants $C,c, \kappa_1$ are independent of $\e,t$  
\begin{enumerate}
\item  $c\e^{1/2+1/12}\leq \|\beta_t\|_{L^2(D_\lambda)} \leq C \e^{1/2+1/12}$
\item If $\rho_t>\!> 1$ is sufficiently large, $|\dot \beta_t| \leq \kappa_1 |\beta_t|$ holds pointwise. 
\item ${\|\dot \beta_t\|_{L^2(D_\lambda)}}\leq \kappa_1{\|\beta_t\|_{L^2(D_\lambda)}}$  \ and \ ${\|\dot \beta_t\|_{L^2(\del D_\lambda)}}\leq \kappa_1{\|\beta_t\|_{L^2(\del D_\lambda)}}$
\end{enumerate}
\label{betaasymptoticsappendix}  
\end{lm}

\begin{proof}
We omit the subscript $t$ on $\rho_t$ from the notation. Recall that up to an $\e$-independent normalization constant $\beta_t$ is a linear combination of  \bea \beta_1&=& \begin{pmatrix} 0 \\ e^{-H}\rho^{-1/2} \end{pmatrix} \otimes 1+ \begin{pmatrix}h_1(\rho) \alpha_1^H \\ -\overline h_1(\rho)\beta^H_1\end{pmatrix}\otimes 1+ \begin{pmatrix}h_1(\rho) \alpha_2^H \\ -\overline h_1(\rho)\beta^H_2\end{pmatrix}\otimes j\\ \beta_j&=& \begin{pmatrix} 0 \\ e^{-H}\rho^{-1/2} \end{pmatrix} \otimes j+ \begin{pmatrix}h_j(\rho) \alpha_1^H \\ -\overline h_j(\rho)\beta^H_1\end{pmatrix}\otimes 1+ \begin{pmatrix}h_j(\rho) \alpha_2^H \\ -\overline h_j(\rho)\beta^H_2\end{pmatrix}\otimes j\eea 
where $h_j(\rho)= h_j^t(\rho)$ is the $t$-parameterized family of solutions (subject to the boundary condition \refeq{doubleaps}) to $$(-\Delta -|\Phi^H|^2)h_1=\mu_\R( \beta^\circ_1, \Phi^H)$$
where $\beta_1^\circ=(0, e^{-H}\rho^{-1/2}) \otimes 1 $ is the first term above and likewise with $\otimes j$. The kernel elements are the ones satisfying $\mu_\C=0$.

 We claim now that for $\rho>>0$, $h^t_1, h^t_j$ and their $t$-derivatives take the form  take the form \begin{equation}h^t_1=p_1(t)\frac{1}{\rho e^{i\theta}}+\frac{1}{\rho^{3/2}}g_1 \hspace{1.5cm}\dot h^t_1=q_1(t)\frac{1}{\rho e^{i\theta}}+\frac{1}{\rho^{3/2}}g_2\label{hasymptotics}\end{equation}
for $\gamma>0$ and and where $p_1(t), q_1(t)$ are functions depending only on $\Phi_0$, and $\|g_1\|_{C^0}$, $\|g_2\|_{C^0}$ are bounded independent of $\e$. The same holds for $j$.  

Given this claim, the first statement of the lemma that $\beta_t$ has leading order $\rho^{-1/2}$ follows. To see this, write $\beta_t=w_1(t) \beta_1(t) + w_j(t)\beta_j (t)$ for the element whose complex span is the subspace defined by the condition $\mu_\C=0$.  Since the components of $\Phi^H$ are given by $\alpha_{1}^H= e^H c(t) \rho^{1/2}$ and $\alpha_2^H=-e^H \overline d(t)\rho^{1/2}$ we find the top two components of $\beta_t$ are 
\medskip  $$(w_1p_1 + w_j p_j)c(t)\rho^{-1/2} + O(\rho^{-1})\hspace{.3cm}\text{and}\hspace{.3cm}-(w_1p_1 + w_j p_j)) \overline d(t) \rho^{-1/2} + O (\rho^{-1}) $$ \noindent for the $\otimes 1$ and $\otimes j$ parts respectively. Since for each $t$ the condition $|c(t)|^2 + |d(t)|^2>0$ holds, the $\rho^{-1/2}$ order term vanishing would imply that $(w_1(t)p_1(t)+ w_j(t)p_j(t))=0$ for all $t$. But if this were the case, then the bottom two components are $w_1 \beta_1^\circ $ and $w_2 \beta_2^\circ$ respectively, which clearly has non-vanishing leading order, since we cannot have both $w_1=w_2=0$ else $\beta_t=0$.

\medskip

We now prove the claim that equation \ref{hasymptotics} holds. The same exact argument applies for $h_1, h_j$, so we prove it for $h_1$. Recall that by definition $h_1$ is the solution (subject to the boundary conditions \refeq{doubleaps}) of

\be (-\Delta - |\Phi^H|^2)h_1 = \mu_\R(\beta_1^\circ,\Phi^H)\label{h1defeq}\ee where the right-hand side is given by 
$$\mu_\R(\beta^\circ_1, \Phi^H)=( e^{-H}\rho^{-1/2}) \cdot \beta^H= e^{-2H}d(t)e^{-i\theta}$$
which is bounded by a constant independent of $r,\e$. We are going to construct the leading order term of $h_1$ by hand. 
To this end, take $\chi_\e$ to be a cutoff equal to 0 outside a region of radius $\rho=\rho_0 \e^{-1/6}$ for some $\rho_0$ independent of $\e$, and set \be \widetilde h_1= -\frac{d(t)e^{-i\theta}}{2 K(t)\rho} e^{-2H\chi_\e} \label{h1tildedef}.\ee It is $O(\rho^{-1})$. Ideally, we would want to take the definition $\widetilde h_1= -\mu(\beta_1^\circ,\Phi^H)/|\Phi^H|^2$, so that it would solve $(-\Delta -|\Phi^H|^2)\widetilde h_1= -\mu(\beta_1^\circ,\Phi^H) + \Delta \widetilde h_1$, but this may not satisfy the required boundary conditions, which is the reason for the cut-off. Notice however, that $|\Phi^H|^2 \sim 2K(t)\rho$ up to an exponentially small error, and $H$ is exponentially small, so our definition of $\widetilde h_1$ is exponentially close to the desired one. Moreover, when $\widetilde h_1$ is defined by $(\refeq{h1tildedef})$, it satisfies that boundary conditions (\refeq{doubleaps}) because $\widetilde h_1$ has only negative modes on the boundary, and $$\delbar \widetilde h_1=0$$ once $\chi_\e=0$ since $1/(\rho e^{i\theta})$ is holomorphic. 

 Then one has 

$$(-\Delta - |\Phi^H|^2)\widetilde h_1=  \mu_\R(\beta_1^\circ,\Phi^H) +E$$

\medskip 
\noindent where $E= \Delta \widetilde h_1+ O (e^{-\rho})=O(\rho^{-3})$. The exponentially small term arises from the difference $|\Phi^H|^2 - 2K(t)\rho$ and the difference $e^{-2H}- e^{-2H\chi_\e}$.  The true solution to (\refeq{h1defeq}) is therefore given by $h_1=\widetilde h_1 + f_1$ where $f_1$ is the unique solution subject to the boundary condition (\refeq{doubleaps}) of \be (-\Delta -|\Phi^H|^2) f_1 = E.\label{f1equation}\ee
\noindent and once $\rho>>0$ then $\widetilde h_1$ constitutes the leading order term asserted in the claim. 
\smallskip 

It remains to show that the remainder term coming from $f_1$ has  a $C^0$ bound as desired. Consider the function $\chi \rho^{3/2} f_1$ where $\chi$ is a radially symmetric cutoff vanishing at the origin and equal to $1$ outside the ball of radius $\rho=1$. Indeed, by Lemma \ref{CLemma} one has $\|f_1\|_{T\mathcal G^\C}\leq \|E\|_{L^2}\leq C$, where   applying the above operator to this, we have 

\be
(-\Delta - |\Phi^H|^2)  \chi \rho^{3/2}f_1= \chi \rho^{3/2} E -f_1\Delta (\chi \rho^{3/2})-\nabla(\chi\rho^{3/2-\gamma})\cdot \nabla f_1.
\label{rho3/2est}
\ee

\medskip 
\noindent Next, since $E=O({\rho^{-3}})$ then the first term lies in $L^2$. Since $\Delta(\chi \rho^{3/2})=O(\rho^{-1/2})$ and $\nabla (\chi\rho^{3/2})=O(\rho^{1/2})$, the second and third terms have $L^2$-norm bounded by $\|f_1\|_{T\mathcal G^\C}\leq C$. Thus the right hand side has $L^2$-norm bounded by a constant independent of $\e$. It follows from the invertibility of the operator $-\Delta -|\Phi^H|^2$ from Lemma  \ref{CLemma} (notice the multiplication by $\rho$ does not affect the boundary conditions) that \be \|\chi \rho^{3/2}f_1\|_{T\mathcal G^\C} \leq C\label{f1estimate}\ee where $C$ is still independent of $\e$. This implies in turn that  \be \|\chi \rho^{3/2}f_1\|_{C^0(D_\e)}\leq C. \label{c0embeddingbound}\ee

This follows from the Sobolev embedding. Indeed, notice that on the unit disk $D_1$, the Sobolev embedding $H^{1+\gamma}\hookrightarrow C^0$ and the Gagliardo-Nirenberg Interpolation inequality  (see equation (1.4) in \cite{NonintegerGNInterpolation} for the non-integer version of this inequality, and apply this to $|u|$).  yield $$\|u\|_{C^0(\overline{D_1})}\leq C \left(\|u\|_{L^2(D_1)}\right)^{\tfrac{1-\gamma}{2}}\left( \|u\|_{L^{2,2}(D_1)}\right)^{\tfrac{\gamma}{2}}$$

\noindent for any $\gamma>0$. Scaling functions on the  disk of size $\e^{-1/6}$ to the unit disk, the first norm $\|u\|_{L^2}$ scales like $\e^{1/6}$, while the $L^{2,2}$ norm scales like $\e^{-1/6}$, while the left-hand-side is independent of scaling. Thus for $\gamma<1/2$ one has  
$$\|u\|_{C^0(\overline{D_\e})}\leq C \left(\|u\|_{L^2(D_\e)}\right)^{\tfrac{1-\gamma}{2}}\left( \|u\|_{L^{2,2}(D_\e)}\right)^{\tfrac{\gamma}{2}}$$
and both these norms are bounded by the $T\mathcal G^\C$-norm on $D_\e$. Once \ref{c0embeddingbound} is established, we simply take $g=\rho^{3/2}f_1$ on the region where $\rho$ is large enough that $\chi=1$.

Now  the same argument may be repeated for $\dot f_t$. The difference is that we must add the term $\del_t |\Phi^H|^2f_t$  to the right hand side of \ref{f1equation} and observe that $\|\rho^{3/2}\del_t|\Phi^H|^2f_t\|_{L^2}\leq C$ by \ref{f1estimate}. Using this, and noting the asymptotics of $E$ are unaffected by differentiating with respect to $t$, it is not hard to show in analogy with \ref{rho3/2est} that 

$$(-\Delta - |\Phi^H|^2)  \chi \rho^{3/2}\dot f_1= \chi \rho^{3/2} \dot E +\chi \rho^{3/2} \del_t|\Phi^H|^2 f_1 -\dot f_1\Delta (\chi \rho^{3/2})-\nabla(\chi\rho^{3/2-\gamma})\cdot \nabla \dot f_1$$

\noindent has a right-hand side in $L^2$. It follows that $\dot f_t=\rho^{-3/2} g_2$ where $\|g_2\|_{C^0}\leq C$ as well. This concludes the bound on the smaller term asserted in (\refeq{hasymptotics}). For the first term of (\refeq{hasymptotics}) we note $$\d{}{t} p_1(t) \frac{1}{\rho e^{-i\theta}}= \dot p_1(t) \frac{1}{\rho e^{i\theta}} + p_1(t) \frac{1}{\rho^2 e^{i\theta}}\dot \rho=q_1(t)\frac{1}{\rho e^{i\theta}} $$ since $\dot \rho = \tfrac{2\dot K}{3K} \rho$. This completes the claim and thus the statements on the asymptotics of $\beta_t$. 

The bounds asserted in the lemma now follow readily. Let $\Lambda$ be a radius after which the $\rho^{-3/2}$ term is negligible, then$$\|\beta_t\|_{L^2(D_\e)}\geq c \left(\int_{\rho\geq \Lambda} |\rho^{-1/2}|^2  r dr\right)^{1/2}  \geq c \e^{2/3}\left(\int_{\Lambda}^{\e^{-1/3}} d\rho\right)^{1/2}\geq c\e^{1/2}$$ 

\noindent while the fact that $\beta_t$ is bounded over the origin yields $|\beta_t|\leq C\rho^{-1/2}$ and reversing the above inequalities show the upper bound.

The leading order term in each component consists of a product $p_i(t)\rho_t^{-1/2}$ where $p_i(t)$ depends only on $\Phi_0$ and the derivatives of its leading coefficients. Since $\d{\rho_t}{t}= \rho \tfrac{K'(t)}{K(t)}$ one has $$\d{}{t} p_i(t)\rho_t^{-1/2}= p_i'(t)\rho_{t}^{-1/2} - p_i(t)\tfrac{K'(t)}{2K}\rho_t^{-1/2}$$ so the leading term of the derivatives is bounded above by a constant times $\rho_t^{-1/2}$ as well. And for $\rho>>0$  the bound $\|g_2\|_{C^0}\leq C$ in \ref{hasymptotics} shows that the sub-leading order are negligible for the derivative as well. Thus there is a pointwise bound $$|\dot \beta_t|\leq C \rho_t^{-1/2}.$$ 
since we know $\dot \beta_t$ is smooth across the origin. This implies $$\|\dot \beta_t\|_{L^2}\lesssim \e^{1/2+1/12}\lesssim \|\beta_t\|_{L^2}, $$
and a pointwise bound $|\dot \beta_t| \leq C |\beta_t|$ once $\rho_t>>0$. The bound on the ratio of the integrals over $\del D_\e$ follows.

\end{proof}


{\small

\medskip

\bibliographystyle{amsplain}
\bibliography{Bibliography_edited}
\end{document}